\documentclass{amsbook}

\usepackage{amsmath,amsthm,amsfonts,amsbsy,amssymb,thmtools,hyperref,fullpage,bm,graphicx,comment}
\usepackage{arydshln}
\usepackage[margin=1in,footskip=0.3in]{geometry}
\usepackage{tikz}
\usepackage{comment}
\usepackage{enumitem}
\usepackage{float}
\usepackage{mathabx}
\usetikzlibrary{graphs,patterns,decorations.markings,arrows,matrix}
\usetikzlibrary{calc,decorations.pathmorphing,decorations.pathreplacing,shapes}
\newcommand{\dotr}[1]{#1^{\bullet}} 

\DeclareMathOperator{\inv}{inv}
\numberwithin{equation}{section}
\newtheorem{thm}{Theorem}[section]
\newtheorem{prop}[thm]{Proposition}
\newtheorem{lem}[thm]{Lemma}

\newtheorem{defn}[thm]{Definition}
\newtheorem{ex}[thm]{Example}
\theoremstyle{remark}
\newtheorem{rmk}[thm]{Remark}

\definecolor{myBlue}{RGB}{0,0,255}
\definecolor{myRed}{RGB}{255,0,0}
\definecolor{myGreen}{RGB}{0,255,0}
\definecolor{myYellow}{RGB}{230,230,0}
\definecolor{myCyan}{RGB}{0,255,255}
\definecolor{myMagenta}{RGB}{255,0,255}
\definecolor{myOrange}{RGB}{255,165,0}
\definecolor{myPurple}{RGB}{128,0,128}
\definecolor{myBrown}{RGB}{165,42,42}
\definecolor{myTeal}{RGB}{0,128,128}

\renewcommand{\tikz}[3]{
\begin{tikzpicture}[scale=#1,baseline=#2,>=stealth]
#3
\end{tikzpicture}}

\colorlet{lgray}{white!85!black}

\def\leq{\leqslant}
\def\geq{\geqslant}

\newcommand{\bra}[1]{\left\langle #1\right|}
\newcommand{\ket}[1]{\left|#1\right\rangle}

\numberwithin{equation}{section}
\numberwithin{section}{chapter}
\numberwithin{subsection}{section}

\renewcommand{\bar}[1]{#1^{\bullet}}

\usepackage{relsize}

\makeatletter
\newsavebox{\@brx}
\newcommand{\llangle}[1][]{\savebox{\@brx}{\(\m@th{#1\langle}\)}%
  \mathopen{\copy\@brx\kern-0.5\wd\@brx\usebox{\@brx}}}
\newcommand{\rrangle}[1][]{\savebox{\@brx}{\(\m@th{#1\rangle}\)}%
  \mathclose{\copy\@brx\kern-0.5\wd\@brx\usebox{\@brx}}}
\makeatother

\begin{document}

\title{\Huge{Rational symmetric functions from \\ the Izergin--Korepin 19-vertex model}}

\author{
\LARGE{Alexandr Garbali}
\footnote{\large{
School of Mathematics and Statistics, University of Melbourne, Parkville, Victoria, Australia. E-mail: \texttt{alexandr.garbali@unimelb.edu.au}
}}
\quad
\LARGE{Weiying Guo}
\footnote{\large{
School of Mathematics and Statistics, University of Melbourne, Parkville, Victoria, Australia. E-mail: \texttt{guwg@student.unimelb.edu.au}
}}
\quad
\LARGE{Michael Wheeler}
\footnote{\large{
School of Mathematics and Statistics, University of Melbourne, Parkville, Victoria, Australia. E-mail: \texttt{wheelerm@unimelb.edu.au}
}}
}

\maketitle

{\large{\bf Abstract.}
Starting from the Izergin--Korepin 19-vertex model in the quadrant, we introduce two families of rational multivariate functions $F_S$ and $G_S$; these are in direct analogy with functions introduced by Borodin \cite{Borodin17} in the context of the higher-spin 6-vertex model in the quadrant. 

\bigskip

\begin{enumerate}
\item 
We prove that $F_S(x_1,\dots,x_N;\bm{z})$ and $G_S(y_1,\dots,y_M;\bm{z})$ are symmetric functions in their alphabets $(x_1,\dots,x_N)$ and $(y_1,\dots,y_M)$, and pair together to yield a Cauchy identity. Both properties are consequences of the Yang--Baxter equation of the model.

\medskip
\item We show that, in an appropriate limit of the spectral parameters $\bm{z}$, $F_S$ tends to a stable symmetric function denoted $H_S$. This leads to a simplified version of the Cauchy identity with a fully factorized kernel, and suggests self-duality of the functions $H_S$.

\medskip
\item We obtain a symmetrization formula for the function $F_S(x_1,\dots,x_N;\bm{z})$, which exhibits its symmetry in $(x_1,\dots,x_N)$. In contrast to the 6-vertex model, where $F^{6{\rm V}}_S(x_1,\dots,x_N;\bm{z})$ is cast as a sum over the symmetric group $\mathfrak{S}_N$, the symmetrization formula in the 19-vertex model is over a larger set of objects that we define; we call these objects 2-permutations. 

\medskip
\item
As a byproduct of the proof of our symmetrization formula, we obtain explicit formulas for the monodromy matrix elements of the 19-vertex model in a basis that renders them totally spatially symmetric.
\end{enumerate}
}

\setcounter{tocdepth}{1}
\makeatletter
\def\l@subsection{\@tocline{2}{0pt}{2.5pc}{5pc}{}}
\makeatother
\tableofcontents

\let\cleardoublepage\clearpage

\chapter{Introduction}

\section{Background}

The study of symmetric functions via integrable lattice models has become a very well established topic over the past two decades, elegantly blending techniques from algebraic combinatorics, mathematical physics, and representation theory. Perhaps the earliest incarnation of this correspondence may be traced back to the Jacobi--Trudi formula for Schur polynomials:
\begin{align}
\label{schurJT}
s_{\lambda}(x_1,\dots,x_N)
=
\det\Big( h_{\lambda_i-i+j}(x_1,\dots,x_N) \Big),
\end{align}
where $\lambda$ is a partition and $h_{k}(x_1,\dots,x_N)$ denotes the $k$-th complete symmetric function with finite alphabet $(x_1,\dots,x_N)$ \cite[Chapter I]{Macdonald}. As is well known, $h_{k}(x_1,\dots,x_N)$ may be expressed as the partition function of a single monotone up-right path on the square lattice, such that the path takes a total of $k$ horizontal and $N$ vertical steps, and where a horizontal step receives a weight of $x_i$ when it occurs in the $i$-th row of the lattice. By application of the Lindstr\"om--Gessel--Viennot Theorem, equation \eqref{schurJT} then yields the famous non-intersecting lattice path formulation of the Schur polynomials; see, for example, \cite{Stembridge90}. Such a construction strongly hints at connections with integrable vertex models, and indeed this turns out to be the case; numerous publications have explored partition function representations for the Schur polynomials, either in terms of the five-vertex model \cite{Zinn-Justin09,BrubakerBF}, or the phase model \cite{Bogoliubov05}.

This connection extends more deeply into general symmetric polynomial theory. By now it is known that the Hall--Littlewood \cite{Tsilevich06,WheelerZinn-Justin16}, $q$--Whittaker \cite{Korff13,BorodinWheeler21}, Lascoux--Leclerc--Thibon \cite{CorteelGKM,AggarwalBW21}, Grothendieck \cite{WheelerZinn-Justin19} and Macdonald polynomials \cite{CantiniGW,GarbaliGW,GarbaliWheeler20} may all be cast in terms of partition functions of appropriate vertex models. These representations usually provide elegant proofs of key properties of the symmetric functions in question, including their symmetry, branching rules, Cauchy identities and (in some cases) monomial positivity. The vertex models at play are various reductions of the six-vertex model --- based on the quantized affine Lie algebra $\mathcal{U}_q(A^{(1)}_1)$ --- and reductions of its higher rank analogues and/or higher dimensional representations.

A more novel direction, leading to deep connections with integrable probability, was discovered in the paper of Borodin \cite{Borodin17}. Rather than attempting to identify a given partition function with a known family in the general hierarchy, the approach of \cite{Borodin17} was to invent a new family of multivariate symmetric functions by studying the generic (higher-spin) six-vertex model in the quadrant. In contrast to the known families listed above, these new objects had the distinction of depending {\it rationally} on their underlying alphabet. Using little more than the Yang--Baxter integrability of the model, Borodin was able to prove several key properties of his rational functions, the most non-trivial of which being their Cauchy identities and orthogonality with respect to a certain deformation of the Macdonald scalar product \cite[Chapter VI, Section 9]{Macdonald}. The work of \cite{Borodin17}, and those that immediately succeeded it \cite{CorwinPetrov16,BorodinPetrov16}, led to a surge of developments in integrable probability, since the (stochastic) higher-spin six-vertex model degenerates to a host of known stochastic processes including the asymmetric simple exclusion process, zero-range processes, semi-directed polymers and the Kardar--Parisi--Zhang equation. For a survey of these developments, and the extension of many of these ideas to the setting of higher-rank (or {\it coloured}) vertex models, we refer the reader to the introduction of \cite{BorodinWheeler18}.

\section{Goal of present work}
The aim of this paper is to extend the approach followed in \cite{Borodin17} to another known vertex model --- based on the {\it twisted} quantized affine Lie algebra $\mathcal{U}_q(A^{(2)}_2)$ --- thereby defining a new family of rational symmetric multivariate functions, whose properties we study. This vertex model is usually called the {\it Izergin--Korepin nineteen-vertex model} or {\it IK model} for short.

The IK model, introduced in \cite{IzerginKorepin81}, belongs to the general framework of solutions to the quantum Yang--Baxter equation \cite{Jimbo86b,Bazhanov87}. 
Since its introduction this model has been studied in the context of quantum spin chains \cite{Vichirko83,Tarasov88,Warnaar92,Pimenta14,Garbali16,Lu24}, loop models \cite{Nienhuis90a,Nienhuis90b,Feher15,Garbali17a,Garbali17b,Morin19,Morin23} and in relation to discrete holomorphicity \cite{Gier13,Ikhlef13}.

In contrast with the six-vertex model, whose underlying state space is the set $\{0,1\}$, the state space in the IK model is three-dimensional; we denote its states by elements of the set $\{0,1,2\}$. In the integrable probability community it is standard to view the state assigned to a vertex edge as labelling the number of lattice {\it paths} which are collinear with that edge; accordingly, vertices in the IK model admit either 0, 1 or 2 paths per edge.

An important feature of the stochastic six-vertex model studied in the work \cite{GwaSpohn92,BorodinCG,Borodin17} is its {\it sum-to-unity property}; this means that the row sums of the $R$-matrix of the model, $R(x;q)$, are all equal to 1. Combined with an appropriate positivity check for the vertex weights, this allows one to build a Markov process of randomly evolving paths in the quadrant; this process has several reductions to known stochastic systems \cite{BorodinPetrov16,Aggarwal17}. In the current work we find a non-standard gauge of the Izergin--Korepin model in which $R(x;q)$ also has the sum-to-unity property, and use this exclusively throughout the text. Unfortunately, as is easily checked, there is no choice of the spectral parameter $x \in \mathbb{R}$ and deformation parameter $q \in \mathbb{R}$ that renders all entries of our $R$-matrix positive; this rules out any immediate probabilistic utility of the model.

Although probabilistic applications remain out of reach, we closely mimic the setup of \cite{Borodin17} and study the IK model in the quadrant, constructing two families of multivariate rational functions as partition functions with prescribed boundary conditions; we denote them $F_S(x_1,\dots,x_N;\bm{z})$ and $G_S(y_1,\dots,y_M;\bm{z})$, where $(x_1,\dots,x_N)$ and $(y_1,\dots,y_M)$ are spectral parameters associated to the rows of the lattice, $\bm{z} = (z_1,z_2,z_3,\dots)$ are spectral parameters associated to columns, and $S$ denotes a 2-string (see Definition \ref{def:n-string-intro}, below). Using the Yang--Baxter equation and a handful of other intrinsic properties of the model, we prove the symmetry of $F_S(x_1,\dots,x_N;\bm{z})$ and $G_S(y_1,\dots,y_M;\bm{z})$ in their first alphabets, Cauchy summation identities, and a symmetrization formula for the first family. For more details on our results, we refer the reader to Sections \ref{ssec:intro-ik}--\ref{ssec:intro-twist}. 

To the best of our knowledge, this is the first paper to explore symmetric functions defined within a vertex model outside of the standard $\mathcal{U}_q(A^{(1)}_n)$ series\footnote{There should be no confusion here between the quantized affine Lie algebra that indexes the vertex model, and the root system underlying the symmetric function in question. For example, there have been some papers devoted to connections between symplectic characters/Koornwinder polynomials and vertex models; see \cite{Cantini16,Finn17}. The former objects are based on the root system $(\widecheck{C}_N,C_N)$, and yet the vertex models used to study them were all members of the $\mathcal{U}_q(A^{(1)}_n)$ series.}.

\section{Disambiguation}

The Yang--Baxter equation admits more than one nineteen-vertex solution \cite{Lima-Santos99}. The {\it Zamolodchikov--Fateev nineteen-vertex model} or {\it ZF model} is another example \cite{Zamolodchikov80}, and at the level of lattice configurations it is combinatorially identical to the IK model. The main point of difference is that the weights assigned to vertices in the ZF model can be obtained from the {\it fusion procedure} \cite{KulishRS} applied to the six-vertex model (it corresponds to taking symmetric tensor representations of weight 2 of the underlying $\mathcal{U}_q(A^{(1)}_1)$ algebra, which lifts the two-dimensional modules of the six-vertex model to three-dimensional ones), whereas the weights assigned to vertices in the IK model cannot be recovered via fusion of any simpler model.

It follows that the IK rational symmetric functions studied in this work {\it cannot} be obtained via fusion of the functions studied in \cite{Borodin17}; but their direct analogues in the ZF model can be. Although we will not pursue this point further in the present work, we remark that the ZF model served as an important experimental tool for developing formulas with a nineteen-vertex flavour starting from known six-vertex ones.

\bigskip

Let us now give a more detailed overview of our results.

\section{Izergin--Korepin model with sum-to-unity property}
\label{ssec:intro-ik}

The Izergin--Korepin $\mathcal{U}_q(A^{(2)}_2)$ vertex model is an assignment of complex-valued weights to vertices: 
\begin{align}
\label{IK-weights-intro}
W_{z/x}(i,j;k,\ell)
\longleftrightarrow
\tikz{0.7}{-0.1cm}{
\draw[lgray,line width=1.5pt,->] (-1,0) -- (1,0);
\draw[lgray,line width=1.5pt,->] (0,-1) -- (0,1);
\node[left] at (-1,0) {$j$};\node[right] at (1,0) {$\ell$};
\node[below] at (0,-1) {$i$};\node[above] at (0,1) {$k$};
\node[left] at (-1.5,0) {$x\rightarrow$};
\node[below] at (0,-1.7) {$\uparrow$};
\node[below] at (0,-2.4) {$z$};
},
\quad\forall\ i,j,k,\ell \in \{0,1,2\}.
\end{align}
Here $x$ and $z$ are variables associated to horizontal and vertical lines of the vertex, called {\it spectral parameters}; the vertex weight is a rational function in their ratio, $z/x$. $W_{z/x}(i,j;k,\ell)$ also depends rationally on $q$, the {\it quantum deformation parameter}, but we shall suppress this dependence in most of our notations. 

The model has a standard conservation property: $W_{z/x}(i,j;k,\ell) = 0$ unless $i+j = k+\ell$. Enforcing this constraint leads to nineteen possible non-zero vertex weights. In this work we present the model in a certain gauge that appears to be novel: our vertex weights satisfy the sum-to-unity property
\begin{align}
\label{sum-to-unity-intro}
\sum_{0 \leq k,\ell \leq 2}
W_{z/x}(i,j;k,\ell)
=
1,
\qquad
\forall\
i,j \in \{0,1,2\}.
\end{align}
Equation \eqref{sum-to-unity-intro} would appear to pave the way to potential probabilistic applications; however, we were not able to find any choice of the parameters $x,z,q$ such that all vertex weights $W_{z/x}(i,j;k,\ell)$ are positive. Nevertheless, we chose to adopt vertex weight conventions such that \eqref{sum-to-unity-intro} holds, since those weights have the greatest visual similarity with the conventions of \cite{BorodinWheeler18} in the stochastic six-vertex setting. For the explicit definition of our vertex weights and the Yang--Baxter equation of the model, we refer the reader to Section \ref{sec:IKmodel-relations} of the text.

\section{Symmetric functions}
\label{ssec:intro-sym}

Before defining our families of rational functions, we introduce the notion of {\it $n$-strings}, which are the combinatorial objects used to index them\footnote{
In the higher-spin six-vertex model of \cite{Borodin17}, all symmetric functions are indexed by partitions. In that work, there is a natural bijection between states and partitions, in view of the fact that the top vertical lattice edges admit any non-negative integer number of paths, leading to a vector $(m_0,m_1,m_2,\dots)$, where $m_i \in \mathbb{Z}_{\geq 0}\ \forall\ i \in \mathbb{Z}_{\geq 0}$. That vector then provides a partition $\lambda = 0^{m_0} 1^{m_1} 2^{m_2}\dots$ written in multiplicative notation.

In the present text, all modules are finite dimensional (we take fundamental representations both in the six-vertex and nineteen-vertex cases). Reading the top vertical lattice edges leads to vectors of the form $(S_1,S_2,S_3,\dots)$ with $S_i \in \{0,1\}$ in the six-vertex setting and $S_i \in \{0,1,2\}$ in the nineteen-vertex model. At least in the latter case, there is no clear bijection with partitions, and accordingly we index all of our functions by strings.}:

\begin{defn}[Definition \ref{def:n-string} in text]\label{def:n-string-intro}
An $n$-string is an infinite vector $S=(S_1,S_2,S_3,\dots)$ with $0 \leq S_i \leq n$ for all $i \in \mathbb{N}$, and obeying the finiteness property $\exists\ M \in \mathbb{N}$ such that $S_i = 0$ for all $i>M$. We write $|S| = \sum_{i \geq 1} S_i$ for the sum of components of any $n$-string and refer to this quantity as its {\it weight}. Let $\mathfrak{s}(n)$ denote the set of all $n$-strings.

Whenever $S_i=0$ for all $i>M$, we make the identification $(S_1,S_2,S_3,\dots) \equiv (S_1,\dots,S_M)$; we may do this without loss of generality, since the symmetric functions $F_S$ and $G_S$ that we define are invariant under the removal of tails of zeros from $S$. 
\end{defn}

Fix an integer $N \geq 1$ and let $S \in \mathfrak{s}(2)$ be an arbitrary 2-string such that $|S|= 2N$. Our first family of rational functions is denoted $F_S(x_1,\dots,x_N;\bm{z})$; here $(x_1,\dots,x_N)$ denotes an arbitrary alphabet of cardinality $N$ and $\bm{z} = (z_1,z_2,z_3,\dots)$ is a secondary infinite alphabet. We define them as partition functions constructed from the vertices \eqref{IK-weights-intro}:
\begin{align}
\label{F-picture-intro}
F_S(x_1,\dots,x_N;\bm{z})
:=
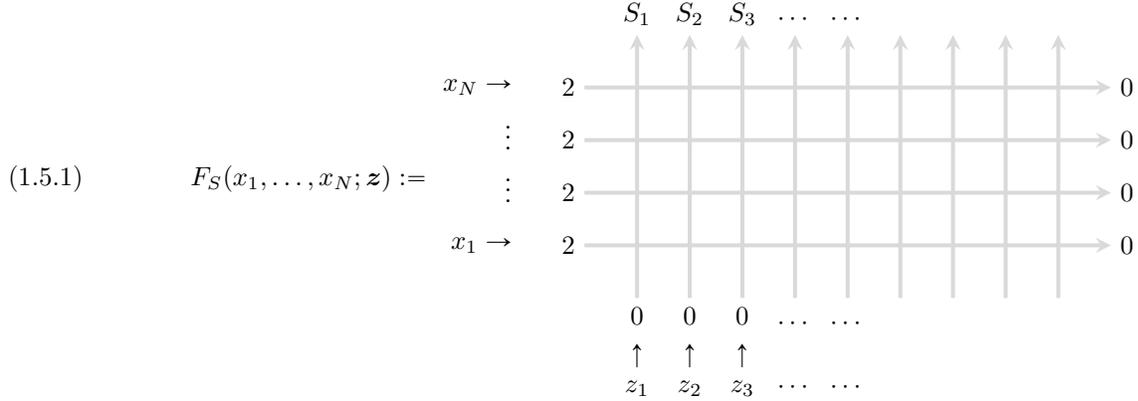
\tikz{0.7}{1.5cm}{
\foreach \y in {1,...,4}
\draw[lgray,line width=1.5pt,->] (0,\y) -- (10,\y);
\foreach \y in {1,...,4}
\node[left] at (0,\y) {2};
\foreach \y in {1,...,4}
\node[right] at (10,\y) {0};
\foreach \x in {1,...,9}
\draw[lgray,line width=1.5pt,->] (\x,0) -- (\x,5);
\node[left] at (-1.2,1) {$x_1\rightarrow$};
\node[left] at (-1.2,2.2) {$\vdots$};
\node[left] at (-1.2,3.2) {$\vdots$};
\node[left] at (-1.2,4) {$x_N\rightarrow$};
\node[above] at (1,5) {$S_1$};
\node[above] at (2,5) {$S_2$};
\node[above] at (3,5) {$S_3$};
\node[above] at (4,5) {$\cdots$};
\node[above] at (5,5) {$\cdots$};
\node[below] at (1,0) {0};
\node[below] at (2,0) {0};
\node[below] at (3,0) {0};
\node[below] at (4,-0.2) {$\cdots$};
\node[below] at (5,-0.2) {$\cdots$};
\node[below] at (1,-0.7) {$\uparrow$};
\node[below] at (2,-0.7) {$\uparrow$};
\node[below] at (3,-0.7) {$\uparrow$};
\node[below] at (1,-1.4) {$z_1$};
\node[below] at (2,-1.4) {$z_2$};
\node[below] at (3,-1.4) {$z_3$};
\node[below] at (4,-1.4) {$\cdots$};
\node[below] at (5,-1.4) {$\cdots$};
}
\end{align}
In this picture, all bottom incoming and right outgoing edges are assigned the state 0. Every left incoming edge is assigned the state 2, while top outgoing edges correspond with the entries of $S$.

The object \eqref{F-picture-intro} is to be interpreted as a partition function in the statistical mechanical sense: one works out all possible labelling of internal lattice edges and assigns a weight to each such configuration, which is the product of all individual vertices that comprise it. Summing these weights across all possible configurations yields the rational function $F_S(x_1,\dots,x_N;\bm{z})$. See the left panel of Example \ref{ex:FandG} for a sample configuration of this partition function.

Introduce another integer $M \geq 1$ and a new alphabet $(y_1,\dots,y_M)$; we take $S$ and $\bm{z}$ to be as previously. Our second family of rational functions is denoted $G_S(y_1,\dots,y_M;\bm{z})$ and defined as follows:
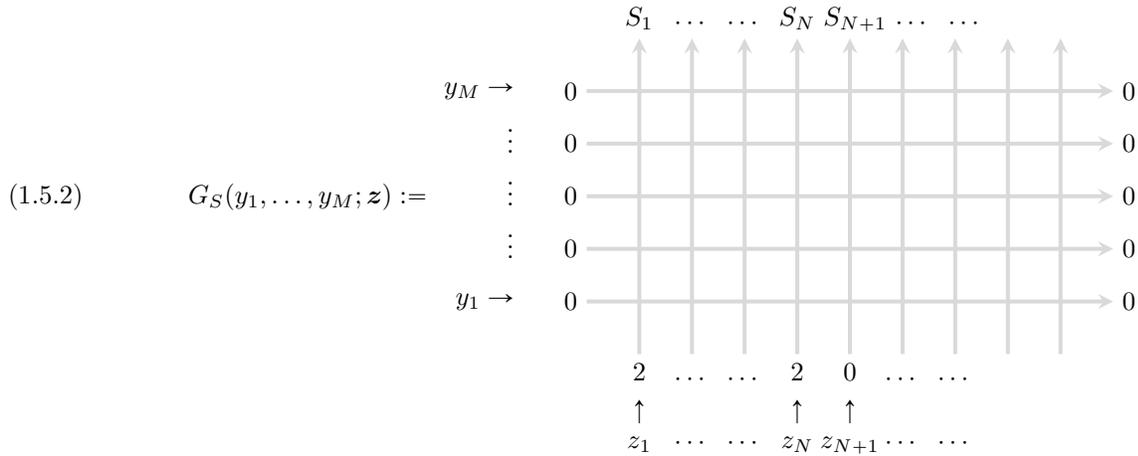
\begin{align}
\label{G-picture-intro}
G_S(y_1,\dots,y_M;\bm{z})
:=
\tikz{0.7}{2cm}{
\foreach \y in {1,...,5}
\draw[lgray,line width=1.5pt,->] (0,\y) -- (10,\y);
\foreach \x in {1,...,9}
\draw[lgray,line width=1.5pt,->] (\x,0) -- (\x,6);
\foreach \y in {1,...,5}
\node[left] at (0,\y) {0};
\foreach \y in {1,...,5}
\node[right] at (10,\y) {0};
\node[left] at (-1.2,1) {$y_1\rightarrow$};
\node[left] at (-1.2,2.2) {$\vdots$};
\node[left] at (-1.2,3.2) {$\vdots$};
\node[left] at (-1.2,4.2) {$\vdots$};
\node[left] at (-1.2,5) {$y_M\rightarrow$};
\node[above] at (1,6) {$S_1$};
\node[above] at (2,6) {$\cdots$};
\node[above] at (3,6) {$\cdots$};
\node[above] at (4,6) {$S_N$};
\node[above] at (5.1,5.95) {$S_{N+1}$};
\node[above] at (6.2,6) {$\cdots$};
\node[above] at (7.2,6) {$\cdots$};
\node[below] at (1,0) {2};
\node[below] at (2,-0.2) {$\cdots$};
\node[below] at (3,-0.2) {$\cdots$};
\node[below] at (4,0) {2};
\node[below] at (5,0) {0};
\node[below] at (6,-0.2) {$\cdots$};
\node[below] at (7,-0.2) {$\cdots$};
\node[below] at (1,-0.7) {$\uparrow$};
\node[below] at (4,-0.7) {$\uparrow$};
\node[below] at (5,-0.7) {$\uparrow$};
\node[below] at (1,-1.4) {$z_1$};
\node[below] at (2,-1.4) {$\cdots$};
\node[below] at (3,-1.4) {$\cdots$};
\node[below] at (4,-1.4) {$z_N$};
\node[below] at (5,-1.4) {$z_{N+1}$};
\node[below] at (6,-1.4) {$\cdots$};
\node[below] at (7,-1.4) {$\cdots$};
}
\end{align}
In this picture, all left incoming and right outgoing edges are assigned the state 0. The first $N$ incoming bottom edges are assigned the state 2; the remaining ones are all set to 0. As in the case of the functions \eqref{F-picture-intro}, the top outgoing edges correspond with the entries of $S$. We emphasize that in this definition, the number of horizontal rows, $M$, is independent of $N$. See the right panel of Example \ref{ex:FandG} for a sample configuration of this partition function.

\begin{ex}
\label{ex:FandG}
In this example we take $N=4$, $M=5$ and $S=(1,2,0,1,1,2,1)$. We give sample lattice configurations of the functions $F_S(x_1,x_2,x_3,x_4;\bm{z})$ (on the left) and $G_S(y_1,y_2,y_3,y_4,y_5;\bm{z})$ (on the right):
\begin{align*}
\tikz{0.7}{1.5cm}{
\foreach \y in {1,...,4}
\draw[lgray,line width=1.5pt] (0,\y) -- (8,\y);
\foreach \y in {1,...,4}
\node[left] at (0,\y) {$2$};
\foreach \y in {1,...,4}
\node[right] at (8,\y) {$0$};
\foreach \x in {1,...,7}
\draw[lgray,line width=1.5pt] (\x,0) -- (\x,5);
\node[left] at (-0.5,1) {$x_1\rightarrow$};
\node[left] at (-0.5,2) {$x_2 \rightarrow$};
\node[left] at (-0.5,3) {$x_3 \rightarrow$};
\node[left] at (-0.5,4) {$x_4 \rightarrow$};
\node[above] at (1,5) {$1$};
\node[above] at (2,5) {$2$};
\node[above] at (3,5) {$0$};
\node[above] at (4,5) {$1$};
\node[above] at (5,5) {$1$};
\node[above] at (6,5) {$2$};
\node[above] at (7,5) {$1$};
\node[below] at (1,0) {$0$};
\node[below] at (2,0) {$0$};
\node[below] at (3,0) {$0$};
\node[below] at (4,0) {$0$};
\node[below] at (5,0) {$0$};
\node[below] at (6,0) {$0$};
\node[below] at (7,0) {$0$};
\node[below] at (1,-0.7) {$\uparrow$};
\node[below] at (2,-0.7) {$\uparrow$};
\node[below] at (3,-0.7) {$\uparrow$};
\node[below] at (4,-0.7) {$\uparrow$};
\node[below] at (5,-0.7) {$\uparrow$};
\node[below] at (6,-0.7) {$\uparrow$};
\node[below] at (7,-0.7) {$\uparrow$};
\node[below] at (1,-1.4) {$z_1$};
\node[below] at (2,-1.4) {$z_2$};
\node[below] at (3,-1.4) {$z_3$};
\node[below] at (4,-1.4) {$z_4$};
\node[below] at (5,-1.4) {$z_5$};
\node[below] at (6,-1.4) {$z_6$};
\node[below] at (7,-1.4) {$z_7$};
\draw[line width = 2,->,rounded corners] (0,4.2) -- (1,4.2) -- (1,5);
\draw[line width = 2,->,rounded corners] (0,4) -- (1.9,4) -- (1.9,5);
\draw[line width = 2,->,rounded corners] (0,3.1) -- (1,3.1) -- (1,3.8) -- (2.1,3.8) -- (2.1,5);
\draw[line width = 2,->,rounded corners] (0,2.9) -- (2.9,2.9) -- (2.9,4) -- (4,4) -- (4,5);
\draw[line width = 2,->,rounded corners] (0,2.2) -- (3.1,2.2) -- (3.1,3) -- (4.8,3) -- (4.8,3.8) -- (5,4.2) -- (5,5);
\draw[line width = 2,->,rounded corners] (0,2) -- (5,2) -- (5,4) -- (5.9,4) -- (5.9,5);
\draw[line width = 2,->,rounded corners] (0,1.1) -- (4,1.1) -- (4,1.8) -- (5.2,1.8) -- (5.2,3) -- (6.1,3) -- (6.1,5);
\draw[line width = 2,->,rounded corners] (0,0.9) -- (6,0.9) -- (6,2) -- (7,2) -- (7,5);
}
\qquad
\tikz{0.7}{1.5cm}{
\foreach \y in {1,...,5}
\draw[lgray,line width=1.5pt] (0,\y) -- (8,\y);
\foreach \y in {1,...,5}
\node[left] at (0,\y) {$0$};
\foreach \y in {1,...,5}
\node[right] at (8,\y) {$0$};
\foreach \x in {1,...,7}
\draw[lgray,line width=1.5pt] (\x,0) -- (\x,6);
\node[left] at (-0.5,1) {$y_1\rightarrow$};
\node[left] at (-0.5,2) {$y_2 \rightarrow$};
\node[left] at (-0.5,3) {$y_3 \rightarrow$};
\node[left] at (-0.5,4) {$y_4 \rightarrow$};
\node[left] at (-0.5,5) {$y_5 \rightarrow$};
\node[above] at (1,6) {$1$};
\node[above] at (2,6) {$2$};
\node[above] at (3,6) {$0$};
\node[above] at (4,6) {$1$};
\node[above] at (5,6) {$1$};
\node[above] at (6,6) {$2$};
\node[above] at (7,6) {$1$};
\node[below] at (1,0) {$2$};
\node[below] at (2,0) {$2$};
\node[below] at (3,0) {$2$};
\node[below] at (4,0) {$2$};
\node[below] at (5,0) {$0$};
\node[below] at (6,0) {$0$};
\node[below] at (7,0) {$0$};
\node[below] at (1,-0.7) {$\uparrow$};
\node[below] at (2,-0.7) {$\uparrow$};
\node[below] at (3,-0.7) {$\uparrow$};
\node[below] at (4,-0.7) {$\uparrow$};
\node[below] at (5,-0.7) {$\uparrow$};
\node[below] at (6,-0.7) {$\uparrow$};
\node[below] at (7,-0.7) {$\uparrow$};
\node[below] at (1,-1.4) {$z_1$};
\node[below] at (2,-1.4) {$z_2$};
\node[below] at (3,-1.4) {$z_3$};
\node[below] at (4,-1.4) {$z_4$};
\node[below] at (5,-1.4) {$z_5$};
\node[below] at (6,-1.4) {$z_6$};
\node[below] at (7,-1.4) {$z_7$};
\draw[line width = 2,->,rounded corners] (0.9,0) -- (0.9,6);
\draw[line width = 2,->,rounded corners] (1.1,0) -- (1.1,4) -- (1.9,4) -- (1.9,6);
\draw[line width = 2,->,rounded corners] (1.9,0) -- (1.9,2.8) -- (2.1,3.2) -- (2.1,6);
\draw[line width = 2,->,rounded corners] (2.1,0) -- (2.1,3) -- (3,3) -- (3,5) -- (4,5) -- (4,6);
\draw[line width = 2,->,rounded corners] (2.9,0) -- (2.9,2.1) -- (4,2.1) -- (4,4) -- (5,4) -- (5,6);
\draw[line width = 2,->,rounded corners] (3.1,0) -- (3.1,1.9) -- (4.9,1.9) -- (4.9,3.2) -- (5.9,3.2) -- (5.9,6);
\draw[line width = 2,->,rounded corners] (3.9,0) -- (3.9,1.1) -- (5.1,1.1) -- (5.1,3) -- (6.1,3) -- (6.1,6);
\draw[line width = 2,->,rounded corners] (4.1,0) -- (4.1,0.9) -- (6,0.9) -- (6,2.8) -- (7,2.8) -- (7,6);
}
\end{align*}
\end{ex}

\begin{thm}[Theorem \ref{thm:IK-FG-sym} in text]
For any $S \in \mathfrak{s}(2)$, $F_S(x_1,\dots,x_N;\bm{z})$ and $G_S(y_1,\dots,y_M;\bm{z})$ are symmetric in their primary alphabets.
\end{thm}

This theorem is an immediate consequence of the Yang--Baxter equation of the model and the particular choice of boundary conditions that we assign to left incoming and right outgoing edges; see Section \ref{ssec:IK-symmetry} of the text.

\section{Cauchy identity}
\label{ssec:intro-cauch}

Having constructed our families of symmetric rational functions, we proceed to derive some of their properties. The first of these is their {\it Cauchy identity}:
\begin{thm}[Theorem \ref{thm:IK-cauchy} in text]
\label{thm:cauchy-FG-intro}
Fix alphabets $(x_1,\dots,x_N)$, $(y_1,\dots,y_M)$ and an infinite alphabet $\bm{z} = (z_1,z_2,z_3,\dots)$ which satisfy the constraints \eqref{eqres1-inhom}, \eqref{eqres2-inhom}. We then have the summation identity
\begin{multline}
\label{cauchy-FG-intro}
\sum_{S \in \mathfrak{s}(2)} 
c_{S}(q)
F_{S}(x_1,\dots,x_N;\bm{z})
G_{S}(y_1,\dots,y_M;q^{-3}\bm{z}^{-1})
\\
=
q^{N(2N+1)}
F_{(2^{N})}(x_1,\dots,x_N;\bm{z}) 
\prod_{i=1}^{N}
\prod_{j=1}^{M}
\frac{(1-q^2x_iy_j)(1-q^3x_iy_j)}{(1-x_iy_j)(1-qx_iy_j)},
\end{multline}
where the sum is taken over all 2-strings, the coefficient $c_S(q)$ in the summand is given by \eqref{eqcsfact1}, and where the function appearing on the right hand side is $F_T(x_1,\dots,x_N;\bm{z})$ with $T = (2^N,0,0,0,\dots)$.
\end{thm}

Equation \eqref{cauchy-FG-intro} is the first key result of this paper; its proof proceeds in several steps. We begin by noticing a certain {\it flip symmetry} of the function $G_S(y_1,\dots,y_M;\bm{z})$ that allows it to be written as a partition function similar to \eqref{G-picture-intro}, but having undergone reflection about its central horizontal line and complementation of the states that live on horizontal lattice edges. This symmetry allows one to express the left hand side of \eqref{cauchy-FG-intro} as a single partition function in the IK model; after permuting the rows of this partition function by $NM$ applications of the Yang--Baxter equation, we arrive at the right hand side of \eqref{cauchy-FG-intro}.

\section{Stable symmetric functions}
\label{ssec:intro-stab}

In the general classification of Cauchy summation identities, Theorem \ref{thm:cauchy-FG-intro} is said to be of {\it skew} type, as the function $F_{(2^N)}(x_1,\dots,x_N;\bm{z})$ resurfaces on its right hand side. Motivated by the quest for a Cauchy identity with a manifestly factorized right hand side, we were led to introduce a third family of symmetric rational functions in this work, which we denote $H_S(x_1,\dots,x_N;\bm{z})$. For their partition function definition, we refer the reader to Section \ref{sec:IK_stable} of the text. 

These functions have several important properties which we believe render them more canonical than either of the families $F_S(x_1,\dots,x_N;\bm{z})$ and $G_S(y_1,\dots,y_M;\bm{z})$. Their first property is that they are indexed by 2-strings $S \in \mathfrak{s}(2)$ of {\it arbitrary even weight} $|S| = 2K \leq 2N$, and serve as generalizations of $F_S(x_1,\dots,x_N;\bm{z})$. The latter are recovered when $S$ is chosen to have maximal weight:
\begin{prop}[Remark \ref{rmk:stable-cases} in text]
Let $S \in \mathfrak{s}(2)$ be a 2-string with $|S| = 2N$. The following equality holds:
\begin{align*}
H_S(x_1,\dots,x_N;\bm{z})
=
F_S(x_1,\dots,x_N;\bm{z}).
\end{align*}
\end{prop}
Not only are the functions $F_S(x_1,\dots,x_N;\bm{z})$ special cases of $H_S(x_1,\dots,x_N;\bm{z})$, the converse also turns out to be true:
\begin{thm}[Theorem \ref{speicalpart1} in text]
Fix two integers $0 \leq K < N$, and $S \in \mathfrak{s}(2)$ such that $|S| = 2K$. Let $\bm{u} \cup \bm{z}$ denote the extended alphabet $(u_1,\dots,u_{N-K}) \cup (z_1,z_2,z_3,\dots)$, and write $2^{N-K} \cup S = (2^{N-K},S_1,S_2,S_3,\dots)$. The following reduction holds:
\begin{align*}
\lim_{u_1 \rightarrow \infty}
\cdots
\lim_{u_{N-K} \rightarrow \infty}
\Big\{
F_{2^{N-K} \cup S}(x_1,\dots,x_N;\bm{u} \cup \bm{z})
\Big\}
=
H_S(x_1,\dots,x_N; \bm{z}).
\end{align*}
\end{thm}
Their second key property is that they are {\it stable} with respect to their primary alphabet 
$(x_1,\dots,x_N)$:
\begin{prop}[Proposition \ref{prop:IK-stab} in text]
For any 2-string $S \in \mathfrak{s}(2)$, we have
\begin{align}
\label{stability-intro}
H_S(x_1,\dots,x_{N-1},x_N=\infty;\bm{z})
=
\left\{
\begin{array}{ll}
H_S(x_1,\dots,x_{N-1};\bm{z}),
&
\quad
|S| \leq 2N-2,
\\ \\
0,
&
\quad
{\rm otherwise},
\end{array}
\right.
\end{align}
where the size of the primary alphabet on the right hand side of \eqref{stability-intro} is reduced to $N-1$.
\end{prop}
The property \eqref{stability-intro} places the functions $H_S$ on a similar footing to the stable spin Hall--Littlewood functions defined in \cite{BorodinWheeler21}. Like the latter objects, the functions $H_S$ pair together in their own Cauchy identity; this is their third key property: 
\begin{thm}[Theorem \ref{thm:IK-stable-cauchy} in text]
Fix alphabets $(x_1,\dots,x_N)$, $(y_1,\dots,y_M)$ and an infinite alphabet $\bm{z} = (z_1,z_2,z_3,\dots)$ which satisfy the constraints \eqref{eqres1-inhom}, \eqref{eqres2-inhom}. We then have the summation identity
\begin{align}
\label{cauchy-HH-intro}
\sum_{S \in \mathfrak{s}(2)}
(-1)^{|S|/2}
q^{4MN-|S|^2}
c_S(q)
H_{S}(x_1,\dots,x_N;\bm{z})
H_{S}(y_1,\dots,y_M;q^{-3}\bm{z}^{-1}) 
=
\prod_{i=1}^{N}
\prod_{j=1}^{M}
\frac{(1-q^2x_iy_j)(1-q^3x_iy_j)}{(1-x_iy_j)(1-qx_iy_j)}, 
\end{align}
where the sum is taken over all 2-strings and the coefficient $c_S(q)$ in the summand is given by \eqref{eqcsfact1}.
\end{thm}

Equation \eqref{cauchy-HH-intro} achieves our aim of writing a Cauchy summation identity with a factorized right hand side; its proof follows by taking a relatively routine limit of \eqref{cauchy-FG-intro}.

\section{Symmetrization identity}
\label{ssec:intro-symmetrization}

Although the partition function \eqref{F-picture-intro} provides an explicit construction of $F_S(x_1,\dots,x_N;\bm{z})$, it is far from being a closed formula. Motivated by the goal of finding a {\it symmetrization identity} for the functions $F_S(x_1,\dots,x_N;\bm{z})$, we were led to the following definition and theorem.

\begin{defn}[Definition \ref{defn:M-set} in text]\label{defn:M-set-intro}
Fix two integers $n,N \geq 1$ and let $S \in \mathfrak{s}(n)$ be an $n$-string. An $n$-permutation matrix of size $N$ and profile $S$ is an $N \times \infty$ matrix with all row sums equal to $n$, and $i$-th column sum equal to $S_i$ for all $i \in \mathbb{N}$. We let $\mathfrak{M}_n(N,S)$ denote the set of all such matrices\footnote{It is a potentially interesting problem to compute the cardinality of $\mathfrak{M}_n(N,S)$, although we will not report on this in the current text.}.
\end{defn} 

\begin{rmk}
Note that $\mathfrak{M}_1(N,1^N)$ is just the standard $N$-dimensional representation of $\mathfrak{S}_N$.
\end{rmk}

\begin{ex}
Let $n=N=2$ and $S=(1,0,2,1)$, where we truncate $S$ by deleting its infinite tail of zeros. The set $\mathfrak{M}_2(2,S)$ consists of 4 matrices:
\begin{align*}
\mathfrak{M}_2(2,S)
=
\left\{
\begin{pmatrix}
0 & 0 & 2 & 0
\\
1 & 0 & 0 & 1
\end{pmatrix},
\begin{pmatrix}
1 & 0 & 0 & 1
\\
0 & 0 & 2 & 0
\end{pmatrix},
\begin{pmatrix}
1 & 0 & 1 & 0
\\
0 & 0 & 1 & 1
\end{pmatrix},
\begin{pmatrix}
0 & 0 & 1 & 1
\\
1 & 0 & 1 & 0
\end{pmatrix}
\right\}.
\end{align*}
\end{ex}

\begin{thm}[Theorem \ref{thm:symmetriz-IK} in text]\label{thm:sym-intro}
Fix an integer $N \geq 1$ and let $S \in \mathfrak{s}(2)$ be any 2-string such that $|S|=2N$. We then have
\begin{align}
\label{sym-IK-intro}
F_S(x_1,\dots,x_N;\bm{z})
=
\sum_{\sigma \in \mathfrak{M}_2(N,S)}
\prod_{1 \leq i<j \leq N} \Delta_{\sigma(i),\sigma(j)}(x_i,x_j;\bm{z})
\prod_{i=1}^{N}
F_{\sigma(i)}(x_i;\bm{z}),
\end{align}
where $\sigma(i)$ denotes the $i$-th row of the 2-permutation matrix $\sigma$, $F_{\sigma(i)}(x_i;\bm{z})$ is a single-row ($N=1$) function \eqref{F-picture-intro}, and
$\Delta_{U,V}(x,y;\bm{z}) : \mathfrak{s}(2) \times \mathfrak{s}(2) \rightarrow \mathbb{Q}(x,y;\bm{z})$ is an explicit rational function given by 
\begin{align*}
\Delta_{U,V}(x,y;\bm{z}) = 
\begin{cases} 
\smallskip
\dfrac{(y - q^2x)(y - q^3x)}{(y - x)(y - q x)}, 
&
\begin{tabular}{c||c|c|c}
    U & 2 & \\ \hline
    V & & 2 &
\end{tabular},\ \
\begin{tabular}{c||c|c|c}
    U & 2 &  &\\ \hline
    V & & 1& 1
\end{tabular},\ \
\begin{tabular}{c||c|c|c}
    U & 1 & 1 &\\ \hline
    V & & & 2
\end{tabular},\ \
\begin{tabular}{c||c|c|c|c}
    U & 1 & 1 & & \\ \hline
    V &  &  & 1 & 1  
\end{tabular}
\\
\smallskip
\dfrac{(y - q^2x)(q^2y - x)}{(y - x)^2}, 
&
\begin{tabular}{c||c|c|c}
    U & 1 & & 1 \\ \hline
    V & & 2 & 
\end{tabular},\ \
\begin{tabular}{c||c|c|c|c}
    U & 1 & & & 1 \\ \hline
    V & & 1 & 1 & 
\end{tabular}
\\
\smallskip
\dfrac{(y-q^2x)^2 (x-qy)}
{(y - x)^2 (y-q x)},
&
\begin{tabular}{c||c|c|c|c}
    U & 1 &  & 1 & \\ \hline
    V &  & 1 & & 1 
\end{tabular}
\end{cases}
\end{align*}
\begin{align*}
\Delta_{U,V}(x,y;\bm{z}) = \begin{cases}
\dfrac{(q^2x-y)(x y+q^2x y-q^3y z_k-q^3x z_k)}{q(1-q)(y-x)(y-qx)z_k}, \quad  &\begin{tabular}{c||c|c|c}
    U & 1 & 1 &  \\ \hline
    V & 1 &  & 1 \\ \hdashline[1pt/1pt]
    & k & &
\end{tabular}
\medskip
\\
\dfrac{(y-q^2x)(xy+q^2xy-q^2yz_k-q^3xz_k)}{q(1-q)(y-x)^2z_k}, 
&\begin{tabular}{c||c|c|c}
    U & 1 & 1 &  \\ \hline
    V &  & 1 & 1 \\ \hdashline[1pt/1pt]
    &  & k & 
\end{tabular}
\medskip
\\
\dfrac{(q^2x-y)(xy+q^2xy-q^2yz_k-q^2xz_k)}{(1-q)(y-x)(y-qx)z_k},
&\begin{tabular}{c||c|c|c}
    U & 1 &  & 1 \\ \hline
    V &  & 1 & 1 \\ \hdashline[1pt/1pt]
    &  & & k
\end{tabular}
\end{cases}
\end{align*}

\begin{align*}
\Delta_{U,V}(x,y;\bm{z}) = 
\frac{(xy+ q^2xy-q^3yz_k-q^3xz_k)(xy+q^2xy-q^2yz_{\ell}-q^2xz_{\ell})}{q(1-q)^2(y-qx)(x-qy)z_kz_{\ell}}, \quad\quad 
\begin{tabular}{c||c|c}
    U & 1 & 1  \\ \hline
    V & 1 & 1 \\ \hdashline[1pt/1pt]
    & k & $\ell$ 
\end{tabular}
\end{align*}
where any case of $\Delta_{U,V}(x,y;\bm{z})$ not listed above is determined by imposing the symmetry
\begin{align}
\label{eq:delta-sym-intro}
\Delta_{U,V}(x,y;\bm{z})
=
\Delta_{V,U}(y,x;\bm{z}).
\end{align}
Here we have made use of the following shorthand, for any two rows $U=(U_1,U_2,\dots)$, $V=(V_1,V_2,\dots)$ of $\sigma$ and $a,b \in \{1,2\}$:
\begin{align*}
&
\begin{tabular}{c||c|c|c}
    U & a & &\\ \hline
    V & & b &
\end{tabular}
\qquad
\text{ indicates that } a \text{ in } U \text{ occurs before } b \text{ in } V,
\\
&
\begin{tabular}{c||c|c|c}
    U & & a & \\ \hline
    V & b & &
\end{tabular}
\qquad
\text{ indicates that } a \text{ in } U \text{ occurs after } b \text{ in } V,
\\
&
\begin{tabular}{c||c|c|c}
    U & a &\phantom{b} &\\ \hline
    V & b & &
\end{tabular}
\qquad
\text{ indicates that } a \text{ in } U \text{ occurs at the same position as } b \text{ in } V,
\\
&
\begin{tabular}{c||c|c|c}
     U & a & & \\ \hline
     V &  & b &\\
     \hdashline[1pt/1pt]
    & k & $\ell$ &
\end{tabular}
\qquad
\text{ indicates that } U_k=a \text{ and } V_{\ell}=b, \text{ with } 1 \leq k < \ell.
\end{align*}
\end{thm}

\begin{rmk}
In view of the symmetry property \eqref{eq:delta-sym-intro} of the bivariate factors appearing in \eqref{sym-IK-intro}, it is easy to see that $F_S(x_1,\dots,x_N;\bm{z})$ is invariant under the pairwise swaps $x_i \leftrightarrow x_j$. Hence, \eqref{sym-IK-intro} manifestly exhibits the symmetry of $F_S(x_1,\dots,x_N;\bm{z})$ in its first alphabet.
\end{rmk}

Equation \eqref{sym-IK-intro} may appear somewhat unusual compared with traditional symmetrization identities, however it is quite easy to cast the corresponding result for the six-vertex model in a completely analogous way:
\begin{thm}[Theorem \ref{thm:F6sym} in text]
Fix an integer $N \geq 1$ and let $S \in \mathfrak{s}(1)$ be any 1-string such that $|S| = N$. There holds
\begin{align}
\label{6v-sym-intro}
F^{6{\rm V}}_S(x_1,\dots,x_N;\bm{z})
=
\sum_{\sigma \in \mathfrak{M}_1(N,S)}
\prod_{1 \leq i<j \leq N} \Delta^{6{\rm V}}_{\sigma(i),\sigma(j)}(x_i,x_j)
\prod_{i=1}^{N}
F^{6{\rm V}}_{\sigma(i)}(x_i;\bm{z}),
\end{align}
where $\sigma(i)$ denotes the $i$-th row of the 1-permutation matrix $\sigma$, and where we have defined
\begin{align*}
\Delta^{6{\rm V}}_{U,V}(x,y)
=
\left\{
\begin{array}{ll}
\dfrac{y-qx}{y-x},
&
\qquad
\begin{tabular}{c||c|c|c}
    U & 1 & \\ \hline
    V & & 1 &
\end{tabular}
\quad
\text{(1 in $U$ occurs before 1 in $V$)},
\\ \\
\dfrac{x-qy}{x-y},
&
\qquad
\begin{tabular}{c||c|c|c}
    U &  & 1 \\ \hline
    V & 1 &  &
\end{tabular}
\quad
\text{(1 in $U$ occurs after 1 in $V$)},
\end{array}
\right.
\end{align*}
for any $(U,V) \in \mathfrak{s}(1) \times \mathfrak{s}(1)$ such that $|U| = |V| = 1$.
\end{thm}

Equation \eqref{sym-IK-intro} is the second key result of this paper; it was also by far the most difficult to formulate and prove. The main challenge in generalizing the six-vertex result \eqref{6v-sym-intro} to the IK case is that it is not {\it a priori}\/ clear over what set of objects one should symmetrize, such that summands remain factorized. We initially arrived at a similar formula to \eqref{sym-IK-intro} in the ZF model, by performing an appropriate fusion of the six-vertex formula \eqref{6v-sym-intro}. From there, aided by computer experiments, we were able to conjecture \eqref{sym-IK-intro} in the IK setting. Even with an explicit conjecture in clear view, the proof of \eqref{sym-IK-intro} remained a formidable challenge, since many of the techniques that expedite the proof of the six-vertex result \eqref{6v-sym-intro} break down at the level of the IK model. The main challenge that one faces is that standard Lagrange interpolation techniques are more involved, since partition functions in the IK model (when viewed as polynomials in their spectral parameters) have double the degree of their six-vertex counterparts; as such, one requires twice the number of interpolating points.

Eventually we were able to find a set of Lagrange interpolation based properties for the functions $F_S(x_1,\dots,x_N;\bm{z})$, that uniquely characterize them. In order to show that the formula \eqref{sym-IK-intro} satisfies these properties, it was necessary to find a subtle algebraic reformulation of it in terms of a family of operators called {\it twisted columns}, as we describe in the next subsection.
 
\section{Twisted columns}
\label{ssec:intro-twist}

Our key algebraic tool in the proof of Theorem \ref{thm:sym-intro} is the introduction of a family of three operators $\Gamma_0(z),\Gamma_1(z),\Gamma_2(z) \in {\rm End}(V_1 \otimes \cdots \otimes V_N)$ which live in an $N$-fold tensor product of the three-dimensional space $V \cong \mathbb{C}^3$. These operators depend implicitly on the full alphabet $(x_1,\dots,x_N)$ and the parameter $q$, as well as a single variable $z$; for notational compactness we shall only show dependence on the latter. 

The simplest of these operators is $\Gamma_0(z)$:
\begin{align*}
\Gamma_0(z)
=
\bigotimes_{i=1}^{N}
\begin{pmatrix}
1 & 0 & 0
\\
0 & \dfrac{1-z/x_i}{1-q^2 z/x_i} & 0
\\
0 & 0 & \dfrac{(1-z/x_i)(1-qz/x_i)}{(1-q^2 z/x_i)(1-q^3 z/x_i)}
\end{pmatrix}_i
\in {\rm End}(V_1 \otimes \cdots \otimes V_N).
\end{align*}
The remaining operators $\Gamma_1(z)$ and $\Gamma_2(z)$ are more complicated, but nevertheless fully explicit; see Section \ref{ssec:IK-twist} of the text. Our main result concerning these operators is the following:
\begin{thm}[Proposition \ref{prop:gam-sym} and Theorem \ref{thm:IK-YB-twist} in text]
\label{thm:yb-alg-intro}
For any $k \in \{0,1,2\}$ and $1 \leq i,j \leq N$, the operator $\Gamma_k(z)$ is invariant under simultaneous permutation of the spaces $V_i \otimes V_j$ in which it acts and the variables $(x_i,x_j)$ upon which it depends. Further, $\langle \Gamma_0(z),\Gamma_1(z),\Gamma_2(z) \rangle$ forms a subalgebra of the Yang--Baxter algebra for the Izergin--Korepin model.
\end{thm}

The above result states the operators $\Gamma_0(z),\Gamma_1(z),\Gamma_2(z)$ satisfy the same exchange relations as the column-to-column operators used to build the partition function \eqref{F-picture-intro}, and are totally spatially symmetric (unlike the original operators). It is therefore reasonable to expect that $\Gamma_0(z),\Gamma_1(z),\Gamma_2(z)$ constitute {\it Drinfeld twists} of our starting column-to-column operators; this is a highly non-trivial result, since to the best of our knowledge very little is known about factorizing twists in the $U_q(A^{(2)}_2)$ setting.

Using Theorem \ref{thm:yb-alg-intro}, combined with the verification of appropriate (Lagrange interpolation) properties, we are finally able to show that
\begin{align}
\label{twist-F-intro}
F_S(x_1,\dots,x_N;\bm{z})
=
\llangle 2^N\| 
\Gamma_{S_1}(z_1) \Gamma_{S_2}(z_2) \Gamma_{S_3}(z_3) 
\cdots
\|0^N \rrangle,
\end{align}
where $\|0^N\rrangle \in V_1 \otimes \cdots \otimes V_N$ and $\llangle 2^N \| \in V_1^{*} \otimes \cdots \otimes V_N^{*}$ denote highest (and lowest) weight vectors in the space in which the operators act (and its dual). This allows us to complete the proof of Theorem \ref{thm:sym-intro}, since the right hand side of \eqref{twist-F-intro} may be readily expanded to yield the right hand side of \eqref{sym-IK-intro}.

\section{Future directions}

The work performed in this paper suggests a number of potential continuations; we elaborate briefly on some of these below.

\begin{itemize}[wide,labelindent=0pt]

\item {\it Orthogonality.} One of the most important properties of the rational symmetric functions constructed in \cite{Borodin17} is their orthogonality with respect to a certain integral inner product; this result may in fact be traced back to the even earlier work \cite{BorodinCPSa,BorodinCPSb}. The orthogonality relations of the functions, together with their Cauchy identity, allow for the construction of a generalized Fourier theory (complete with a forwards and backwards transform that compose as the identity). Notably, the proof of this orthogonality property seems to depart entirely from Yang--Baxter calculus, and rather proceeds by direct formulas that lay bare the pole structure of the participating functions. 

It would be rewarding to extend as much of this framework as possible to the symmetric functions introduced in our text. Such an endeavour is likely to be extremely challenging, and would require a detailed understanding of the pole structure in the symmetrization identity \eqref{sym-IK-intro} for $F_S(x_1,\dots,x_N;\bm{z})$.

\medskip

\item {\it Fusion.} Arguably the most important potential follow-up to this work would be to perform fusion of the symmetric functions in question. This is a subject that is very well-developed in context of the $\mathcal{U}_q(A^{(1)}_n)$ series; explicit formulas for the fused $R$-matrices acting in symmetric tensor representations were obtained in \cite{Mangazeev14,BosnjakMangazeev16}, and the procedure of {\it stochastic fusion} was developed in the works \cite{CorwinPetrov16,Borodin17,BorodinWheeler18}. This combinatorial understanding of fusion, together with explicit formulas for the weights of the fused models, allows one to readily extend results from the six-vertex model to its higher-dimensional representations. In fact, the rational symmetric functions introduced in \cite{Borodin17} were {\it a priori} fused in one direction (vertical lines corresponded to $\mathfrak{sl}(2)$ Verma modules).

Aside from leading to a richer class of multivariate rational functions, fusion may potentially allow one to access symmetric {\it polynomials} as certain special limits. For example, as is well-known from the work of \cite{Borodin17}, in the limit where the weight of the underlying symmetric tensor representation tends to infinity, one recovers the classical Hall--Littlewood polynomials of type $A_N$. It would be interesting to investigate the same limit in the context of the $\mathcal{U}_q(A^{(2)}_2)$ model.

\medskip

\item {\it Higher-rank analogues.} The IK model is just the first model in the $\mathcal{U}_q(A^{(2)}_{2n})$ series. The $R$-matrices of this family all have a common structure; see \cite{Jimbo86b}. 

It would be natural to use these higher-rank analogues of the IK model to define multivariate rational functions from ensembles of {\it coloured} lattice paths, similarly to what was done in \cite{BorodinWheeler18} in the context of $\mathcal{U}_q(A^{(1)}_n)$ models. It is expected that the resulting functions would be {\it nonsymmetric}, similarly to how the partition functions defined in \cite{BorodinWheeler18} gave a nonsymmetric decomposition of the original six-vertex functions introduced in \cite{Borodin17}. Investigations along these lines would also potentially shed light on the orthogonality direction mentioned above.    
\end{itemize}

\section{About the layout of this paper}

Our text is divided into two main chapters. Chapter \ref{chapter:6-vertex} reviews the construction of rational symmetric functions via partition functions in the six-vertex model, while Chapter \ref{chapter:19-vertex} contains our new results on the Izergin--Korepin nineteen-vertex model.

Most of the material in Chapter \ref{chapter:6-vertex} has appeared in previous works. The definition of the rational symmetric functions $F^{\rm 6V}_S$ and $G^{\rm 6V}_S$ in Section \ref{sec:6v-functions} is inherited from the {\it spin Hall--Littlewood functions} of \cite{Borodin17} and \cite{BorodinPetrov16}: in particular, one takes a simple limit of the higher-spin vertex model in the latter papers which maps it to the stochastic six-vertex model. The Cauchy identity listed in Section \ref{sec:6cauchy} is, in turn, a special case of the skew Cauchy identities of \cite{Borodin17} and \cite{BorodinPetrov16}. The stable symmetric functions $H^{\rm 6V}_S$ and their Cauchy identity, given in Sections \ref{sec:6_stable}--\ref{ssec:6v-stable-cauchy}, are closely inspired by the stable spin Hall--Littlewood functions of \cite[Section 5.3]{BorodinWheeler21}. Finally the symmetrization formula for $F^{\rm 6V}_S$ in Section \ref{sec:6-sym} is deduced as a special case of the corresponding symmetrization identity in \cite{Borodin17} and \cite{BorodinPetrov16}, although our proof (based on exchange relations and Lagrange interpolation) appears to be new.

Even though Chapter \ref{chapter:6-vertex} contains no substantive new material, we have chosen to present it as a warm-up to Chapter \ref{chapter:19-vertex} for several reasons. The first reason is that it sets many of our notations: almost every quantity introduced in Chapter \ref{chapter:6-vertex} has a direct analogue in Chapter \ref{chapter:19-vertex}, and we employ common notations for these. The second reason is historical: in our investigations of the Izergin--Korepin model, the passage to many of our results was not straightforward. In several cases, it was necessary to proceed by first understanding the corresponding calculation in the six-vertex model, and we have tried to reflect this lineage in the text itself. The final reason is for the convenience of the reader: more experienced readers may prefer to read Chapters \ref{chapter:6-vertex} and \ref{chapter:19-vertex} in parallel, as this will immediately illustrate what aspects of the constructions are different or more complicated within the Izergin--Korepin model. Less experienced readers may prefer to read the two parts independently: both Chapters \ref{chapter:6-vertex} and \ref{chapter:19-vertex} are entirely self-contained, and function as standalone texts in their own right.

\section{Acknowledgments}

We thank the Australian Research Council (ARC) for generous support. AG was partially supported by the ARC DECRA DE210101264. MW was partially supported by the ARC Future Fellowship  FT200100981. Both AG and MW were partially supported by the ARC Discovery Projects DP190102897 and DP240101787. WG was supported by an Australian Government Research Training Program Scholarship.

\chapter{Six-vertex model}
\label{chapter:6-vertex}

In this chapter we study the stochastic six-vertex model in the quadrant, and review how it gives rise to a pair of rational symmetric functions $F_S(x_1,\dots,x_N;\bm{z})$ and $G_S(y_1,\dots,y_M;\bm{z})$ with a host of important algebraic properties.

Section \ref{sec:6V-relations} gives the definition of the stochastic six-vertex model and its intrinsic properties; we shall follow the same conventions as in \cite{BorodinWheeler18}. In Section \ref{sec:6row} we use the model to define row-to-row operators, and derive some of their properties. Sections \ref{sec:6v-functions}--\ref{ssec:6v-stable-cauchy} contain the algebraic construction of the functions $F_S(x_1,\dots,x_N;\bm{z})$ and $G_S(y_1,\dots,y_M;\bm{z})$, their stable analogues $H_S(x_1,\dots,x_N;\bm{z})$, as well as the Cauchy identities that they satisfy. 

Finally, in Section \ref{sec:6-sym} we state and prove a symmetrization identity for $F_S(x_1,\dots,x_N;\bm{z})$. Although this result is not new, it is worthwhile for the reader to study this proof in some detail, as the method adopted there is ultimately employed to prove the corresponding symmetrization identity in the more complicated Izergin--Korepin model.

As we have mentioned, almost all of the results in this chapter are known (or may be deduced) from previous works, and we refer the reader to \cite{Borodin17,BorodinPetrov15,BorodinPetrov16,BorodinWheeler21} for much of the original theory.

\section{Definition of stochastic six-vertex model and basic relations}
\label{sec:6V-relations}

We begin by writing down the $R$-matrix and explain how it gives rise to a vertex model, in Section \ref{ssec:6R}. Section \ref{ssec:6YB} contains the unitarity and Yang--Baxter relations of the model, and Section \ref{ssec:6dot} introduces dotted analogues of the vertex weights; the latter differ from the usual vertex weights up to an overall normalization factor.

\subsection{$R$-matrix and vertices}
\label{ssec:6R}

\begin{defn}
The $R$-matrix for the stochastic six-vertex model is a $4 \times 4$ matrix given by
\begin{equation}
\label{eqrm1}
R_{ab}(y/x) =\sum_{i,j,k,\ell \in \{0,1\}} W_{y/x}(i,j;k,\ell) E_a^{j,\ell} \otimes E_b^{i,k}
\end{equation}
where $E_{a}^{j,\ell}$ and $E_{b}^{i,k}$ are $2 \times 2$ elementary matrices acting on the two-dimensional spaces $V_{a}\cong \mathbb{C}^2$ and $V_{b}\cong \mathbb{C}^2$, respectively. The matrix entries satisfy the conservation property
\begin{align}
\label{conserve-6}
W_{y/x}(i,j;k,\ell)
=
0,
\qquad
\text{if}
\qquad
i+j \not= k+\ell,
\end{align}
which leaves six possible non-vanishing choices of $(i,j;k,\ell)$. Writing the matrix \eqref{eqrm1} explicitly, one then finds that
\begin{equation}
\label{eqrm61}
R_{ab}(y/x) = 
\begin{bmatrix}
  \begin{array}{cccc}
    W_{y/x}(0,0;0,0) & 0 & 0 & 0 \\
    0 & W_{y/x}(1,0;1,0) & W_{y/x}(1,0;0,1) & 0 \\
   0 & W_{y/x}(0,1;1,0) & W_{y/x}(0,1;0,1) & 0 \\
    0 & 0 & 0 & W_{y/x}(1,1;1,1) \\
  \end{array}
\end{bmatrix}_{ab}
\end{equation}
where the subscript $[\cdot]_{ab}$ is used to indicate the action of $R_{ab}(y/x)$ in ${\rm End}(V_a \otimes V_b)$. 

\smallskip
The non-zero matrix entries of \eqref{eqrm61} are identified with vertices as follows:
\begin{equation}
\label{eqdefn6w}
\raisebox{-22mm}{
\begin{tikzpicture}[scale=0.8,baseline = {(0,-2.25)}]
\draw[lgray,line width=1.5pt,->] (14,0) -- (16,0);
\draw[lgray,line width=1.5pt,->] (15,-1) -- (15,1); 
\node at (13.8,0) {$j$};
\node at (16.2,0) {$\ell$};
\node at (15,-1.2) (A1) {$i$};
\node at (15,1.2) {$k$};
\node at (12.7,0) {$x$};
\node at (15,-2.2) (A) {$y$};
\draw[->,thick] (13,0) -- (13.5,0);
\draw[->,thick] (A) -- (A1);
\node at (10.75,0) {$W_{y/x}(i,j;k,\ell)=$};
\node at (19,0) {$\qquad \text{ for }\quad  i,j,k,\ell \in \{0,1\},$}; 
\end{tikzpicture}}
\end{equation}

and we tabulate the six possibilities in Figure \ref{fig:6v-weights}.
\end{defn}

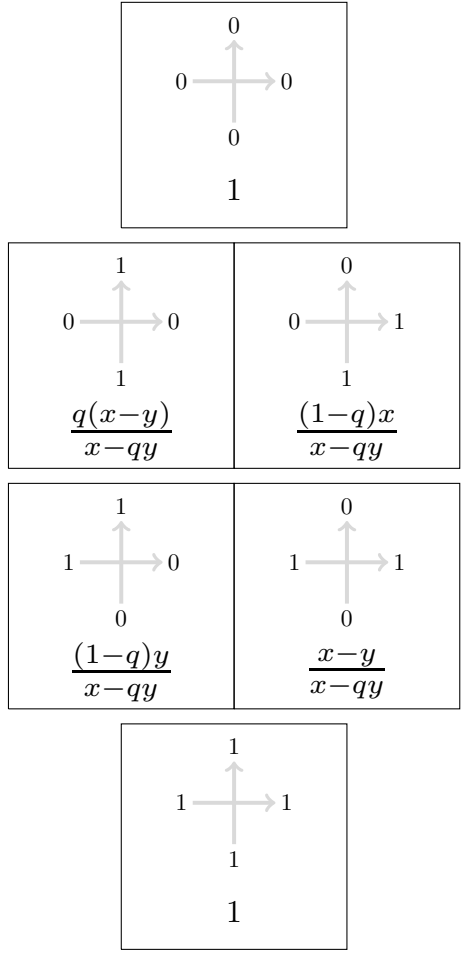
\begin{figure}
\begin{tikzpicture}
	\draw[scale = 1.5] (0,0) -- (2,0) -- (2,2) -- (0,2) -- (0,0);
	\begin{scope}[shift = {(0,0.2)}]
	\draw[lgray,line width=1.5pt,->] (0.95,1.75) -- (2.05,1.75);
	\draw[lgray,line width=1.5pt,->] (1.5,1.2) -- (1.5,2.3);
	\end{scope}
	\node[scale = 0.9] at (1.5,1.2) {$0$};
	\node[scale = 0.9] at (0.8,1.95) {$0$};
	\node[scale = 0.9] at (2.2,1.95) {$0$};
	\node[scale = 0.9] at (1.5,2.7) {$0$};
	\node[scale = 1.2] at (1.5,0.5) {$1$};
\begin{scope}[shift = {(-1.5,-3.2)}]
	\draw[scale = 1.5] (0,0) -- (2,0) -- (2,2) -- (0,2) -- (0,0);
	\begin{scope}[shift = {(0,0.2)}]
	\draw[lgray,line width=1.5pt,->] (0.95,1.75) -- (2.05,1.75);
	\draw[lgray,line width=1.5pt,->] (1.5,1.2) -- (1.5,2.3);
	\end{scope}
	\node[scale = 0.9] at (1.5,1.2) {$1$};
	\node[scale = 0.9] at (0.8,1.95) {$0$};
	\node[scale = 0.9] at (2.2,1.95) {$0$};
	\node[scale = 0.9] at (1.5,2.7) {$1$};
	\node[scale = 1.5] at (1.5,0.5) {$\frac{q(x-y)}{x-q y}$};
\end{scope}
\begin{scope}[shift = {(1.5,-3.2)}]
	\draw[scale = 1.5] (0,0) -- (2,0) -- (2,2) -- (0,2) -- (0,0);
	\begin{scope}[shift = {(0,0.2)}]
	\draw[lgray,line width=1.5pt,->] (0.95,1.75) -- (2.05,1.75);
	\draw[lgray,line width=1.5pt,->] (1.5,1.2) -- (1.5,2.3);
	\end{scope}
	\node[scale = 0.9] at (1.5,1.2) {$1$};
	\node[scale = 0.9] at (0.8,1.95) {$0$};
	\node[scale = 0.9] at (2.2,1.95) {$1$};
	\node[scale = 0.9] at (1.5,2.7) {$0$};
	\node[scale = 1.5] at (1.5,0.5) {$\frac{(1-q)x}{x-qy}$};
\end{scope}
\begin{scope}[shift = {(-1.5,-6.4)}]
	\draw[scale = 1.5] (0,0) -- (2,0) -- (2,2) -- (0,2) -- (0,0);
	\begin{scope}[shift = {(0,0.2)}]
	\draw[lgray,line width=1.5pt,->] (0.95,1.75) -- (2.05,1.75);
	\draw[lgray,line width=1.5pt,->] (1.5,1.2) -- (1.5,2.3);
	\end{scope}
	\node[scale = 0.9] at (1.5,1.2) {$0$};
	\node[scale = 0.9] at (0.8,1.95) {$1$};
	\node[scale = 0.9] at (2.2,1.95) {$0$};
	\node[scale = 0.9] at (1.5,2.7) {$1$};
	\node[scale = 1.5] at (1.5,0.5) {$\frac{(1-q)y}{x-qy}$};
\end{scope}
\begin{scope}[shift = {(1.5,-6.4)}]
	\draw[scale = 1.5] (0,0) -- (2,0) -- (2,2) -- (0,2) -- (0,0);
	\begin{scope}[shift = {(0,0.2)}]
	\draw[lgray,line width=1.5pt,->] (0.95,1.75) -- (2.05,1.75);
	\draw[lgray,line width=1.5pt,->] (1.5,1.2) -- (1.5,2.3);
	\end{scope}
	\node[scale = 0.9] at (1.5,1.2) {$0$};
	\node[scale = 0.9] at (0.8,1.95) {$1$};
	\node[scale = 0.9] at (2.2,1.95) {$1$};
	\node[scale = 0.9] at (1.5,2.7) {$0$};
	\node[scale = 1.5] at (1.5,0.5) {$\frac{x-y}{x-qy}$};
\end{scope}
\begin{scope}[shift = {(0,-9.6)}]
	\draw[scale = 1.5] (0,0) -- (2,0) -- (2,2) -- (0,2) -- (0,0);
	\begin{scope}[shift = {(0,0.2)}]
	\draw[lgray,line width=1.5pt,->] (0.95,1.75) -- (2.05,1.75);
	\draw[lgray,line width=1.5pt,->] (1.5,1.2) -- (1.5,2.3);
	\end{scope}
	\node[scale = 0.9] at (1.5,1.2) {$1$};
	\node[scale = 0.9] at (0.8,1.95) {$1$};
	\node[scale = 0.9] at (2.2,1.95) {$1$};
	\node[scale = 0.9] at (1.5,2.7) {$1$};
	\node[scale = 1.2] at (1.5,0.5) {$1$};
\end{scope}
\end{tikzpicture}
\caption{The weights of the stochastic six-vertex model.}
\label{fig:6v-weights}
\end{figure}

\begin{prop}[Sum-to-unity property]
For any fixed $i,j \in \{0,1\}$, we have 
\begin{equation}
\label{eqsumtounity1}
\sum_{k = 0}^1 \sum_{\ell = 0}^1 W_{y/x}(i,j;k,\ell) = 1.
\end{equation}
\end{prop}

\begin{proof}
By direct verification on the vertex weights of Figure \ref{fig:6v-weights}: the rows of that figure correspond with the rows of the matrix \eqref{eqrm61}.
\end{proof}

\begin{rmk}
The sum-to-unity property \eqref{eqsumtounity1} is necessary for the probabilistic interpretation of the $R$-matrix \eqref{eqrm61}. Indeed, by making suitable choices for the parameters $x,y,q$, such as $0 \leq y < x \leq 1$ and $0 \leq q \leq 1$, one finds that all entries of the matrix \eqref{eqrm61} are positive; it is thus a stochastic matrix. The stochastic property of the $R$-matrix \eqref{eqrm61} has been used to construct Markov processes in the quadrant in numerous prior works 
\cite{GwaSpohn92,BorodinCG,BorodinPetrov16,BorodinWheeler18}.
\end{rmk}

\subsection{Unitarity and Yang--Baxter relations}
\label{ssec:6YB}

\begin{prop}
Let $R_{ab}(x)$ denote the $R$-matrix defined in equation $\eqref{eqrm61}$, whose components are given by \eqref{eqdefn6w} and Figure \ref{fig:6v-weights}. The $R$-matrix satisfies the unitarity and Yang--Baxter relations:
\begin{alignat}{2}
R_{ab}(y/x)R_{ba}(x/y) &= {\rm id}, \label{eqpropybesix1}\\
R_{ab}(y/x)R_{ac}(z/x)R_{bc}(z/y)& = R_{bc}(z/y)R_{ac}(z/x)R_{ab}(y/x), \label{eqpropybesix2}
\end{alignat}
where the labels $a,b,c$ indicate the spaces upon which the $R$-matrix acts; here \eqref{eqpropybesix1} holds as an identity in ${\rm End}(V_a \otimes V_b)$, and \eqref{eqpropybesix2} as an identity in ${\rm End}(V_a \otimes V_b \otimes V_c)$.

Equations $\eqref{eqpropybesix1}$ and $\eqref{eqpropybesix2}$ have the following pictorial representations:
\begin{equation}
\label{equnitary1}
\sum_{k_1,k_2\in \{0,1\}}
\raisebox{-15mm}{\begin{tikzpicture}[scale = 0.4]
\draw[->,lgray,line width=1.5pt,rounded corners] (0,0) -- (3,0) -- (3,3);
\draw[->,lgray,line width=1.5pt,rounded corners] (1.5,-1.5) -- (1.5,1.5) -- (4.5,1.5);
\node at (-0.5,0) {\footnotesize $i_1$};
\node at (1.5,-2) {\footnotesize $i_2$};
\node at (3.4,-0.2) {\footnotesize $k_1$};
\node at (1.5,2.2) {\footnotesize $k_2$};
\node at (5,1.5) {\footnotesize $j_2$};
\node at (2.85,3.65) {\footnotesize $j_1$};
\node at (-2.5,0) {\footnotesize{$x$}};
\node at (1.5,-3.75) {\footnotesize{$y$}};
\draw[->,thick] (-2,0) -- (-1,0);
\draw[->,thick] (1.5,-3.35) -- (1.5,-2.5);
\end{tikzpicture}}
=
\bm{1}_{i_1=j_1}
\bm{1}_{i_2=j_2}
\end{equation}
\begin{equation}
\label{eqybesix1}
\sum_{k_1,k_2,k_3 \in \{0,1\}}
\begin{tikzpicture}[scale = 0.9,every node/.style={scale=0.8},baseline={([yshift=-0.25ex]current bounding box.center)}]
\draw[lgray,line width=1.5pt] (-2,1) -- (-1,0);
\draw[lgray,line width=1.5pt] (-2,0) -- (-1,1);
\draw[->,lgray,line width=1.5pt] (-1,1) -- (0,1);
\draw[->,lgray,line width=1.5pt] (-1,0) -- (0,0);
\draw[->,lgray,line width=1.5pt] (-0.5,-0.5) -- (-0.5,1.5);
\node at (-3,1) (A1) {$x$};
\node at (-3,0) (B1) {$y$};
\node[left] at (-1.9,1) (A) {$i_1$};
\node[left] at (-1.9,0) (B) {$i_2$};
\draw[->,thick] (A1) -- (A);
\draw[->,thick] (B1) -- (B);
\node at (-1,1.3) {$k_2$};
\node at (-1,-0.3) {$k_1$};
\node at (-0.5,1.75) {$j_3$};
\node at (-0.5,-0.75) {$i_3$};
\draw[->,thick] (-0.5,-1.3) -- (-0.5,-1);
\node at (-0.5,-1.5) {$z$};
\node at (0.35,1) {$j_2$};
\node at (0.35,0) {$j_1$};
\node at (-0.5, 0.5) {$k_3$};
\end{tikzpicture}
=\sum_{k_1,k_2,k_3 \in \{0,1\}}
\begin{tikzpicture}[scale = 0.8,every node/.style={scale=0.8},baseline={([yshift=-0.25ex]current bounding box.center)}]
\draw[lgray,line width=1.5pt] (4.5,1) -- (5.5,1);
\draw[lgray,line width=1.5pt] (4.5,0) -- (5.5,0);
\draw[->,lgray,line width=1.5pt] (5.5,1) -- (6.5,0);
\draw[->,lgray,line width=1.5pt] (5.5,0) -- (6.5,1);
\draw[->,lgray,line width=1.5pt] (5,-0.5) -- (5,1.5);
\node at (3.5,1) (C1) {$x$};
\node at (3.5,0) (D1) {$y$};
\node[left] at (4.6,1) (C) {$i_1$};
\node[left] at (4.6,0) (D) {$i_2$};
\draw[->,thick] (C1) -- (C);
\draw[->,thick] (D1) -- (D);
\node at (5,1.75) {$j_3$};
\node at (5,-0.75) {$i_3$};
\draw[->,thick] (5,-1.3) -- (5,-1);
\node at (5,-1.5) {$z$};
\node at (6.85,1) {$j_2$};
\node at (6.85,0) {$j_1$};
\node at (5.5,1.25) {$k_1$};
\node at (5.5,-0.25) {$k_2$};
\node at (5,0.5) {$k_3$};
\end{tikzpicture}
\end{equation}
where in both equations the indices $i_1,i_2,i_3,j_1,j_2,j_3$ are assigned fixed values in $\{0,1\}$.
\end{prop}

\begin{proof}
Equations \eqref{eqpropybesix1} and \eqref{eqpropybesix2} can be verified by direct calculation.
\end{proof}

\subsection{Dotted vertices}
\label{ssec:6dot}

For technical reasons, in the sequel we will also require a normalized version of the weights  $W_{y/x}(i,j;k,\ell)$; namely,
\begin{equation}
\label{eqcrsym61}
\dotr{W}_{y/x}(i,j;k,\ell) 
= 
\frac{W_{y/x}(i,j;k,\ell)}{W_{y/x}(0,1;0,1)}
=
\left( \dfrac{x-qy}{x-y} \right)
W_{y/x}(i,j;k,\ell).
\end{equation}
To distinguish these weights from the regular ones, we shall place a dot in the center of each vertex; see Figure \ref{fig:6v-dot}.

\begin{figure}
\begin{tikzpicture}
	\draw[scale = 1.5] (0,0) -- (2,0) -- (2,2) -- (0,2) -- (0,0);
	\begin{scope}[shift = {(0,0.2)}]
	\draw[lgray,line width=1.5pt,->] (0.95,1.75) -- (2.05,1.75);
	\draw[lgray,line width=1.5pt,->] (1.5,1.2) -- (1.5,2.3);
	\end{scope}
	\node at (1.5,1.95) [circle,fill,inner sep=1.5pt] {};
	\node[scale = 0.9] at (1.5,1.2) {$0$};
	\node[scale = 0.9] at (0.8,1.95) {$0$};
	\node[scale = 0.9] at (2.2,1.95) {$0$};
	\node[scale = 0.9] at (1.5,2.7) {$0$};
	\node[scale = 1.5] at (1.5,0.5) {$\frac{x-qy}{x-y}$};
\begin{scope}[shift = {(-1.5,-3.2)}]
	\draw[scale = 1.5] (0,0) -- (2,0) -- (2,2) -- (0,2) -- (0,0);
	\begin{scope}[shift = {(0,0.2)}]
	\draw[lgray,line width=1.5pt,->] (0.95,1.75) -- (2.05,1.75);
	\draw[lgray,line width=1.5pt,->] (1.5,1.2) -- (1.5,2.3);
	\end{scope}
	\node[scale = 0.9] at (1.5,1.2) {$1$};
	\node[scale = 0.9] at (0.8,1.95) {$0$};
	\node[scale = 0.9] at (2.2,1.95) {$0$};
	\node[scale = 0.9] at (1.5,2.7) {$1$};
	\node at (1.5,1.95) [circle,fill,inner sep=1.5pt] {};
	\node[scale = 1.2] at (1.5,0.5) {$q$};
\end{scope}
\begin{scope}[shift = {(1.5,-3.2)}]
	\draw[scale = 1.5] (0,0) -- (2,0) -- (2,2) -- (0,2) -- (0,0);
	\begin{scope}[shift = {(0,0.2)}]
	\draw[lgray,line width=1.5pt,->] (0.95,1.75) -- (2.05,1.75);
	\draw[lgray,line width=1.5pt,->] (1.5,1.2) -- (1.5,2.3);
	\end{scope}
	\node[scale = 0.9] at (1.5,1.2) {$1$};
	\node[scale = 0.9] at (0.8,1.95) {$0$};
	\node[scale = 0.9] at (2.2,1.95) {$1$};
	\node[scale = 0.9] at (1.5,2.7) {$0$};
	\node at (1.5,1.95) [circle,fill,inner sep=1.5pt] {};
	\node[scale = 1.5] at (1.5,0.5) {$\frac{(1-q)x}{x-y}$};
\end{scope}
\begin{scope}[shift = {(-1.5,-6.4)}]
	\draw[scale = 1.5] (0,0) -- (2,0) -- (2,2) -- (0,2) -- (0,0);
	\begin{scope}[shift = {(0,0.2)}]
	\draw[lgray,line width=1.5pt,->] (0.95,1.75) -- (2.05,1.75);
	\draw[lgray,line width=1.5pt,->] (1.5,1.2) -- (1.5,2.3);
	\end{scope}
	\node[scale = 0.9] at (1.5,1.2) {$0$};
	\node[scale = 0.9] at (0.8,1.95) {$1$};
	\node[scale = 0.9] at (2.2,1.95) {$0$};
	\node[scale = 0.9] at (1.5,2.7) {$1$};
	\node at (1.5,1.95) [circle,fill,inner sep=1.5pt] {};
	\node[scale = 1.5] at (1.5,0.5) {$\frac{(1-q)y}{x-y}$};
\end{scope}
\begin{scope}[shift = {(1.5,-6.4)}]
	\draw[scale = 1.5] (0,0) -- (2,0) -- (2,2) -- (0,2) -- (0,0);
	\begin{scope}[shift = {(0,0.2)}]
	\draw[lgray,line width=1.5pt,->] (0.95,1.75) -- (2.05,1.75);
	\draw[lgray,line width=1.5pt,->] (1.5,1.2) -- (1.5,2.3);
	\end{scope}
	\node[scale = 0.9] at (1.5,1.2) {$0$};
	\node[scale = 0.9] at (0.8,1.95) {$1$};
	\node[scale = 0.9] at (2.2,1.95) {$1$};
	\node[scale = 0.9] at (1.5,2.7) {$0$};
	\node at (1.5,1.95) [circle,fill,inner sep=1.5pt] {};
	\node[scale = 1.2] at (1.5,0.5) {$1$};
\end{scope}
\begin{scope}[shift = {(0,-9.6)}]
	\draw[scale = 1.5] (0,0) -- (2,0) -- (2,2) -- (0,2) -- (0,0);
	\begin{scope}[shift = {(0,0.2)}]
	\draw[lgray,line width=1.5pt,->] (0.95,1.75) -- (2.05,1.75);
	\draw[lgray,line width=1.5pt,->] (1.5,1.2) -- (1.5,2.3);
	\end{scope}
	\node[scale = 0.9] at (1.5,1.2) {$1$};
	\node[scale = 0.9] at (0.8,1.95) {$1$};
	\node[scale = 0.9] at (2.2,1.95) {$1$};
	\node[scale = 0.9] at (1.5,2.7) {$1$};
	\node[scale = 1.5] at (1.5,0.5) {$\frac{x-qy}{x-y}$};
	\node at (1.5,1.95) [circle,fill,inner sep=1.5pt] {};
\end{scope}
\end{tikzpicture}
\caption{The weights of the stochastic six-vertex model after applying the normalization \eqref{eqcrsym61}.}
\label{fig:6v-dot}
\end{figure}
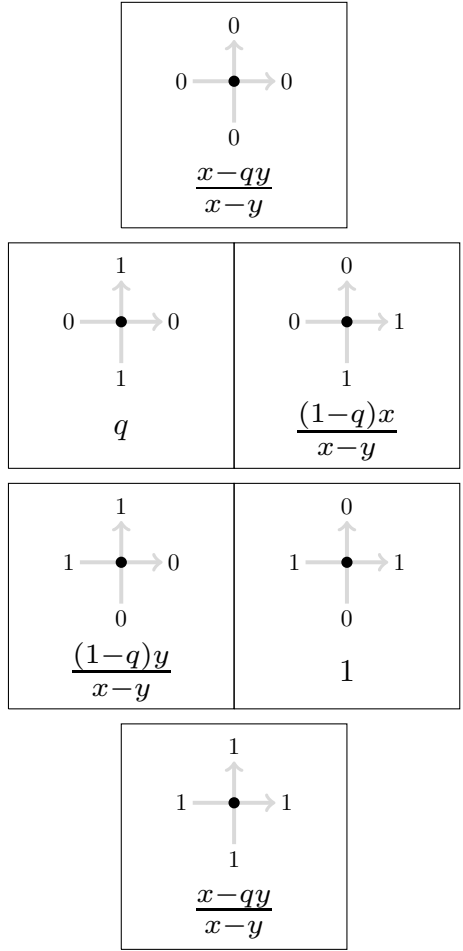
 
\begin{prop}
\label{flipsix1}
The vertices of Figure \ref{fig:6v-dot} can be related to those of Figure \ref{fig:6v-weights} under $x \mapsto x^{-1}$ and $y \mapsto q^{-1} y^{-1}$, via the following symmetry:
\begin{equation}
\label{eqnormandor61}
\dotr{W}_{y/x}(i,j;k,\ell) 
= 
(-1)^{i-k}q^{\bar{j}}
W_{x/y/q}(k,\bar{j};i,\bar{\ell}) \quad \text{ for all } \quad i,j,k,\ell \in \{0,1\},
\end{equation}
where $\bar{j}=1-j$ and $\bar{\ell}=1-\ell$. Pictorially, the vertices are related in the following way:
\[
\begin{tikzpicture}[scale = 0.8,every node/.style={scale=0.8}]
\draw[lgray,line width=1.5pt,->] (-1,0) -- (1,0);
\draw[lgray,line width=1.5pt,->] (0,-1) -- (0,1);
\node at (0,-1.25) {$i$};
\node at (0,1.25) {$k$};
\node at (-1.25,0) {$j$};
\node at (1.25,0) {$\ell$};
\node at (0,0) [circle,fill,inner sep=1.5pt] {};
\node at (-2.25,0) {$x$};
\draw[thick,->] (0,-2) -- (0,-1.45);
\draw[thick,->] (-2,0)  -- (-1.5,0);
\node[below] at (0,-2.05) {$y$};
\node[right] at (2,0) {$=$};
\draw[lgray,line width=1.5pt,->] (5,0) -- (7,0);
\draw[lgray,line width=1.5pt,->] (6,-1) -- (6,1);
\node at (6,-1.25) {$k$};
\node at (6,1.25) {$i$};
\node[left] at (5,0) {$\bar{j}$};
\node at (7.325,0) {$\bar{\ell}$};
\node at (3,0) {$x^{-1}$};
\draw[thick,->] (6,-2) -- (6,-1.45);
\draw[thick,->] (3.5,0) -- (4,0);
\node[below] at (6,-2.05) {$q^{-1} y^{-1}$};
\node[scale=1.25] at (9,0) {\qquad $\times\ \ (-1)^{i-k} q^{\bar{j}}.$};
\end{tikzpicture}
\]
\end{prop}
\begin{proof}
By direct computation. 
\end{proof}

\section{Row operators and the Yang--Baxter algebra}
\label{sec:6row}

This section contains the fundamental algebraic tools that underpin most of the calculations in this chapter. In Section \ref{ssec:6-finite-row} we introduce row-operators in finite size, and list some of their commutation relations. In Section \ref{ssec:6-infinite-row} we study the limit of these operators when the length of the row tends to infinity, and examine what happens to the associated commutation relations.

\subsection{Row operators and commutation relations}
\label{ssec:6-finite-row}

We begin by introducing the vector space $\mathbb{V}^{(L)}$, which is obtained by taking an $L$-fold tensor product of local spaces $V_i \cong \mathbb{C}^2$:
\begin{equation}
\label{6VL}
\mathbb{V}^{(L)} = V_1 \otimes V_2 \otimes \cdots \otimes V_L.
\end{equation}
In the partition functions that we subsequently consider, each vector space $V_i$ will correspond to a vertical lattice line; as such, the natural operators for acting on $\mathbb{V}^{(L)}$ are obtained by considering rows of vertices of length $L$ within the model \eqref{eqdefn6w}. In particular, for all $a,b \in \{0,1\}$, we define the {\it finite} row operator $\mathcal{T}_{a,b}(x;\bm{z}) \in {\rm End}(\mathbb{V}^{(L)})$ as follows:
\begin{align}
\mathcal{T}_{a,b}(x;\bm{z}): \bigotimes_{k = 1}^{L} \ket{i_k}_{k} 
&\mapsto \sum_{(j_1,\dots,j_L) 
\in \{0,1\}^L}  \left(\raisebox{-15mm}{\smash{
\begin{tikzpicture}[scale = 0.8,every node/.style={scale=0.8}]
\node at (-1.8,0) (D1) {$x$};
\node at (-1,0) (D) {};
\draw[lgray,line width=1.5pt] (-1,0) -- (3,0);
\draw[lgray,line width=1.5pt,->] (4.2,0) -- (7.2,0);
\draw[<-,lgray,line width=1.5pt] (1,1) -- (1,-1);
\draw[<-,lgray,line width=1.5pt] (0,1) -- (0,-1);
\draw[<-,lgray,line width=1.5pt] (2,1) -- (2,-1);
\draw[<-,lgray,line width=1.5pt] (5.2,1) -- (5.2,-1);
\draw[<-,lgray,line width=1.5pt] (6.2,1) -- (6.2,-1);
\node at (-1,-0.25) {$a$};
\node[right] at (3.3,0) {$\dots$};
\node at (0,1.25) {$i_1$};
\node at (1,1.25) {$i_2$};
\node at (2,1.25) {$\dots$};
\node at (5.2,1.25) {$\dots$};
\node at (6.2,1.25) {$i_L$};
\node at (0,-1.25) (A) {$j_1$};
\node at (1,-1.25) (B) {$j_2$};
\node at (2,-1.25) {$\dots$};
\node at (5.2,-1.25) {$\dots$};
\node at (6.2,-1.25) (C) {$j_L$};
\node at (7.2,-0.25) {$b$};
\node at (0,-2.25) (A1) {$z_1$};
\node at (1,-2.25) (B1) {$z_2$};
\node at (2,-2.25)  {$\dots$};
\node at (6.2,-2.25) (C1) {$z_L$};
\draw[->,thick] (A1) -- (A);
\draw[->,thick] (B1) -- (B);
\draw[->,thick] (C1) -- (C);
\draw[->,thick] (D1) -- (D);
\end{tikzpicture}}}\right)
\bigotimes_{k = 1}^{L} 
\ket{j_k}_{k}, 
\label{eqrowop6}
\end{align}
where $(i_1,\dots,i_L) \in \{0,1\}^L$ selects a fixed, arbitrary vector in $\mathbb{V}^{(L)}$. The row operators $\mathcal{T}_{a,b}(x;\bm{z})$ depend on both the horizontal spectral parameter $x$ and the vertical ones, $\bm{z} = (z_1,z_2,\dots,z_L)$. However, in practice, the dependence on $\bm{z}$ will have little interest for us, and we usually suppress it from our notations by writing $\mathcal{T}_{a,b}(x;\bm{z}) \equiv \mathcal{T}_{a,b}(x)$.

We shall also require a dotted version of the row operators \eqref{eqrowop6}, denoted $\dotr{\mathcal{T}}_{a,b}(x;\bm{z})$. These are obtained by simply replacing each vertex in \eqref{eqrowop6} by its dotted counterpart \eqref{eqcrsym61}:
\begin{align}
\dotr{\mathcal{T}}_{a,b}(x;\bm{z}): \bigotimes_{k = 1}^{L} \ket{i_k}_{k} 
&\mapsto \sum_{(j_1,\dots,j_L) 
\in \{0,1\}^L}  \left(\raisebox{-15mm}{\smash{
\begin{tikzpicture}[scale = 0.8,every node/.style={scale=0.8}]
\node at (-1.8,0) (D1) {$x$};
\node at (-1,0) (D) {};
\draw[lgray,line width=1.5pt] (-1,0) -- (3,0);
\draw[lgray,line width=1.5pt,->] (4.2,0) -- (7.2,0);
\draw[<-,lgray,line width=1.5pt] (1,1) -- (1,-1);
\draw[<-,lgray,line width=1.5pt] (0,1) -- (0,-1);
\draw[<-,lgray,line width=1.5pt] (2,1) -- (2,-1);
\draw[<-,lgray,line width=1.5pt] (5.2,1) -- (5.2,-1);
\draw[<-,lgray,line width=1.5pt] (6.2,1) -- (6.2,-1);
\node at (-1,-0.25) {$a$};
\node[right] at (3.3,0) {$\dots$};
\node at (0,1.25) {$i_1$};
\node at (1,1.25) {$i_2$};
\node at (2,1.25) {$\dots$};
\node at (5.2,1.25) {$\dots$};
\node at (6.2,1.25) {$i_L$};
\node at (0,-1.25) (A) {$j_1$};
\node at (1,-1.25) (B) {$j_2$};
\node at (2,-1.25) {$\dots$};
\node at (5.2,-1.25) {$\dots$};
\node at (6.2,-1.25) (C) {$j_L$};
\node at (7.2,-0.25) {$b$};
\node at (0,-2.25) (A1) {$z_1$};
\node at (1,-2.25) (B1) {$z_2$};
\node at (2,-2.25)  {$\dots$};
\node at (6.2,-2.25) (C1) {$z_L$};
\draw[->,thick] (A1) -- (A);
\draw[->,thick] (B1) -- (B);
\draw[->,thick] (C1) -- (C);
\draw[->,thick] (D1) -- (D);
\node at (0,0) [circle,fill,inner sep=1.5pt] {};
\node at (1,0) [circle,fill,inner sep=1.5pt] {};
\node at (2,0) [circle,fill,inner sep=1.5pt] {};
\node at (5.2,0) [circle,fill,inner sep=1.5pt] {};
\node at (6.2,0) [circle,fill,inner sep=1.5pt] {};
\end{tikzpicture}}}\right)
\bigotimes_{k = 1}^{L} 
\ket{j_k}_{k}. 
\label{eqrowop6dot}
\end{align}

\begin{rmk}
The operators \eqref{eqrowop6} correspond with the four entries of the {\it monodromy matrix} of the six-vertex model, and satisfy a collection of $2^4$ bilinear relations known as the {\it Yang--Baxter algebra}. Only a few of these relations will be important for our purposes: namely, the ones that allow us to prove the symmetry of our families of multivariate functions, and their Cauchy identities. These will be discussed below.
\end{rmk}

\begin{prop}
\label{prop:6-Tcommute}
For any two integers $a,b \in \{0,1\}$, the operators $\mathcal{T}_{a,b}(x)$ and $\mathcal{T}_{a,b}(y)$ commute:
\begin{align}
\label{6-Tcommute}
[\mathcal{T}_{a,b}(x),\mathcal{T}_{a,b}(y)]=0.
\end{align}
\end{prop}

\begin{proof}
This follows from a standard Yang--Baxter argument. One begins by considering arbitrary matrix elements of the product $\mathcal{T}_{a,b}(x)\mathcal{T}_{a,b}(y)$:
\begin{align*}
\bigotimes_{k = 1}^{L} \bra{j_k}_{k}
\mathcal{T}_{a,b}(x)\mathcal{T}_{a,b}(y) 
\bigotimes_{k = 1}^{L} \ket{i_k}_{k}
=
\raisebox{-12mm}{\begin{tikzpicture}[scale=0.6,every node/.style={scale=0.6}]
\node at (2,1) {$x$};
\node at (2,2) {$y$};
\foreach \i in {1,2} {
\draw[->,thick] (2.25,\i) -- (3,\i);
}
\foreach \i in {4} {
\draw[lgray,line width=1.5pt,->] (\i,0) -- (\i,3);
}
\foreach \i in {1,2} {
\draw[lgray,line width=1.5pt] (3.5,\i) -- (4.5,\i);
}
\foreach \i in {1,2} {
\node at (5,\i) {$\dots$};
}
\foreach \i in {6,7} {
\draw[lgray,line width=1.5pt,->] (\i,0) -- (\i,3);
}
\foreach \i in {1,2} {
\draw[lgray,line width=1.5pt,->] (5.5,\i) -- (7.5,\i);
}
\node at (3.3,1) {$a$};
\node at (3.3,2) {$a$};
\node[right] at (7.5,1) {$b$};
\node[right] at (7.5,2) {$b$};
\node[above] at (4,3) {$i_1$};
\node[above] at (6,3) {$\dots$};
\node[above] at (7,3) {$i_L$};
\node[below] at (4,0) {$j_1$};
\node[below] at (6,0) {$\dots$};
\node[below] at (7,0) {$j_L$};
\end{tikzpicture}}
\end{align*}
In view of the fact that both of the external left edges take the value $a \in \{0,1\}$, one may insert an $R$-matrix vertex at the left of this picture, without altering the value of the partition function:
\begin{align}
\label{eq:6-insertR-left}
\bigotimes_{k = 1}^{L} \bra{j_k}_{k}
\mathcal{T}_{a,b}(x)\mathcal{T}_{a,b}(y) 
\bigotimes_{k = 1}^{L} \ket{i_k}_{k}
=
\raisebox{-12mm}{\begin{tikzpicture}[scale=0.6,every node/.style={scale=0.6}]
\begin{scope}[shift = {(-1.2,0)}]
\node at (2,1) {$y$};
\node at (2,2) {$x$};
\foreach \i in {1,2} {
\draw[->,thick] (2.25,\i) -- (3,\i);
}
\end{scope}
\draw[lgray,line width=1.5pt] (2.5,2) -- (3.5,1);
\draw[lgray,line width=1.5pt] (2.5,1) -- (3.5,2);
\foreach \i in {4} {
\draw[lgray,line width=1.5pt,->] (\i,0) -- (\i,3);
}
\foreach \i in {1,2} {
\draw[lgray,line width=1.5pt] (3.5,\i) -- (4.5,\i);
}
\foreach \i in {1,2} {
\node at (5,\i) {$\dots$};
}
\foreach \i in {6,7} {
\draw[lgray,line width=1.5pt,->] (\i,0) -- (\i,3);
}
\foreach \i in {1,2} {
\draw[lgray,line width=1.5pt,->] (5.5,\i) -- (7.5,\i);
}
\node[left] at (2.3,1) {$a$};
\node[left] at (2.3,2) {$a$};
\node[right] at (7.5,1) {$b$};
\node[right] at (7.5,2) {$b$};
\node[above] at (4,3) {$i_1$};
\node[above] at (6,3) {$\dots$};
\node[above] at (7,3) {$i_L$};
\node[below] at (4,0) {$j_1$};
\node[below] at (6,0) {$\dots$};
\node[below] at (7,0) {$j_L$};
\node at (3.5,1) {$*$};
\node at (3.5,2) {$*$};
\end{tikzpicture}}
\end{align}
Indeed, this step is justified by the fact that both of the edges marked $*$ are forced to assume the value $a$ (by conservation of paths), and the fact that $W_{y/x}(a,a;a,a)=1$ for $a \in \{0,1\}$; {\it cf.} the top and bottom rows of Figure \ref{fig:6v-weights}. From here, we use the Yang--Baxter equation \eqref{eqybesix1} repeatedly to transfer the inserted $R$-matrix all the way to the right edges of the lattice:
\begin{align*}
\bigotimes_{k = 1}^{L} \bra{j_k}_{k}
\mathcal{T}_{a,b}(x)\mathcal{T}_{a,b}(y) 
\bigotimes_{k = 1}^{L} \ket{i_k}_{k}
=
\raisebox{-12mm}{\begin{tikzpicture}
[scale=0.6,every node/.style={scale=0.6}]
\begin{scope}[shift = {(0,0)}]
\node at (2,1) {$y$};
\node at (2,2) {$x$};
\foreach \i in {1,2} {
\draw[->,thick] (2.25,\i) -- (3,\i);
}
\end{scope}
\draw[lgray,line width=1.5pt,->] (7.5,2) -- (8.5,1);
\draw[lgray,line width=1.5pt,->] (7.5,1) -- (8.5,2);
\foreach \i in {4} {
\draw[lgray,line width=1.5pt,->] (\i,0) -- (\i,3);
}
\foreach \i in {1,2} {
\draw[lgray,line width=1.5pt] (3.5,\i) -- (4.5,\i);
}
\foreach \i in {1,2} {
\node at (5,\i) {$\dots$};
}
\foreach \i in {6,7} {
\draw[lgray,line width=1.5pt,->] (\i,0) -- (\i,3);
}
\foreach \i in {1,2} {
\draw[lgray,line width=1.5pt] (5.5,\i) -- (7.5,\i);
}
\node[left] at (3.5,1) {$a$};
\node[left] at (3.5,2) {$a$};
\node[right] at (8.5,1) {$b$};
\node[right] at (8.5,2) {$b$};
\node[above] at (4,3) {$i_1$};
\node[above] at (6,3) {$\dots$};
\node[above] at (7,3) {$i_L$};
\node[below] at (4,0) {$j_1$};
\node[below] at (6,0) {$\dots$};
\node[below] at (7,0) {$j_L$};
\node at (7.5,1) {$*$};
\node at (7.5,2) {$*$};
\end{tikzpicture}}
\end{align*}
One concludes by deleting the $R$-matrix from the right of the picture; the justification for doing so is again the fact that the edges marked $*$ are frozen to the values $b$, and the trivial weight $W_{y/x}(b,b;b,b)=1$ for $b \in \{0,1\}$. The net result of this calculation is the switching of horizontal lattice lines; we conclude that
\begin{align*}
\bigotimes_{k = 1}^{L} \bra{j_k}_{k}
\mathcal{T}_{a,b}(x)\mathcal{T}_{a,b}(y) 
\bigotimes_{k = 1}^{L} \ket{i_k}_{k}
=
\bigotimes_{k = 1}^{L} \bra{j_k}_{k}
\mathcal{T}_{a,b}(y)\mathcal{T}_{a,b}(x) 
\bigotimes_{k = 1}^{L} \ket{i_k}_{k}
\end{align*}
for all $(i_1,\dots,i_L), (j_1,\dots,j_L) \in \{0,1\}^L$, which is precisely \eqref{6-Tcommute}.
\end{proof}

\begin{prop}
\label{prop:6-commute-nontrivial}
The following commutation relation holds:
\begin{align}
\label{6-eqrowrel2}
\mathcal{T}_{1,1}(x) \mathcal{T}_{1,0}(y)
=
W_{y/x}(0,1;0,1)
\mathcal{T}_{1,0}(y) \mathcal{T}_{1,1}(x)
+ 
W_{y/x}(1,0;0,1)
\mathcal{T}_{1,1}(y) \mathcal{T}_{1,0}(x).
\end{align}
\end{prop}

\begin{proof}
The proof of \eqref{6-eqrowrel2} follows by similar reasoning to that of Proposition \ref{prop:6-Tcommute}. One again begins by considering the matrix elements of the product $\mathcal{T}_{1,1}(x) \mathcal{T}_{1,0}(y)$, represented as a two-row partition function, and inserting an $R$-matrix vertex at the left of the lattice (using the fact that $W_{y/x}(1,1;1,1)=1$), similarly to \eqref{eq:6-insertR-left}. This time, after repeated application of the Yang--Baxter equation, the $R$-matrix vertex emerges from the right of the picture but it cannot simply be removed modulo freezing. Rather, one must sum over the possible values of the edges which connect this vertex to the lattice, resulting in the equation
\begin{align}
\label{6-eqybepic2}
\raisebox{-10mm}{\begin{tikzpicture}[scale=0.6,every node/.style={scale=0.6}]
\foreach \i in {24} {
\draw[lgray,line width=1.5pt,->] (\i,0) -- (\i,3);
}
\foreach \i in {1,2} {
\draw[lgray,line width=1.5pt] (23.5,\i) -- (24.5,\i);
}
\foreach \i in {1,2} {
\node at (25,\i) {$\dots$}; 
}
\foreach \i in {26,27} {
\draw[lgray,line width=1.5pt,->] (\i,0) -- (\i,3);
}
\foreach \i in {1,2} {
\draw[lgray,line width=1.5pt,->] (25.5,\i) -- (27.5,\i);
}
\node at (24,-0.2) {$j_1$};
\node at (26,-0.2) {$\dots$};
\node at (27,-0.2) {$j_L$};
\node at (24,3.2) {$i_1$};
\node at (26,3.2) {$\dots$};
\node at (27,3.2) {$i_L$};
\node at (21,2) {$x$};
\node at (21,1) {$y$};
\foreach \i in {1,2} {
\draw[thick,->] (21.3,\i) -- (22,\i);
}
\draw[lgray,line width=1.5pt] (22.5,2) -- (23.5,1);
\draw[lgray,line width=1.5pt] (22.5,1) -- (23.5,2);
\node at (22.3,1) {$1$};
\node at (22.3,2) {$1$};
\node at (27.7,2) {$0$};
\node at (27.7,1) {$1$};
\end{tikzpicture}}
\quad
=
\raisebox{-8mm}{\begin{tikzpicture}[scale=0.6,every node/.style={scale=0.6}]
\begin{scope}[shift = {(18,-4)}]
\foreach \i in {10.5} {
\draw[lgray,line width=1.5pt,->] (\i,0) -- (\i,3);
}
\node at (10.5,3.2) {$i_1$};
\node at (12.5,3.2) {$\dots$};
\node at (13.5,3.2) {$i_L$};
\node at (10.5,-0.2) {$j_1$};
\node at (12.5,-0.2) {$\dots$};
\node at (13.5,-0.2) {$j_L$};
\node at (9,1) {$y$};
\node at (9,2) {$x$};
\foreach \i in {1,2} {
\draw[->,thick] (9.2,\i) -- (9.6,\i);
}
\foreach \i in {1,2} {
\draw[lgray,line width=1.5pt] (10,\i) -- (11,\i);
}
\foreach \i in {1,2} {
\node at (11.5,\i) {$\dots$};
} 
\foreach \i in {12.5,13.5} {
\draw[lgray,line width=1.5pt,->] (\i,0) -- (\i,3);
} 
\foreach \i in {1,2} {
\draw[lgray,line width=1.5pt] (12,\i) -- (14,\i);
}
\draw[lgray,line width=1.5pt,->] (14,2) -- (15,1);
\draw[lgray,line width=1.5pt,->] (14,1) -- (15,2);
\node at (9.8,1) {$1$};
\node at (9.8,2) {$1$};
\node at (15.2,1) {$1$};
\node at (15.2,2) {$0$};
\node at (14,0.8) {$0$};
\node at (14,2.2) {$1$};
\end{scope}
\begin{scope}[shift = {(25.5,-4)}]
\node at (8.35,1.5) {$+$};
\foreach \i in {10.5} {
\draw[lgray,line width=1.5pt,->] (\i,0) -- (\i,3);
}
\node at (10.5,3.2) {$i_1$};
\node at (12.5,3.2) {$\dots$};
\node at (13.5,3.2) {$i_L$};
\node at (10.5,-0.2) {$j_1$};
\node at (12.5,-0.2) {$\dots$};
\node at (13.5,-0.2) {$j_L$};
\node at (9,1) {$y$};
\node at (9,2) {$x$};
\foreach \i in {1,2} {
\draw[->,thick] (9.2,\i) -- (9.6,\i);
}
\foreach \i in {1,2} {
\draw[lgray,line width=1.5pt] (10,\i) -- (11,\i);
}
\foreach \i in {1,2} {
\node at (11.5,\i) {$\dots$};
} 
\foreach \i in {12.5,13.5} {
\draw[lgray,line width=1.5pt,->] (\i,0) -- (\i,3);
} 
\foreach \i in {1,2} {
\draw[lgray,line width=1.5pt] (12,\i) -- (14,\i);
}
\draw[lgray,line width=1.5pt,->] (14,2) -- (15,1);
\draw[lgray,line width=1.5pt,->] (14,1) -- (15,2);
\node at (9.8,1) {$1$};
\node at (9.8,2) {$1$};
\node at (15.2,1) {$1$};
\node at (15.2,2) {$0$};
\node at (14,0.8) {$1$};
\node at (14,2.2) {$0$};
\end{scope}
\end{tikzpicture}}
\end{align}
which is valid for all $(i_1,\dots,i_L), (j_1,\dots,j_L) \in \{0,1\}^L$. Matching the vertices which appear in this identity with their symbolic form \eqref{eqdefn6w}, we read off \eqref{6-eqrowrel2}.
\end{proof}

\subsection{Row operators in the infinite-volume limit}
\label{ssec:6-infinite-row}

Before discussing the lift of our operators to rows of infinite length, we require a further definition.

\begin{defn}
\label{def:n-string}
An $n$-string is an infinite vector $S=(S_1,S_2,S_3,\dots)$ with $0 \leq S_i \leq n$ for all $i \in \mathbb{N}$, and obeying the finiteness property $\exists\ M \in \mathbb{N}$ such that $S_i = 0$ for all $i>M$. Whenever $S_i=0$ for all $i>M$, we make the identification $(S_1,S_2,S_3,\dots) \equiv (S_1,\dots,S_M)$. We write $|S| = \sum_{i \geq 1} S_i$ for the sum of components of any $n$-string and refer to this quantity as its {\it weight}. Let $\mathfrak{s}(n)$ denote the set of all $n$-strings. 
\end{defn}

The functions studied in this chapter will be indexed by $1$-strings (the $n=1$ case of Definition \ref{def:n-string}). We begin by defining a suitable completion of the vector space $\mathbb{V}^{(L)}$, obtained by taking $L \rightarrow \infty$.\footnote{Equivalently, we shall work with partition functions with infinitely many vertical lines in the rightward direction.} To that end, we define
\begin{align}
\label{6infV}
\mathbb{V}
=
\bigoplus_{S \in \mathfrak{s}(1)} \mathbb{C} \ket{S},
\end{align}
with the set $\mathfrak{s}(1)$ of $1$-strings given by Definition \ref{def:n-string}. 
We shall refer to $\mathbb{V}$ as the \textit{infinite-volume limit} of the space $\mathbb{V}^{(L)}$.

Some care is needed to extend the action of the finite row-operators \eqref{eqrowop6} and \eqref{eqrowop6dot} to the infinite volume; indeed, one may encounter convergence issues that are not present in finite size. We shall not attempt to address this issue for all of the possible choices of row-operators, but only for the subset of them that is needed in this work. This is done in the following proposition:
\begin{prop}[\cite{BorodinPetrov16} Sections 4.2 and 4.3]
\label{prop:6-ACD-define}
The objects
\begin{align*}
\mathcal{A}(x)
:=
\lim_{L \rightarrow \infty}
\mathcal{T}_{0,0}(x),
\qquad
\mathcal{C}(x)
:=
\lim_{L \rightarrow \infty}
\mathcal{T}_{1,0}(x),
\qquad
\dotr{\mathcal{D}}(x)
:=
\lim_{L \rightarrow \infty}
\dotr{\mathcal{T}}_{1,1}(x),
\end{align*}
are well-defined as operators in ${\rm End}(\mathbb{V})$.
\end{prop}

\begin{proof}
The matrix elements of $\mathcal{A}(x)$ and $\mathcal{C}(x)$ are represented by one-row partition functions of the form
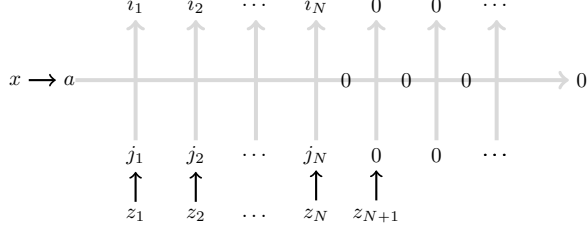
\[
\begin{tikzpicture}[scale = 0.8,every node/.style={scale=0.8}]
\node at (-2,0) (D1) {$x$};
\node at (-1.2,0) (D) {};
\draw[lgray,line width=1.5pt,->] (-1,0) -- (7.2,0);
\foreach \i in {0,1,2,3,4,5,6} {
\draw[<-,lgray,line width=1.5pt] (\i,1) -- (\i,-1);
}
\node at (-1.1,0) {$a$};
\node at (0,1.25) {$i_1$};
\node at (1,1.25) {$i_2$};
\node at (2,1.25) {$\dots$};
\node at (3,1.25) {$i_{N}$};
\node at (4,1.25) {$0$};
\node at (5,1.25) {$0$};
\node at (6,1.25) {$\dots$};
\node at (0,-1.25) (A) {$j_1$};
\node at (1,-1.25) (B) {$j_2$};
\foreach \i in {2,6} {
\node at (\i,-1.25) {$\dots$};
}
\node at (3,-1.25) {$j_{N}$};
\node at (4,-1.25) {$0$};
\node at (5,-1.25) {$0$};
\node at (6,-1.25) (C) {$\dots$};
\node[right] at (7.2,0) {$0$};
\node at (0,-2.25) (A1) {$z_1$};
\node at (1,-2.25) (B1) {$z_2$};
\node at (3,-2.25) {$z_N$};
\node at (4,-2.25) {$z_{N+1}$};
\node at (2,-2.25)  {$\dots$};
\draw[->,thick] (3,-2) -- (3,-1.5);
\draw[->,thick] (4,-2) -- (4,-1.5);
\draw[->,thick] (A1) -- (A);
\draw[->,thick] (B1) -- (B);
\draw[->,thick] (D1) -- (D);
\foreach \i in {3.5,4.5,5.5} {
\node at (\i,0) {$0$};
}
\end{tikzpicture}
\]
where $a \in \{0,1\}$. By virtue of the fact that $\exists\ N: i_k = j_k=0$ for all $k>N$, it is clear that all vertices in this diagram past the $N$-th column must take the form $W_{z_k/x}(0,0;0,0)=1$. As such, viewed as operators in ${\rm End}(\mathbb{V})$, the matrix entries of $\mathcal{A}(x)$ and $\mathcal{C}(x)$ consist only of finite products of rational functions.

A similar argument works for the matrix elements of $\dotr{\mathcal{D}}(x)$, which are represented by one-row partition functions using dotted vertices:
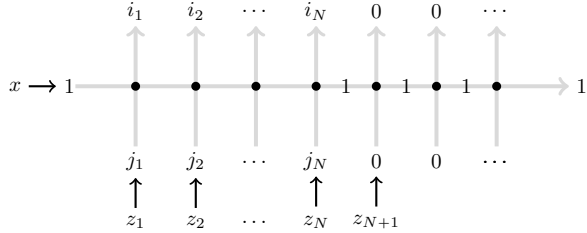
\[
\begin{tikzpicture}[scale = 0.8,every node/.style={scale=0.8}]
\node at (-2,0) (D1) {$x$};
\node at (-1.2,0) (D) {};
\draw[lgray,line width=1.5pt,->] (-1,0) -- (7.2,0);
\foreach \i in {0,1,2,3,4,5,6} {
\draw[<-,lgray,line width=1.5pt] (\i,1) -- (\i,-1);
}
\node at (-1.1,0) {$1$};
\node at (0,1.25) {$i_1$};
\node at (1,1.25) {$i_2$};
\node at (2,1.25) {$\dots$};
\node at (3,1.25) {$i_{N}$};
\node at (4,1.25) {$0$};
\node at (5,1.25) {$0$};
\node at (6,1.25) {$\dots$};
\node at (0,-1.25) (A) {$j_1$};
\node at (1,-1.25) (B) {$j_2$};
\foreach \i in {2,6} {
\node at (\i,-1.25) {$\dots$};
}
\node at (3,-1.25) {$j_{N}$};
\node at (4,-1.25) {$0$};
\node at (5,-1.25) {$0$};
\node at (6,-1.25) (C) {$\dots$};
\node[right] at (7.2,0) {$1$};
\node at (0,-2.25) (A1) {$z_1$};
\node at (1,-2.25) (B1) {$z_2$};
\node at (3,-2.25) {$z_N$};
\node at (4,-2.25) {$z_{N+1}$};
\node at (2,-2.25)  {$\dots$};
\draw[->,thick] (3,-2) -- (3,-1.5);
\draw[->,thick] (4,-2) -- (4,-1.5);
\draw[->,thick] (A1) -- (A);
\draw[->,thick] (B1) -- (B);
\draw[->,thick] (D1) -- (D);
\foreach \i in {3.5,4.5,5.5} {
\node at (\i,0) {$1$};
}
\foreach \i in {0,1,2,3,4,5,6} {
\node at (\i,0) [circle,fill,inner sep=1.5pt] {};
}
\end{tikzpicture}
\]
This time, $\exists\ N$ such that all vertices past the $N$-th column must take the form $\dotr{W}_{z_k/x}(0,1;0,1)=1$, and we once again conclude that entries of $\dotr{\mathcal{D}}(x)$ are meaningful.
\end{proof}

\begin{prop}
Operators $\mathcal{A}(x)$, $\mathcal{C}(x)$, $\dotr{\mathcal{D}}(x)$ satisfy the following commutation relations in ${\rm End}(\mathbb{V})$:
\begin{align}
\label{6-ACD}
[\mathcal{A}(x),\mathcal{A}(y)]
=
[\mathcal{C}(x),\mathcal{C}(y)]
=
[\dotr{\mathcal{D}}(x),\dotr{\mathcal{D}}(y)]
=
0.
\end{align}
\end{prop}

\begin{proof}
The relations \eqref{6-ACD} are a direct transcription of the $(a,b)=(0,0)$, $(a,b)=(1,0)$ cases of \eqref{6-Tcommute}, as well as a normalized version of the $(a,b)=(1,1)$ case. All of these relations survive the transition to infinite volume, in view of Proposition \ref{prop:6-ACD-define}.
\end{proof}

\begin{prop}[\cite{BorodinPetrov16} Equation (4.12)]
\label{propforCauchy61}
Fix $q,x,y,z_k \in \mathbb{C}$ such that for all $k \in \mathbb{Z}_{>0}$, the following condition holds:
\begin{align}
\label{6-eqres2}
\begin{tikzpicture}[baseline=(current bounding box.center)]
\node[scale = 0.8] at (-1.5,0) {$x$};
\node[scale = 0.8] at (-0.1,-1.2) {$z_k$};
\draw[thick,->] (-1.3,0) -- (-0.9,0);
\draw[thick,<-] (-0.1,-0.75) -- (-0.1,-1.05);
\draw[thick] (-1.7,-0.8) -- (-1.7,0.7);
\draw[->,lgray,line width=1.5pt] (-0.6,0) -- (0.4,0);
\draw[->,lgray,line width=1.5pt] (-0.1,-0.5) -- (-0.1,0.5);
\node[scale = 0.7] at (-0.7,0) {$0$};
\node[scale = 0.7] at (-0.1,-0.6) {$0$};
\node[scale = 0.7] at (-0.1,0.65) {$0$};
\node[scale = 0.7] at (0.5,0) {$0$};
\node at (-0.1,0) [circle,fill,inner sep=1.5pt] {};
\node at (0.75,0) {$\cdot$};
\begin{scope}[shift = {(2.5,0)}]
\node[scale = 0.8] at (-1.5,0) {$y$};
\node[scale = 0.8] at (-0.1,-1.2) {$z_k$};
\draw[thick,->] (-1.3,0) -- (-0.9,0);
\draw[thick,<-] (-0.1,-0.75) -- (-0.1,-1.05);
\draw[thick] (0.7,-0.8) -- (0.7,0.7);
\draw[->,lgray,line width=1.5pt] (-0.6,0) -- (0.4,0);
\draw[->,lgray,line width=1.5pt] (-0.1,-0.5) -- (-0.1,0.5);
\node[scale = 0.7] at (-0.7,0) {$1$};
\node[scale = 0.7] at (-0.1,-0.6) {$0$};
\node[scale = 0.7] at (-0.1,0.65) {$0$};
\node[scale = 0.7] at (0.5,0) {$1$};
\end{scope}
\end{tikzpicture}
=
\bigg| 
\frac{x-qz_k}{x-z_k} 
\cdot 
\frac{y-z_k}{y-qz_k}
\bigg|
<
\epsilon
<1,
\end{align}
for some constant $\epsilon \in \mathbb{R}$ which is independent of $k$. One then has the following commutation relation in ${\rm End}(\mathbb{V})$:
\begin{equation}
\label{6-eqspexchange1}
\dotr{\mathcal{D}}(x)\mathcal{C}(y)
=
\left( \frac{x-y}{x-qy} \right)
\mathcal{C}(y)\dotr{\mathcal{D}}(x).
\end{equation}
\end{prop}
\begin{proof}
Since it will be important in what follows, we note here the version of \eqref{6-eqrowrel2} obtained by dividing through by $\prod_{i=1}^{L} W_{z_i/x}(0,1;0,1)$. This converts all row operators carrying spectral parameter $x$ into their dotted counterparts:
\begin{align}
\label{6-pre-limit}
\dotr{\mathcal{T}}_{1,1}(x) \mathcal{T}_{1,0}(y)
=
W_{y/x}(0,1;0,1)
\mathcal{T}_{1,0}(y) \dotr{\mathcal{T}}_{1,1}(x)
+ 
W_{y/x}(1,0;0,1)
\mathcal{T}_{1,1}(y) \dotr{\mathcal{T}}_{1,0}(x).
\end{align}
We argue that in the limit $L \rightarrow \infty$, only the left hand side and the first term on the right hand side of \eqref{6-pre-limit} have non-vanishing matrix elements when evaluated on $\mathbb{V}$: in fact, it is clear that these terms do indeed produce \eqref{6-eqspexchange1} under the limit in question. Our attention then shifts to showing that matrix elements of the second term on the right hand side vanish as $L \rightarrow \infty$.

Examining the matrix elements of $\mathcal{T}_{1,1}(y) \dotr{\mathcal{T}}_{1,0}(x)$ as $L \rightarrow \infty$, one finds that they are given by partition functions of the form
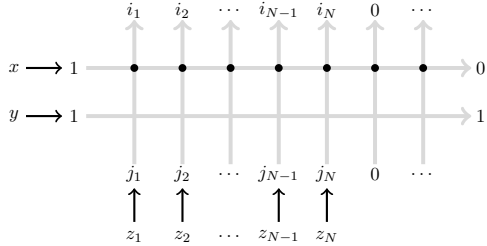
\begin{equation*}
\begin{tikzpicture}[scale=0.64,every node/.style={scale=0.7},baseline={([yshift=-0.5ex]current bounding box.center)}]
\draw[lgray,line width=1.5pt,->] (0,1) -- (8,1);
\draw[lgray,line width=1.5pt,->] (0,0) -- (8,0);
\foreach \i in {1,2,3,4,5,6,7} {
\draw[lgray,line width=1.5pt,<-] (\i,2) -- (\i,-1);
}
\foreach \i in {1,2} {
\draw[->,thick] (\i,-2.2) -- (\i,-1.5);
}
\foreach \i in {4,5} {
\draw[->,thick] (\i,-2.2) -- (\i,-1.5);
}
\foreach \i in {1,2}
{
\node at (\i,-2.5) {$z_{\i}$};
}
\node at (3,-2.5) {$\dots$};
\node at (4,-2.5) {$z_{N-1}$};
\node at (5,-2.5) {$z_{N}$};
\node at (-0.25,0) {$1$};
\node at (-0.25,1) {$1$};
\node at (8.2,1) {$0$};
\node at (8.2,0) {$1$};
\node at (1,2.2) {$i_1$};
\node at (2,2.2) {$i_2$};
\node at (1,-1.2) {$j_1$};
\node at (2,-1.2) {$j_2$};
\node at (3,-1.2) {$\dots$};
\node at (4,-1.2) {$j_{N-1}$};
\node at (5,-1.2) {$j_{N}$};
\foreach \i in {7} {
\node at (\i,-1.2) {$\dots$};
}
\node at (6,-1.2) {$0$};
\node at (3,2.2) {$\dots$};
\node at (4,2.2) {$i_{N-1}$};
\node at (5,2.2) {$i_{N}$};
\foreach \i in {7} {
\node at (\i,2.2) {$\dots$};
}
\node at (6,2.2) {$0$};
\node at (-1.5,0) {$y$};
\node at (-1.5,1) {$x$};
\draw[thick,->] (-1.25,0) -- (-0.5,0);
\draw[thick,->] (-1.25,1) -- (-0.5,1);
\foreach \i in {1,2,3,4,5,6,7} {
\node at (\i,1) [circle,fill,inner sep=1.5pt] {};
}
\end{tikzpicture}
\end{equation*}
where as previously, $\exists\ N: i_k = j_k = 0$ for all $k>N$. It is then not difficult to show that all columns beyond the $N$-th one will produce pairs of vertices of the form \eqref{6-eqres2}. Since we assume that the absolute values of these quantities are uniformly bounded below $1$, and they occur infinitely often, one finds that the above partition function vanishes.
\end{proof}

\section{Rational symmetric functions from the six-vertex model}
\label{sec:6v-functions}

The goal of this section is to introduce families of multivariate rational functions, derived as partition functions in the stochastic six-vertex model. Section \ref{ssec:6F} contains the definition of the multivariate rational function $F_S$ and Section \ref{ssec:6G} introduces its partner, the multivariate rational function $G_S$. Finally, in Section \ref{ssec:symmetry-6}, we prove that both functions are symmetric in their primary alphabet.

\subsection{Multivariate rational function $F_S$}
\label{ssec:6F}

\begin{defn}
\label{defnraf1}
Fix $N \in \mathbb{Z}_{>0}$ and let $S \in \mathfrak{s}(1)$ be a $1$-string such that $|S|=N$. Fix also two alphabets $(x_1,\dots,x_N)$, $\bm{z} = (z_1,z_2,\dots)$. We define the multivariate rational function
\begin{equation}
\label{eqrff1}
F_{S}(x_1,\dots,x_N;\bm{z}) = \bra{\varnothing} \mathcal{C}(x_1;\bm{z})\cdots \mathcal{C}(x_N;\bm{z}) \ket{S},
\end{equation}
where the row operators $\mathcal{C}(x;\bm{z})$ are given by Proposition \ref{prop:6-ACD-define}, while
\begin{equation*}
\ket{S} = \bigotimes_{k = 1}^{\infty} \ket{S_{k}}_{k}
\quad \text{ and } \quad 
\bra{\varnothing} = \bigotimes_{k = 1}^{\infty} \bra{0}_k
\end{equation*}
denote states in $\mathbb{V}$ and its dual $\mathbb{V}^{*}$. We refer to $(x_1,\dots,x_N)$ and $(z_1,z_2,\dots)$ as the primary and secondary alphabets of $F_S$.
\end{defn}

Translating \eqref{eqrff1} into its partition function form, one has that
\begin{equation}
\label{picofsixf1}
F_S(x_1,\dots,x_N;\bm{z})
=
\tikz{0.7}{1.5cm}{
\foreach \y in {1,...,4}
\draw[lgray,line width=1.5pt,->] (0,\y) -- (10,\y);
\foreach \y in {1,...,4}
\node[left] at (0,\y) {1};
\foreach \y in {1,...,4}
\node[right] at (10,\y) {0};
\foreach \x in {1,...,9}
\draw[lgray,line width=1.5pt,->] (\x,0) -- (\x,5);
\node[left] at (-1.2,1) {$x_1\rightarrow$};
\node[left] at (-1.2,2.2) {$\vdots$};
\node[left] at (-1.2,3.2) {$\vdots$};
\node[left] at (-1.2,4) {$x_N\rightarrow$};
\node[above] at (1,5) {$S_1$};
\node[above] at (2,5) {$S_2$};
\node[above] at (3,5) {$S_3$};
\node[above] at (4,5) {$\cdots$};
\node[above] at (5,5) {$\cdots$};
\node[below] at (1,0) {0};
\node[below] at (2,0) {0};
\node[below] at (3,0) {0};
\node[below] at (4,-0.2) {$\cdots$};
\node[below] at (5,-0.2) {$\cdots$};
\node[below] at (1,-0.7) {$\uparrow$};
\node[below] at (2,-0.7) {$\uparrow$};
\node[below] at (3,-0.7) {$\uparrow$};
\node[below] at (1,-1.4) {$z_1$};
\node[below] at (2,-1.4) {$z_2$};
\node[below] at (3,-1.4) {$z_3$};
\node[below] at (4,-1.4) {$\cdots$};
\node[below] at (5,-1.4) {$\cdots$};
}
\end{equation}
where all bottom incoming and right outgoing edges are assigned the state $0$. Every left incoming edge is assigned the state $1$, while top outgoing edges correspond with the entries of $S$.

\begin{rmk}
The multivariate functions \eqref{picofsixf1} are a special case of functions previously studied in \cite{Borodin17,BorodinPetrov16}. The functions studied in those earlier works are more general in two ways: (a) They depend on an extra complex parameter $s$, in addition to the variables $(x_1,\dots,x_N)$ and $q$; (b) The underlying vertex model in \cite{Borodin17,BorodinPetrov16} is a {\it higher-spin} version of the six-vertex model, and this permits the indices on vertical lattice edges (including all $S_i$, $i \geq 1$ at the top of the lattice) to take arbitrary non-negative integer values. One recovers the functions \eqref{picofsixf1} from those of \cite{Borodin17,BorodinPetrov16} by restricting each $S_i$, $i \geq 1$ to lie in $\{0,1\}$, and setting $s=q^{-1/2}$.

From an alternative perspective, the functions \eqref{picofsixf1} may be considered as {\it pre-fused} analogues of their higher-spin counterparts in \cite{Borodin17,BorodinPetrov16}: namely, one recovers the latter objects by applying the {\it fusion procedure} to the former. For more information on this matter, we refer the reader to \cite[Section 3.6]{BorodinWheeler18}.
\end{rmk}

\subsection{Multivariate rational function $G_S$}
\label{ssec:6G}

\begin{defn}
Fix $N \in \mathbb{Z}_{>0}$ and let $S \in \mathfrak{s}(1)$ be a $1$-string such that $|S|=N$. Fix a further integer $M \in \mathbb{Z}_{>0}$ and two alphabets $(y_1,\dots,y_M)$, $\bm{z} = (z_1,z_2,\dots)$. We define the multivariate rational function
\begin{equation}
\label{eqrfg2}
G_{S}(y_1,\dots,y_M;\bm{z}) = \bra{1^N,\varnothing} \mathcal{A}(y_1;\bm{z}) \cdots \mathcal{A}(y_M;\bm{z}) \ket{S},
\end{equation}
where the row operators $\mathcal{A}(y_i;\bm{z})$ are given by Proposition \ref{prop:6-ACD-define}, while
\begin{equation*}
\ket{S} = \bigotimes_{k = 1}^{\infty} \ket{S_{k}}_{k}
\quad \text{ and } \quad 
\bra{1^N,\varnothing} = \bigg(\bigotimes_{k = 1}^{N} \bra{1}_k\bigg) \otimes \bigg( \bigotimes_{k = N+1}^{\infty} \bra{0}_{k}\bigg)
\end{equation*}
denote states in $\mathbb{V}$ and its dual $\mathbb{V}^{*}$. We refer to $(y_1,\dots,y_M)$ and $(z_1,z_2,\dots)$ as the primary and secondary alphabets of $G_S$.
\end{defn}

Translating \eqref{eqrfg2} into its partition function form, one has that 
\begin{align}
\label{6-eqdualrf1}
G_S(y_1,\dots,y_M;\bm{z})
=
\tikz{0.7}{2cm}{
\foreach \y in {1,...,5}
\draw[lgray,line width=1.5pt,->] (0,\y) -- (10,\y);
\foreach \x in {1,...,9}
\draw[lgray,line width=1.5pt,->] (\x,0) -- (\x,6);
\foreach \y in {1,...,5}
\node[left] at (0,\y) {0};
\foreach \y in {1,...,5}
\node[right] at (10,\y) {0};
\node[left] at (-1.2,1) {$y_1\rightarrow$};
\node[left] at (-1.2,2.2) {$\vdots$};
\node[left] at (-1.2,3.2) {$\vdots$};
\node[left] at (-1.2,4.2) {$\vdots$};
\node[left] at (-1.2,5) {$y_M\rightarrow$};
\node[above] at (1,6) {$S_1$};
\node[above] at (2,6) {$\cdots$};
\node[above] at (3,6) {$\cdots$};
\node[above] at (4,6) {$S_N$};
\node[above] at (5.1,5.95) {$S_{N+1}$};
\node[above] at (6.2,6) {$\cdots$};
\node[above] at (7.2,6) {$\cdots$};
\node[below] at (1,0) {1};
\node[below] at (2,-0.2) {$\cdots$};
\node[below] at (3,-0.2) {$\cdots$};
\node[below] at (4,0) {1};
\node[below] at (5,0) {0};
\node[below] at (6,-0.2) {$\cdots$};
\node[below] at (7,-0.2) {$\cdots$};
\node[below] at (1,-0.7) {$\uparrow$};
\node[below] at (4,-0.7) {$\uparrow$};
\node[below] at (5,-0.7) {$\uparrow$};
\node[below] at (1,-1.4) {$z_1$};
\node[below] at (2,-1.4) {$\cdots$};
\node[below] at (3,-1.4) {$\cdots$};
\node[below] at (4,-1.4) {$z_N$};
\node[below] at (5,-1.4) {$z_{N+1}$};
\node[below] at (6,-1.4) {$\cdots$};
\node[below] at (7,-1.4) {$\cdots$};
}
\end{align}
In this picture, all left incoming and right outgoing edges are assigned the state 0. The first $N$ incoming bottom edges are assigned the state 1; the remaining ones are all set to 0. As in the case of the functions \eqref{picofsixf1}, the top outgoing edges correspond with the entries of $S$. We emphasize that in this definition, the number of horizontal rows, $M$, is independent of $N$.

\begin{rmk}
By applying the fusion procedure to the multivariate functions \eqref{6-eqdualrf1}, one is able to construct higher-spin analogues in which a total of $N$ paths enter the lattice via its first column, with no other paths entering via the remaining columns. The functions that arise from such a procedure were also studied in the previous works \cite{Borodin17,BorodinPetrov16}.
\end{rmk}

\subsection{Symmetry in primary alphabet}
\label{ssec:symmetry-6}

\begin{thm}
\label{thm:6-sym-FG}
The rational functions $F_{S}(x_1,\dots,x_N;\bm{z})$ and $G_{S}(y_1,\dots,y_M;\bm{z})$, defined in $\eqref{eqrff1}$ and $\eqref{eqrfg2}$, are symmetric in their primary alphabets.
\end{thm}

\begin{proof}
The symmetry of $F_{S}(x_1,\dots,x_N;\bm{z})$ in $(x_1,\dots,x_N)$ follows from the commutation relation $[\mathcal{C}(x_i;\bm{z}),\mathcal{C}(x_j;\bm{z})] = 0$ for $i \not= j$; see equation \eqref{6-ACD}. In a similar vein, the symmetry of $G_{S}(y_1,\dots,y_M;\bm{z})$ in $(y_1,\dots,y_M)$ follows from the commutation relation $[\mathcal{A}(y_i;\bm{z}),\mathcal{A}(y_j;\bm{z})] = 0$; see again \eqref{6-ACD}.
\end{proof}

\begin{rmk}
For generic $S \in \mathfrak{s}(1)$, neither 
$F_{S}(x_1,\dots,x_N;\bm{z})$, $G_{S}(y_1,\dots,y_M;\bm{z})$ have any symmetry properties in their secondary alphabet $\bm{z} = (z_1,z_2,\dots)$. Rather, the secondary alphabet in these functions plays an analogous role to the secondary alphabet of simpler symmetric functions, such as factorial Schur and double Grothendieck polynomials.
\end{rmk}

\section{Cauchy identities in the six-vertex model}
\label{sec:6cauchy}

In Section \ref{ssec:6v-cauchy}, we state a Cauchy summation identity for the rational functions $F_{S}$ and $G_{S}$ defined in $\eqref{eqrff1}$ and $\eqref{eqrfg2}$. The proof of this identity is split over the subsequent sections. Section \ref{ssec:6-flip} introduces a further rational function $\dotr{G}_{S}$, which is related to the original $G_S$ via the symmetry \eqref{6-eqdottondot}. Making use of this alternative formulation of $G_S$, one is able to cast the Cauchy identity in algebraic form, and compute it using commutation relations among row operators; this is done in Section \ref{ssec:6-cauchy-proof}. All of these ideas originally appeared in \cite{Borodin17,BorodinPetrov16}, in the setting of the higher-spin six-vertex model.

\subsection{Cauchy identity}
\label{ssec:6v-cauchy}

\begin{thm}
\label{thmci1}
Let $N,M \in \mathbb{Z}_{>0}$ be fixed positive integers. Fix parameters $q,x_i,y_j,z_k \in \mathbb{C}$ such that for all $i \in \{1,\dots,N\}$, $j \in \{1,\dots,M\}$ and $k \in \mathbb{Z}_{>0}$, the following conditions hold: 
\begin{equation}
\label{eqCauchyidentitywDW}
\begin{tikzpicture}[baseline={(0,0)}]
\node[scale = 0.8] at (-1.5,0) {$y_j^{-1}$};
\node[scale = 0.8] at (-0.1,-1.2) {$z_k$};
\draw[thick,->] (-1.3,0) -- (-0.9,0);
\draw[thick,<-] (-0.1,-0.75) -- (-0.1,-1.05);
\draw[thick] (-1.9,-0.8) -- (-1.9,0.7);
\draw[->,lgray,line width=1.5pt] (-0.6,0) -- (0.4,0);
\draw[->,lgray,line width=1.5pt] (-0.1,-0.5) -- (-0.1,0.5);
\node[scale = 0.7] at (-0.7,0) {$0$};
\node[scale = 0.7] at (-0.1,-0.6) {$0$};
\node[scale = 0.7] at (-0.1,0.65) {$0$};
\node[scale = 0.7] at (0.5,0) {$0$};
\node at (-0.1,0) [circle,fill,inner sep=1.5pt] {};
\node at (0.75,0) {$\cdot$};
\begin{scope}[shift = {(2.5,0)}]
\node[scale = 0.8] at (-1.5,0) {$x_i$};
\node[scale = 0.8] at (-0.1,-1.2) {$z_k$};
\draw[thick,->] (-1.3,0) -- (-0.9,0);
\draw[thick,<-] (-0.1,-0.75) -- (-0.1,-1.05);
\draw[thick] (0.8,-0.8) -- (0.8,0.7);
\draw[->,lgray,line width=1.5pt] (-0.6,0) -- (0.4,0);
\draw[->,lgray,line width=1.5pt] (-0.1,-0.5) -- (-0.1,0.5);
\node[scale = 0.7] at (-0.7,0) {$1$};
\node[scale = 0.7] at (-0.1,-0.6) {$0$};
\node[scale = 0.7] at (-0.1,0.65) {$0$};
\node[scale = 0.7] at (0.5,0) {$1$};
\end{scope}
\end{tikzpicture}
=
\bigg| \frac{(1-q y_jz_k)}{(1-y_j z_k)} \cdot 
\frac{(x_i-z_k)}{(x_i-qz_k)}\bigg|
<
\epsilon
<1,
\end{equation}
for some constant $\epsilon \in \mathbb{R}$ which is independent of $i,j,k$. For such a choice of parameters, the rational functions \eqref{eqrff1} and \eqref{eqrfg2} satisfy the summation identity
\begin{equation}
\label{eqccsix5}
\sum_{S \in \mathfrak{s}(1)} 
c_{S}(q)
F_{S}(x_1,\dots,x_N;\bm{z})
G_{S}(y_1,\dots,y_M;q^{-1}\bm{z}^{-1}) 
=
q^{N(N+1)/2}
F_{(1^{N})}(x_1,\dots,x_N;\bm{z}) 
\prod_{i=1}^N \prod_{j=1}^M \frac{1-qx_iy_j}{1-x_iy_j},
\end{equation}
where the sum is taken over all $1$-strings $S$ such that $|S|=N$. Here $\bm{z} = (z_1,z_2,\dots)$ as usual, while $q^{-1}\bm{z}^{-1} = (q^{-1}z_1^{-1},q^{-1}z_2^{-1},\dots)$. $F_{(1^{N})}$ appearing on the right hand side corresponds to the function \eqref{eqrff1} for $S=(1^N,0,0,\dots)$, and the constant $c_S(q)$ is given by
\begin{equation}
\label{eqcs1}
c_{S}(q) = \prod_{k=1}^{\infty} q^{k S_k}.
\end{equation}
\end{thm}
The proof of Theorem \ref{thmci1} is deferred to Section \ref{ssec:6-cauchy-proof}.

\begin{rmk}
\label{rmk:dwpf}
The quantity $F_{(1^N)}(x_1,\dots,x_N;\bm{z})$, which depends only on the variables $(x_1,\dots,x_N)$ and $(z_1,\dots,z_N)$, is known as the \textit{domain-wall partition function} \cite{Korepin82}. It admits an explicit $N \times N$ determinant formula \cite{Izergin87}:
\begin{align}
\label{eq:DWPF}
F_{(1^N)}(x_1,\dots,x_N;\bm{z})
=
\frac{\prod_{i,j=1}^{N} (x_i-z_j)}
{\prod_{1 \leq i<j \leq N} (x_i-x_j)(z_j-z_i)}
\det_{1 \leq i,j \leq N}
\left[ 
\frac{(1-q)z_j}{(x_i-z_j)(x_i-q z_j)} 
\right].
\end{align}
For the details of how the formula \eqref{eq:DWPF} is proved, we refer the reader to Section \ref{ssec:6-sym}.
\end{rmk}

\subsection{Flip symmetry of $G_S$}
\label{ssec:6-flip}

\begin{defn}
Fix $N \in \mathbb{Z}_{>0}$ and let $S \in \mathfrak{s}(1)$ be a $1$-string such that $|S|=N$. Fix a further integer $M \in \mathbb{Z}_{>0}$ and two alphabets $(y_1,\dots,y_M)$, $\bm{z} = (z_1,z_2,\dots)$. We define the multivariate rational function
\begin{equation}
\label{6-eqrff2}
\dotr{G}_{S}(y_1,\dots,y_M;\bm{z}) = \bra{S} \dotr{\mathcal{D}}(y_M;\bm{z}) \cdots \dotr{\mathcal{D}}(y_1;\bm{z}) \ket{1^N,\varnothing},
\end{equation}
where the row operators $\dotr{\mathcal{D}}(y_i;\bm{z})$ are given by Proposition \ref{prop:6-ACD-define}, while
\begin{equation*}
\ket{1^N,\varnothing} = \bigg(\bigotimes_{k = 1}^{N} \ket{1}_k\bigg) \otimes \bigg( \bigotimes_{k = N+1}^{\infty} \ket{0}_{k}\bigg)
\quad \text{ and } \quad 
\bra{S} = \bigotimes_{k = 1}^{\infty} \bra{S_{k}}_{k}
\end{equation*}
denote states in $\mathbb{V}$ and its dual $\mathbb{V}^{*}$.
\end{defn}

Translating \eqref{6-eqrff2} into its partition function form, it is given by
\begin{align}
\label{6-eqdualdot}
\dotr{G}_S(y_1,\dots,y_M;\bm{z})
=
\tikz{0.7}{2cm}{
\foreach \y in {1,...,5}
\draw[lgray,line width=1.5pt,->] (0,\y) -- (10,\y);
\foreach \x in {1,...,9}
\draw[lgray,line width=1.5pt,->] (\x,0) -- (\x,6);
\foreach \y in {1,...,5}
\node[left] at (0,\y) {1};
\foreach \y in {1,...,5}
\node[right] at (10,\y) {1};
\node[left] at (-1.2,1) {$y_M\rightarrow$};
\node[left] at (-1.2,2.2) {$\vdots$};
\node[left] at (-1.2,3.2) {$\vdots$};
\node[left] at (-1.2,4.2) {$\vdots$};
\node[left] at (-1.2,5) {$y_1\rightarrow$};
\node[above] at (1,6) {1};
\node[above] at (2,6) {$\cdots$};
\node[above] at (3,6) {$\cdots$};
\node[above] at (4,6) {1};
\node[above] at (5,6) {0};
\node[above] at (6.2,6) {$\cdots$};
\node[above] at (7.2,6) {$\cdots$};
\node[below] at (1,0) {$S_1$}; 
\node[below] at (2,-0.2) {$\cdots$};
\node[below] at (3,-0.2) {$\cdots$};
\node[below] at (4,0) {$S_N$}; 
\node[below] at (5.2,0) {$S_{N+1}$};
\node[below] at (6.4,-0.2) {$\cdots$};
\node[below] at (7.4,-0.2) {$\cdots$};
\node[below] at (1,-0.7) {$\uparrow$};
\node[below] at (4,-0.7) {$\uparrow$};
\node[below] at (5,-0.7) {$\uparrow$};
\node[below] at (1,-1.4) {$z_1$};
\node[below] at (2,-1.4) {$\cdots$};
\node[below] at (3,-1.4) {$\cdots$};
\node[below] at (4,-1.4) {$z_N$};
\node[below] at (5,-1.4) {$z_{N+1}$};
\node[below] at (6,-1.4) {$\cdots$};
\node[below] at (7,-1.4) {$\cdots$};
\foreach\x in {1,...,9}{
\foreach\y in {1,...,5}{
\node at (\x,\y) [circle,fill,inner sep=1.5pt] {};
}};
}
\end{align}
where the quantity \eqref{6-eqdualdot} is well-defined, despite the change in normalization of the underlying vertex weights. Indeed, one sees that the vertices in all columns sufficiently far to the right are frozen with weight $\dotr{W}_{z_j/y_i}(0,1;0,1)=1$. 

\begin{prop}
\label{prop:6-flip-sym}
Fix an integer $M \in \mathbb{Z}_{>0}$ and a $1$-string $S \in \mathfrak{s}(1)$ such that $|S|=N$. The rational functions \eqref{eqrfg2} and \eqref{6-eqrff2} are related under the following symmetry:
\begin{equation}
\label{6-eqdottondot}
q^{-N(N+1)/2}
c_{S}(q)
G_{S}(y_1,\dots,y_M;q^{-1}\bm{z}^{-1}) 
=  
\dotr{G}_{S}(y_1^{-1},\dots,y_M^{-1};\bm{z}),
\end{equation}
where the constant $c_{S}(q)$ is defined in \eqref{eqcs1}.
\end{prop}

\begin{proof}
The proof is by repeated application of the local symmetry relation \eqref{eqnormandor61}. To make the procedure more transparent, we begin by stating a version of \eqref{eqnormandor61} that may be applied to towers of vertices. In particular, one notes that
\begin{align}
\label{6-tower-sym}
\prod_{a=1}^{M} q^{m \bar{j}_a}
\begin{tikzpicture}[scale=0.8,baseline=0cm]
\draw[->,lgray,line width=1.5pt] (4,-2.75) -- (4,2.75);
\foreach \i in {-2,-1,0,1,2}
{\draw[->,lgray,line width=1.5pt] (3,\i) -- (5,\i);}
\node at (2.75,-2) {$j_M$};
\node at (2.75,-1) {$\vdots$};
\foreach \i in {0.1,1.1} {
\node at (2.75,\i) {$\vdots$};
}
\node at (2.75,2) {$j_1$};
\node at (5.35,-2) {$\ell_M$};
\node at (5.35,-1) {$\vdots$};
\foreach \i in {0.1,1.1} {
\node at (5.35,\i) {$\vdots$};
}
\node at (5.35,2) {$\ell_1$};
\node at (1.2,-2) {$y_M^{-1}$};
\node at (1.2,-1) {$\vdots$};
\node at (1.1,2) {$y_1^{-1}$};
\foreach \i in {0,1} {
\node at (1.2,\i) {$\vdots$};
}
\foreach \i in {-2,2} {
\draw[->,thick] (1.6,\i) -- (2.3,\i);
}
\node at (4,-3) {$i$};
\node at (4,3) {$k$};
\draw[->,thick] (4,-3.9) -- (4,-3.3);
\node at (4,-4.1) {$z$};
\foreach\y in {-2,...,2}{
\node at (4,\y) [circle,fill,inner sep=1.5pt] {};
}
\end{tikzpicture}
=
\prod_{a=1}^{M} q^{(m+1) \bar{\ell}_a}
\dfrac{q^{(m+1)i} (-1)^i}
{q^{(m+1)k} (-1)^k}
\begin{tikzpicture}[scale=0.8,baseline=0cm]
\begin{scope}[shift = {(8,0)}]
\draw[->,lgray,line width=1.5pt] (4,-2.75) -- (4,2.75);
\foreach \i in {-2,-1,0,1,2}
{\draw[->,lgray,line width=1.5pt] (3,\i) -- (5,\i);}
\node at (2.75,-2) {$\bar{j}_1$};
\node at (2.75,-1) {$\vdots$};
\foreach \i in {0.1,1.1} {
\node at (2.75,\i) {$\vdots$};
}
\node at (2.75,2) {$\bar{j}_M$};
\node at (5.35,-2) {$\bar{\ell}_1$};
\node at (5.35,-1) {$\vdots$};
\foreach \i in {0.1,1.1} {
\node at (5.35,\i) {$\vdots$};
}
\node at (5.35,2) {$\bar{\ell}_M$};
\node at (1.3,-2) {$y_1$};
\node at (1.3,-1) {$\vdots$};
\node at (1.3,2) {$y_M$};
\foreach \i in {0,1} {
\node at (1.3,\i) {$\vdots$};
}
\foreach \i in {-2,2} {
\draw[->,thick] (1.6,\i) -- (2.3,\i);
}
\node at (4,-3) {$k$};
\node at (4,3) {$i$};
\draw[->,thick] (4,-3.9) -- (4,-3.3);
\node at (4,-4.2) {$q^{-1}z^{-1}$};
\end{scope}
\end{tikzpicture}
\end{align}
where $m$ is an arbitrary integer. Equation \eqref{6-tower-sym} is a straightforward consequence of \eqref{eqnormandor61}, combined with conservation of paths through vertices, and appropriate telescopic cancellations of factors assigned to vertical edges.

The proof concludes by taking the lattice definition \eqref{6-eqdualdot} of $\dotr{G}_S$ and applying the relation \eqref{6-tower-sym} column by column (with $m$ initially equal to $0$, and rising in value by $1$ after each application of the relation). In particular, one finds that the procedure stabilizes: all indices $i,k,\bar{\ell}_a$ on the right hand side of \eqref{6-tower-sym} are ultimately $0$, as one traverses sufficiently far to the right through the lattice. Keeping track of all factors acquired through this process, we read off the relation
\begin{align*}
\dotr{G}_{S}(y_1^{-1},\dots,y_M^{-1};\bm{z})
=
\dfrac{1}{q^{N(N+1)/2}}
\prod_{k=1}^{\infty} q^{kS_k}
G_{S}(y_1,\dots,y_M;q^{-1}\bm{z}^{-1}),
\end{align*}
which is precisely \eqref{6-eqdottondot}.
\end{proof}

\subsection{Proof of the Cauchy identity}
\label{ssec:6-cauchy-proof}

We now return to the proof of Theorem \ref{thmci1}.

\begin{proof}
In view of the symmetry property \eqref{6-eqdottondot}, to demonstrate the Cauchy identity stated in \eqref{eqccsix5} we need to show that
\begin{equation}
\label{eqcuachyid6}
\sum_{S \in \mathfrak{s}(1)}F_{S}(x_1,\dots,x_N;\bm{z})\dotr{G}_{S}(y_1^{-1},\dots,y_M^{-1};\bm{z}) = 
F_{(1^N)}(x_1,\dots,x_N;\bm{z})
\prod_{i=1}^{N}
\prod_{j=1}^{M}
\frac{1-qx_iy_j}{1-x_iy_j}.
\end{equation}
Define an expectation value
\begin{equation}
\label{6-eqexpectation1}
\mathcal{E}_{N,M} 
= 
\bra{\varnothing} 
\mathcal{C}(x_1)
\dots 
\mathcal{C}(x_N) 
\dotr{\mathcal{D}}(y_M^{-1}) 
\dots 
\dotr{\mathcal{D}}(y_1^{-1}) 
\ket{1^N,\varnothing}.
\end{equation}
Inserting the identity operator $\sum_{S\in \mathfrak s(1)} \ket{S}\bra{S}$ between the operators $\mathcal{C}(x_N) \dotr{\mathcal{D}}(y_M^{-1})$ on the right hand side of equation $\eqref{6-eqexpectation1}$, we obtain
\begin{align*}
\mathcal{E}_{N,M}
&= 
\sum_{S \in \mathfrak{s}(1)}
\bra{\varnothing} 
\mathcal{C}(x_1)\dots \mathcal{C}(x_N) 
\ket{S}\bra{S} 
\dotr{\mathcal{D}}(y_M^{-1})
\dots \dotr{\mathcal{D}}(y_1^{-1}) 
\ket{1^N,\varnothing}
\\
&=
\sum_{S \in \mathfrak{s}(1)} 
F_{S}(x_1,\dots,x_N;\bm{z})
\dotr{G}_{S}(y_1^{-1},\dots,y_M^{-1};\bm{z}),
\end{align*}
where the final line follows from the algebraic definitions \eqref{eqrff1} and \eqref{6-eqrff2}. This matches the left hand side of \eqref{eqcuachyid6}.

On the other hand, commuting all $\mathcal{C}(x_i)$ and $\dotr{\mathcal{D}}(y_j^{-1})$ operators using the relation \eqref{6-eqspexchange1} (noting that both the relation and convergence constraint \eqref{6-eqres2} should be rewritten with $y \mapsto x_i$, $x \mapsto y_j^{-1}$), we find that
\begin{align*}
\mathcal{E}_{N,M} 
&= 
\prod_{i=1}^{N}
\prod_{j=1}^{M}
\frac{1-qx_iy_j}{1-x_iy_j}
\bra{\varnothing} 
\dotr{\mathcal{D}}(y_M^{-1})\cdots \dotr{\mathcal{D}}(y_1^{-1})
\mathcal{C}(x_1)\dots \mathcal{C}(x_N)\ket{1^N,\varnothing}.
\end{align*}
We conclude by noting that
$
\bra{\varnothing} \dotr{\mathcal{D}}(y_j^{-1})
=
\bra{\varnothing},
$
after which one has
\begin{align*}
\mathcal{E}_{N,M}
= 
\prod_{i=1}^{N}
\prod_{j=1}^{M}
\frac{1-qx_iy_j}{1-x_iy_j}
\bra{\varnothing} \mathcal{C}(x_1)\dots \mathcal{C}(x_N) \ket{1^N,\varnothing}
=
\prod_{i=1}^{N}
\prod_{j=1}^{M}
\frac{1-qx_iy_j}{1-x_iy_j}
F_{(1^N)}(x_1,\dots,x_N;\bm{z}), 
\end{align*}
which is precisely the right hand side of \eqref{eqcuachyid6}.
\end{proof}

\section{Stable symmetric functions}
\label{sec:6_stable}

It is desirable that a basis of the ring of symmetric functions exhibits {\it stability}. This property describes the invariance of the functions in question when a number of variables within their alphabet are specialized appropriately. For instance, given a Schur polynomial $s_{\lambda}(x_1,\dots,x_N)$ indexed by an arbitrary partition $\lambda$, setting $x_N=0$ yields the Schur polynomial $s_{\lambda}(x_1,\dots,x_{N-1})$ in an alphabet of one size smaller than originally. Such a stability property allows one to extend the definition of symmetric polynomials to alphabets with infinitely many variables \cite[Chapter I]{Macdonald}.

Our goal in this section is to generalize the definition \eqref{picofsixf1} of the rational symmetric functions $F_S$ in a way that renders them stable; this generalization results in yet another family of functions denoted $H_S$. These functions are defined in Section \ref{ssec:6-H}, and a direct limiting procedure that constructs them is outlined in Section \ref{ssec:6stab-construct}. Their basic properties, namely their symmetry and stability, are proved in Section \ref{ssec:6-stab}.

\subsection{Multivariate rational function $H_S$}
\label{ssec:6-H}

\begin{defn}
Fix an integer $N \in \mathbb{Z}_{>0}$ and a vector $I = (I_1,\dots,I_N) \in \{0,1\}^N$. We define the inversion number ${\rm inv}(I)$ as follows:
\begin{equation}
\label{eqinversionfor6}
\inv(I) = 
\sum_{1 \leq i<j \leq N} 
I_i (1-I_j).
\end{equation}
This coincides with the usual notion of inversions, since ${\rm inv}(I)$ is clearly equal to the cardinality of the set $\{a<b:I_a>I_b\}$.
\end{defn}

\begin{defn}
Fix two integers $K \in \mathbb{Z}_{\geq 0}$, $N \in \mathbb{Z}_{>0}$ such that $0 \leq K \leq N$ and let $S \in \mathfrak{s}(1)$ be a $1$-string such that $|S|=K$. We introduce the rational function $H_S$ as follows:
\begin{equation}
\label{eqstable6}
H_{S}(x_1,\dots,x_N;\bm{z}) = 
\sum_{I : |I| = K}
q^{-\inv(I)}
F_{S}^{I}(x_1,\dots,x_N;\bm{z}),
\end{equation}
with the sum taken over all vectors 
$I \in \{0,1\}^{N}$ such that $|I| = K$, and where
\begin{equation}
\label{eqstable6-pic}
\raisebox{-37.5mm}{
\begin{tikzpicture}[scale = 0.65]
\foreach \i in {0,1,2,3,4,6,7} {
    \draw[->,lgray,line width=1.5pt] (\i,0) -- (\i,7);}
    \foreach \i in {1,2,3,4,5,6} {
    \draw[lgray,line width=1.5pt] (-1,\i) -- (4.3,\i);}
    \foreach \i in {1,2,3,4,5,6} {
    \draw[->,lgray,line width=1.5pt] (5.5,\i) -- (8,\i);}
    \foreach \i in {1,2,3,4,5,6} {
    \node at (5,\i) {$\dots$};}
    \foreach \i in {1,2,6} {
    \node at (8.3,\i) {$0$};
    }
    \foreach \i in {3.2,4.2,5.2} {
    \node at (8.3,\i) {$\vdots$};
    }
    \node[left] at (-1.8,1) {$x_1 \to$};
    \node[left] at (-1.8,2) {$x_2 \to$};
    \node[left] at (-1.8,3) {$\vdots$};
    \node[left] at (-1.8,4) {$\vdots$};
    \node[left] at (-1.8,5) {$\vdots$};
    \node[left] at (-1.8,6) {$x_N \to$};
    \node at (-1.4,1) {$I_1$};
    \node at (-1.4,2) {$I_2$};
    \node at (-1.4,3) {$\vdots$};
    \node at (-1.4,4) {$\vdots$};
    \node at (-1.4,5) {$\vdots$};
    \node at (-1.4,6) {$I_{N}$};
    \node[left] at (-3.8,3.5) {$F_{S}^{I}(x_1,\dots,x_N;\bm{z})=$};
    \node at (0,7.5) {$S_1$};
    \node at (1,7.5) {$S_2$};
    \node at (2,7.5) {$S_3$};
    \node at (3,7.5) {$\dots$};
    \node at (4,7.5) {$\dots$};
    \node at (0,-2) {$z_1$};
    \node at (1,-2) {$z_2$};
    \node at (2,-2) {$z_3$};
    \node at (3,-2) {$\dots$};
    \node at (4,-2) {$\dots$};
    \draw[thick,->] (0,-1.5) -- (0,-0.85);
    \draw[thick,->] (1,-1.5) -- (1,-0.85);
    \draw[thick,->] (2,-1.5) -- (2,-0.85);
    \node at (0,-0.5) {$0$};
    \node at (1,-0.5) {$0$};
    \node at (2,-0.5) {$0$};
    \node at (3,-0.5) {$\dots$};
    \node at (4,-0.5) {$\dots$};
\end{tikzpicture}}
\end{equation}
which is interpreted as a partition function in the same vein as \eqref{picofsixf1}.
\end{defn}

\begin{rmk} 
The functions \eqref{eqstable6} generalize the family $F_S$. Indeed, taking $K=N$ trivializes the summation on the right hand side, and one is forced to have $(I_1,\dots,I_N) = (1^N)$. Since ${\rm inv}(I) = 0$ when $I=(1^N)$, and $F_S^{(1^N)}(x_1,\dots,x_N;\bm{z}) = F_S(x_1,\dots,x_N;\bm{z})$, one finds that $H_S = F_S$ when $K=N$.

At the other extreme, when $K=0$, no particles may enter the lattice \eqref{eqstable6-pic} and it freezes into a product of vertices $W_{z_j/x_i}(0,0;0,0)=1$. We therefore have $H_{(0,0,0,\dots)}(x_1,\dots,x_N;\bm{z}) = 1$ for any $N \geq 1$.
\end{rmk}

\subsection{Limiting procedure to obtain $H_S$}
\label{ssec:6stab-construct}

\begin{thm}
\label{6-speicalpart1}
Fix two integers $K \in \mathbb{Z}_{\geq 0}$, $N \in \mathbb{Z}_{>0}$ such that $0 \leq K \leq N$, and let $S \in \mathfrak{s}(1)$ be a $1$-string such that $|S|=K$. Fix three alphabets $(x_1,\dots,x_N)$, ${\bm u} = (u_1,\dots,u_{N-K})$ and ${\bm z} = (z_1,z_2,\dots)$ of arbitrary parameters. The following relation then holds:
\begin{equation}
\label{6-eqspeical1}
H_{S}(x_1,\dots,x_N;\bm{z}) 
= 
\lim_{\bm{u} \to \infty} 
\frac{F_{(1^{N-K},S)}(x_1,\dots,x_N;\bm{u} \cup \bm{z})}{F_{(1^{N-K})}(*;\bm{u})},
\end{equation}
where we write $\bm{u} \to \infty$ as shorthand for the limits $u_i \rightarrow \infty$, $1 \leq i \leq N-K$. The notation $(*;\bm{u})$ indicates that the primary alphabet of $F_{(1^{N-K})}$ may be chosen arbitrarily. 
\end{thm}

\begin{proof}
Our first observation is that all of the vertex weights in Figure \ref{fig:6v-weights} have a well-defined $y \rightarrow \infty$ limit. Moreover, due to the fact that the weights \eqref{eqdefn6w} depend on the spectral parameters $x,y$ only via their ratio, it is clear that all weights become independent of $x$ after taking $y \rightarrow \infty$. We shall reflect this fact by suppressing the value of horizontal spectral parameters in most of the partition functions that follow.

Now consider the lattice model interpretation of $F_{(1^{N-K},S)}(x_1,\dots,x_N;\bm{u} \cup \bm{z})$; see equation \eqref{picofsixf1}. Sending $\bm{u} \to \infty$, we obtain the following:
\begin{equation}
\label{eqqexchange6rel1}
\begin{tikzpicture}[scale = 0.6,every node/.style={scale=0.5},baseline=(current bounding box.center)]
\node[scale = 2] at (-6.5,3) {$\displaystyle\lim_{\bm{u}\to \infty} F_{(1^{N-K},S)}(x_1,\dots,x_N;\bm{u} \cup \bm{z}) = \sum_{I:|I| = K}$};
\node[scale = 1.5] at (3.5,-1.5) {$\infty$};
\node[scale = 1.5] at (4.5,-1.5) {$\dots$};
\node[scale = 1.5] at (5.5,-1.5) {$\infty$};
\draw[thick,->] (3.5,-1.2) -- (3.5,-0.5);
\draw[thick,->] (5.5,-1.2) -- (5.5,-0.5);
\foreach \i in {3.5,4.5,5.5} {
\draw[lgray,line width=1.5pt,->] (\i,0) -- (\i,7);
}
\foreach \i in {1,2,3,4,5,6} {
\draw[lgray,line width=1.5pt] (2.5,\i) -- (6,\i);
}
\draw[red,dashed] (6.75,7) -- (6.75,0);
\foreach \i in {1,2,6} {
\node[scale = 1.25] at (2.25,\i) {$1$};
}
\foreach \i in {3,4,5} {
\node at (2.25,\i) {$\vdots$};
}
\foreach \i in {3.5,5.5} {
\node[scale = 1.25] at (\i,-0.25) {$0$};
}
\node[scale = 1.25] at (3.5,7.25) {$1$};
\node at (4.5,7.25) {$\dots$};
\node[scale = 1.25] at (5.5,7.25) {$1$};
\foreach \i in {3,4,5} {
\node at (6.2,\i) {$\vdots$};
}
\node[scale = 1.25] at (6.4,1) {$I_1$};
\node[scale = 1.25] at (6.4,2) {$I_2$};
\node[scale = 1.25] at (6.4,6) {$I_N$};
\foreach \i in {3,4,5} {
\node at (7.2,\i) {$\vdots$};
}
\node[scale = 1.25] at (7.2,1) {$I_1$};
\node[scale = 1.25] at (7.2,2) {$I_2$};
\node[scale = 1.25] at (7.2,6) {$I_N$};
\foreach \i in {8,9,10,11,12,13} {
\draw[lgray,line width=1.5pt,->] (\i,0) -- (\i,7);
}
\foreach \i in {1,2,3,4,5,6} {
\draw[lgray,line width=1.5pt,->] (7.5,\i) -- (13.5,\i);
}
\node at (4.5,-0.25) {$\dots$};
\foreach \i in {8,9} {
\node[scale = 1.25] at (\i,-0.25) {$0$};
}
\foreach \i in {10,11,12,13} {
\node at (\i,-0.25) {$\dots$};
}
\node[scale = 1.25] at (8,7.25) {$S_1$};
\node[scale = 1.25] at (9,7.25) {$S_2$};
\foreach \i in {10,11,12,13} {
\node at (\i,7.25) {$\dots$};
}
\foreach \i in {1,2,6} {
\node[scale = 1.25] at (13.75,\i) {$0$};
}
\node[scale = 1.5] at (8,-1.5) {$z_1$};
\node[scale = 1.5] at (9,-1.5) {$z_2$};
\foreach \i in {3,4,5} {
\node at (13.75,\i) {$\vdots$};}
\draw[thick,->] (8,-1.2) -- (8,-0.5);
\draw[thick,->] (9,-1.2) -- (9,-0.5);
\node[scale = 1.5] at (0.5,1) {$x_1$};
\node[scale = 1.5] at (0.5,2) {$x_2$};
\node[scale = 1.5] at (0.5,6) {$x_{N}$};
\foreach \i in {3,4,5} {
\node[scale = 1.5] at (0.5,\i) {$\vdots$};
}
\foreach \i in {1,2,6} {
\draw[thick,->] (1,\i) -- (2,\i);
}
\end{tikzpicture}
\end{equation}
where the red dashed line is used to separate out the first $N-K$ vertical lines of the lattice. In view of the comment above, we note that in these first $N-K$ columns, one may assign arbitrary values to the horizontal spectral parameters. To the right of the red dashed line we recognize the lattice definition of $F_{S}^{I}(x_1,\dots,x_N;\bm{z})$. Comparing with the definition \eqref{eqstable6} of $H_S$, equation \eqref{6-eqspeical1} holds provided that
\begin{equation}
\label{6-eqlatticeb1}
\begin{tikzpicture}[scale = 0.6,every node/.style={scale=0.5},baseline=1.9cm]
\foreach \i in {4,5,6} {
\draw[lgray,line width=1.5pt,->] (\i,0) -- (\i,7);
}
\foreach \i in {1,2,3,4,5,6} {
\draw[lgray,line width=1.5pt] (3,\i) -- (6.5,\i);
}
\foreach \i in {1,2,6} {
\node[scale = 1.25] at (2.8,\i) {$1$};
}
\foreach \i in {3,4,5} {
\node at (2.8,\i) {$\vdots$};
}
\foreach \i in {4,6} {
\node[scale = 1.25] at (\i,-0.2) {$0$};
\node[scale = 1.25] at (\i,7.2) {$1$};
}
\node at (5,7.2) {$\dots$};
\node at (5,-0.2) {$\dots$};
\foreach \i in {3,4,5} {
\node at (6.8,\i) {$\vdots$};
}
\node[scale = 1.25] at (6.8,1) {$I_1$};
\node[scale = 1.25] at (6.8,2) {$I_2$};
\node[scale = 1.25] at (6.8,6) {$I_{N}$};
\draw[->,thick] (4,-1.2) -- (4,-0.5);
\draw[->,thick] (6,-1.2) -- (6,-0.5);
\node[scale = 1.5] at (4,-1.5) {$\infty$};
\node[scale = 1.5] at (5,-1.5) {$\dots$};
\node[scale = 1.5] at (6,-1.5) {$\infty$};
\node[scale = 1.25] at (1,1) {$*$};
\node[scale = 1.25] at (1,2) {$*$};
\node[scale = 1.25] at (1,6) {$*$};
\draw[thick,->] (1.5,1) -- (2.5,1);
\draw[thick,->] (1.5,2) -- (2.5,2);
\draw[thick,->] (1.5,6) -- (2.5,6);
\end{tikzpicture}
=
q^{-\inv(I)}
\lim_{\bm{u}\to \infty}
F_{(1^{N-K})}(*;\bm{u})
\end{equation}
for all $(I_1,\dots,I_N) \in \{0,1\}^N$ such that $|I|=K$. To prove equation \eqref{6-eqlatticeb1} we require the following lemma, that allows us to incrementally modify the boundary conditions at the right edges of the partition function.

\begin{lem}
\label{lem:qexchange6}
Fix arbitrary vectors $S,T \in \{0,1\}^N$. The following equality of partition functions holds:
\begin{equation}
\label{6qexchange1}
q \times \left(
\begin{tikzpicture}[scale=0.64,every node/.style={scale=0.7},baseline={([yshift=-0.5ex]current bounding box.center)}]
\draw[lgray,line width=1.5pt] (0,1) -- (2.5,1);
\draw[lgray,line width=1.5pt] (0,0) -- (2.5,0);
\draw[lgray,line width=1.5pt,->] (3.5,1) -- (6,1);
\draw[lgray,line width=1.5pt,->] (3.5,0) -- (6,0);
\node[scale = 1.5] at (3,0) {$\dots$};
\node[scale = 1.5] at (3,1) {$\dots$};
\foreach \i in {1,2,4,5} {
\draw[lgray,line width=1.5pt,<-] (\i,2) -- (\i,-1);
}
\foreach \i in {1,2} {
\draw[->,thick] (\i,-2.2) -- (\i,-1.5);
}
\foreach \i in {4,5} {
\draw[->,thick] (\i,-2.2) -- (\i,-1.5);
}
\node at (1,-2.5) {$\infty$};
\node at (2,-2.5) {$\infty$};
\node at (4,-2.5) {$\infty$};
\node at (5,-2.5) {$\infty$};
\node at (-0.25,0) {$1$};
\node at (-0.25,1) {$1$};
\node at (6.2,1) {$0$};
\node at (6.2,0) {$1$};
\foreach \i in {1,2} {
\node at (\i,2.2) {$T_{\i}$};
}
\node at (4,2.2) {$T_{N-1}$};
\node at (5,2.2) {$T_{N}$};
\foreach \i in {1,2} {
\node at (\i,-1.2) {$S_{\i}$};
}
\node at (4,-1.2) {$S_{N-1}$};
\node at (5,-1.2) {$S_{N}$};
\node at (-1.5,0) {$*$};
\node at (-1.5,1) {$*$};
\draw[thick,->] (-1.25,0) -- (-0.5,0);
\draw[thick,->] (-1.25,1) -- (-0.5,1);
\end{tikzpicture}\right)
= 
\begin{tikzpicture}[scale=0.64,every node/.style={scale=0.7},baseline={([yshift=-0.5ex]current bounding box.center)}]
\node at (8.5,0) {$*$};
\node at (8.5,1) {$*$};
\draw[->,thick] (9,0) -- (9.75,0);
\draw[->,thick] (9,1) -- (9.75,1);
\draw[lgray,line width=1.5pt] (10.25,1) -- (12.75,1);
\draw[lgray,line width=1.5pt,->] (13.75,1) -- (16.25,1);
\draw[lgray,line width=1.5pt] (10.25,0) -- (12.75,0);
\draw[lgray,line width=1.5pt,->] (13.75,0) -- (16.25,0);
\node[scale = 1.5] at (13.25,0) {$\dots$};
\node[scale = 1.5] at (13.25,1) {$\dots$};
\foreach \i in {11.25,12.25,14.25,15.25} 
{
\draw[lgray,line width=1.5pt,<-] (\i,2) -- (\i,-1);
}
\node at (10,1) {$1$};
\node[below] at (10,0.25) {$1$};
\node at (16.5,1) {$1$};
\node at (16.5,0) {$0$};
\foreach \i in {1,2}{
\node at (\i+10.25,2.2) {$T_{\i}$};
}
\node at (14.25,2.2) {$T_{N-1}$};
\node at (15.25,2.2) {$T_{N}$};
\foreach \i in {1,2} {
\node at (\i+10.25,-1.2) {$S_{\i}$};
}
\node at (14.25,-1.2) {$S_{N-1}$};
\node at (15.25,-1.2) {$S_{N}$};
\foreach \i in {11.2,12.2} {
\draw[->,thick] (\i,-2.2) -- (\i,-1.5);
}
\foreach \i in {14.2,15.2} {
\draw[->,thick] (\i,-2.2) -- (\i,-1.5);
}
\node at (11.2,-2.5) {$\infty$};
\node at (12.2,-2.5) {$\infty$};
\node at (14.2,-2.5) {$\infty$};
\node at (15.2,-2.5) {$\infty$};
\end{tikzpicture}
\end{equation}

\end{lem}

\begin{proof}
Following very similar ideas to Proposition \ref{prop:6-commute-nontrivial}, one may derive the following exchange relation between row-operators (of arbitrary length):
\begin{align}
\label{6-general-commute}
\mathcal{T}_{1,1}(x;\bm{z}) \mathcal{T}_{1,0}(y;\bm{z})
=
\sum_{c,d:c+d=1}
W_{y/x}(d,c;0,1)
\mathcal{T}_{1,d}(y;\bm{z}) 
\mathcal{T}_{1,c}(x;\bm{z}).
\end{align}
After sending $\bm{z} \rightarrow \infty$, the row operators in \eqref{6-general-commute} become independent of their primary arguments:
\begin{align}
\label{6-general-commute2}
\mathcal{T}_{1,1}(*;\infty) 
\mathcal{T}_{1,0}(*;\infty)
=
\sum_{c,d:c+d=1}
W_{y/x}(d,c;0,1)
\mathcal{T}_{1,d}(*;\infty) 
\mathcal{T}_{1,c}(*;\infty),
\end{align}
which holds for arbitrary values of the parameter $y/x$. Consulting the weights in Figure \ref{fig:6v-weights} and sending $y/x \rightarrow \infty$, one finds that
\begin{align*}
W_{\infty}(d,c;0,1)
=
\bm{1}_{c=1}
\cdot
\bm{1}_{d=0}
\cdot
q^{-1},
\end{align*}
and \eqref{6qexchange1} then follows directly from \eqref{6-general-commute2}.

\end{proof}

An immediate corollary of the exchange relation \eqref{6qexchange1} is that the quantity
\begin{align}
\label{6-chi-inv}
\chi(I_1,\dots,I_N)
:=
q^{\inv(I)}
\times
\left(
\begin{tikzpicture}[scale = 0.6,every node/.style={scale=0.5},baseline=1.6cm]
\foreach \i in {4,5,6} {
\draw[lgray,line width=1.5pt,->] (\i,0) -- (\i,7);
}
\foreach \i in {1,2,3,4,5,6} {
\draw[lgray,line width=1.5pt] (3,\i) -- (6.5,\i);
}
\foreach \i in {1,2,6} {
\node[scale = 1.25] at (2.8,\i) {$1$};
}
\foreach \i in {3,4,5} {
\node at (2.8,\i) {$\vdots$};
}
\foreach \i in {4,6} {
\node[scale = 1.25] at (\i,-0.2) {$0$};
\node[scale = 1.25] at (\i,7.2) {$1$};
}
\node at (5,7.2) {$\dots$};
\node at (5,-0.2) {$\dots$};
\foreach \i in {3,4,5} {
\node at (6.8,\i) {$\vdots$};
}
\node[scale = 1.25] at (6.8,1) {$I_1$};
\node[scale = 1.25] at (6.8,2) {$I_2$};
\node[scale = 1.25] at (6.8,6) {$I_{N}$};
\draw[->,thick] (4,-1.2) -- (4,-0.5);
\draw[->,thick] (6,-1.2) -- (6,-0.5);
\node[scale = 1.5] at (4,-1.5) {$\infty$};
\node[scale = 1.5] at (5,-1.5) {$\dots$};
\node[scale = 1.5] at (6,-1.5) {$\infty$};
\node[scale = 1.25] at (1,1) {$*$};
\node[scale = 1.25] at (1,2) {$*$};
\node[scale = 1.25] at (1,6) {$*$};
\draw[thick,->] (1.5,1) -- (2.5,1);
\draw[thick,->] (1.5,2) -- (2.5,2);
\draw[thick,->] (1.5,6) -- (2.5,6);
\end{tikzpicture} \right)
\end{align}
is invariant for all $(I_1,\dots,I_N) \in \{0,1\}^N$ such that $|I|=K$. Indeed, any reordering of $(I_1,\dots,I_N)$ may be achieved by iteration of the relation \eqref{6qexchange1}. We are therefore at liberty to choose $(I_1,\dots,I_N)$ in any way that facilitates the computation of $\chi(I_1,\dots,I_N)$. Choosing $(I_1,\dots,I_N)=(0^{N-K},1^K)$, equation \eqref{6-chi-inv} becomes
\begin{align}
\label{6-chi-inv2}
\chi(0^{N-K},1^K)
=
\begin{tikzpicture}[scale = 0.6,every node/.style={scale=0.5},baseline=1.9cm]
\begin{scope}[shift = {(0,1)}]
\foreach \i in {4,5,6} {
\draw[lgray,line width=1.5pt,->] (\i,3.75) -- (\i,7);
}
\foreach \i in {4,5,6} {
\draw[lgray,line width=1.5pt] (3,\i) -- (6.5,\i);
}
\node[scale = 1.25] at (6.8,4) {$1$};
\node[scale = 1.25] at (6.8,5) {$\vdots$};
\node[scale = 1.25] at (6.8,6) {$1$};
\node[scale = 1.25] at (1,4) {$*$};
\node[scale = 1.25] at (1,6) {$*$};
\draw[thick,->] (1.5,4) -- (2.5,4);
\draw[thick,->] (1.5,6) -- (2.5,6);
\foreach \i in {4,6} {
\node[scale = 1.25] at (\i,7.2) {$1$};
}
\node at (5,7.2) {$\dots$};
\foreach \i in {4,6} {
\node[scale = 1.25] at (2.8,\i) {$1$};
}
\node at (2.8,5) {$\vdots$};
\node at (4,3.5) {$1$};
\node at (5,3.5) {$\dots$};
\node at (6,3.5) {$1$};
\end{scope}
\draw[red,dashed] (3,4.1) -- (6.5,4.1);
\node at (4,3.75) {$1$};
\node at (5,3.75) {$\dots$};
\node at (6,3.75) {$1$};
\foreach \i in {4,5,6} {
\draw[lgray,line width=1.5pt] (\i,0) -- (\i,3.5);
}
\foreach \i in {1,2,3} {
\draw[lgray,line width=1.5pt] (3,\i) -- (6.5,\i);
}
\foreach \i in {1,3} {
\node[scale = 1.25] at (2.8,\i) {$1$};
}
\foreach \i in {2} {
\node at (2.8,\i) {$\vdots$};
}
\foreach \i in {4,6} {
\node[scale = 1.25] at (\i,-0.2) {$0$};
}
\node at (5,-0.2) {$\dots$};
\node[scale = 1.25] at (6.8,1) {$0$};
\node[scale = 1.25] at (6.8,2) {$\vdots$};
\node[scale = 1.25] at (6.8,3) {$0$};
\draw[->,thick] (4,-1.2) -- (4,-0.5);
\draw[->,thick] (6,-1.2) -- (6,-0.5);
\node[scale = 1.5] at (4,-1.5) {$\infty$};
\node[scale = 1.5] at (5,-1.5) {$\dots$};
\node[scale = 1.5] at (6,-1.5) {$\infty$};
\node[scale = 1.25] at (1,1) {$*$};
\node[scale = 1.25] at (1,3) {$*$};
\draw[thick,->] (1.5,1) -- (2.5,1);
\draw[thick,->] (1.5,3) -- (2.5,3);
\draw [decorate,decoration={brace,amplitude=5pt,mirror,raise=4ex}]
  (4,-0.8) -- (6,-0.8) node[midway,yshift=-5.5em]{$N-K$};
\draw [decorate,decoration={brace,amplitude=5pt,mirror,raise=4ex}]
  (6.25,1) -- (6.25,3) node[midway,xshift=7em]{$N-K$};
  \draw [decorate,decoration={brace,amplitude=5pt,mirror,raise=4ex}]
  (6.25,5) -- (6.25,7) node[midway,xshift=5.75em]{$K$};
\end{tikzpicture}
=
\lim_{\bm{u}\to \infty}
F_{(1^{N-K})}(*;\bm{u}),
\end{align}
where the final equality follows by noting that all vertices above the red dashed line are frozen: all of these vertices have the weight $W_{\infty}(1,1;1,1)=1$. Comparison of \eqref{6-chi-inv} and \eqref{6-chi-inv2} proves \eqref{6-eqlatticeb1}, thereby completing the proof of Theorem \ref{6-speicalpart1}.
\end{proof}

\subsection{An aside about the domain-wall partition function}

We now compute the limit of the domain-wall partition function $\lim_{\bm{u}\to \infty} F_{(1^{N-K})}(*;\bm{u})$ that surfaced throughout the computations of the previous subsection. We do this mainly for completeness, as we will not require it in the sections which follow.  
 
\begin{prop}
\label{domanwallpart1}
Define the following quantity:
\begin{equation*}
Z_{N}(q)
:=
\lim_{\bm{u} \to \infty} F_{(1^N)}(*;\bm{u}).
\end{equation*}
We then have the explicit evaluation
\begin{align}
\label{eqspsixdomainwall1}
Z_N(q)
=
\prod_{i = 1}^{N} (1-q^{-i}).
\end{align}
\end{prop}
\begin{proof}
When the ratio $y/x$ is set to $\infty$, the weights of the $R$-matrix \eqref{eqrm61} become
\begin{align*}
&
W_{\infty}(0,0;0,0) = W_{\infty}(1,1;1,1) = 
W_{\infty}(1,0;1,0) = 1,
\\
&
W_{\infty}(1,0;0,1) = 0, \quad
W_{\infty}(0,1;1,0) = -\frac{1-q}{q}, \quad
W_{\infty}(0,1;0,1) = \frac{1}{q}.
\end{align*}
Since $W_{\infty}(1,0;0,1) = 0$, the configurations in $F_{(1^N)}(*;\bm{u})$ which survive the limit $\bm{u} \rightarrow \infty$ are in bijection with the permutations of $\mathfrak{S}_N$. To see this, we note that the non-vanishing configurations of $F_{(1^N)}(*;\bm{u})$ have a unique vertex $W_{\infty}(0,1;1,0)$ in each row and column of the lattice. Replacing each such vertex by a $1$ and all other vertices by $0$, one recovers a permutation matrix $\sigma \in \mathfrak{S}_N$ (see Example \ref{ex:dwlimit}).

Now let us fix $\sigma \in \mathfrak{S}_N$ and compute the weight of the corresponding configuration. One finds that the frequencies of the vertices $W_{\infty}(0,1;1,0)$ and $W_{\infty}(0,1;0,1)$ are given by
\begin{equation*}
\#_{\sigma}\{ W_{\infty}(0,1;1,0) \} = N, 
\quad 
\#_{\sigma}\{ W_{\infty}(0,1;0,1) \} = 
\binom{N}{2} - \inv(\sigma),
\end{equation*}
where $\#_{\sigma}\{W_{\infty}(i,j;k,\ell)\}$ denotes the count of vertices of type $W_{\infty}(i,j;k,\ell)$ within the configuration $\sigma$. Summing the weights across each configuration $\sigma \in \mathfrak{S}_N$ we then obtain
\begin{equation}
Z_N(q) 
= 
(-1)^N 
(1-q)^N 
q^{-\binom{N+1}{2}} 
\sum_{\sigma \in \mathfrak{S}_N} 
q^{\inv(\sigma)}.
\end{equation}
The sum on the right hand side is the Poincar\'e polynomial of $\mathfrak{S}_N$, which yields the factorization
\begin{equation}
(-1)^N 
(1-q)^N 
q^{-\binom{N+1}{2}}
\sum_{\sigma \in \mathfrak{S}_N} 
q^{\inv(\sigma)} 
= 
\prod_{i = 1}^{N} (1-q^{-i}),
\end{equation}
completing the proof of \eqref{eqspsixdomainwall1}.
\end{proof}

\begin{ex}\label{ex:dwlimit}
For $N = 3$, the lattice configurations of $Z_3(q)$ are listed as follows (left panel):
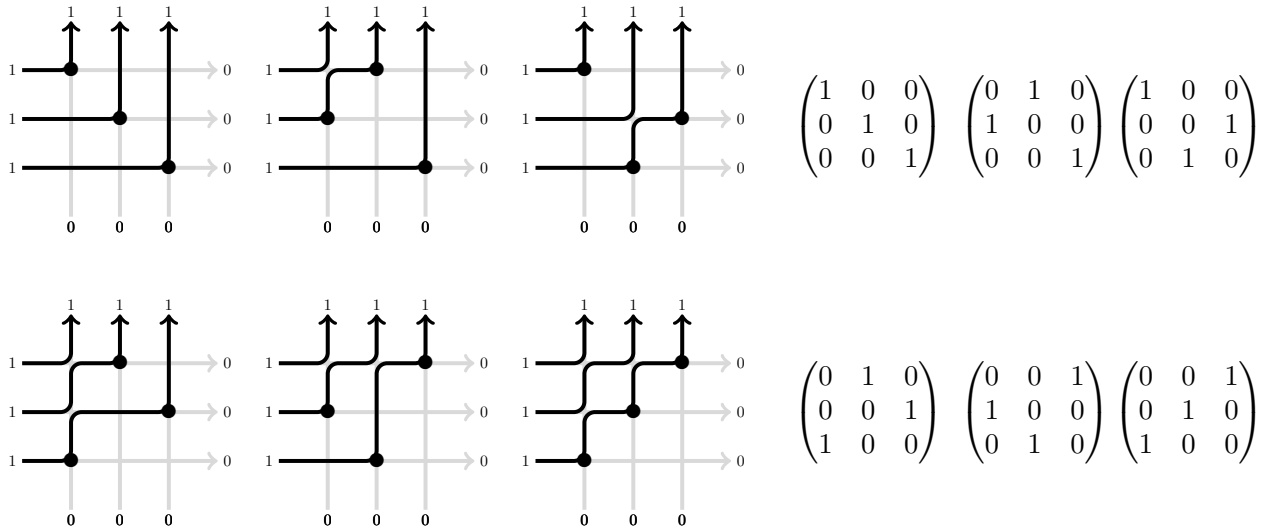
\[
\begin{tikzpicture}[scale = 0.65,every node/.style={scale=0.6}]
\foreach \i in {1,2,3}
{
\draw[->,lgray,line width=1.5pt] (0,\i) -- (4,\i);
}
\foreach \i in {1,2,3} {
\draw[->,lgray,line width=1.5pt] (\i,0) -- (\i,4);
}
\foreach \i in {1,2,3} {
\node at (4.2,\i) {$0$};
\foreach \i in {1,2,3} {
\node at (\i,-0.2) {$0$};
}
}
\foreach \i in {1,2,3} {
\node at (-0.2,\i) {$1$};
}
\foreach \i in {1,2,3} {
\node at (\i,4.2) {$1$};
}
\draw[line width = 1.5,->,rounded corners] (0,1) -- (3,1) -- (3,4);
\draw[line width = 1.5,->,rounded corners] (0,2) -- (2,2) -- (2,4);
\draw[line width = 1.5,->,rounded corners] (0,3) -- (1,3) -- (1,4);
\node[scale = 2.5] at (1,3) {$\bullet$};
\node[scale = 2.5] at (2,2) {$\bullet$};
\node[scale = 2.5] at (3,1) {$\bullet$};
\begin{scope}[shift = {(-0.75,0)}]
\foreach \i in {1,2,3}
{
\draw[->,lgray,line width=1.5pt] (6,\i) -- (10,\i);
}
\foreach \i in {7,8,9} {
\draw[->,lgray,line width=1.5pt] (\i,0) -- (\i,4);
}
\foreach \i in {1,2,3} {
\node at (10.2,\i) {$0$};
\foreach \i in {7,8,9} {
\node at (\i,-0.2) {$0$};
}
}
\foreach \i in {1,2,3} {
\node at (5.8,\i) {$1$};
}
\foreach \i in {7,8,9} {
\node at (\i,4.2) {$1$};
}
\draw[line width = 1.5,->,rounded corners] (6,3) -- (7,3) -- (7,4);
\draw[line width = 1.5,->,rounded corners] (6,2) -- (7,2) -- (7,3) -- (8,3) -- (8,4);
\draw[line width = 1.5,->,rounded corners] (6,1) -- (9,1) -- (9,3) -- (9,4);
\node[scale = 2.5] at (8,3) {$\bullet$};
\node[scale = 2.5] at (7,2) {$\bullet$};
\node[scale = 2.5] at (9,1) {$\bullet$};
\end{scope}
\begin{scope}[shift = {(-1.5,0)}]
\foreach \i in {1,2,3}
{
\draw[->,lgray,line width=1.5pt] (12,\i) -- (16,\i);
}
\foreach \i in {13,14,15} {
\draw[->,lgray,line width=1.5pt] (\i,0) -- (\i,4);
}
\foreach \i in {1,2,3} {
\node at (16.2,\i) {$0$};
\foreach \i in {13,14,15} {
\node at (\i,-0.2) {$0$};
}
}
\foreach \i in {1,2,3} {
\node at (11.8,\i) {$1$};
}
\foreach \i in {13,14,15} {
\node at (\i,4.2) {$1$};
}
\draw[line width = 1.5,->,rounded corners] (12,3) -- (13,3) -- (13,4);
\draw[line width = 1.5,->,rounded corners] (12,2) -- (14,2) -- (14,4);
\draw[line width = 1.5,->,rounded corners] (12,1) -- (14,1) -- (14,2) -- (15,2) -- (15,4);
\node[scale = 2.5] at (13,3) {$\bullet$};
\node[scale = 2.5] at (15,2) {$\bullet$};
\node[scale = 2.5] at (14,1) {$\bullet$};
\end{scope}
\foreach \i in {-5,-4,-3}
{
\draw[->,lgray,line width=1.5pt] (0,\i) -- (4,\i);
}
\foreach \i in {1,2,3} {
\draw[->,lgray,line width=1.5pt] (\i,-6) -- (\i,-2);
}
\foreach \i in {-5,-4,-3} {
\node at (4.2,\i) {$0$};
\foreach \i in {1,2,3} {
\node at (\i,-6.2) {$0$};
}
}
\foreach \i in {-5,-4,-3} {
\node at (-0.2,\i) {$1$};
}
\foreach \i in {1,2,3} {
\node at (\i,-1.8) {$1$};
}
\draw[line width = 1.5,->,rounded corners] (0,-4) -- (1,-4) -- (1,-3) -- (2,-3) -- (2,-2);
\draw[line width = 1.5,->,rounded corners] (0,-3) -- (1,-3) -- (1,-2);
\draw[line width = 1.5,->,rounded corners] (0,-5) -- (1,-5) -- (1,-4) -- (3,-4) -- (3,-2);
\node[scale = 2.5] at (2,-3) {$\bullet$};
\node[scale = 2.5] at (3,-4) {$\bullet$};
\node[scale = 2.5] at (1,-5) {$\bullet$};
\begin{scope}[shift = {(-0.75,0)}]
\foreach \i in {-5,-4,-3}
{
\draw[->,lgray,line width=1.5pt] (6,\i) -- (10,\i);
}
\foreach \i in {7,8,9} {
\draw[->,lgray,line width=1.5pt] (\i,-6) -- (\i,-2);
}
\foreach \i in {-5,-4,-3} {
\node at (10.2,\i) {$0$};
\foreach \i in {7,8,9} {
\node at (\i,-6.2) {$0$};
}
}
\foreach \i in {-5,-4,-3} {
\node at (5.8,\i) {$1$};
}
\foreach \i in {7,8,9} {
\node at (\i,-1.8) {$1$};
}
\draw[line width = 1.5,->,rounded corners] (6,-3) -- (7,-3) -- (7,-2);
\draw[line width = 1.5,->,rounded corners] (6,-4) -- (7,-4) -- (7,-3) -- (8,-3) -- (8,-2);
\draw[line width = 1.5,->,rounded corners] (6,-5) -- (8,-5) -- (8,-3) -- (9,-3) -- (9,-2);
\node[scale = 2.5] at (9,-3) {$\bullet$};
\node[scale = 2.5] at (7,-4) {$\bullet$};
\node[scale = 2.5] at (8,-5) {$\bullet$};
\end{scope}
\begin{scope}[shift = {(-1.5,0)}]
\foreach \i in {-5,-4,-3}
{
\draw[->,lgray,line width=1.5pt] (12,\i) -- (16,\i);
}
\foreach \i in {13,14,15} {
\draw[->,lgray,line width=1.5pt] (\i,-6) -- (\i,-2);
}
\foreach \i in {-5,-4,-3} {
\node at (16.2,\i) {$0$};
\foreach \i in {13,14,15} {
\node at (\i,-6.2) {$0$};
}
}
\foreach \i in {-5,-4,-3} {
\node at (11.8,\i) {$1$};
}
\foreach \i in {13,14,15} {
\node at (\i,-1.8) {$1$};
}
\draw[line width = 1.5,->,rounded corners] (12,-3) -- (13,-3) -- (13,-2);
\draw[line width = 1.5,->,rounded corners] (12,-4) -- (13,-4) -- (13,-3) -- (14,-3) -- (14,-2);
\draw[line width = 1.5,->,rounded corners] (12,-5) -- (13,-5) -- (13,-4) -- (14,-4) -- (14,-3) -- (15,-3) -- (15,-2);
\node[scale = 2.5] at (15,-3) {$\bullet$};
\node[scale = 2.5] at (14,-4) {$\bullet$};
\node[scale = 2.5] at (13,-5) {$\bullet$};
\end{scope} 
\node[scale = 1.8] at (17.25,1.85) {$\begin{pmatrix}
	1 & 0 & 0\\
	0 & 1 & 0\\
	0 & 0 & 1
\end{pmatrix}$};
\node[scale = 1.8] at (20.65,1.85) {$\begin{pmatrix}
	0 & 1 & 0\\
	1 & 0 & 0\\
	0 & 0 & 1
\end{pmatrix}$};
\node[scale = 1.8] at (23.8,1.85) {$\begin{pmatrix}
	1 & 0 & 0\\
	0 & 0 & 1\\
	0 & 1 & 0
\end{pmatrix}$};
\begin{scope}[shift = {(0,-4.35)}]
\node[scale = 1.8] at (17.25,0.35) {$\begin{pmatrix}
	0 & 1 & 0\\
	0 & 0 & 1\\
	1 & 0 & 0
\end{pmatrix}$};
\node[scale = 1.8] at (20.65,0.35) {$\begin{pmatrix}
	0 & 0 & 1\\
	1 & 0 & 0\\
	0 & 1 & 0
\end{pmatrix}$};
\node[scale = 1.8] at (23.8,0.35) {$\begin{pmatrix}
	0 & 0 & 1\\
	0 & 1 & 0\\
	1 & 0 & 0
\end{pmatrix}$};
\end{scope}
\end{tikzpicture}
\]
Once can recognize the six permutation matrices of $\mathfrak{S}_{3}$ by writing $1$ in place of each vertex $W_\infty(0,1;1,0)$ and $0$ in place of all other vertices (right panel). The weights of these configurations sum up to yield 
\[
\lim_{\bm{u} \to \infty} F_{(1^3)}(x_1,x_2,x_3;\bm{u}) = \bigg(\frac{q-1}{q}\bigg)^{3} \bigg(1 + q^{-1} + q^{-1} +  q^{-2} +q^{-2}+ q^{-3} \bigg) = \prod_{i = 1}^{3}(1-q^{-i}).
\]
\end{ex}

\subsection{Symmetry and stability of $H_S$}
\label{ssec:6-stab}

The first property of the functions $\eqref{eqstable6}$ is their symmetry with respect to their primary alphabet:
\begin{prop}
Fix two integers $K \in \mathbb{Z}_{\geq 0}$, $N \in \mathbb{Z}_{>0}$ such that $0 \leq K \leq N$ and let $S \in \mathfrak{s}(1)$ be a $1$-string such that $|S|=K$. Then the rational function $H_{S}(x_1,\dots,x_N;\bm{z})$ is symmetric with respect to $(x_1,\dots,x_N)$.
\end{prop}

\begin{proof}
This is an immediate consequence of Theorem \ref{6-speicalpart1}, in view of the fact that the right hand side of \eqref{6-eqspeical1} is symmetric in $(x_1,\dots,x_N)$.
\end{proof}

The second property is that the rational functions $\eqref{eqstable6}$ are {\it stable} with respect to their primary alphabet:

\begin{prop}
Fix two integers $K \in \mathbb{Z}_{\geq 0}$, $N \in \mathbb{Z}_{>0}$ such that $0 \leq K \leq N$ and let $S \in \mathfrak{s}(1)$ be a $1$-string such that $|S|=K$. We then have the reduction property
\begin{align}
\label{stability-6v}
H_S(x_1,\dots,x_{N-1},x_N=\infty;\bm{z})
=
\left\{
\begin{array}{ll}
H_S(x_1,\dots,x_{N-1};\bm{z}),
&
\quad
|S| \leq N-1,
\\ \\
0,
&
\quad
|S|=N.
\end{array}
\right.
\end{align}
\end{prop}

\begin{proof}
Since $H_S(x_1,\dots,x_N;\bm{z})$ is symmetric with respect to its primary alphabet, we may equivalently study the limit as $x_1 \rightarrow \infty$. This limit affects the bottommost row of vertices in the partition function \eqref{eqstable6-pic}. Consulting the table in Figure \ref{fig:6v-weights}, one finds that all weights $W_{z/x}(i,j;k,\ell)$ have a well-defined limit as $x \rightarrow \infty$, and furthermore
\begin{align}
\label{6-vanishing-weights}
\lim_{x \rightarrow \infty}
W_{z/x}(0,1;1,0)
=
0.
\end{align}
It follows that $F^{I}_S(x_1=\infty,x_2,\dots,x_N;\bm{z})$ vanishes if the vertex \eqref{6-vanishing-weights} occurs within the first row of the partition function \eqref{eqstable6-pic}. 

Let us begin by assuming that $|S|<N$. The only way to prevent the appearance of the vertex \eqref{6-vanishing-weights} is to fix $I_1=0$, which causes the entire bottom row of the lattice to freeze away (since it is then comprised solely of $W_{0}(0,0;0,0)$ vertices). As such, for any $I=(I_1,\dots,I_N)$ we find that
\begin{align*}
F^{I}_S(x_1=\infty,x_2,\dots,x_N;\bm{z})
=
\bm{1}_{I_1=0}
\cdot
F^{\tilde{I}}_S(x_2,\dots,x_N;\bm{z}),
\qquad
\tilde{I}=(I_2,\dots,I_N),
\end{align*}
which translates into the property
\begin{align}
\label{6-H-reduce}
H_{S}(x_1=\infty,x_2,\dots,x_N;\bm{z}) = 
\sum_{\tilde{I} : |\tilde{I}| = K} 
q^{-\inv(\tilde{I})}
F_{S}^{\tilde{I}}(x_2,\dots,x_N;\bm{z})
=
H_{S}(x_2,\dots,x_N;\bm{z}),
\end{align}
with the sum taken over vectors 
$\tilde{I}=(I_2,\dots,I_N)$, and where the first equality in \eqref{6-H-reduce} makes use of the fact that 
\begin{align*}
\inv(0,I_2,\dots,I_N) = \inv(I_2,\dots,I_N) = \inv(\tilde{I}).
\end{align*}
This completes the proof of \eqref{stability-6v} in the case $|S| < N$. 

The case $|S|=N$ is handled very easily, since in this situation the vector that indexes the left incoming edges of the partition function \eqref{eqstable6-pic} is forced, by conservation, to be $(I_1,\dots,I_N)=(1^N)$. As such, it is impossible to choose $I_1=0$, and $F^{I}_S(x_1=\infty,x_2,\dots,x_N;\bm{z})$ vanishes identically.
\end{proof}

\section{Cauchy identity for stable symmetric functions}
\label{ssec:6v-stable-cauchy}

In Section \ref{sec:6cauchy} we stated a Cauchy-type identity between the rational symmetric functions $\eqref{eqrff1}$ and $\eqref{eqrfg2}$. Although the right hand side of the identity \eqref{eqccsix5} may be expressed in terms of Izergin's determinant \cite{Izergin87}, it does not fully factorize. In this section we show that the stable symmetric functions \eqref{eqstable6} are more natural from this point of view, and satisfy a Cauchy summation identity with a fully factorized right hand side.

\begin{thm}
\label{thm:6-stable-cauchy}
Let $N,M \in \mathbb{Z}_{>0}$ be fixed positive integers. Fix parameters $q,x_i,y_j,z_k \in \mathbb{C}$ such that for all $i \in \{1,\dots,N\}$, $j \in \{1,\dots,M\}$ and $k \in \mathbb{Z}_{>0}$, the condition \eqref{eqCauchyidentitywDW} holds. We then have the summation identity
\begin{align}
\label{eq:6-stable-cauchy}
\sum_{S \in \mathfrak{s}(1)}
(-1)^{|S|}
q^{MN-|S|^2}
c_S(q)
H_{S}(x_1,\dots,x_N;\bm{z})
H_{S}(y_1,\dots,y_M;q^{-1}\bm{z}^{-1}) 
=
\prod_{i=1}^{N}
\prod_{j=1}^{M}
\frac{1-q x_iy_j}{1-x_iy_j}, 
\end{align}
where the constant $c_S(q)$ is given by \eqref{eqcs1}.
\end{thm}

\begin{proof}
Our starting point is the identity \eqref{eqcuachyid6}, in which we choose the secondary alphabet to be $\bm{u} \cup \bm{z}$, where $\bm{u} = (u_1,\dots,u_N)$ and $\bm{z} = (z_1,z_2,\dots)$:
\begin{align*}
\sum_{S \in \mathfrak{s}(1)}
F_{S}(x_1,\dots,x_N;\bm{u}\cup\bm{z})
\dotr{G}_{S}(y_1^{-1},\dots,y_M^{-1};\bm{u}\cup\bm{z}) 
=
F_{(1^N)}(x_1,\dots,x_N;\bm{u}\cup\bm{z})
\prod_{i=1}^{N}
\prod_{j=1}^{M}
\frac{1-qx_iy_j}{1-x_iy_j}.
\end{align*}
Dividing this equation by $F_{(1^N)}(*;\bm{u})$ and taking the limit $\bm{u} \rightarrow \infty$, in exactly the same vein as \eqref{6-eqspeical1}, we read off the identity
\begin{align}
\label{eq:u-lim-H-6}
\lim_{\bm{u} \rightarrow \infty}
\frac{1}{F_{(1^N)}(*;\bm{u})}
\cdot
\sum_{S \in \mathfrak{s}(1)}
F_{S}(x_1,\dots,x_N;\bm{u}\cup\bm{z})
\dotr{G}_{S}(y_1^{-1},\dots,y_M^{-1};\bm{u}\cup\bm{z}) 
=
\prod_{i=1}^{N}
\prod_{j=1}^{M}
\frac{1-qx_iy_j}{1-x_iy_j},
\end{align}
where the right hand side of \eqref{eq:u-lim-H-6} follows from \eqref{6-eqspeical1} with $S = (0,0,0,\dots)$. It remains to show that the left hand side of \eqref{eq:u-lim-H-6} matches that of \eqref{eq:6-stable-cauchy}.

For that purpose, we define the quantity
\begin{align*}
K_{N,M} := 
\lim_{\bm{u}\to \infty} 
\sum_{S \in \mathfrak s(1)} 
F_{S}(x_1,\dots,&x_N;\bm{u} \cup\bm{z}) 
\dotr{G}_{S}(y_1^{-1},\dots,y_M^{-1};\bm{u}\cup \bm{z}),
\end{align*}
which has the graphical representation
\begin{equation}
\label{eq:K-leftmost-columns-6}
\raisebox{-35mm}{
\begin{tikzpicture}[scale = 0.7,every node/.style={scale=0.7}]
\node[scale = 1.5] at (-5.5,4) {$K_{N,M}=
\displaystyle\sum_{\substack{(I_1,\dots,I_N) \\ (J_1,\dots,J_M) \\ |I|=|J|}}$};
\node[scale = 1.25] at (-2,1) {$x_1$};
\foreach \i in {2,6,7} {
\node at (-2,\i) {$\vdots$};
}
\node[scale = 1.25] at (-2,3) {$x_N$};
\node[scale = 1.25] at (-2,5) {$y_M^{-1}$};
\node[scale = 1.25] at (-2,8) {$y_1^{-1}$};
\foreach \i in {1,2,3,5,6,7,8} {
\draw[thick,->] (-1.5,\i) -- (-0.5,\i);
}
\foreach \i in {1,2,3,5,6,7,8} {
\draw[lgray,line width=1.5pt,->] (0,\i) -- (3,\i);
}
\foreach \i in {1,2,3,5,6,7,8} {
\draw[lgray,line width=1.5pt,->] (4,\i) -- (9,\i);
}
\foreach \i in {0.5,1.5,2.5} {
\draw[lgray,line width=1.5pt,->] (\i,0.5) -- (\i,8.5);
} 
\foreach \i in {4.5,5.5,6.5,7.5,8.5} {
\draw[lgray,line width=1.5pt,->] (\i,0.5) -- (\i,8.5);
} 
\foreach \i in {1,2,3,5,6,7,8} {
\node at (-0.2,\i) {$1$};
}
\foreach \i in {0.5,1.5,2.5,4.5,5.5,6.5} {
\node at (\i,0.2) {$0$};
}
\node at (3.5,1) {$I_1$};
\node at (3.5,2) {$\vdots$};
\node at (3.5,3) {$I_{N}$};
\node at (3.5,5) {$\bar{J}_M$};
\node at (3.5,6) {$\vdots$};
\node at (3.5,7) {$\vdots$};
\node at (3.5,8) {$\bar{J}_1$};
\node at (7.5,0.25) {$\dots$};
\node at (8.5,0.25) {$\dots$};
\node at (9.2,1) {$0$};
\node at (9.2,2) {$\vdots$};
\node at (9.2,3) {$0$};
\node at (9.2,5) {$1$};
\node at (9.2,6) {$1$};
\node at (9.2,7.1) {$\vdots$};
\node at (9.2,8) {$1$};
\foreach \i in {0.5,1.5,2.5} {
\node at (\i,8.8) {$1$};
}
\foreach \i in {4.5,5.5,6.5} {
\node at (\i,8.8) {$0$};
}
\foreach \i in {7.5,8.5} {
\node at (\i,8.8) {$\dots$};
}
\foreach \i in {0.5,1.5,2.5,4.5,5.5,6.5} {
\draw[thick,->] (\i,-0.8) -- (\i,-0.2);
}
\foreach \i in {0.5,1.5,2.5} {
\node at (\i,-1) {$\infty$};
}
\node at (4.5,-1) {$z_1$};
\node at (5.5,-1) {$z_2$};
\node at (6.5,-1) {$z_3$};
\foreach \i in {0.5,1.5,2.5,4.5,5.5,6.5,7.5,8.5} { \foreach \j in {5,6,7,8} {
\node at (\i,\j)[circle,fill,inner sep=2pt]{};
}}
\end{tikzpicture}}
\end{equation}
The leftmost $N$ columns of this lattice correspond with the variables $(u_1,\dots,u_N)$ that have been sent to infinity, while the remaining columns correspond with the secondary alphabet $(z_1,z_2,\dots)$ that remains free. We perform summation over the vectors $(I_1,\dots,I_N)$ and $(J_1,\dots,J_M)$ such that $\sum_{i=1}^{N} I_i = \sum_{i=1}^{M} J_i$, noting that this implies $\sum_{i=1}^{N} I_i + \sum_{i=1}^{M} \bar{J}_i = M$, so that conservation of paths is respected in the rightmost region of the partition function.

We turn our attention firstly to the leftmost $N$ columns of \eqref{eq:K-leftmost-columns-6}. We see that this region takes a very similar form to the quantity \eqref{6-chi-inv} defined in Section \ref{ssec:6stab-construct}. Indeed, the only significant difference is that top $M \times N$ block of these columns uses dotted vertices, which differ from their undotted counterparts only up to normalization. We thus find that
\begin{align}
\label{eq:K-leftmost-compute-6}
\begin{tikzpicture}[scale = 0.6,every node/.style={scale=0.6},baseline=2cm]
\foreach \i in {1,2,3} {
\draw[lgray,line width=1.5pt,->] (\i,0) -- (\i,8);
}
\foreach \i in {1,2,3,4,5,6,7} {
\draw[lgray,line width=1.5pt] (0,\i) -- (3.5,\i);
}
\foreach \i in {1,2,3,4,5,6,7} {
\node[scale = 1.25] at (-0.2,\i) {$1$};
}
\foreach \i in {1,3} {
\node[scale = 1.25] at (\i,-0.2) {$0$};
\node[scale = 1.25] at (\i,8.2) {$1$};
}
\node at (2,8.2) {$\dots$};
\node at (2,-0.2) {$\dots$};
\foreach \i in {2} {
\node at (3.8,\i) {$\vdots$};
}
\node[scale = 1.25] at (3.8,3) {$I_N$};
\node[scale = 1.25] at (3.8,1) {$I_1$};
\node[scale = 1.25] at (3.8,7) {$\bar{J}_1$};
\node[scale = 1.25] at (3.8,4) {$\bar{J}_M$};
\node at (3.8,5) {$\vdots$};
\node at (3.8,6) {
$\vdots$
};
\draw[->,thick] (1,-1.2) -- (1,-0.5);
\draw[->,thick] (3,-1.2) -- (3,-0.5);
\node[scale = 1.5] at (1,-1.5) {$\infty$};
\node[scale = 1.5] at (2,-1.5) {$\dots$};
\node[scale = 1.5] at (3,-1.5) {$\infty$};
\node[scale = 1.25] at (-1.8,1) {$*$};
\node[scale = 1.25] at (-1.8,2) {$\vdots$};
\node[scale = 1.25] at (-1.8,3) {$*$};
\node[scale = 1.25] at (-1.8,4) {$*$};
\node[scale = 1.25] at (-1.8,7) {$*$};
\foreach \i in {5,6} {
\node at (-1.8,\i) {$\vdots$};
}
\foreach \i in {1,2,3,4,5,6,7} {
\draw[thick,->] (-1.5,\i) -- (-0.5,\i);
}
\foreach \i in {1,2,3} { \foreach \j in {4,5,6,7}
{\node at (\i,\j) [circle,fill,inner sep=2pt]{};};}
\end{tikzpicture}
=
q^{MN}
\times
\frac{\chi(I_1,\dots,I_N,\bar{J}_M,\dots,\bar{J}_1)}{q^{\inv(I)+\inv(J)+|I||J|}}
\end{align}
with $\chi$ defined in \eqref{6-chi-inv}, and where the factor of $q^{MN}$ arises from the above-mentioned normalization discrepancy between dotted and undotted vertices, while the $q$-inversions in the denominator of \eqref{eq:K-leftmost-compute-6} are computed using the fact that
\begin{align*}
\inv(I_1,\dots,I_N,\bar{J}_M,\dots,\bar{J}_1)
&=
\inv(I_1\dots,I_N)+\inv(\bar{J}_M,\dots,\bar{J}_1)+|I||J|
\\
&=
\inv(I_1\dots,I_N)+\inv(J_1,\dots,J_M)+|I||J|.
\end{align*}
By the invariance of $\chi$ under permutation of its arguments, we may choose $(I_1,\dots,I_N,\bar{J}_M,\dots,\bar{J}_1) = (0^N,1^M)$, which by \eqref{6-chi-inv2} yields 
\begin{align*}
\chi(I_1,\dots,I_N,\bar{J}_M,\dots,\bar{J}_1)
=
\chi(0^N,1^M) = \lim_{\bm{u}\to \infty} F_{(1^{N})}(*;\bm{u}).
\end{align*}
We have thus derived the equation
\begin{align}
\label{eq:K-right-flip-6}
\lim_{\bm{u}\to \infty} 
\frac{K_{N,M}}{F_{(1^{N})}(*;\bm{u})}
=
q^{MN}
\displaystyle\sum_{\substack{(I_1,\dots,I_N) \\ (J_1,\dots,J_M) \\ |I|=|J|}}
\frac{1}
{q^{\inv(I)+\inv(J)+|I||J|}}
\displaystyle\sum_{S\in \mathfrak s(1)}\raisebox{-35mm}{
\begin{tikzpicture}[scale = 0.7,every node/.style={scale=0.7}]
\foreach \i in {4.5,5.5,6.5,7.5,8.5} {
\draw[lgray,line width=1.5pt,->] (\i,4.5) -- (\i,8.5);
} 
\begin{scope}[shift = {(3.5,0)}]
\foreach \i in {1,2,3,5,6,7,8} {
\draw[thick,->] (-1.5,\i) -- (-0.5,\i);
}
\node[scale = 1.25] at (-2,1) {$x_1$};
\foreach \i in {2,6,7} {
\node at (-2,\i) {$\vdots$};
}
\node[scale = 1.25] at (-2,3) {$x_N$};
\node[scale = 1.25] at (-2,5) {$y_M^{-1}$};
\node[scale = 1.25] at (-2,8) {$y_1^{-1}$};
\end{scope}
\foreach \i in {1,2,3,5,6,7,8} {
\draw[lgray,line width=1.5pt,->] (4,\i) -- (9,\i);
}
\foreach \i in {4.5,5.5,6.5,7.5,8.5} {
\draw[lgray,line width=1.5pt,->] (\i,0.5) -- (\i,3.5);
} 
\foreach \i in {4.5,5.5,6.5} {
\node at (\i,0.2) {$0$};
}
\node at (3.5,1) {$I_1$};
\node at (3.5,2) {$\vdots$};
\node at (3.5,3) {$I_{N}$};
\node at (3.5,5) {$\bar{J}_M$};
\node at (3.5,6) {$\vdots$};
\node at (3.5,7) {$\vdots$};
\node at (3.5,8) {$\bar{J}_1$};
\node at (4.5,4) {$S_{1}$};
\node at (5.5,4) {$S_{2}$};
\node at (6.5,4) {$S_{3}$};
\node at (7.5,0.25) {$\dots$};
\node at (8.5,0.25) {$\dots$};
\node at (7.5,4) {$\dots$};
\node at (8.5,4) {$\dots$};
\node at (9.2,1) {$0$};
\node at (9.2,2) {$\vdots$};
\node at (9.2,3) {$0$};
\node at (9.2,5) {$1$};
\node at (9.2,6) {$1$};
\node at (9.2,7.1) {$\vdots$};
\node at (9.2,8) {$1$};
\foreach \i in {4.5,5.5,6.5} {
\node at (\i,8.8) {$0$};
}
\foreach \i in {7.5,8.5} {
\node at (\i,8.8) {$\dots$};
}
\foreach \i in {4.5,5.5,6.5} {
\draw[thick,->] (\i,-0.8) -- (\i,-0.2);
}
\node at (4.5,-1) {$z_1$};
\node at (5.5,-1) {$z_2$};
\node at (6.5,-1) {$z_3$};
\foreach \i in {4.5,5.5,6.5,7.5,8.5} { \foreach \j in {5,6,7,8} {
\node at (\i,\j)[circle,fill,inner sep=2pt]{};
}}
\end{tikzpicture}}
\end{align}
where summation over $S \in \mathfrak{s}(1)$ is restricted to $1$-strings such that $|S|=|I|=|J|$. 

The final step is to apply a symmetry relation to the top half of the lattice on the right hand side of \eqref{eq:K-right-flip-6}, similarly to that of Proposition \ref{prop:6-flip-sym}. Letting the top half of the lattice be denoted by
\begin{equation}
\label{eq:6-Fdot}
\raisebox{-37.5mm}{
\begin{tikzpicture}[scale = 0.6]
\foreach \i in {0,1,2,3,4,6,7} {
    \draw[->,lgray,line width=1.5pt] (\i,0) -- (\i,7);}
    \foreach \i in {1,2,3,4,5,6} {
    \draw[lgray,line width=1.5pt] (-1,\i) -- (4.3,\i);}
    \foreach \i in {1,2,3,4,5,6} {
    \draw[->,lgray,line width=1.5pt] (5.5,\i) -- (8,\i);}
    \foreach \i  in {0,1,2,3,4,6,7} { \foreach \j in {1,2,3,4,5,6} {
\node at (\i,\j)     [circle,fill, inner sep=1.5pt] {};}
}
    \foreach \i in {1,2,3,4,5,6} {
    \node at (5,\i) {$\dots$};}
    \foreach \i in {1,2,6} {
    \node at (8.3,\i) {\tiny $1$};
    }
    \foreach \i in {3,4,5} {
    \node at (8.3,\i+0.2) {$\vdots$};
    }
    \node[left] at (-1.8,1) {$y_M^{-1} \to$};
    \foreach \i in {3,4,5} {
    \node at (-2.2,\i) {$\vdots$};
    }
    \node[left] at (-1.8,6) {$y_1^{-1} \to$};
    \node at (-1.4,1) {\tiny $\bar{J}_M$};
    \node at (-1.4,6) {\tiny $\bar{J}_{1}$};
    \node[left] at (-3.8,3.5) {$F_{S}^{\bullet J}(y_1^{-1},\dots,y_M^{-1};\bm{z})=$};
    \node at (0,7.5) {\tiny $0$};
    \node at (1,7.5) {\tiny $0$};
    \node at (2,7.5) {\tiny $0$};
    \node at (3,7.5) {$\dots$};
    \node at (4,7.5) {$\dots$};
    \node at (0,-2) {$z_1$};
    \node at (1,-2) {$z_2$};
    \node at (2,-2) {$z_3$};
    \node at (3,-2) {$\dots$};
    \node at (4,-2) {$\dots$};
    \draw[thick,->] (0,-1.5) -- (0,-0.85);
    \draw[thick,->] (1,-1.5) -- (1,-0.85);
    \draw[thick,->] (2,-1.5) -- (2,-0.85);
    \node at (0,-0.5) {\tiny $S_1$};
    \node at (1,-0.5) {\tiny $S_2$};
    \node at (2,-0.5) {\tiny $S_3$};
    \node at (3,-0.5) {$\dots$};
    \node at (4,-0.5) {$\dots$};
\end{tikzpicture}}
\end{equation}
one may prove the relation
\begin{align*}
F^{\bullet J}_S(y_1^{-1},\dots,y_M^{-1};\bm{z})
=
(-1)^{|S|}
c_S(q)
F^{J}_S(y_1,\dots,y_M;q^{-1}\bm{z}^{-1})
\end{align*}
by repeated application of \eqref{6-tower-sym} to the columns of the partition function \eqref{eq:6-Fdot}, where $c_S(q)$ is given by \eqref{eqcs1} and $F^{J}_S(y_1,\dots,y_M;q^{-1}\bm{z}^{-1})$ by \eqref{eqstable6-pic}. Combining everything, we have shown that
\begin{multline*}
\lim_{\bm{u}\to \infty} \frac{K_{N,M}}{F_{(1^{N})}(*;\bm{u})}
=
\sum_{S \in \mathfrak{s}(1)}
(-1)^{|S|}
q^{MN-|S|^2}
c_S(q)
\times
\\
\displaystyle
\sum_{\substack{(I_1,\dots,I_N) \\ (J_1,\dots,J_M) \\ |I|=|J|=|S|}}
\frac{1}{q^{\inv(I)+\inv(J)}}
F^{I}_S(x_1,\dots,x_N;\bm{z})
F^{J}_S(y_1,\dots,y_M;q^{-1}\bm{z}^{-1})
\\
=
\sum_{S \in \mathfrak{s}(1)}
(-1)^{|S|}
q^{MN-|S|^2}
c_S(q)
H_S(x_1,\dots,x_N;\bm{z})
H_S(y_1,\dots,y_M;q^{-1}\bm{z}^{-1}),
\end{multline*}
completing the proof of \eqref{eq:6-stable-cauchy}.
\end{proof}

\section{Properties of rational symmetric functions and symmetrization formula}
\label{sec:6-sym}

We conclude the chapter by stating and proving a symmetrization formula for the rational symmetric functions \eqref{eqrff1}. Our proof is based on setting up, in Sections \ref{ssec:6-exchange} and \ref{ssec:6-rec}, a collection of properties which uniquely determine $F_{S}(x_1,\dots,x_N;\bm{z})$. One is then able to check that the proposed symmetrization identity, given in Section \ref{ssec:6-sym}, satisfies all of these properties.

\subsection{Exchange relations}
\label{ssec:6-exchange}

\begin{prop}
\label{prop:6-exch}

Fix an integer $k \geq 1$ and a $1$-string $S \in \mathfrak{s}(1)$ such that $S_k=0$ and $S_{k+1}=1$. The following relation then holds:
\begin{multline}
\label{eq:6-exchange}
W_{z_{k+1}/z_k}(1,0;1,0)
F_S(x_1,\dots,x_N;\bm{z})
=
\\
F_{\sigma_k \cdot S}(x_1,\dots,x_N;\sigma_k \cdot \bm{z})
-
W_{z_{k+1}/z_k}(0,1;1,0)
F_{\sigma_k \cdot S}(x_1,\dots,x_N;\bm{z}),
\end{multline}
where $\sigma_k\cdot S = (\dots,S_{k+1},S_k,\dots)$ and $\sigma_k \cdot \bm{z} = (\dots,z_{k+1},z_k,\dots)$ denote transposition of the $k$-th and $(k+1)$-th elements of $S,\bm{z}$.
\end{prop}

\begin{proof}
The proof is by application of the Yang--Baxter equation \eqref{eqybesix1} to the $k$-th and $(k+1)$-th columns of the partition function in question. 

In particular, starting from the lattice representation of $F_{\sigma_k \cdot S}(x_1,\dots,x_N;\sigma_k \cdot \bm{z})$, one may attach a frozen $W_{z_{k+1}/z_k}(0,0;0,0)$ vertex at the base of the $k$-th and $(k+1)$-th columns. By repeated application of the Yang--Baxter equation, this vertex emerges from the top of the same columns, leading to the identity
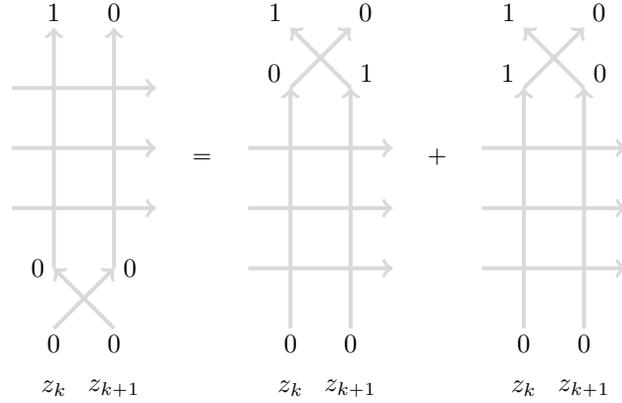
\begin{align*}
\begin{tikzpicture}[scale=0.8,baseline=1.4cm]
    \foreach \i in {3.5,4.5} {
    \draw[->,lgray,line width=1.5pt] (\i,0) -- (\i,4);}
    \foreach \i in {1,2,3} {
    \draw[->,lgray,line width=1.5pt] (2.8,\i) -- (5.2,\i);}
    \draw[<-,lgray,line width=1.5pt] (3.5,0) -- (4.5,-1);
    \draw[<-,lgray,line width=1.5pt] (4.5,0) -- (3.5,-1);
    \node[left] at (3.5,0) {$0$};
    \node[right] at (4.5,0) {$0$};
    \node at (3.5,4.25) {$1$};
    \node at (4.5,4.25) {$0$};
    \node at (3.5,-1.25) {$0$};
    \node at (4.5,-1.25) {$0$};
    \node at (3.5,-2) {$z_k$};
    \node at (4.5,-2) {$z_{k+1}$};
\end{tikzpicture}
\quad
=
\quad
\begin{tikzpicture}[scale=0.8,baseline=1.4cm]
    \foreach \i in {6.5,7.5} {
    \draw[->,lgray,line width=1.5pt] (\i,-1) -- (\i,3);}
    \foreach \i in {1,2,3} {
    \draw[->,lgray,line width=1.5pt] (5.8,\i-1) -- (8.2,\i-1);}
    \draw[<-,lgray,line width=1.5pt] (6.5,4) -- (7.5,3);
    \draw[<-,lgray,line width=1.5pt] (7.5,4) -- (6.5,3);
    \node at (6.25,4.25) {$1$};
    \node at (7.75,4.25) {$0$};
    \node at (6.5,-1.25) {$0$};
    \node at (7.5,-1.25) {$0$};
    \node at (6.5,-2) {$z_k$};
    \node at (7.5,-2) {$z_{k+1}$};
    \node[left] at (6.5,3.25) {$0$};
    \node[right] at (7.5,3.25) {$1$};
\end{tikzpicture}
\quad
+
\quad
\begin{tikzpicture}[scale=0.8,baseline=1.4cm]
    \foreach \i in {6.5,7.5} {
    \draw[->,lgray,line width=1.5pt] (\i,-1) -- (\i,3);}
    \foreach \i in {1,2,3} {
    \draw[->,lgray,line width=1.5pt] (5.8,\i-1) -- (8.2,\i-1);}
    \draw[<-,lgray,line width=1.5pt] (6.5,4) -- (7.5,3);
    \draw[<-,lgray,line width=1.5pt] (7.5,4) -- (6.5,3);
    \node at (6.25,4.25) {$1$};
    \node at (7.75,4.25) {$0$};
    \node at (6.5,-1.25) {$0$};
    \node at (7.5,-1.25) {$0$};
    \node at (6.5,-2) {$z_k$};
    \node at (7.5,-2) {$z_{k+1}$};
    \node[left] at (6.5,3.25) {$1$};
    \node[right] at (7.5,3.25) {$0$};
\end{tikzpicture}
\end{align*}
where the states assigned to external horizontal edges are arbitrary, but held fixed on both sides of the equation. Equation \eqref{eq:6-exchange} then follows by rearrangement.
\end{proof}

Explicitly writing the vertex weights in \eqref{eq:6-exchange}, after dividing it through by $W_{z_{k+1}/z_k}(1,0;1,0)$, we have
\begin{align}
\label{eq:6-exchange2}
F_S(x_1,\dots,x_N;\bm{z})
=
\frac{z_k-q z_{k+1}}{q(z_k-z_{k+1})}
F_{\sigma_k \cdot S}(x_1,\dots,x_N;\sigma_k \cdot \bm{z})
-
\frac{(1-q)z_{k+1}}{q(z_k-z_{k+1})}
F_{\sigma_k \cdot S}(x_1,\dots,x_N;\bm{z}).
\end{align}
The exchange relation \eqref{eq:6-exchange2} may be used repeatedly to express $F_S$ as a linear combination of domain-wall partition functions $F_{(1^N)}$, for arbitrary $S \in \mathfrak{s}(1)$.

\subsection{Recursion relations}
\label{ssec:6-rec}

Recall the definition of the domain-wall partition function $F_{(1^N)}$; see Remark \ref{rmk:dwpf}. In this section we record a list of properties satisfied by this object, which determine it uniquely; these properties are originally due to Korepin \cite{Korepin82}.

\begin{prop}
\label{prop:6-properties}
Fix an integer $N \in \mathbb{Z}_{>0}$. The domain-wall partition function $F_{(1^N)}(x_1,\dots,x_N;\bm{z})$ satisfies the following properties:

\begin{enumerate}[wide,labelindent=0pt]

\item The function $F_{(1^N)}(x_1,\dots,x_N;\bm{z})$ is symmetric with respect to its primary alphabet $(x_1,\dots,x_N)$.
\medskip

\item Introduce the normalized function 
\begin{equation*}
\tilde{F}_{(1^N)}(x_1,\dots,x_N;\bm{z}) 
= 
\prod_{i=1}^{N}\prod_{j=1}^{N}
(x_i-q z_j)
F_{(1^N)}(x_1,\dots,x_N;\bm{z}).
\end{equation*}
Then $\tilde{F}_{(1^N)}$ is a polynomial in $z_1$ of degree $N$.
\medskip

\item The function $F_{(1^N)}$ vanishes at $z_1=0$:
\begin{align*}
F_{(1^N)}(x_1,\dots,x_N;\bm{z})
\Big|_{z_1 = 0} = 0.
\end{align*}
\medskip

\item The function $F_{(1^N)}$ has the following recursion relation with respect to $z_1$:
\begin{align*}
F_{(1^N)}(x_1,\dots,x_N;\bm{z})
\Big|_{z_1 = x_N} = 
F_{(1^{N-1})}(x_1,\dots,x_{N-1};\widehat{\bm{z}}_1), 
\end{align*}
where $\widehat{\bm{z}}_1 = (z_2,z_3,\dots)$ denotes the secondary alphabet with the omission of $z_1$.
\medskip

\item In the case $N=1$ we have the explicit evaluation
\begin{align*}
F_{(1)}(x_1;\bm{z}) 
= 
W_{z_1/x_1}(0,1;1,0) 
= 
\frac{(1-q)z_1}{x_1-q z_1}.
\end{align*}

\end{enumerate}
\end{prop}

\begin{proof}
All of the properties (1)--(5) follow directly from the partition function \eqref{picofsixf1} under the choice $S=(1^N,0,0,\ldots)$.

\begin{enumerate}[wide,labelindent=0pt]
\item The symmetry of $F_{(1^N)}(x_1,\dots,x_N;\bm{z})$ in $(x_1,\dots,x_N)$ has already been established in Theorem \ref{thm:6-sym-FG}.
\medskip

\item Multiplying the partition function by $\prod_{i=1}^{N}\prod_{j=1}^{N}
(x_i-q z_j)$ is equivalent to converting the underlying vertex weights to their polynomial form. In particular, we define
\begin{align}
\label{6-poly-weights}
\tilde{W}_{x,z}(i,j;k,\ell)
:=
(x-q z)
W_{z/x}(i,j;k,\ell),
\end{align}
which are polynomials in $x$ and $z$ for all $i,j,k,\ell \in \{0,1\}$. The function $\tilde{F}_{(1^N)}(x_1,\dots,x_N;\bm{z})$ is then given by the partition function setup \eqref{picofsixf1}, but using the weights \eqref{6-poly-weights}. Examining this partition function, it is clear that all dependence on the parameter $z_1$ is via the leftmost column of the lattice. Since the weights \eqref{6-poly-weights} of individual vertices in that column are degree-$1$ polynomials in $z_1$, it is then clear that $\tilde{F}_{(1^N)}(x_1,\dots,x_N;\bm{z})$ is a degree-$N$ polynomial in $z_1$.
\medskip

\item This again follows by studying the leftmost column of the partition function under consideration. It is easily seen that any configuration of this column must give rise to exactly one of the following vertex weights:
\begin{align*}
W_{z_1/x_i}(0,1;1,0)
&=
\frac{(1-q)z_1}{x_i-q z_1},
\end{align*}
where $i$ takes values in $\{1,\dots,N\}$. This weight vanishes at $z_1=0$, and the claim is then immediate.
\medskip

\item Take the lattice representation of $F_{(1^N)}(x_1,\dots,x_N;\bm{z})$ and use the symmetry in $(x_1,\dots,x_N)$ to permute the horizontal spectral parameters, so that $x_N$ is assigned to the lowest row of the lattice, followed by $(x_1,\dots,x_{N-1})$ assigned to the remaining rows. In any lattice configuration, the bottom-left vertex must then take one of the following forms:
\begin{align*}
W_{z_1/x_N}(0,1;0,1)
=
\frac{x_N-z_1}{x_N-q z_1},
\qquad
W_{z_1/x_N}(0,1;1,0)
=
\frac{(1-q)z_1}{x_N-q z_1}.
\end{align*}
The first of these weights vanishes at $z_1=x_N$; the second becomes $W_{1}(0,1;1,0)=1$. It is then easy to show that, after specializing $z_1=x_N$, the first row and column of $F_{(1^N)}(x_1,\dots,x_N;\bm{z})$ freeze out entirely:
\begin{equation}
\label{6-eqfreez1}
F_{(1^N)}(x_1,\dots,x_N;\bm{z})
\Big|_{z_1 = x_N}
=
\begin{tikzpicture}[scale = 0.7,every node/.style={scale=0.6},baseline=2.5cm]
\foreach \i in {-1,0,1,2,3,4} {
    \draw[->,lgray,line width=1.5pt] (\i,0) -- (\i,7);}
    \foreach \i in {1,2,3,4,5,6} {
    \draw[->,lgray,line width=1.5pt] (-2,\i) -- (5,\i);}
    \foreach \i in {1,2,6} {
    \node at (5.3,\i) {$0$};
    }
    \foreach \i in {3,4,5} {
    \node at (5.3,\i) {$\vdots$};
    }
    \node[left,scale = 1.25] at (-3.8,1) {$x_N$};
    \node[left,scale = 1.25] at (-3.8,2) {$x_1$};
    \node[left] at (-3.8,3) {$\vdots$};
    \node[left] at (-3.8,4) {$\vdots$};
    \node[left] at (-3.8,5) {$\vdots$};
    \node[left,scale = 1.25] at (-3.8,6) {$x_{N-1}$};
    \draw[thick,->] (-3.6,1) -- (-2.6,1);
    \draw[thick,->] (-3.6,2) -- (-2.6,2);
    \draw[thick,->] (-3.6,6) -- (-2.6,6);
    \foreach \i in {1,2,6} {
    \node at (-2.4,\i) {$1$};
    }
    \node at (-2.4,3) {$\vdots$};
    \node at (-2.4,4) {$\vdots$};
    \node at (-2.4,5) {$\vdots$};
    \node at (-1,1.5) {$1$};
    \foreach \i in {-1,0} {
    \node at (\i,7.3) {$1$};
    }
    \foreach \i in {4} {
    \node at (\i,7.3) {$1$};
    }
    \foreach \i in {1,2,3} {
    \node at (\i,7.3) {$\dots$};
    }
    \draw[->,thick] (0,-1) -- (0,-0.5);
    \draw[->,thick] (-1,-1) -- (-1,-0.5);
    \draw[->,thick] (4,-1) -- (4,-0.5);
    \node[scale = 1.25] at (-1,-1.2) {$x_N$};
    \node[scale = 1.25] at (0,-1.2) {$z_2$};
    \foreach \i in {1,2,3}{
    \node[scale = 1.25] at (\i,-1.2) {$\dots$};}
    \node[scale = 1.25] at (4,-1.2) {$z_{N}$};
    \foreach \i in {-1,0,4} {
    \node at (\i,-0.2) {$0$};
    }
    \node at (1,-0.2) {$\dots$};
    \foreach \i in {2,3} {
    \node at (\i,-0.2) {$\dots$};
    }
    \foreach \i in {2,3,4,5,6} {
    \node at (-0.5,\i) {$1$};
    }
    \node at (-0.5,1) {$0$};
    \foreach \i in {0,1,2,3,4} {
    \node at (\i,1.5) {$0$};
    }
    \node at (-1,1) [circle,fill,inner sep=1.5pt,red] {};
    \end{tikzpicture}
\end{equation}
The total weight of the frozen regions in \eqref{6-eqfreez1} is identically $1$, and the remaining (unfrozen) portion of the lattice is simply $F_{(1^{N-1})}(x_1,\dots,x_{N-1};\widehat{\bm{z}}_1)$, yielding the desired recursion.
\medskip

\item This property is evident.

\end{enumerate}
\end{proof}

\subsection{Proof of uniqueness}
\label{ssec:unique-6}

\begin{thm}
\label{thm:unique-6}
The functions $F_S(x_1,\dots,x_N;\bm{z})$ are uniquely characterized by the exchange relations of Proposition \ref{prop:6-exch} and by the properties listed in Proposition \ref{prop:6-properties}.
\end{thm}

\begin{proof}
Suppose that $Z_S(x_1,\dots,x_N;\bm{z})$, $S \in \mathfrak{s}(1)$, $|S|=N$ is a family of rational functions that satisfies all of the conditions of Propositions \ref{prop:6-exch} and \ref{prop:6-properties}. Our task is to show that, necessarily,
\begin{align}
\label{ZS=FS-6}
Z_S(x_1,\dots,x_N;\bm{z})
=
F_S(x_1,\dots,x_N;\bm{z}).
\end{align}
We begin with the observation that for any $1$-string $S$, one is able to write
\begin{align}
\label{Z-F-order-6}
Z_S(x_1,\dots,x_N;\bm{z})
-
F_S(x_1,\dots,x_N;\bm{z})
=
\sum_{\sigma}
d_{\sigma}^{S} 
\Big(
Z_{(1^N)}(x_1,\dots,x_N;\sigma \cdot \bm{z})
-
F_{(1^N)}(x_1,\dots,x_N;\sigma \cdot \bm{z})
\Big),
\end{align}
where the sum is over permutations of the secondary alphabet $\bm{z}$, while $d_{\sigma}^{S}$ are appropriate coefficients. Indeed, Proposition \ref{prop:6-exch} ensures that one may write
\begin{align}
\label{F-order-6}
F_S(x_1,\dots,x_N;\bm{z})
=
\sum_{\sigma}
d_{\sigma}^{S} 
F_{(1^N)}(x_1,\dots,x_N;\sigma \cdot \bm{z})
\end{align}
for appropriate coefficients $d_{\sigma}^{S}$. By virtue of the assumption that $Z_S(x_1,\dots,x_N;\bm{z})$ satisfies the same exchange relations, it must then be possible to write an analogue of equation \eqref{F-order-6} for $Z_S(x_1,\dots,x_N;\bm{z})$ with the {\it same}\/ expansion coefficients; the claim \eqref{Z-F-order-6} follows immediately.

Thanks to equation \eqref{Z-F-order-6}, the task of proving \eqref{ZS=FS-6} is reduced to showing that
\begin{align}
\label{ZT=FT-6}
Z_{(1^N)}(x_1,\dots,x_N;\bm{z})
=
F_{(1^N)}(x_1,\dots,x_N;\bm{z}).
\end{align}
The rest of the proof proceeds by induction on $N$; namely, the size of the primary alphabet. Let $\mathcal{P}_N$ denote the proposition that \eqref{ZT=FT-6} holds. 

We begin by noting that $\mathcal{P}_1$ is true, by virtue of property (5). Now we proceed with the inductive hypothesis that $\mathcal{P}_{N-1}$ is true for some $N \geq 2$, and consider the function
\begin{align}
\label{eq:delta-T-6}
\Delta_N=
\prod_{i=1}^{N}\prod_{j=1}^{N}(x_i-q z_j)
\Big(
Z_{(1^N)}(x_1,\dots,x_N;\bm{z})
-
F_{(1^N)}(x_1,\dots,x_N;\bm{z})
\Big).
\end{align}
By property (2) of Proposition \ref{prop:6-properties}, $\Delta_N$ is a polynomial in $z_1$ of degree $N$. The symmetry property (1), together with properties (3) and (4), specifies $\Delta_N$ at $N+1$ values of $z_1$. In particular, property (3) becomes $\Delta_N \Big|_{z_1=0}=0$. Property (4) becomes
\begin{align*}
\Delta_N
\Big|_{z_1 = x_k} = 
\prod_{j = 2}^{N}(x_k-q z_j)
\prod_{i = 1}^{N} (x_i - q x_k)
\Delta_{N-1}(x_1,\dots,\widehat{x}_k,\dots,x_N;
\widehat{\bm{z}}_1)
=
0
\end{align*}
for all $1 \leq k \leq N$, where the final equality with $0$ is due to the assumption that $\mathcal{P}_{N-1}$ holds. Collectively, these properties show that $\Delta_N$ vanishes at $N+1$ values of $z_1$; it is therefore identically zero for arbitrary $z_1$. Hence $\mathcal{P}_N$ is true, and the proof is complete by induction on $N$.

\end{proof}

\subsection{Symmetrization formula}
\label{ssec:6-sym}

\begin{defn}\label{defn:M-set}
Fix two integers $n,N \geq 1$ and let $S \in \mathfrak{s}(n)$ be an $n$-string. An $n$-permutation matrix of size $N$ and profile $S$ is an $N \times \infty$ matrix with all row sums equal to $n$, and $i$-th column sum equal to $S_i$ for all $i \in \mathbb{N}$. We let $\mathfrak{M}_n(N,S)$ denote the set of all such matrices.
\end{defn}

In what follows we let $S \in \mathfrak{s}(1)$ be a $1$-string, and denote by $\mathfrak{M}_1(N,S)$ the set of all $1$-permutation matrices of size $N$ and profile $S$. These are $N \times \infty$ matrices with row sums equal to $1$, and $i$-th column sum equal to $S_i$ for all $i \in \mathbb{Z}_{>0}$ (see {\normalfont Definition \ref{defn:M-set}}, above). 

The rows of the matrices $\sigma\in\mathfrak{M}_1(N,S)$ are given by unit vectors $e_k$, with $1$ at position $k$ and $0$ at all remaining positions. We note that when $S=e_k$, $k \in \mathbb{Z}_{>0}$, the function \eqref{eqrff1} may be computed explicitly:
\begin{align}
\label{eq:F6-one_row}
F_{e_k}(x;{\bm z}) 
= 
\prod_{j=1}^{k-1} 
\frac{x-z_j}{x-qz_j}
\cdot
\frac{(1-q)z_k}{x-qz_k}.
\end{align}
Indeed, this corresponds to the case where the partition function \eqref{picofsixf1} consists of just a single row. One finds that it permits a unique configuration whose weight is given by a product of $k-1$ vertices of the type $W_{z_j/x}(0,1;0,1)$ and one vertex $W_{z_k/x}(0,1;1,0)$, which results precisely in the equation \eqref{eq:F6-one_row}.

\begin{thm}\label{thm:F6sym}
Fix an integer $N \geq 1$ and let $S \in \mathfrak{s}(1)$ be a $1$-string such that $|S|=N$. The rational symmetric function \eqref{eqrff1} is given by the explicit formula
\begin{align}
\label{symformulasix1}
F_S(x_1,\dots,x_N;\bm{z})
=
\sum_{\sigma \in \mathfrak{M}_1(N,S)}
\prod_{1 \leq i<j \leq N} \Delta_{\sigma(i),\sigma(j)}(x_i,x_j)
\prod_{i=1}^{N}
F_{\sigma(i)}(x_i;\bm{z}),
\end{align}
where $\sigma(i)$ denotes the $i$-th row (read from top to bottom) of the $1$-permutation matrix $\sigma$, and $F_{\sigma(i)}(x_i;\bm{z})$ is given by \eqref{eq:F6-one_row}. For any $U,V \in \mathfrak{s}(1)$ such that $|U| = |V| =1$, we define
\begin{align*}
\Delta_{U,V}(x,y)
=
\left\{
\begin{array}{ll}
\dfrac{y-qx}{y-x},
&
\qquad
\text{1 in $U$ occurs before 1 in $V$},
\\ \\
\dfrac{x-qy}{x-y},
&
\qquad
\text{1 in $U$ occurs after 1 in $V$}.
\end{array}
\right.
\end{align*}
\end{thm}

\begin{rmk}
It is easy to show that \eqref{symformulasix1} is equivalent to the more standard symmetrization identity
\begin{align}
\label{symformulasix1-SN}
F_S(x_1,\dots,x_N;\bm{z})
=
\sum_{\sigma \in \mathfrak{S}_N}
\prod_{1 \leq i<j \leq N}
\frac{x_{\sigma(j)}-q x_{\sigma(i)}}
{x_{\sigma(j)}- x_{\sigma(i)}}
\prod_{i=1}^{N}
F_{e_{k_i}}(x_{\sigma(i)};\bm{z}),
\end{align}
where $\{k_1 < \cdots < k_N\}$ are the coordinates of ones in the string $S$; namely, $S_{k_i} = 1$ for all $1 \leq i \leq N$. The reason why we present the somewhat less natural-looking formula \eqref{symformulasix1}, rather than simply \eqref{symformulasix1-SN}, is that the former has an immediate generalization to the setting of the Izergin--Korepin model (see Section \ref{ssec:IK-sym-formula}) while we do not know of any straightforward way to generalize the latter.
\end{rmk}

\begin{proof}
Throughout the proof, we shall let
\begin{align}
\label{eq:gS-define-6}
g_S(x_1,\dots,x_N;\bm{z})
=
\sum_{\sigma \in \mathfrak{S}_N}
\prod_{1 \leq i<j \leq N}
\frac{x_{\sigma(j)}-q x_{\sigma(i)}}
{x_{\sigma(j)}- x_{\sigma(i)}}
\prod_{i=1}^{N}
F_{e_{k_i}}(x_{\sigma(i)};\bm{z})
\end{align}
denote the right hand side of \eqref{symformulasix1-SN}. The proof is in two steps. In the first stage, we show that $g_S(x_1,\dots,x_N;\bm{z})$ satisfies the exchange relation \eqref{eq:6-exchange2}. In the second stage, we show that $g_{(1^N)}(x_1,\dots,x_N;\bm{z})$ satisfies the conditions of Proposition \ref{prop:6-properties}. In view of Theorem \ref{thm:unique-6}, these checks suffice to show that $g_S(x_1,\dots,x_N;\bm{z}) = F_S(x_1,\dots,x_N;\bm{z})$ for general $S \in \mathfrak{s}(1)$.

To validate the exchange relation \eqref{eq:6-exchange2}, we fix a $1$-string $S$ such that $S_{\ell}=0$, $S_{\ell+1}=1$. One finds that all factors in the sum \eqref{eq:gS-define-6} are symmetric in $(z_{\ell},z_{\ell+1})$ and invariant under the switch $S \leftrightarrow \sigma_{\ell} \cdot S$, with the exception of the unique one-row function $F_{e_{k_i}}(x_{\sigma(i)};\bm{z})$, where $i$ is chosen such that $k_i =\ell+1$. The exchange relation \eqref{eq:6-exchange2} is then immediately verified on this one-row function.

To check the properties of Proposition \ref{prop:6-properties}, we set $S=(1^N)$ in \eqref{eq:gS-define-6}:
\begin{align}
\label{eq:gS-DW}
g_{(1^N)}(x_1,\dots,x_N;\bm{z})
=
\sum_{\sigma \in \mathfrak{S}_N}
\prod_{1 \leq i<j \leq N}
\frac{x_{\sigma(j)}-q x_{\sigma(i)}}
{x_{\sigma(j)}- x_{\sigma(i)}}
\prod_{i=1}^{N}
F_{e_i}(x_{\sigma(i)};\bm{z}).
\end{align}

\begin{enumerate}[wide, labelindent=0pt]
\item The symmetry in $(x_1,\dots,x_N)$ is manifest from of \eqref{eq:gS-DW}.
\medskip

\item Multiplying \eqref{eq:gS-DW} by $\prod_{i=1}^{N} \prod_{j=1}^{N} (x_i-qz_j)$ converts each factor $F_{e_i}(x_{\sigma(i)};\bm{z})$ to a degree-$1$ polynomial in $z_1$. Hence the renormalized version of \eqref{eq:gS-DW} itself is a polynomial in $z_1$ of degree $N$.
\medskip

\item The function $F_{e_1}(x_{\sigma(1)};\bm{z})$ vanishes at $z_1=0$, and the vanishing of \eqref{eq:gS-DW} at the same point is then immediate.
\medskip

\item The functions $F_{e_i}(x_N;\bm{z})$, $2 \leq i \leq N$ all vanish at $z_1=x_N$. Setting $z_1=x_N$ in \eqref{eq:gS-DW}, we are then forced to have $\sigma(1)=N$. Writing $(\sigma(2),\dots,\sigma(N))=(\rho(1),\dots,\rho(N-1)) \in \mathfrak{S}_{N-1}$ for the remaining elements of the permutation $\sigma$, one finds that
\begin{multline}
\label{eq:gS-DW-rec}
g_{(1^N)}(x_1,\dots,x_N;\bm{z})
\Big|_{z_1=x_N}
=
\\
\prod_{j=1}^{N-1}
\frac{x_j-q x_N}{x_j-x_N}
F_{e_1}(x_N;\bm{z})
\Big|_{z_1=x_N}
\times
\sum_{\rho \in \mathfrak{S}_{N-1}}
\prod_{1 \leq i<j \leq N-1}
\frac{x_{\rho(j)}-q x_{\rho(i)}}{x_{\rho(j)}-x_{\rho(i)}}
\prod_{i=1}^{N-1}
F_{e_{i+1}}(x_{\rho(i)};\bm{z}) 
\Big|_{z_1=x_N}.
\end{multline}
The recursion $g_{(1^N)}(x_1,\dots,x_N;\bm{z})|_{z_1 = x_N} = 
g_{(1^{N-1})}(x_1,\dots,x_{N-1};\widehat{\bm{z}}_1)$ then follows after using the fact that
\begin{align*}
F_{e_{i+1}}(x_{\rho(i)};\bm{z})
=
\frac{x_{\rho(i)}-z_1}{x_{\rho(i)}-qz_1}
F_{e_i}(x_{\rho(i)};\widehat{\bm{z}}_1),
\qquad
\forall\ 1 \leq i \leq N-1,
\end{align*}
which is an immediate consequence of the definition \eqref{eq:F6-one_row}.

\item For $N=1$, the summation over $\mathfrak{S}_N$ is trivial, and one sees that $g_{(1)}(x_1;\bm{z}) = F_{e_1}(x_1;\bm{z}) = F_{(1)}(x_1;\bm{z})$.

\end{enumerate}

\end{proof}

\begin{rmk}
Just as we did throughout the previous proof, one may also show that the determinant formula \eqref{eq:DWPF} satisfies the properties of Proposition \ref{prop:6-properties}; indeed, this was its original proof due to Izergin \cite{Izergin87}. Outside of showing that they obey the same set of uniquely-characterizing properties, we do not know of a direct way to relate the formulas \eqref{eq:DWPF} and \eqref{symformulasix1-SN} at $S=(1^N)$.
\end{rmk}

\newpage

\chapter{Nineteen-vertex model}
\label{chapter:19-vertex}

This chapter contains the main results of our text. We introduce two families of rational symmetric functions arising from the 
$R$-matrix of the Izergin--Korepin nineteen-vertex model \cite{IzerginKorepin81,Jimbo86b,Lima-Santos99} and study a number of their properties. Where possible, our constructions proceed directly in parallel with those of the six-vertex model and employ common notations. While this does create a double-usage of many symbols and functions, we believe that there is no potential for confusion due to the partitioning of our text into chapters.

In Section \ref{sec:IKmodel-relations}, we write down the $R$-matrix of the model in a certain choice of gauge differing from the one in \cite{Jimbo86b}: our choice endows the $R$-matrix with a sum-to-unity property, namely, all of its rows sum to $1$. We state the Yang--Baxter equation and the unitarity relation for our $R$-matrix. In analogy with the six-vertex model, we also define dotted vertices (obtained as normalizations of undotted ones) and observe a symmetry between the dotted and undotted vertices. 

In Section \ref{sec:IKrow} we introduce row operators for the model, and in Section \ref{sec:IKrational} we use them to construct the rational functions $F_{S}$ and $G_{S}$. Key properties of these functions, including their symmetry in their (primary) alphabet and Cauchy identity, are proved in Sections \ref{sec:IKrational} and \ref{sec:IKcauchy}. In analogy with the six-vertex model, the Cauchy identity obtained at this point does not have a factorized right hand side: rather, one sees the appearance of the domain-wall partition function of the Izergin--Korepin model, previously studied in \cite{Garbali16}.

Subsequently, in Sections \ref{sec:IK_stable}--\ref{ssec:IK-stable-cauchy} we introduce a further family of rational symmetric functions $H_{S}$ (which have a stability property with respect to their primary alphabet) and offer a simplified version of the Cauchy identity with a fully factorized right hand side. The proof of these facts rests on demonstrating how to degenerate both the functions $F_S$ and $G_S$ to $H_{S}$.

By far the most challenging result of the chapter, and indeed of the whole text, is contained in Sections \ref{sec:19-prop}--\ref{sec:19-sym}: we present a symmetrization formula for the family of symmetric functions $F_S$. This formula is conceptually very similar to its six-vertex analogue in Section \ref{sec:6-sym}, in the sense that it expresses $F_S$ as a sum over products of one-variable wavefunctions, dressed by appropriate two-variable scattering factors. However, the proof of this formula ends up being far more involved than in the six-vertex case. The crux of the difficulty may be seen by analysing the weights of the Izergin--Korepin vertex model: once normalized to be polynomials, one finds that they are typically of degree-two in their spectral parameters, rather than degree-one as in the six-vertex case. This means that any type of Lagrange interpolation proof necessarily requires twice the number of interpolating points as previously; this creates technical problems both for setting up the required recursion relations and for validating them on the formula itself.

\section{Definition of Izergin--Korepin nineteen-vertex model and basic relations}
\label{sec:IKmodel-relations}

We begin by writing down the $R$-matrix and explain how it gives rise to a vertex model, in Section \ref{ssec:19R}. Section \ref{ssec:19YB} contains the unitarity and Yang--Baxter relations of the model, and Section \ref{ssec:19dot} introduces dotted analogues of the vertex weights.

\subsection{$R$-matrix and vertices}
\label{ssec:19R}

\begin{defn}
The $R$-matrix for the Izergin--Korepin model is a $9 \times 9$ matrix given by
\begin{equation}
\label{eqrmnineteen1}
R_{ab}(y/x) =\sum_{i,j,k,\ell \in \{0,1,2\}} W_{y/x}(i,j;k,\ell) E_a^{j,\ell} \otimes E_b^{i,k}
\end{equation}
where $E_{a}^{j,\ell}$ and $E_{b}^{i,k}$ are $3 \times 3$ elementary matrices acting on the three--dimensional spaces $V_{a}\cong \mathbb{C}^3$ and $V_{b}\cong \mathbb{C}^3$, respectively. The matrix entries $W_{y/x}(i,j;k,\ell)$ are identified with vertices as follows:
\begin{equation}
\label{eqdefn19w}
\raisebox{-22mm}{
\begin{tikzpicture}[scale=0.8,baseline = {(0,-2.25)}]
\draw[lgray,line width=1.5pt,->] (14,0) -- (16,0);
\draw[lgray,line width=1.5pt,->] (15,-1) -- (15,1); 
\node at (13.8,0) {$j$};
\node at (16.2,0) {$\ell$};
\node at (15,-1.2) (A1) {$i$};
\node at (15,1.2) {$k$};
\node at (12.75,0) {$x$};
\node at (15,-2.2) (A) {$y$};
\draw[->,thick] (13,0) -- (13.5,0);
\draw[->,thick] (A) -- (A1);
\node at (10.8,0) {$W_{y/x}(i,j;k,\ell)=$};
\node at (19,0) {$\qquad \text{ for }\quad  i,j,k,\ell \in \{0,1,2\},$}; 
\end{tikzpicture}}
\end{equation}
where the arguments $(i,j;k,\ell)$ are read in clockwise manner around the vertex. The vertex weights satisfy the conservation property
\begin{align}
\label{conserve}
W_{y/x}(i,j;k,\ell)
=
0,
\qquad
\text{if}
\qquad
i+j \not= k+\ell,
\end{align}
which leaves nineteen possible non-vanishing choices of $(i,j;k,\ell)$. We tabulate all of these possibilities, together with their explicit weights, in Figure \ref{fig:19weights}.
\end{defn}

\begin{rmk}
While the $R$-matrix for the Izergin--Korepin model has nineteen non-vanishing weights, it is crucial to differentiate it from the nineteen-vertex Zamolodchikov--Fateev model \cite{Zamolodchikov80,Lima-Santos99}, which is derived from fusion of the six-vertex model. Indeed, if one were to work with the latter model, most results could be obtained from suitable fusion of results in Chapter \ref{chapter:6-vertex}. The nineteen-vertex Izergin--Korepin model does not afford us such a luxury. 

It should also be emphasized that the vertices \eqref{eqdefn19w} are {\it not} coloured\footnote{In other words, arising from the quantized affine Lie algebra $\mathcal{U}_{q}(\widehat{\mathfrak{sl}_3})$.} variants of the six-vertex model, as were studied in \cite{BorodinWheeler18}. Instead, the states attached to each edge denote the {\it total number} of paths present, and the conservation property \eqref{conserve} means that the total flux of paths through a vertex is preserved. This approach contrasts with coloured vertex models, where edge labels indicate the {\it colour} of the path present, rather than the quantity.
\end{rmk}

\begin{rmk}
The model presented in Figure \ref{fig:19weights} may seem intimidating at first glance, but it should be noted that in practical calculations one usually only needs explicit knowledge of a subset of the vertex weights. For example, for the purpose of deriving the symmetry of our functions and their Cauchy identities (see Sections \ref{sec:IKrational} and \ref{sec:IKcauchy}), it turns out that only explicit knowledge of the weights
\begin{align*}
W_{y/x}(0,0;0,0) = 1,
\qquad
W_{y/x}(2,2;2,2) = 1,
\qquad
W_{y/x}(0,2;0,2) = 
\dfrac{(x-y)(x-qy)}{(x-q^2y)(x-q^3y)}
\end{align*}
is needed; the remaining weights are implicitly treated as providing a solution to the Yang--Baxter equation.

A further nice aspect of the model in Figure \ref{fig:19weights} is that two copies of the six-vertex model are embedded within it. Indeed, if one restricts to $(i,j;k,\ell)$ such that $0 \leq i+j \leq 1$ or $3 \leq i+j \leq 4$, one finds that the resulting weights $W_{y/x}(i,j;k,\ell)$ match with six-vertex weights, up to replacing $q$ in the latter model by $q^2$.
\end{rmk}

\begin{prop}[Sum-to-unity property]
\label{nineteensumtounity}
For any fixed $i,j \in \{0,1,2\}$, we have 
\begin{equation}
\sum_{k = 0}^2 \sum_{\ell = 0}^2 W_{y/x}(i,j;k,\ell) = 1.
\end{equation}
\end{prop}

\begin{proof}
This is by direct verification, using the weights in Figure \ref{fig:19weights}. Indeed, each row of that figure has fixed values for the incoming indices $i,j$, meaning that one simply sums the weights along each row to deduce the result.
\end{proof}

\begin{rmk}
Similarly to the scenario in the six-vertex model, the $R$-matrix introduced in equation $\eqref{eqrmnineteen1}$ obeys the sum-to-unity criterion, as outlined in Proposition $\ref{nineteensumtounity}$. This property of the $R$-matrix is necessary for the purposes of constructing Markov processes in the quadrant, but not sufficient: one also requires that all entries of the $R$-matrix be non-negative. We were unable to find any parameter range of $x,y,q$ for which the entries obey such a requirement, other than very restrictive ones such as $0 \leq y<x \leq 1$, $q=0$ (when non-negativity becomes manifest).
\end{rmk}

\newgeometry{bottom=0.1cm,top=0.1cm,left=2cm,right=2cm}
\begin{figure}
\centering
\begin{center}
\begin{tabular}{|c|}
\hline
\begin{tikzpicture}[scale=0.5]
\draw[lgray,line width=1.5pt,->] (-3,2.5) -- (-1,2.5);
\draw[lgray,line width=1.5pt,->] (-2,1.5) -- (-2,3.5);
\node[below] at (-2,1.5) {\footnotesize{$0$}};
\node[above] at (-2,3.5) {\footnotesize{$0$}};
\node[left] at (-3,2.5) {\footnotesize{$0$}};
\node[right] at (-1,2.5) {\footnotesize{$0$}};
\node at (-2,0.3) {$1$};
\end{tikzpicture}
\\
\hline
\end{tabular}
\end{center}
\begin{center}
\begin{tabular}{|c|c|}
\hline
\begin{tikzpicture}[scale=0.5]
\draw[lgray,line width=1.5pt,->] (1,2.5) -- (3,2.5);
\draw[lgray,line width=1.5pt,->] (2,1.5) -- (2,3.5);
\node[below] at (2,1.5) {\footnotesize{$1$}};
\node[above] at (2,3.5) {\footnotesize{$1$}};
\node[left] at (1,2.5) {\footnotesize{$0$}};
\node[right] at (3,2.5) {\footnotesize{$0$}};
\node[scale = 1.2] at (2,-0.2) {\footnotesize{$\dfrac{q^2(x-y)}{x-q^2y}$}};
\end{tikzpicture}
&
\begin{tikzpicture}[scale=0.5]
\draw[lgray,line width=1.5pt,->] (1,-1.5) -- (3,-1.5);
\draw[lgray,line width=1.5pt,->] (2,-2.5) -- (2,-0.5);
\draw[lgray,line width=1.5pt,->] (2,-2.5) -- (2,-1.5) -- (3,-1.5);
\node[below] at (2,-2.5) {\footnotesize{$1$}};
\node[above] at (2,-0.5) {\footnotesize{$0$}};
\node[left] at (1,-1.5) {\footnotesize{$0$}};
\node[right] at (3,-1.5) {\footnotesize{$1$}};
\node[scale = 1.2] at (2,-4.2) {\footnotesize{$\dfrac{(1-q^2)x}{x-q^2y}$}};
\end{tikzpicture}
\\
\hline
\end{tabular}
\end{center}
\begin{center}
\begin{tabular}{|c|c|c|}
\hline
\begin{tikzpicture}[scale=0.5]
\draw[lgray,line width=1.5pt,->] (9,2.5) -- (11,2.5);
\draw[lgray,line width=1.5pt,->] (10,1.5) -- (10,3.5);
\node[below] at (10,1.5) {\footnotesize{$2$}};
\node[above] at (10,3.5) {\footnotesize{$2$}};
\node[left] at (9,2.5) {\footnotesize{$0$}};
\node[right] at (11,2.5) {\footnotesize{$0$}};
\node[scale = 1.2] at (10,-0.2) {\footnotesize{$\dfrac{q^4(x-y)(x-qy)}{(x-q^2y)(x-q^3y)}$}};
\end{tikzpicture}
&
\begin{tikzpicture}[scale=0.5]
\draw[lgray,line width=1.5pt] (9,-1.5) -- (11,-1.5);
\draw (10,-2.5) -- (10,-0.5);
\draw[lgray,line width=1.5pt] (10,-2.5) -- (10,-1.5);
\draw[lgray,line width=1.5pt,->] (10,-1.5) -- (10,-0.5);
\draw[lgray,line width=1.5pt,->] (10,-1.5) -- (11,-1.5);
\node[below] at (10,-2.5) {\footnotesize{$2$}};
\node[above] at (10,-0.5) {\footnotesize{$1$}};
\node[left] at (9,-1.5) {\footnotesize{$0$}};
\node[right] at (11,-1.5) {\footnotesize{$1$}};
\node[scale = 1.2] at (10,-4.2) {\footnotesize{$-\dfrac{q(1-q)^2(1+q)x(x-y)}{(x-q^2y)(x-q^3y)}
$}};
\end{tikzpicture}
&
\begin{tikzpicture}[scale=0.5]
\draw[lgray,line width=1.5pt] (9,-5.5) -- (11,-5.5);
\draw[lgray,line width=1.5pt,->] (10,-6.5) -- (10,-4.5);
\draw[lgray,line width=1.5pt,->] (10,-6.5) -- (10,-5.5) -- (11,-5.5);
\node[below] at (10,-6.5) {\footnotesize{$2$}};
\node[above] at (10,-4.5) {\footnotesize{$0$}};
\node[left] at (9,-5.5) {\footnotesize{$0$}};
\node[right] at (11,-5.5) {\footnotesize{$2$}};
\node[scale = 1.2] at (10,-8.2) {\footnotesize{$\dfrac{(1-q^2)x(x+qx-qy-q^3y)}{(x-q^2y)(x-q^3y)}
$}};
\end{tikzpicture}
\\
\hline
\end{tabular}
\end{center}
\begin{center}
\begin{tabular}{|c|c|}
\hline
\begin{tikzpicture}[scale=0.5]
\draw[lgray,line width=1.5pt,->] (1,2.5) -- (3,2.5);
\draw[lgray,line width=1.5pt,->] (2,1.5) -- (2,3.5);
\node[below] at (2,1.5) {\footnotesize{$0$}};
\node[above] at (2,3.5) {\footnotesize{$1$}};
\node[left] at (1,2.5) {\footnotesize{$1$}};
\node[right] at (3,2.5) {\footnotesize{$0$}};
\node[scale = 1.2] at (2,-0.2) 
{\footnotesize{$\dfrac{(1-q^2)y}{x-q^2y}$}};
\end{tikzpicture}
&
\begin{tikzpicture}[scale=0.5]
\draw[lgray,line width=1.5pt,->] (1,-1.5) -- (3,-1.5);
\draw[lgray,line width=1.5pt,->] (2,-2.5) -- (2,-0.5);
\draw[lgray,line width=1.5pt,->] (2,-2.5) -- (2,-1.5) -- (3,-1.5);
\node[below] at (2,-2.5) {\footnotesize{$0$}};
\node[above] at (2,-0.5) {\footnotesize{$0$}};
\node[left] at (1,-1.5) {\footnotesize{$1$}};
\node[right] at (3,-1.5) {\footnotesize{$1$}};
\node[scale = 1.2] at (2,-4.2) {\footnotesize{$\dfrac{x-y}{x-q^2y}$}};
\end{tikzpicture}
\\
\hline
\end{tabular}
\end{center}
\begin{center}
\begin{tabular}{|c|c|c|}
\hline
\begin{tikzpicture}[scale=0.5]
\draw[lgray,line width=1.5pt,->] (9,2.5) -- (11,2.5);
\draw[lgray,line width=1.5pt,->] (10,1.5) -- (10,3.5);
\node[below] at (10,1.5) {\footnotesize{$1$}};
\node[above] at (10,3.5) {\footnotesize{$2$}};
\node[left] at (9,2.5) {\footnotesize{$1$}};
\node[right] at (11,2.5) {\footnotesize{$0$}};
\node[scale = 1.2] at (10,-0.2) {\footnotesize{$-\dfrac{q^4(1+q)(x-y)y}{(x-q^2y)(x-q^3y)}$}};
\end{tikzpicture}
&
\begin{tikzpicture}[scale=0.5]
\draw[lgray,line width=1.5pt] (9,-1.5) -- (11,-1.5);
\draw (10,-2.5) -- (10,-0.5);
\draw[lgray,line width=1.5pt] (10,-2.5) -- (10,-1.5);
\draw[lgray,line width=1.5pt,->] (10,-1.5) -- (10,-0.5);
\draw[lgray,line width=1.5pt,->] (10,-1.5) -- (11,-1.5);
\node[below] at (10,-2.5) {\footnotesize{$1$}};
\node[above] at (10,-0.5) {\footnotesize{$1$}};
\node[left] at (9,-1.5) {\footnotesize{$1$}};
\node[right] at (11,-1.5) {\footnotesize{$1$}};
\node[scale = 1.2] at (10,-4.2) {\footnotesize{$\dfrac{-qx^2+(1+q)(1-q^2+q^4)xy-q^4y^2}{(x-q^2y)(x-q^3y)}$}};
\end{tikzpicture}
&
\begin{tikzpicture}[scale=0.5]
\draw[lgray,line width=1.5pt] (9,-5.5) -- (11,-5.5);
\draw[lgray,line width=1.5pt,->] (10,-6.5) -- (10,-4.5);
\draw[lgray,line width=1.5pt,->] (10,-6.5) -- (10,-5.5) -- (11,-5.5);
\node[below] at (10,-6.5) {\footnotesize{$1$}};
\node[above] at (10,-4.5) {\footnotesize{$0$}};
\node[left] at (9,-5.5) {\footnotesize{$1$}};
\node[right] at (11,-5.5) {\footnotesize{$2$}};
\node[scale = 1.2] at (10,-8.2) {\footnotesize{$\dfrac{(1+q)x(x-y)}{(x-q^2y)(x-q^3y)}$}};
\end{tikzpicture}
\\
\hline
\end{tabular}
\end{center}
\begin{center}
\begin{tabular}{|c|c|}
\hline
\begin{tikzpicture}[scale=0.5]
\draw[lgray,line width=1.5pt,->] (1,2.5) -- (3,2.5);
\draw[lgray,line width=1.5pt,->] (2,1.5) -- (2,3.5);
\node[below] at (2,1.5) {\footnotesize{$2$}};
\node[above] at (2,3.5) {\footnotesize{$2$}};
\node[left] at (1,2.5) {\footnotesize{$1$}};
\node[right] at (3,2.5) {\footnotesize{$1$}};
\node[scale = 1.2] at (2,-0.2) {\footnotesize{$\dfrac{q^2(x-y)}{x-q^2y}$}};
\end{tikzpicture}
&
\begin{tikzpicture}[scale=0.5]
\draw[lgray,line width=1.5pt,->] (1,-1.5) -- (3,-1.5);
\draw[lgray,line width=1.5pt,->] (2,-2.5) -- (2,-0.5);
\draw[lgray,line width=1.5pt,->] (2,-2.5) -- (2,-1.5) -- (3,-1.5);
\node[below] at (2,-2.5) {\footnotesize{$2$}};
\node[above] at (2,-0.5) {\footnotesize{$1$}};
\node[left] at (1,-1.5) {\footnotesize{$1$}};
\node[right] at (3,-1.5) {\footnotesize{$2$}};
\node[scale = 1.2] at (2,-4.2) {\footnotesize{$\dfrac{(1-q^2)x}{x-q^2y}$}};
\end{tikzpicture}
\\
\hline
\end{tabular}
\end{center}
\begin{center}
\begin{tabular}{|c|c|c|}
\hline
\begin{tikzpicture}[scale=0.5]
\draw[lgray,line width=1.5pt,->] (9,2.5) -- (11,2.5);
\draw[lgray,line width=1.5pt,->] (10,1.5) -- (10,3.5);
\node[below] at (10,1.5) {\footnotesize{$0$}};
\node[above] at (10,3.5) {\footnotesize{$2$}};
\node[left] at (9,2.5) {\footnotesize{$2$}};
\node[right] at (11,2.5) {\footnotesize{$0$}};
\node[scale = 1.2] at (10,-0.2) {\footnotesize{$\dfrac{(1-q^2)(x+q^2x-q^2y-q^3y)y}{(x-q^2y)(x-q^3y)}$}};
\end{tikzpicture}
&
\begin{tikzpicture}[scale=0.5]
\draw[lgray,line width=1.5pt] (9,-1.5) -- (11,-1.5);
\draw (10,-2.5) -- (10,-0.5);
\draw[lgray,line width=1.5pt] (10,-2.5) -- (10,-1.5);
\draw[lgray,line width=1.5pt,->] (10,-1.5) -- (10,-0.5);
\draw[lgray,line width=1.5pt,->] (10,-1.5) -- (11,-1.5);
\node[below] at (10,-2.5) {\footnotesize{$0$}};
\node[above] at (10,-0.5) {\footnotesize{$1$}};
\node[left] at (9,-1.5) {\footnotesize{$2$}};
\node[right] at (11,-1.5) {\footnotesize{$1$}};
\node[scale = 1.2] at (10,-4.2) {\footnotesize{$\dfrac{q(1-q)^2(1+q)(x-y)y}{(x-q^2y)(x-q^3y)}$}};
\end{tikzpicture}
&
\begin{tikzpicture}[scale=0.5]
\draw[lgray,line width=1.5pt] (9,-5.5) -- (11,-5.5);
\draw[lgray,line width=1.5pt,->] (10,-6.5) -- (10,-4.5);
\draw[lgray,line width=1.5pt,->] (10,-6.5) -- (10,-5.5) -- (11,-5.5);
\node[below] at (10,-6.5) {\footnotesize{$0$}};
\node[above] at (10,-4.5) {\footnotesize{$0$}};
\node[left] at (9,-5.5) {\footnotesize{$2$}};
\node[right] at (11,-5.5) {\footnotesize{$2$}};
\node[scale = 1.2] at (10,-8.2) {\footnotesize{$\dfrac{(x-y)(x-qy)}{(x-q^2y)(x-q^3y)}$}};
\end{tikzpicture}
\\
\hline
\end{tabular}
\end{center}
\begin{center}
\begin{tabular}{|c|c|}
\hline
\begin{tikzpicture}[scale=0.5]
\draw[lgray,line width=1.5pt,->] (1,2.5) -- (3,2.5);
\draw[lgray,line width=1.5pt,->] (2,1.5) -- (2,3.5);
\node[below] at (2,1.5) {\footnotesize{$1$}};
\node[above] at (2,3.5){\footnotesize{$2$}};
\node[left] at (1,2.5) {\footnotesize{$2$}};
\node[right] at (3,2.5) {\footnotesize{$1$}};
\node[scale = 1.2] at (2,-0.2) {\footnotesize{$\dfrac{(1-q^2)y}{x-q^2y}$}};
\end{tikzpicture}
&
\begin{tikzpicture}[scale=0.5]
\draw[lgray,line width=1.5pt,->] (1,-1.5) -- (3,-1.5);
\draw[lgray,line width=1.5pt,->] (2,-2.5) -- (2,-0.5);
\draw[lgray,line width=1.5pt,->] (2,-2.5) -- (2,-1.5) -- (3,-1.5);
\node[below] at (2,-2.5) {\footnotesize{$1$}};
\node[above] at (2,-0.5) {\footnotesize{$1$}};
\node[left] at (1,-1.5) {\footnotesize{$2$}};
\node[right] at (3,-1.5) {\footnotesize{$2$}};
\node[scale = 1.2] at (2,-4.2) {\footnotesize{$\dfrac{x-y}{x-q^2y}$}};
\end{tikzpicture}
\\
\hline
\end{tabular}
\end{center}
\begin{center}
\begin{tabular}{|c|}
\hline
\begin{tikzpicture}[scale=0.5]
\draw[lgray,line width=1.5pt,->] (-3,2.5) -- (-1,2.5);
\draw[lgray,line width=1.5pt,->] (-2,1.5) -- (-2,3.5);
\node[below] at (-2,1.5) {\footnotesize{$2$}};
\node[above] at (-2,3.5) {\footnotesize{$2$}};
\node[left] at (-3,2.5) {\footnotesize{$2$}};
\node[right] at (-1,2.5) {\footnotesize{$2$}};
\node at (-2,0.3) {$1$};
\end{tikzpicture}
\\
\hline
\end{tabular}
\end{center}
\caption{The weights of the nineteen-vertex Izergin--Korepin model.}
\label{fig:19weights}
\end{figure}
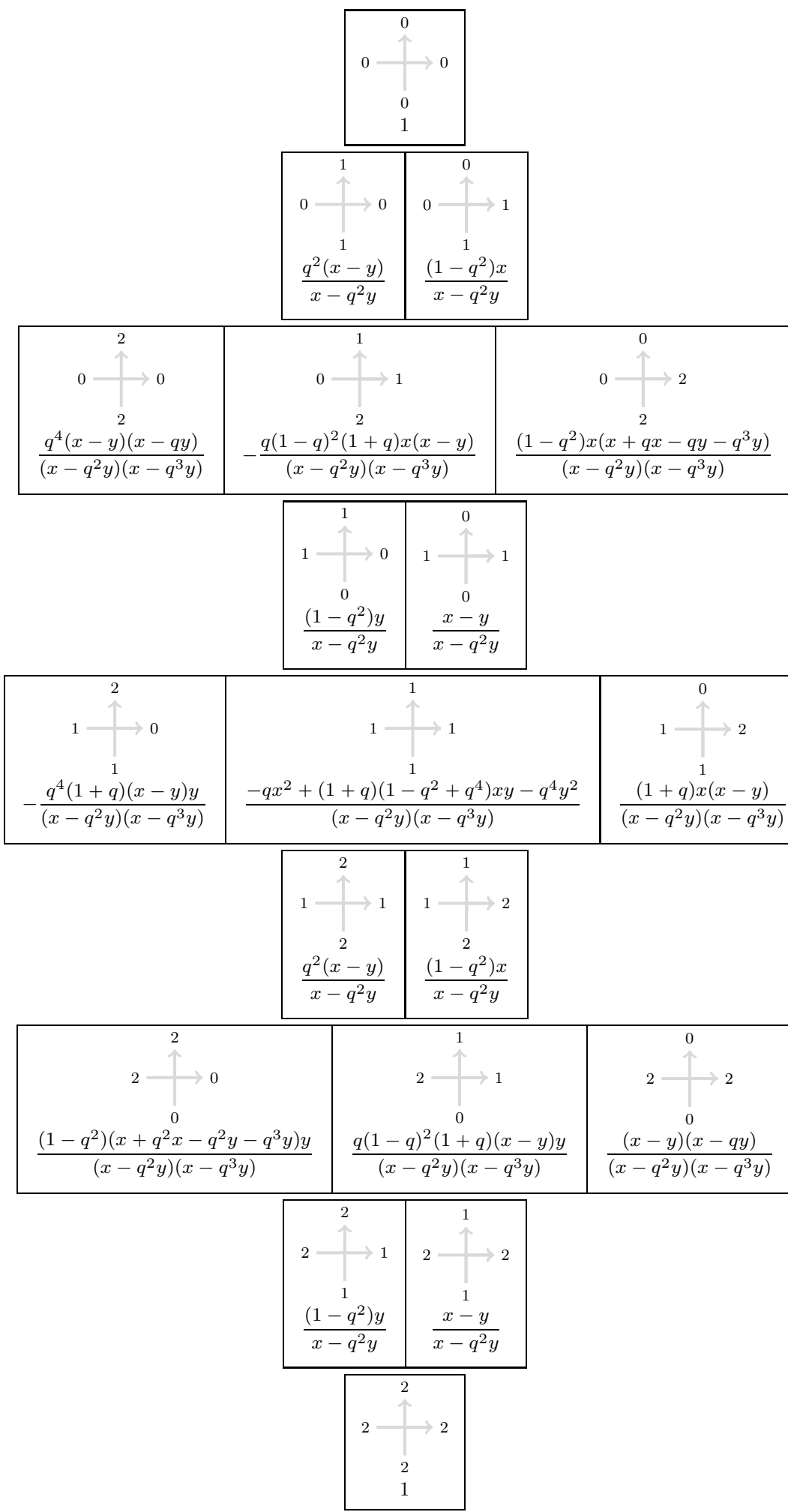

\restoregeometry
\newpage

\subsection{Unitarity and Yang--Baxter relations}
\label{ssec:19YB}

\begin{prop}
Let $R_{ab}(x)$ denote the $R$-matrix defined in equation $\eqref{eqrmnineteen1}$, whose components are given by \eqref{eqdefn19w} and Figure \ref{fig:19weights}. The $R$-matrix satisfies the unitarity and Yang--Baxter relations:
\begin{alignat}{2}
R_{ab}(y/x)R_{ba}(x/y) &= {\rm id}, \label{eqpropybe1}\\
R_{ab}(y/x)R_{ac}(z/x)R_{bc}(z/y)& = R_{bc}(z/y)R_{ac}(z/x)R_{ab}(y/x), \label{eqpropybe2}
\end{alignat}
where the labels $a,b,c$ indicate the spaces upon which the $R$-matrix acts; here \eqref{eqpropybe1} holds as an identity in ${\rm End}(V_a \otimes V_b)$, and \eqref{eqpropybe2} as an identity in ${\rm End}(V_a \otimes V_b \otimes V_c)$.

Equations $\eqref{eqpropybe1}$ and $\eqref{eqpropybe2}$ have the following pictorial representations:

\begin{equation}
\label{eq19unitary1}
\sum_{k_1,k_2\in \{0,1,2\}}
\raisebox{-15mm}{\begin{tikzpicture}[scale = 0.4]
\draw[->,lgray,line width=1.5pt,rounded corners] (0,0) -- (3,0) -- (3,3);
\draw[->,lgray,line width=1.5pt,rounded corners] (1.5,-1.5) -- (1.5,1.5) -- (4.5,1.5);
\node at (-0.5,0) {\footnotesize $i_1$};
\node at (1.5,-2) {\footnotesize $i_2$};
\node at (3.4,-0.2) {\footnotesize $k_1$};
\node at (1.5,2.2) {\footnotesize $k_2$};
\node at (5,1.5) {\footnotesize $j_2$};
\node at (2.85,3.65) {\footnotesize $j_1$};
\node at (-2.5,0) {\footnotesize{$x$}};
\node at (1.5,-3.75) {\footnotesize{$y$}};
\draw[->,thick] (-2,0) -- (-1,0);
\draw[->,thick] (1.5,-3.35) -- (1.5,-2.5);
\end{tikzpicture}}
=
\bm{1}_{i_1=j_1}
\bm{1}_{i_2=j_2}
\end{equation}
\begin{equation}
\label{eq19ybe1}
\sum_{k_1,k_2,k_3 \in \{0,1,2\}}
\begin{tikzpicture}[scale = 0.9,every node/.style={scale=0.8},baseline={([yshift=-0.25ex]current bounding box.center)}]
\draw[lgray,line width=1.5pt] (-2,1) -- (-1,0);
\draw[lgray,line width=1.5pt] (-2,0) -- (-1,1);
\draw[->,lgray,line width=1.5pt] (-1,1) -- (0,1);
\draw[->,lgray,line width=1.5pt] (-1,0) -- (0,0);
\draw[->,lgray,line width=1.5pt] (-0.5,-0.5) -- (-0.5,1.5);
\node at (-3,1) (A1) {$x$};
\node at (-3,0) (B1) {$y$};
\node[left] at (-1.9,1) (A) {$i_1$};
\node[left] at (-1.9,0) (B) {$i_2$};
\draw[->,thick] (A1) -- (A);
\draw[->,thick] (B1) -- (B);
\node at (-1,1.3) {$k_2$};
\node at (-1,-0.3) {$k_1$};
\node at (-0.5,1.75) {$j_3$};
\node at (-0.5,-0.75) {$i_3$};
\draw[->,thick] (-0.5,-1.3) -- (-0.5,-1);
\node at (-0.5,-1.5) {$z$};
\node at (0.35,1) {$j_2$};
\node at (0.35,0) {$j_1$};
\node at (-0.5, 0.5) {$k_3$};
\end{tikzpicture}
=\sum_{k_1,k_2,k_3 \in \{0,1,2\}}
\begin{tikzpicture}[scale = 0.8,every node/.style={scale=0.8},baseline={([yshift=-0.25ex]current bounding box.center)}]
\draw[lgray,line width=1.5pt] (4.5,1) -- (5.5,1);
\draw[lgray,line width=1.5pt] (4.5,0) -- (5.5,0);
\draw[->,lgray,line width=1.5pt] (5.5,1) -- (6.5,0);
\draw[->,lgray,line width=1.5pt] (5.5,0) -- (6.5,1);
\draw[->,lgray,line width=1.5pt] (5,-0.5) -- (5,1.5);
\node at (3.5,1) (C1) {$x$};
\node at (3.5,0) (D1) {$y$};
\node[left] at (4.6,1) (C) {$i_1$};
\node[left] at (4.6,0) (D) {$i_2$};
\draw[->,thick] (C1) -- (C);
\draw[->,thick] (D1) -- (D);
\node at (5,1.75) {$j_3$};
\node at (5,-0.75) {$i_3$};
\draw[->,thick] (5,-1.3) -- (5,-1);
\node at (5,-1.5) {$z$};
\node at (6.85,1) {$j_2$};
\node at (6.85,0) {$j_1$};
\node at (5.5,1.25) {$k_1$};
\node at (5.5,-0.25) {$k_2$};
\node at (5,0.5) {$k_3$};
\end{tikzpicture}
\end{equation}
where in both equations the indices $i_1,i_2,i_3,j_1,j_2,j_3$ are assigned fixed values in $\{0,1,2\}$.
\end{prop}

\begin{proof}
Equations $\eqref{eqpropybe1}$ and $\eqref{eqpropybe2}$ can be verified by direct calculation.
\end{proof}

\subsection{Dotted vertices}
\label{ssec:19dot}
For technical reasons, in the sequel we will also require a normalized version of the weights  $W_{y/x}(i,j;k,\ell)$; namely,
\begin{equation}
\label{eqcr19sym1}
\dotr{W}_{y/x}(i,j;k,\ell) 
= 
\dfrac{W_{y/x}(i,j;k,\ell)}
{W_{y/x}(0,2;0,2)}
=
\dfrac{(x-q^2y)(x-q^3y)}{(x-y)(x-qy)}
W_{y/x}(i,j;k,\ell).
\end{equation}
By analogy with the six-vertex model, we place dots in the center of each vertex $\eqref{eqdefn19w}$ to denote the normalized weight $\dotr{W}_{y/x}(i,j;k,\ell)$. 

It is clear that the normalized weights \eqref{eqcr19sym1} continue to satisfy the Yang--Baxter equation \eqref{eq19ybe1} as written. However, \eqref{eq19unitary1} is not scale-invariant and its right hand side needs to be appropriately normalized in order for the equation to be valid for dotted vertices. 
\begin{prop}
\label{nineteenflip1}
The vertices presented in equation $\eqref{eqcr19sym1}$ can be related to those of equation $\eqref{eqdefn19w}$ under $x \mapsto x^{-1}$ and $y \mapsto q^{-3}y^{-1}$, via the following symmetry:
\begin{equation}
\label{eqnineteencym1}
\dotr{W}_{y/x}(i,j;k,\ell) 
=
q^{2\bar{j}}
(q-1)^{\bm{1}_{k=1}-\bm{1}_{i=1}}
(-q)^{\bm{1}_{k=2}-\bm{1}_{i=2}}
W_{x/y/q^{3}}(k,\bar{j};i,\bar{\ell}) 
\quad \text{ for all } \quad i,j,k,\ell \in \{0,1,2\},
\end{equation}
where $\bar{j} = 2 - j$ and $\bar{\ell} = 2-\ell$. Pictorially, the vertices are related in the following way:
\begin{equation}
\label{eqnormandor1}
\raisebox{-20mm}{
\begin{tikzpicture}[scale = 0.8,every node/.style={scale=0.8}]
\draw[lgray,line width=1.5pt,->] (-1,0) -- (1,0);
\draw[lgray,line width=1.5pt,->] (0,-1) -- (0,1);
\node at (0,-1.25) {$i$};
\node at (0,1.25) {$k$};
\node at (-1.25,0) {$j$};
\node at (1.25,0) {$\ell$};
\node at (0,0) [circle,fill,inner sep=1.5pt] {};
\node at (-2.25,0) {$x$};
\draw[thick,->] (0,-2) -- (0,-1.45);
\draw[thick,->] (-2,0)  -- (-1.5,0);
\node[below] at (0,-2.05) {$y$};
\node[right] at (2,0) {$=$};
\draw[lgray,line width=1.5pt,->] (5,0) -- (7,0);
\draw[lgray,line width=1.5pt,->] (6,-1) -- (6,1);
\node at (6,-1.25) {$k$};
\node at (6,1.25) {$i$};
\node at (4.75,0) {$\bar{j}$};
\node at (7.25,0) {$\bar{\ell}$};
\node at (3.2,0) {$x^{-1}$};
\draw[thick,->] (6,-2) -- (6,-1.45);
\draw[thick,->] (3.7,0) -- (4.2,0);
\node[below] at (6,-2.05) {$q^{-3}y^{-1}$};
\node[scale = 1.25] at (11,0) {\qquad $\times\ \ q^{2\bar{j}}
(q-1)^{\bm{1}_{k=1}-\bm{1}_{i=1}}
(-q)^{\bm{1}_{k=2}-\bm{1}_{i=2}}.$};
\end{tikzpicture}}
\end{equation}
\end{prop}
\begin{proof}
Equation $\eqref{eqnineteencym1}$ can be checked directly on each individual vertex. 
\end{proof}

\section{Row operators and the Yang--Baxter algebra}
\label{sec:IKrow}

This section contains the fundamental algebraic tools that underpin most of the calculations in this chapter. In Section \ref{ssec:IK-finite-row} we introduce row-operators in finite size, and list some of their commutation relations. In Section \ref{ssec:IK-infinite-row} we study the limit of these operators when the length of the row tends to infinity, and examine what happens to the associated commutation relations.

\subsection{Row operators and commutation relations}
\label{ssec:IK-finite-row}

We begin by introducing the vector space $\mathbb{V}^{(L)}$, which is obtained by taking an $L$-fold tensor product of local spaces $V_i \cong \mathbb{C}^3$:
\begin{equation}
\label{IKVL}
\mathbb{V}^{(L)} = V_1 \otimes V_2 \otimes \cdots \otimes V_L.
\end{equation}
In the partition functions that we subsequently consider, each vector space $V_i$ will correspond to a vertical lattice line; as such, the natural operators for acting on $\mathbb{V}^{(L)}$ are obtained by considering rows of vertices of length $L$ within the model \eqref{eqdefn19w}. In particular, for all $a,b \in \{0,1,2\}$, we define the {\it finite} row operator $\mathcal{T}_{a,b}(x;\bm{z}) \in {\rm End}(\mathbb{V}^{(L)})$ as follows:
\begin{align}
\mathcal{T}_{a,b}(x;\bm{z}): \bigotimes_{k = 1}^{L} \ket{i_k}_{k} 
&\mapsto \sum_{(j_1,\dots,j_L) 
\in \{0,1,2\}^L}  \left(\raisebox{-15mm}{\smash{
\begin{tikzpicture}[scale = 0.8,every node/.style={scale=0.8}]
\node at (-1.8,0) (D1) {$x$};
\node at (-1,0) (D) {};
\draw[lgray,line width=1.5pt] (-1,0) -- (3,0);
\draw[lgray,line width=1.5pt,->] (4.2,0) -- (7.2,0);
\draw[<-,lgray,line width=1.5pt] (1,1) -- (1,-1);
\draw[<-,lgray,line width=1.5pt] (0,1) -- (0,-1);
\draw[<-,lgray,line width=1.5pt] (2,1) -- (2,-1);
\draw[<-,lgray,line width=1.5pt] (5.2,1) -- (5.2,-1);
\draw[<-,lgray,line width=1.5pt] (6.2,1) -- (6.2,-1);
\node at (-1,-0.25) {$a$};
\node[right] at (3.3,0) {$\dots$};
\node at (0,1.25) {$i_1$};
\node at (1,1.25) {$i_2$};
\node at (2,1.25) {$\dots$};
\node at (5.2,1.25) {$\dots$};
\node at (6.2,1.25) {$i_L$};
\node at (0,-1.25) (A) {$j_1$};
\node at (1,-1.25) (B) {$j_2$};
\node at (2,-1.25) {$\dots$};
\node at (5.2,-1.25) {$\dots$};
\node at (6.2,-1.25) (C) {$j_L$};
\node at (7.2,-0.25) {$b$};
\node at (0,-2.25) (A1) {$z_1$};
\node at (1,-2.25) (B1) {$z_2$};
\node at (2,-2.25)  {$\dots$};
\node at (6.2,-2.25) (C1) {$z_L$};
\draw[->,thick] (A1) -- (A);
\draw[->,thick] (B1) -- (B);
\draw[->,thick] (C1) -- (C);
\draw[->,thick] (D1) -- (D);
\end{tikzpicture}}}\right)
\bigotimes_{k = 1}^{L} 
\ket{j_k}_{k}, 
\label{eqrowopIK}
\end{align}
where $(i_1,\dots,i_L) \in \{0,1,2\}^L$ selects a fixed, arbitrary vector in $\mathbb{V}^{(L)}$. The row operators $\mathcal{T}_{a,b}(x;\bm{z})$ depend on both the horizontal spectral parameter $x$ and the vertical ones, $\bm{z} = (z_1,z_2,\dots,z_L)$. However, in practice, the dependence on $\bm{z}$ will have little interest for us, and we usually suppress it from our notations by writing $\mathcal{T}_{a,b}(x;\bm{z}) \equiv \mathcal{T}_{a,b}(x)$.

We shall also require a dotted version of the row operators \eqref{eqrowopIK}, denoted $\dotr{\mathcal{T}}_{a,b}(x;\bm{z})$. These are obtained by simply replacing each vertex in \eqref{eqrowopIK} by its dotted counterpart \eqref{eqcr19sym1}:
\begin{align}
\dotr{\mathcal{T}}_{a,b}(x;\bm{z}): \bigotimes_{k = 1}^{L} \ket{i_k}_{k} 
&\mapsto \sum_{(j_1,\dots,j_L) 
\in \{0,1,2\}^L}  \left(\raisebox{-15mm}{\smash{
\begin{tikzpicture}[scale = 0.8,every node/.style={scale=0.8}]
\node at (-1.8,0) (D1) {$x$};
\node at (-1,0) (D) {};
\draw[lgray,line width=1.5pt] (-1,0) -- (3,0);
\draw[lgray,line width=1.5pt,->] (4.2,0) -- (7.2,0);
\draw[<-,lgray,line width=1.5pt] (1,1) -- (1,-1);
\draw[<-,lgray,line width=1.5pt] (0,1) -- (0,-1);
\draw[<-,lgray,line width=1.5pt] (2,1) -- (2,-1);
\draw[<-,lgray,line width=1.5pt] (5.2,1) -- (5.2,-1);
\draw[<-,lgray,line width=1.5pt] (6.2,1) -- (6.2,-1);
\node at (-1,-0.25) {$a$};
\node[right] at (3.3,0) {$\dots$};
\node at (0,1.25) {$i_1$};
\node at (1,1.25) {$i_2$};
\node at (2,1.25) {$\dots$};
\node at (5.2,1.25) {$\dots$};
\node at (6.2,1.25) {$i_L$};
\node at (0,-1.25) (A) {$j_1$};
\node at (1,-1.25) (B) {$j_2$};
\node at (2,-1.25) {$\dots$};
\node at (5.2,-1.25) {$\dots$};
\node at (6.2,-1.25) (C) {$j_L$};
\node at (7.2,-0.25) {$b$};
\node at (0,-2.25) (A1) {$z_1$};
\node at (1,-2.25) (B1) {$z_2$};
\node at (2,-2.25)  {$\dots$};
\node at (6.2,-2.25) (C1) {$z_L$};
\draw[->,thick] (A1) -- (A);
\draw[->,thick] (B1) -- (B);
\draw[->,thick] (C1) -- (C);
\draw[->,thick] (D1) -- (D);
\node at (0,0) [circle,fill,inner sep=1.5pt] {};
\node at (1,0) [circle,fill,inner sep=1.5pt] {};
\node at (2,0) [circle,fill,inner sep=1.5pt] {};
\node at (5.2,0) [circle,fill,inner sep=1.5pt] {};
\node at (6.2,0) [circle,fill,inner sep=1.5pt] {};
\end{tikzpicture}}}\right)
\bigotimes_{k = 1}^{L} 
\ket{j_k}_{k}. 
\label{eqrowopIKdot}
\end{align}

\begin{rmk}
The operators \eqref{eqrowopIK} correspond with the nine entries of the {\it monodromy matrix} of the Izergin--Korepin model, and satisfy a collection of $3^4$ bilinear relations known as the {\it Yang--Baxter algebra}. Only a few of these relations will be important for our purposes: namely, the ones that allow us to prove the symmetry of our families of multivariate functions, and their Cauchy identities. These will be discussed below.
\end{rmk}

\begin{prop}
\label{prop:Tcommute}
For any two integers $a,b \in \{0,2\}$, the operators $\mathcal{T}_{a,b}(x)$ and $\mathcal{T}_{a,b}(y)$ commute:
\begin{align}
\label{IK-Tcommute}
[\mathcal{T}_{a,b}(x),\mathcal{T}_{a,b}(y)]=0.
\end{align}
\end{prop}

\begin{proof}
This follows from a standard Yang--Baxter argument. One begins by considering arbitrary matrix elements of the product $\mathcal{T}_{a,b}(x)\mathcal{T}_{a,b}(y)$:
\begin{align*}
\bigotimes_{k = 1}^{L} \bra{j_k}_{k}
\mathcal{T}_{a,b}(x)\mathcal{T}_{a,b}(y) 
\bigotimes_{k = 1}^{L} \ket{i_k}_{k}
=
\raisebox{-12mm}{\begin{tikzpicture}[scale=0.6,every node/.style={scale=0.6}]
\node at (2,1) {$x$};
\node at (2,2) {$y$};
\foreach \i in {1,2} {
\draw[->,thick] (2.25,\i) -- (3,\i);
}
\foreach \i in {4} {
\draw[lgray,line width=1.5pt,->] (\i,0) -- (\i,3);
}
\foreach \i in {1,2} {
\draw[lgray,line width=1.5pt] (3.5,\i) -- (4.5,\i);
}
\foreach \i in {1,2} {
\node at (5,\i) {$\dots$};
}
\foreach \i in {6,7} {
\draw[lgray,line width=1.5pt,->] (\i,0) -- (\i,3);
}
\foreach \i in {1,2} {
\draw[lgray,line width=1.5pt,->] (5.5,\i) -- (7.5,\i);
}
\node at (3.3,1) {$a$};
\node at (3.3,2) {$a$};
\node[right] at (7.5,1) {$b$};
\node[right] at (7.5,2) {$b$};
\node[above] at (4,3) {$i_1$};
\node[above] at (6,3) {$\dots$};
\node[above] at (7,3) {$i_L$};
\node[below] at (4,0) {$j_1$};
\node[below] at (6,0) {$\dots$};
\node[below] at (7,0) {$j_L$};
\end{tikzpicture}}
\end{align*}
In view of the fact that both of the external left edges take the value $a \in \{0,2\}$, one may insert an $R$-matrix vertex at the left of this picture, without altering the value of the partition function:
\begin{align}
\label{eq:insertR-left}
\bigotimes_{k = 1}^{L} \bra{j_k}_{k}
\mathcal{T}_{a,b}(x)\mathcal{T}_{a,b}(y) 
\bigotimes_{k = 1}^{L} \ket{i_k}_{k}
=
\raisebox{-12mm}{\begin{tikzpicture}[scale=0.6,every node/.style={scale=0.6}]
\begin{scope}[shift = {(-1.2,0)}]
\node at (2,1) {$y$};
\node at (2,2) {$x$};
\foreach \i in {1,2} {
\draw[->,thick] (2.25,\i) -- (3,\i);
}
\end{scope}
\draw[lgray,line width=1.5pt] (2.5,2) -- (3.5,1);
\draw[lgray,line width=1.5pt] (2.5,1) -- (3.5,2);
\foreach \i in {4} {
\draw[lgray,line width=1.5pt,->] (\i,0) -- (\i,3);
}
\foreach \i in {1,2} {
\draw[lgray,line width=1.5pt] (3.5,\i) -- (4.5,\i);
}
\foreach \i in {1,2} {
\node at (5,\i) {$\dots$};
}
\foreach \i in {6,7} {
\draw[lgray,line width=1.5pt,->] (\i,0) -- (\i,3);
}
\foreach \i in {1,2} {
\draw[lgray,line width=1.5pt,->] (5.5,\i) -- (7.5,\i);
}
\node[left] at (2.3,1) {$a$};
\node[left] at (2.3,2) {$a$};
\node[right] at (7.5,1) {$b$};
\node[right] at (7.5,2) {$b$};
\node[above] at (4,3) {$i_1$};
\node[above] at (6,3) {$\dots$};
\node[above] at (7,3) {$i_L$};
\node[below] at (4,0) {$j_1$};
\node[below] at (6,0) {$\dots$};
\node[below] at (7,0) {$j_L$};
\node at (3.5,1) {$*$};
\node at (3.5,2) {$*$};
\end{tikzpicture}}
\end{align}
Indeed, this step is justified by the fact that both of the edges marked $*$ are forced to assume the value $a$ (by conservation of paths), and the fact that $W_{y/x}(a,a;a,a)=1$ for $a \in \{0,2\}$; {\it cf.} the top and bottom rows of Figure \ref{fig:19weights}. From here, we use the Yang--Baxter equation \eqref{eq19ybe1} repeatedly to transfer the inserted $R$-matrix all the way to the right edges of the lattice:
\begin{align*}
\bigotimes_{k = 1}^{L} \bra{j_k}_{k}
\mathcal{T}_{a,b}(x)\mathcal{T}_{a,b}(y) 
\bigotimes_{k = 1}^{L} \ket{i_k}_{k}
=
\raisebox{-12mm}{\begin{tikzpicture}
[scale=0.6,every node/.style={scale=0.6}]
\begin{scope}[shift = {(0,0)}]
\node at (2,1) {$y$};
\node at (2,2) {$x$};
\foreach \i in {1,2} {
\draw[->,thick] (2.25,\i) -- (3,\i);
}
\end{scope}
\draw[lgray,line width=1.5pt,->] (7.5,2) -- (8.5,1);
\draw[lgray,line width=1.5pt,->] (7.5,1) -- (8.5,2);
\foreach \i in {4} {
\draw[lgray,line width=1.5pt,->] (\i,0) -- (\i,3);
}
\foreach \i in {1,2} {
\draw[lgray,line width=1.5pt] (3.5,\i) -- (4.5,\i);
}
\foreach \i in {1,2} {
\node at (5,\i) {$\dots$};
}
\foreach \i in {6,7} {
\draw[lgray,line width=1.5pt,->] (\i,0) -- (\i,3);
}
\foreach \i in {1,2} {
\draw[lgray,line width=1.5pt] (5.5,\i) -- (7.5,\i);
}
\node[left] at (3.5,1) {$a$};
\node[left] at (3.5,2) {$a$};
\node[right] at (8.5,1) {$b$};
\node[right] at (8.5,2) {$b$};
\node[above] at (4,3) {$i_1$};
\node[above] at (6,3) {$\dots$};
\node[above] at (7,3) {$i_L$};
\node[below] at (4,0) {$j_1$};
\node[below] at (6,0) {$\dots$};
\node[below] at (7,0) {$j_L$};
\node at (7.5,1) {$*$};
\node at (7.5,2) {$*$};
\end{tikzpicture}}
\end{align*}
One concludes by deleting the $R$-matrix from the right of the picture; the justification for doing so is again the fact that the edges marked $*$ are frozen to the values $b$, and the trivial weight $W_{y/x}(b,b;b,b)=1$ for $b \in \{0,2\}$. The net result of this calculation is the switching of horizontal lattice lines; we conclude that
\begin{align*}
\bigotimes_{k = 1}^{L} \bra{j_k}_{k}
\mathcal{T}_{a,b}(x)\mathcal{T}_{a,b}(y) 
\bigotimes_{k = 1}^{L} \ket{i_k}_{k}
=
\bigotimes_{k = 1}^{L} \bra{j_k}_{k}
\mathcal{T}_{a,b}(y)\mathcal{T}_{a,b}(x) 
\bigotimes_{k = 1}^{L} \ket{i_k}_{k}
\end{align*}
for all $(i_1,\dots,i_L), (j_1,\dots,j_L) \in \{0,1,2\}^L$, which is precisely \eqref{IK-Tcommute}.
\end{proof}

\begin{rmk}
It is worth mentioning that Proposition \ref{prop:Tcommute} cannot be extended to the situation where $a,b$ take the value $1$. The point where the reasoning breaks is that in the Izergin--Korepin model, a vertex whose incoming states both assume the value $1$ is {\it not} forced to have outgoing states both equal to $1$; see, in particular, the middle row of Figure \ref{fig:19weights}. This lack of freezing prevents one from freely attaching/detaching an $R$-matrix vertex as in the previous argument.
\end{rmk}

\begin{prop}
\label{prop:IK-commute-nontrivial}
The following commutation relation holds:
\begin{multline}
\label{eqrowrel2}
\mathcal{T}_{2,2}(x) \mathcal{T}_{2,0}(y)
=
\\
W_{y/x}(0,2;0,2)
\mathcal{T}_{2,0}(y) \mathcal{T}_{2,2}(x)
+ 
W_{y/x}(1,1;0,2)
\mathcal{T}_{2,1}(y) \mathcal{T}_{2,1}(x)
+ 
W_{y/x}(2,0;0,2)
\mathcal{T}_{2,2}(y) \mathcal{T}_{2,0}(x).
\end{multline}
\end{prop}
\begin{proof}
The proof of $\eqref{eqrowrel2}$ follows by similar reasoning to that of Proposition \ref{prop:Tcommute}. One again begins by considering the matrix elements of the product $\mathcal{T}_{2,2}(x) \mathcal{T}_{2,0}(y)$, represented as a two-row partition function, and inserting an $R$-matrix vertex at the left of the lattice (using the fact that $W_{y/x}(2,2;2,2)=1$), similarly to \eqref{eq:insertR-left}. This time, after repeated application of the Yang--Baxter equation, the $R$-matrix vertex emerges from the right of the picture but it cannot simply be removed modulo freezing. Rather, one must sum over all possible values of the edges which connect this vertex to the lattice, resulting in the equation
\begin{multline}
\label{eqybepic2}
\raisebox{-10mm}{\begin{tikzpicture}[scale=0.6,every node/.style={scale=0.6}]
\foreach \i in {24} {
\draw[lgray,line width=1.5pt,->] (\i,0) -- (\i,3);
}
\foreach \i in {1,2} {
\draw[lgray,line width=1.5pt] (23.5,\i) -- (24.5,\i);
}
\foreach \i in {1,2} {
\node at (25,\i) {$\dots$}; 
}
\foreach \i in {26,27} {
\draw[lgray,line width=1.5pt,->] (\i,0) -- (\i,3);
}
\foreach \i in {1,2} {
\draw[lgray,line width=1.5pt,->] (25.5,\i) -- (27.5,\i);
}
\node at (24,-0.2) {$j_1$};
\node at (26,-0.2) {$\dots$};
\node at (27,-0.2) {$j_L$};
\node at (24,3.2) {$i_1$};
\node at (26,3.2) {$\dots$};
\node at (27,3.2) {$i_L$};
\node at (21,2) {$x$};
\node at (21,1) {$y$};
\foreach \i in {1,2} {
\draw[thick,->] (21.3,\i) -- (22,\i);
}
\draw[lgray,line width=1.5pt] (22.5,2) -- (23.5,1);
\draw[lgray,line width=1.5pt] (22.5,1) -- (23.5,2);
\node at (22.3,1) {$2$};
\node at (22.3,2) {$2$};
\node at (27.7,2) {$0$};
\node at (27.7,1) {$2$};
\end{tikzpicture}}
\quad
=
\\
\raisebox{-8mm}{\begin{tikzpicture}[scale=0.6,every node/.style={scale=0.6}]
\begin{scope}[shift = {(18,-4)}]
\foreach \i in {10.5} {
\draw[lgray,line width=1.5pt,->] (\i,0) -- (\i,3);
}
\node at (10.5,3.2) {$i_1$};
\node at (12.5,3.2) {$\dots$};
\node at (13.5,3.2) {$i_L$};
\node at (10.5,-0.2) {$j_1$};
\node at (12.5,-0.2) {$\dots$};
\node at (13.5,-0.2) {$j_L$};
\node at (9,1) {$y$};
\node at (9,2) {$x$};
\foreach \i in {1,2} {
\draw[->,thick] (9.2,\i) -- (9.6,\i);
}
\foreach \i in {1,2} {
\draw[lgray,line width=1.5pt] (10,\i) -- (11,\i);
}
\foreach \i in {1,2} {
\node at (11.5,\i) {$\dots$};
} 
\foreach \i in {12.5,13.5} {
\draw[lgray,line width=1.5pt,->] (\i,0) -- (\i,3);
} 
\foreach \i in {1,2} {
\draw[lgray,line width=1.5pt] (12,\i) -- (14,\i);
}
\draw[lgray,line width=1.5pt,->] (14,2) -- (15,1);
\draw[lgray,line width=1.5pt,->] (14,1) -- (15,2);
\node at (9.8,1) {$2$};
\node at (9.8,2) {$2$};
\node at (15.2,1) {$2$};
\node at (15.2,2) {$0$};
\node at (14,0.8) {$0$};
\node at (14,2.2) {$2$};
\end{scope}
\begin{scope}[shift = {(25.5,-4)}]
\node at (8.35,1.5) {$+$};
\foreach \i in {10.5} {
\draw[lgray,line width=1.5pt,->] (\i,0) -- (\i,3);
}
\node at (10.5,3.2) {$i_1$};
\node at (12.5,3.2) {$\dots$};
\node at (13.5,3.2) {$i_L$};
\node at (10.5,-0.2) {$j_1$};
\node at (12.5,-0.2) {$\dots$};
\node at (13.5,-0.2) {$j_L$};
\node at (9,1) {$y$};
\node at (9,2) {$x$};
\foreach \i in {1,2} {
\draw[->,thick] (9.2,\i) -- (9.6,\i);
}
\foreach \i in {1,2} {
\draw[lgray,line width=1.5pt] (10,\i) -- (11,\i);
}
\foreach \i in {1,2} {
\node at (11.5,\i) {$\dots$};
} 
\foreach \i in {12.5,13.5} {
\draw[lgray,line width=1.5pt,->] (\i,0) -- (\i,3);
} 
\foreach \i in {1,2} {
\draw[lgray,line width=1.5pt] (12,\i) -- (14,\i);
}
\draw[lgray,line width=1.5pt,->] (14,2) -- (15,1);
\draw[lgray,line width=1.5pt,->] (14,1) -- (15,2);
\node at (9.8,1) {$2$};
\node at (9.8,2) {$2$};
\node at (15.2,1) {$2$};
\node at (15.2,2) {$0$};
\node at (14,0.8) {$1$};
\node at (14,2.2) {$1$};
\end{scope}
\begin{scope}[shift = {(33,-4)}]
\node at (8.35,1.5) {$+$};
\foreach \i in {10.5} {
\draw[lgray,line width=1.5pt,->] (\i,0) -- (\i,3);
}
\node at (10.5,3.2) {$i_1$};
\node at (12.5,3.2) {$\dots$};
\node at (13.5,3.2) {$i_L$};
\node at (10.5,-0.2) {$j_1$};
\node at (12.5,-0.2) {$\dots$};
\node at (13.5,-0.2) {$j_L$};
\node at (9,1) {$y$};
\node at (9,2) {$x$};
\foreach \i in {1,2} {
\draw[->,thick] (9.2,\i) -- (9.6,\i);
}
\foreach \i in {1,2} {
\draw[lgray,line width=1.5pt] (10,\i) -- (11,\i);
}
\foreach \i in {1,2} {
\node at (11.5,\i) {$\dots$};
} 
\foreach \i in {12.5,13.5} {
\draw[lgray,line width=1.5pt,->] (\i,0) -- (\i,3);
} 
\foreach \i in {1,2} {
\draw[lgray,line width=1.5pt] (12,\i) -- (14,\i);
}
\draw[lgray,line width=1.5pt,->] (14,2) -- (15,1);
\draw[lgray,line width=1.5pt,->] (14,1) -- (15,2);
\node at (9.8,1) {$2$};
\node at (9.8,2) {$2$};
\node at (15.2,1) {$2$};
\node at (15.2,2) {$0$};
\node at (14,0.8) {$2$};
\node at (14,2.2) {$0$};
\end{scope}
\end{tikzpicture}}
\end{multline}
which is valid for all $(i_1,\dots,i_L), (j_1,\dots,j_L) \in \{0,1,2\}^L$. Matching the vertices which appear in this identity with their symbolic form \eqref{eqdefn19w}, we read off \eqref{eqrowrel2}.
\end{proof}

\subsection{Row operators in the infinite-volume limit}
\label{ssec:IK-infinite-row}

Our main interest in this work is not the vector space $\mathbb{V}^{(L)}$, but rather a suitable completion of it obtained by taking $L \rightarrow \infty$.\footnote{Equivalently, we shall work with partition functions with infinitely many vertical lines in the rightward direction.} To that end, we define
\begin{align}
\label{IKinfV}
\mathbb{V}
=
\bigoplus_{S \in \mathfrak{s}(2)} \mathbb{C} \ket{S},
\end{align}
where we recall the set $\mathfrak{s}(2)$ of $2$-strings given in Definition \ref{def:n-string}. 
We refer to $\mathbb{V}$ as the \textit{infinite-volume limit} of the space $\mathbb{V}^{(L)}$.

Some care is needed to extend the action of the finite row-operators \eqref{eqrowopIK} and \eqref{eqrowopIKdot} to the infinite volume; indeed, one may encounter convergence issues that are not present in finite size. We shall not attempt to address this issue for all of the possible choices of row-operators, but only for the subset of them that is needed in this work. This is done in the following proposition:
\begin{prop}
\label{prop:IK-ACD-define}
The objects
\begin{align*}
\mathcal{A}(x)
:=
\lim_{L \rightarrow \infty}
\mathcal{T}_{0,0}(x),
\qquad
\mathcal{C}(x)
:=
\lim_{L \rightarrow \infty}
\mathcal{T}_{2,0}(x),
\qquad
\dotr{\mathcal{D}}(x)
:=
\lim_{L \rightarrow \infty}
\dotr{\mathcal{T}}_{2,2}(x),
\end{align*}
are well-defined as operators in ${\rm End}(\mathbb{V})$.
\end{prop}

\begin{proof}
The matrix elements of $\mathcal{A}(x)$ and $\mathcal{C}(x)$ are represented by one-row partition functions of the form
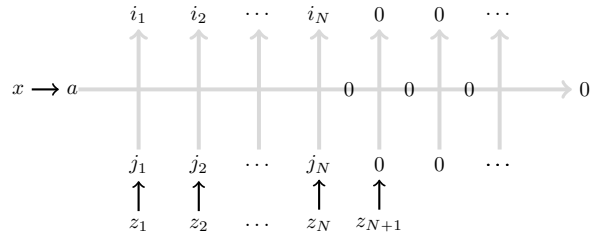
\[
\begin{tikzpicture}[scale = 0.8,every node/.style={scale=0.8}]
\node at (-2,0) (D1) {$x$};
\node at (-1.2,0) (D) {};
\draw[lgray,line width=1.5pt,->] (-1,0) -- (7.2,0);
\foreach \i in {0,1,2,3,4,5,6} {
\draw[<-,lgray,line width=1.5pt] (\i,1) -- (\i,-1);
}
\node at (-1.1,0) {$a$};
\node at (0,1.25) {$i_1$};
\node at (1,1.25) {$i_2$};
\node at (2,1.25) {$\dots$};
\node at (3,1.25) {$i_{N}$};
\node at (4,1.25) {$0$};
\node at (5,1.25) {$0$};
\node at (6,1.25) {$\dots$};
\node at (0,-1.25) (A) {$j_1$};
\node at (1,-1.25) (B) {$j_2$};
\foreach \i in {2,6} {
\node at (\i,-1.25) {$\dots$};
}
\node at (3,-1.25) {$j_{N}$};
\node at (4,-1.25) {$0$};
\node at (5,-1.25) {$0$};
\node at (6,-1.25) (C) {$\dots$};
\node[right] at (7.2,0) {$0$};
\node at (0,-2.25) (A1) {$z_1$};
\node at (1,-2.25) (B1) {$z_2$};
\node at (3,-2.25) {$z_N$};
\node at (4,-2.25) {$z_{N+1}$};
\node at (2,-2.25)  {$\dots$};
\draw[->,thick] (3,-2) -- (3,-1.5);
\draw[->,thick] (4,-2) -- (4,-1.5);
\draw[->,thick] (A1) -- (A);
\draw[->,thick] (B1) -- (B);
\draw[->,thick] (D1) -- (D);
\foreach \i in {3.5,4.5,5.5} {
\node at (\i,0) {$0$};
}
\end{tikzpicture}
\]
where $a \in \{0,2\}$. By virtue of the fact that $\exists\ N: i_k = j_k=0$ for all $k>N$, it is clear that all vertices in this diagram past the $N$-th column must take the form $W_{z_k/x}(0,0;0,0)=1$. As such, viewed as operators in ${\rm End}(\mathbb{V})$, the matrix entries of $\mathcal{A}(x)$ and $\mathcal{C}(x)$ consist only of finite products of rational functions.

A similar argument works for the matrix elements of $\dotr{\mathcal{D}}(x)$, which are represented by one-row partition functions using dotted vertices:
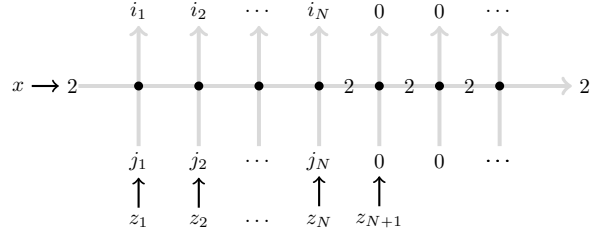
\[
\begin{tikzpicture}[scale = 0.8,every node/.style={scale=0.8}]
\node at (-2,0) (D1) {$x$};
\node at (-1.2,0) (D) {};
\draw[lgray,line width=1.5pt,->] (-1,0) -- (7.2,0);
\foreach \i in {0,1,2,3,4,5,6} {
\draw[<-,lgray,line width=1.5pt] (\i,1) -- (\i,-1);
}
\node at (-1.1,0) {$2$};
\node at (0,1.25) {$i_1$};
\node at (1,1.25) {$i_2$};
\node at (2,1.25) {$\dots$};
\node at (3,1.25) {$i_{N}$};
\node at (4,1.25) {$0$};
\node at (5,1.25) {$0$};
\node at (6,1.25) {$\dots$};
\node at (0,-1.25) (A) {$j_1$};
\node at (1,-1.25) (B) {$j_2$};
\foreach \i in {2,6} {
\node at (\i,-1.25) {$\dots$};
}
\node at (3,-1.25) {$j_{N}$};
\node at (4,-1.25) {$0$};
\node at (5,-1.25) {$0$};
\node at (6,-1.25) (C) {$\dots$};
\node[right] at (7.2,0) {$2$};
\node at (0,-2.25) (A1) {$z_1$};
\node at (1,-2.25) (B1) {$z_2$};
\node at (3,-2.25) {$z_N$};
\node at (4,-2.25) {$z_{N+1}$};
\node at (2,-2.25)  {$\dots$};
\draw[->,thick] (3,-2) -- (3,-1.5);
\draw[->,thick] (4,-2) -- (4,-1.5);
\draw[->,thick] (A1) -- (A);
\draw[->,thick] (B1) -- (B);
\draw[->,thick] (D1) -- (D);
\foreach \i in {3.5,4.5,5.5} {
\node at (\i,0) {$2$};
}
\foreach \i in {0,1,2,3,4,5,6} {
\node at (\i,0) [circle,fill,inner sep=1.5pt] {};
}
\end{tikzpicture}
\]
This time, $\exists\ N$ such that all vertices past the $N$-th column must take the form $\dotr{W}_{z_k/x}(0,2;0,2)=1$, and we once again conclude that entries of $\dotr{\mathcal{D}}(x)$ are meaningful.
\end{proof}

\begin{prop}
Operators $\mathcal{A}(x)$, $\mathcal{C}(x)$, $\dotr{\mathcal{D}}(x)$ satisfy the following commutation relations in ${\rm End}(\mathbb{V})$:
\begin{align}
\label{IK-ACD}
[\mathcal{A}(x),\mathcal{A}(y)]
=
[\mathcal{C}(x),\mathcal{C}(y)]
=
[\dotr{\mathcal{D}}(x),\dotr{\mathcal{D}}(y)]
=
0.
\end{align}
\end{prop}

\begin{proof}
The relations \eqref{IK-ACD} are a direct transcription of the $(a,b)=(0,0)$, $(a,b)=(2,0)$ cases of \eqref{IK-Tcommute}, as well as a normalized version of the $(a,b)=(2,2)$ case. All of these relations survive the transition to infinite volume, in view of Proposition \ref{prop:IK-ACD-define}.
\end{proof}

\begin{prop}
\label{propforCauchyIK1}
Fix $q,x,y,z_k \in \mathbb{C}$ such that for all $k \in \mathbb{Z}_{>0}$, the following conditions hold:
\begin{align}
\label{eqres1}
&
\begin{tikzpicture}[baseline={(0,0)}]
\node[scale = 0.8] at (-1.5,0) {$x$};
\node[scale = 0.8] at (-0.1,-1.2) {$z_k$};
\draw[thick,->] (-1.3,0) -- (-0.9,0);
\draw[thick,<-] (-0.1,-0.75) -- (-0.1,-1.05);
\draw[thick] (-1.7,-0.8) -- (-1.7,0.7);
\draw[->,lgray,line width=1.5pt] (-0.6,0) -- (0.4,0);
\draw[->,lgray,line width=1.5pt] (-0.1,-0.5) -- (-0.1,0.5);
\node[scale = 0.7] at (-0.7,0) {$1$};
\node[scale = 0.7] at (-0.1,-0.6) {$0$};
\node[scale = 0.7] at (-0.1,0.65) {$0$};
\node[scale = 0.7] at (0.5,0) {$1$};
\node at (-0.1,0) [circle,fill,inner sep=1.5pt] {};
\node at (0.75,0) {$\cdot$};
\begin{scope}[shift = {(2.5,0)}]
\node[scale = 0.8] at (-1.5,0) {$y$};
\node[scale = 0.8] at (-0.1,-1.2) {$z_k$};
\draw[thick,->] (-1.3,0) -- (-0.9,0);
\draw[thick,<-] (-0.1,-0.75) -- (-0.1,-1.05);
\draw[thick] (0.8,-0.8) -- (0.8,0.7);
\draw[->,lgray,line width=1.5pt] (-0.6,0) -- (0.4,0);
\draw[->,lgray,line width=1.5pt] (-0.1,-0.5) -- (-0.1,0.5);
\node[scale = 0.7] at (-0.7,0) {$1$};
\node[scale = 0.7] at (-0.1,-0.6) {$0$};
\node[scale = 0.7] at (-0.1,0.65) {$0$};
\node[scale = 0.7] at (0.5,0) {$1$};
\end{scope}
\end{tikzpicture}
=
\bigg| \frac{(x-q^3z_k)}{(x-qz_k)} \cdot \frac{(y-z_k)}{(y-q^{2}z_k)}\bigg|
<
\epsilon
<1,
\\
\label{eqres2}
&
\begin{tikzpicture}[baseline=(current bounding box.center)]
\node[scale = 0.8] at (-1.5,0) {$x$};
\node[scale = 0.8] at (-0.1,-1.2) {$z_k$};
\draw[thick,->] (-1.3,0) -- (-0.9,0);
\draw[thick,<-] (-0.1,-0.75) -- (-0.1,-1.05);
\draw[thick] (-1.7,-0.8) -- (-1.7,0.7);
\draw[->,lgray,line width=1.5pt] (-0.6,0) -- (0.4,0);
\draw[->,lgray,line width=1.5pt] (-0.1,-0.5) -- (-0.1,0.5);
\node[scale = 0.7] at (-0.7,0) {$0$};
\node[scale = 0.7] at (-0.1,-0.6) {$0$};
\node[scale = 0.7] at (-0.1,0.65) {$0$};
\node[scale = 0.7] at (0.5,0) {$0$};
\node at (-0.1,0) [circle,fill,inner sep=1.5pt] {};
\node at (0.75,0) {$\cdot$};
\begin{scope}[shift = {(2.5,0)}]
\node[scale = 0.8] at (-1.5,0) {$y$};
\node[scale = 0.8] at (-0.1,-1.2) {$z_k$};
\draw[thick,->] (-1.3,0) -- (-0.9,0);
\draw[thick,<-] (-0.1,-0.75) -- (-0.1,-1.05);
\draw[thick] (0.7,-0.8) -- (0.7,0.7);
\draw[->,lgray,line width=1.5pt] (-0.6,0) -- (0.4,0);
\draw[->,lgray,line width=1.5pt] (-0.1,-0.5) -- (-0.1,0.5);
\node[scale = 0.7] at (-0.7,0) {$2$};
\node[scale = 0.7] at (-0.1,-0.6) {$0$};
\node[scale = 0.7] at (-0.1,0.65) {$0$};
\node[scale = 0.7] at (0.5,0) {$2$};
\end{scope}
\end{tikzpicture}
=
\bigg| \frac{(x-q^2 z_k)(x-q^3z_k)}{(x-z_k)(x-q z_k)} \cdot \frac{(y-z_k)(y-qz_k)}{(y-q^2z_k)(y-q^3z_k)}\bigg|
<
\epsilon
<1,
\end{align}
for some constant $\epsilon \in \mathbb{R}$ which is independent of $k$. One then has the following commutation relation in ${\rm End}(\mathbb{V})$:
\begin{equation}
\label{eqspexchange1}
\dotr{\mathcal{D}}(x)\mathcal{C}(y)
=
\frac{(x-y)(x-qy)}{(x-q^2y)(x-q^3y)}
\mathcal{C}(y)\dotr{\mathcal{D}}(x).
\end{equation}
\end{prop}
\begin{proof}
Since it will be important in what follows, we note here the version of \eqref{eqrowrel2} obtained by dividing through by $\prod_{i=1}^{L} W_{z_i/x}(0,2;0,2)$. This converts all row operators carrying spectral parameter $x$ into their dotted counterparts:
\begin{multline}
\label{IK-pre-limit}
\dotr{\mathcal{T}}_{2,2}(x) \mathcal{T}_{2,0}(y)
=
\\
W_{y/x}(0,2;0,2)
\mathcal{T}_{2,0}(y) \dotr{\mathcal{T}}_{2,2}(x)
+ 
W_{y/x}(1,1;0,2)
\mathcal{T}_{2,1}(y) \dotr{\mathcal{T}}_{2,1}(x)
+ 
W_{y/x}(2,0;0,2)
\mathcal{T}_{2,2}(y) \dotr{\mathcal{T}}_{2,0}(x).
\end{multline}
We argue that in the limit $L \rightarrow \infty$, only the left hand side and the first term on the right hand side of \eqref{IK-pre-limit} have non-vanishing matrix elements when evaluated on $\mathbb{V}$: in fact, it is clear that these terms do indeed produce \eqref{eqspexchange1} under the limit in question. Our attention then shifts to showing that matrix elements of the second and third terms on the right hand side vanish as $L \rightarrow \infty$.

Examining the matrix elements of $\mathcal{T}_{2,1}(y) \dotr{\mathcal{T}}_{2,1}(x)$ and $\mathcal{T}_{2,2}(y) \dotr{\mathcal{T}}_{2,0}(x)$ as $L \rightarrow \infty$, one finds that they are given by partition functions of the form
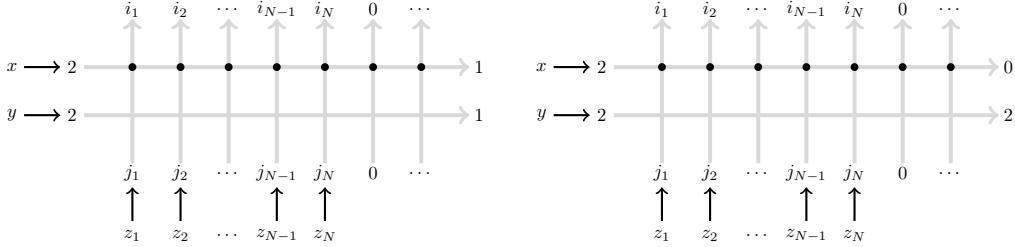
\begin{equation*}
\begin{tikzpicture}[scale=0.64,every node/.style={scale=0.7},baseline={([yshift=-0.5ex]current bounding box.center)}]
\draw[lgray,line width=1.5pt,->] (0,1) -- (8,1);
\draw[lgray,line width=1.5pt,->] (0,0) -- (8,0);
\foreach \i in {1,2,3,4,5,6,7} {
\draw[lgray,line width=1.5pt,<-] (\i,2) -- (\i,-1);
}
\foreach \i in {1,2} {
\draw[->,thick] (\i,-2.2) -- (\i,-1.5);
}
\foreach \i in {4,5} {
\draw[->,thick] (\i,-2.2) -- (\i,-1.5);
}
\foreach \i in {1,2}
{
\node at (\i,-2.5) {$z_{\i}$};
}
\node at (3,-2.5) {$\dots$};
\node at (4,-2.5) {$z_{N-1}$};
\node at (5,-2.5) {$z_{N}$};
\node at (-0.25,0) {$2$};
\node at (-0.25,1) {$2$};
\node at (8.2,1) {$1$};
\node at (8.2,0) {$1$};
\node at (1,2.2) {$i_1$};
\node at (2,2.2) {$i_2$};
\node at (1,-1.2) {$j_1$};
\node at (2,-1.2) {$j_2$};
\node at (3,-1.2) {$\dots$};
\node at (4,-1.2) {$j_{N-1}$};
\node at (5,-1.2) {$j_{N}$};
\foreach \i in {7} {
\node at (\i,-1.2) {$\dots$};
}
\node at (6,-1.2) {$0$};
\node at (3,2.2) {$\dots$};
\node at (4,2.2) {$i_{N-1}$};
\node at (5,2.2) {$i_{N}$};
\foreach \i in {7} {
\node at (\i,2.2) {$\dots$};
}
\node at (6,2.2) {$0$};
\node at (-1.5,0) {$y$};
\node at (-1.5,1) {$x$};
\draw[thick,->] (-1.25,0) -- (-0.5,0);
\draw[thick,->] (-1.25,1) -- (-0.5,1);
\foreach \i in {1,2,3,4,5,6,7} {
\node at (\i,1) [circle,fill,inner sep=1.5pt] {};
}
\begin{scope}[shift = {(11,0)}]
\draw[lgray,line width=1.5pt,->] (0,1) -- (8,1);
\draw[lgray,line width=1.5pt,->] (0,0) -- (8,0);
\foreach \i in {1,2,3,4,5,6,7} {
\draw[lgray,line width=1.5pt,<-] (\i,2) -- (\i,-1);
}
\foreach \i in {1,2} {
\draw[->,thick] (\i,-2.2) -- (\i,-1.5);
}
\foreach \i in {4,5} {
\draw[->,thick] (\i,-2.2) -- (\i,-1.5);
}
\foreach \i in {1,2}
{
\node at (\i,-2.5) {$z_{\i}$};
}
\node at (3,-2.5) {$\dots$};
\node at (4,-2.5) {$z_{N-1}$};
\node at (5,-2.5) {$z_{N}$};
\node at (-0.25,0) {$2$};
\node at (-0.25,1) {$2$};
\node at (8.2,1) {$0$};
\node at (8.2,0) {$2$};
\node at (1,2.2) {$i_1$};
\node at (2,2.2) {$i_2$};
\node at (1,-1.2) {$j_1$};
\node at (2,-1.2) {$j_2$};
\node at (3,-1.2) {$\dots$};
\node at (4,-1.2) {$j_{N-1}$};
\node at (5,-1.2) {$j_{N}$};
\foreach \i in {7} {
\node at (\i,-1.2) {$\dots$};
}
\node at (6,-1.2) {$0$};
\node at (3,2.2) {$\dots$};
\node at (4,2.2) {$i_{N-1}$};
\node at (5,2.2) {$i_{N}$};
\foreach \i in {7} {
\node at (\i,2.2) {$\dots$};
}
\node at (6,2.2) {$0$};
\node at (-1.5,0) {$y$};
\node at (-1.5,1) {$x$};
\draw[thick,->] (-1.25,0) -- (-0.5,0);
\draw[thick,->] (-1.25,1) -- (-0.5,1);
\foreach \i in {1,2,3,4,5,6,7} {
\node at (\i,1) [circle,fill,inner sep=1.5pt] {};
}
\end{scope}
\end{tikzpicture}
\end{equation*}
where as previously, $\exists\ N: i_k = j_k = 0$ for all $k>N$. It is then not difficult to show that, in both pictures, all columns beyond the $N$-th one will produce pairs of vertices either of the form \eqref{eqres1} or \eqref{eqres2}. Since we assume that the absolute values of these quantities are uniformly bounded below $1$, and they occur infinitely often, one finds that both of the above partition functions vanish.
\end{proof}

\section{Rational symmetric functions from the nineteen-vertex model}
\label{sec:IKrational}

The goal of this section is to introduce families of multivariate rational functions, derived as partition functions in the Izergin--Korepin model. Our approach is closely inspired by that of Section \ref{sec:6v-functions} for the six-vertex model.
Section \ref{ssec:IKF} contains the definition of the multivariate rational function $F_S$ and Section \ref{ssec:IKG} introduces its partner, the multivariate rational function $G_S$. Finally, in Section \ref{ssec:IK-symmetry}, we prove that both functions are symmetric in their primary alphabet.

\subsection{Multivariate rational function $F_{S}$}
\label{ssec:IKF}

\begin{defn}
Fix $N \in \mathbb{Z}_{>0}$ and let $S \in \mathfrak{s}(2)$ be a $2$-string such that $|S|=2N$. Fix also two alphabets $(x_1,\dots,x_N)$, $\bm{z} = (z_1,z_2,\dots)$. We define the multivariate rational function
\begin{equation}
\label{eqspsymIK2}
F_{S}(x_1,\dots,x_N;\bm{z}) = \bra{\varnothing} \mathcal{C}(x_1;\bm{z})\cdots\mathcal{C}(x_N;\bm{z}) \ket{S},
\end{equation}
where the row operators 
$\mathcal{C}(x_i;\bm{z})$ are given by Proposition \ref{prop:IK-ACD-define}, while
\begin{equation*}
\ket{S} = \bigotimes_{k = 1}^{\infty} \ket{S_{k}}_{k}
\quad \text{ and } \quad 
\bra{\varnothing} = \bigotimes_{k = 1}^{\infty} \bra{0}_k
\end{equation*}
denote states in $\mathbb{V}$ and its dual $\mathbb{V}^{*}$. We refer to $(x_1,\dots,x_N)$ and $(z_1,z_2,\dots)$ as the primary and secondary alphabets of $F_S$.
\end{defn}

Translating \eqref{eqspsymIK2} into its partition function form, one has that
\begin{equation}
\label{pictureofrational1}
F_S(x_1,\dots,x_N;\bm{z})
=
\tikz{0.7}{1.5cm}{
\foreach \y in {1,...,4}
\draw[lgray,line width=1.5pt,->] (0,\y) -- (10,\y);
\foreach \y in {1,...,4}
\node[left] at (0,\y) {2};
\foreach \y in {1,...,4}
\node[right] at (10,\y) {0};
\foreach \x in {1,...,9}
\draw[lgray,line width=1.5pt,->] (\x,0) -- (\x,5);
\node[left] at (-1.2,1) {$x_1\rightarrow$};
\node[left] at (-1.2,2.2) {$\vdots$};
\node[left] at (-1.2,3.2) {$\vdots$};
\node[left] at (-1.2,4) {$x_N\rightarrow$};
\node[above] at (1,5) {$S_1$};
\node[above] at (2,5) {$S_2$};
\node[above] at (3,5) {$S_3$};
\node[above] at (4,5) {$\cdots$};
\node[above] at (5,5) {$\cdots$};
\node[below] at (1,0) {0};
\node[below] at (2,0) {0};
\node[below] at (3,0) {0};
\node[below] at (4,-0.2) {$\cdots$};
\node[below] at (5,-0.2) {$\cdots$};
\node[below] at (1,-0.7) {$\uparrow$};
\node[below] at (2,-0.7) {$\uparrow$};
\node[below] at (3,-0.7) {$\uparrow$};
\node[below] at (1,-1.4) {$z_1$};
\node[below] at (2,-1.4) {$z_2$};
\node[below] at (3,-1.4) {$z_3$};
\node[below] at (4,-1.4) {$\cdots$};
\node[below] at (5,-1.4) {$\cdots$};
}
\end{equation}
where all bottom incoming and right outgoing edges are assigned the state $0$. Every left incoming edge is assigned the state $2$, while top outgoing edges correspond with the entries of $S$. The object \eqref{pictureofrational1} is to be interpreted as a partition function in the statistical mechanical sense: one works out all possible labelling of internal lattice edges and assigns a weight to each such configuration, which is the product of all individual vertices \eqref{eqdefn19w} that comprise it. Summing these weights across all possible configurations yields the rational function $F_S(x_1,\dots,x_N;\bm{z})$.

\begin{rmk}
A key difference between the rational functions $\eqref{pictureofrational1}$ and their six-vertex counterparts $\eqref{picofsixf1}$ lies in the alteration of the states on the left of the lattice: these all assume the value $2$, rather than $1$ as in the six-vertex case. Paths continue to enter from the left and exit towards the top of the lattice; however, in the Izergin--Korepin model there is increased freedom, as pairs of paths can now split and recombine multiple times under the weights $W_{z/x}(0,2;1,1)$, $W_{z/x}(2,0;1,1)$, $W_{z/x}(1,1;0,2)$, $W_{z/x}(1,1;2,0)$. As such, the functions $F_S(x_1,\dots,x_N;\bm{z})$ contain many more non-trivial configurations than their six-vertex cousins.
\end{rmk}

\begin{rmk}
\label{rmk:IK-truncate}
Since the index of $F_S(x_1,\dots,x_N;\bm{z})$ is a $2$-string, there exists an integer $M$ such that $S_k=0$ for all $k>M$. It is then straightforward to see that all columns to the right of column $M$ freeze in the partition function \eqref{pictureofrational1}: they are comprised solely of vertices $W_{z_k/x_i}(0,0;0,0)$ which have weight $1$. In view of this fact, it is often convenient to identify the function \eqref{pictureofrational1} with the following finite-lattice version:
\begin{align}
\label{eq:IK-F-truncate}
F_S(x_1,\dots,x_N;\bm{z})
=
\tikz{0.6}{1.5cm}{
\foreach \y in {1,...,4}
\draw[lgray,line width=1.5pt,->] (0,\y) -- (8,\y);
\foreach \y in {1,...,4}
\node[left] at (0,\y) {$2$};
\foreach \y in {1,...,4}
\node[right] at (8,\y) {$0$};
\foreach \x in {1,...,7}
\draw[lgray,line width=1.5pt,->] (\x,0) -- (\x,5);
\node[left] at (-0.5,1) {$x_1\rightarrow$};
\node[left] at (-0.5,2) {$\vdots$};
\node[left] at (-0.5,3) {$\vdots$};
\node[left] at (-0.5,4) {$x_N \rightarrow$};
\node[above] at (1,5) {$S_1$};
\node[above] at (2,5) {$S_2$};
\node[above] at (3,5) {$\cdots$};
\node[above] at (4,5) {$\cdots$};
\node[above] at (5,5) {$\cdots$};
\node[above] at (7,5) {$S_M$};
\node[below] at (1,0) {$0$};
\node[below] at (2,0) {$0$};
\node[below] at (3,0) {$\cdots$};
\node[below] at (4,0) {$\cdots$};
\node[below] at (5,0) {$\cdots$};
\node[below] at (7,0) {$0$};
\node[below] at (1,-0.7) {$\uparrow$};
\node[below] at (2,-0.7) {$\uparrow$};
\node[below] at (7,-0.7) {$\uparrow$};
\node[below] at (1,-1.4) {$z_1$};
\node[below] at (2,-1.4) {$z_2$};
\node[below] at (3,-1.4) {$\cdots$};
\node[below] at (4,-1.4) {$\cdots$};
\node[below] at (5,-1.4) {$\cdots$};
\node[below] at (7,-1.4) {$z_M$};
}
\end{align}
and we write $S \equiv (S_1,\dots,S_M)$. In what follows, we will pass freely between the two different formulations \eqref{pictureofrational1} and \eqref{eq:IK-F-truncate}. 
\end{rmk}

\begin{ex}
\label{ex:IKF}
Sample lattice configuration of \eqref{pictureofrational1} when $N=4$, $S=(1,2,0,1,1,2,1,0,0,\dots)$: 
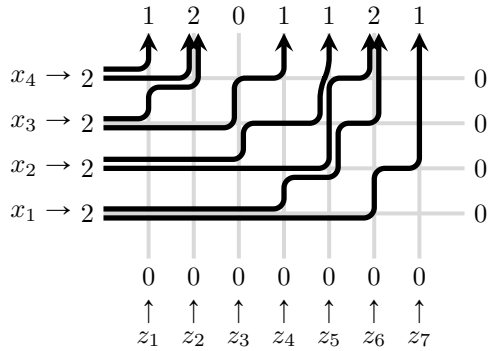
\begin{align*}
\tikz{0.6}{1.5cm}{
\foreach \y in {1,...,4}
\draw[lgray,line width=1.5pt] (0,\y) -- (8,\y);
\foreach \y in {1,...,4}
\node[left] at (0,\y) {$2$};
\foreach \y in {1,...,4}
\node[right] at (8,\y) {$0$};
\foreach \x in {1,...,7}
\draw[lgray,line width=1.5pt] (\x,0) -- (\x,5);
\node[left] at (-0.5,1) {$x_1\rightarrow$};
\node[left] at (-0.5,2) {$x_2 \rightarrow$};
\node[left] at (-0.5,3) {$x_3 \rightarrow$};
\node[left] at (-0.5,4) {$x_4 \rightarrow$};
\node[above] at (1,5) {$1$};
\node[above] at (2,5) {$2$};
\node[above] at (3,5) {$0$};
\node[above] at (4,5) {$1$};
\node[above] at (5,5) {$1$};
\node[above] at (6,5) {$2$};
\node[above] at (7,5) {$1$};
\node[below] at (1,0) {$0$};
\node[below] at (2,0) {$0$};
\node[below] at (3,0) {$0$};
\node[below] at (4,0) {$0$};
\node[below] at (5,0) {$0$};
\node[below] at (6,0) {$0$};
\node[below] at (7,0) {$0$};
\node[below] at (1,-0.7) {$\uparrow$};
\node[below] at (2,-0.7) {$\uparrow$};
\node[below] at (3,-0.7) {$\uparrow$};
\node[below] at (4,-0.7) {$\uparrow$};
\node[below] at (5,-0.7) {$\uparrow$};
\node[below] at (6,-0.7) {$\uparrow$};
\node[below] at (7,-0.7) {$\uparrow$};
\node[below] at (1,-1.4) {$z_1$};
\node[below] at (2,-1.4) {$z_2$};
\node[below] at (3,-1.4) {$z_3$};
\node[below] at (4,-1.4) {$z_4$};
\node[below] at (5,-1.4) {$z_5$};
\node[below] at (6,-1.4) {$z_6$};
\node[below] at (7,-1.4) {$z_7$};
\draw[line width = 2,->,rounded corners] (0,4.2) -- (1,4.2) -- (1,5);
\draw[line width = 2,->,rounded corners] (0,4) -- (1.9,4) -- (1.9,5);
\draw[line width = 2,->,rounded corners] (0,3.1) -- (1,3.1) -- (1,3.8) -- (2.1,3.8) -- (2.1,5);
\draw[line width = 2,->,rounded corners] (0,2.9) -- (2.9,2.9) -- (2.9,4) -- (4,4) -- (4,5);
\draw[line width = 2,->,rounded corners] (0,2.2) -- (3.1,2.2) -- (3.1,3) -- (4.8,3) -- (4.8,3.8) -- (5,4.2) -- (5,5);
\draw[line width = 2,->,rounded corners] (0,2) -- (5,2) -- (5,4) -- (5.9,4) -- (5.9,5);
\draw[line width = 2,->,rounded corners] (0,1.1) -- (4,1.1) -- (4,1.8) -- (5.2,1.8) -- (5.2,3) -- (6.1,3) -- (6.1,5);
\draw[line width = 2,->,rounded corners] (0,0.9) -- (6,0.9) -- (6,2) -- (7,2) -- (7,5);
}
\end{align*}
As stated in Remark \ref{rmk:IK-truncate}, is only necessary to draw the first seven columns of the lattice, as all subsequent columns are necessarily devoid of paths and have trivial weight.
\end{ex}

\subsection{Multivariate rational function $G_{S}$}
\label{ssec:IKG}

\begin{defn}
Fix $N \in \mathbb{Z}_{>0}$ and let $S \in \mathfrak{s}(2)$ be a $2$-string such that $|S|=2N$. Fix a further integer $M \in \mathbb{Z}_{>0}$ and two alphabets $(y_1,\dots,y_M)$, $\bm{z} = (z_1,z_2,\dots)$. We define the multivariate rational function
\begin{equation}
\label{eqdualsym1}
G_{S}(y_1,\dots,y_M;\bm{z}) = \bra{2^N,\varnothing} \mathcal{A}(y_1;\bm{z}) \cdots \mathcal{A}(y_M;\bm{z}) \ket{S},
\end{equation}
where the row operators $\mathcal{A}(y_i;\bm{z})$ are given by Proposition \ref{prop:IK-ACD-define}, while
\begin{equation*}
\ket{S} = \bigotimes_{k = 1}^{\infty} \ket{S_{k}}_{k}
\quad \text{ and } \quad 
\bra{2^N,\varnothing} = \bigg(\bigotimes_{k = 1}^{N} \bra{2}_k\bigg) \otimes \bigg( \bigotimes_{k = N+1}^{\infty} \bra{0}_{k}\bigg)
\end{equation*}
denote states in $\mathbb{V}$ and its dual $\mathbb{V}^{*}$. We refer to $(y_1,\dots,y_M)$ and $(z_1,z_2,\dots)$ as the primary and secondary alphabets of $G_S$.
\end{defn}

Translating \eqref{eqdualsym1} into its partition function form, one has that 
\begin{align}
\label{eqdualrf1}
G_S(y_1,\dots,y_M;\bm{z})
=
\tikz{0.7}{2cm}{
\foreach \y in {1,...,5}
\draw[lgray,line width=1.5pt,->] (0,\y) -- (10,\y);
\foreach \x in {1,...,9}
\draw[lgray,line width=1.5pt,->] (\x,0) -- (\x,6);
\foreach \y in {1,...,5}
\node[left] at (0,\y) {0};
\foreach \y in {1,...,5}
\node[right] at (10,\y) {0};
\node[left] at (-1.2,1) {$y_1\rightarrow$};
\node[left] at (-1.2,2.2) {$\vdots$};
\node[left] at (-1.2,3.2) {$\vdots$};
\node[left] at (-1.2,4.2) {$\vdots$};
\node[left] at (-1.2,5) {$y_M\rightarrow$};
\node[above] at (1,6) {$S_1$};
\node[above] at (2,6) {$\cdots$};
\node[above] at (3,6) {$\cdots$};
\node[above] at (4,6) {$S_N$};
\node[above] at (5.1,5.95) {$S_{N+1}$};
\node[above] at (6.2,6) {$\cdots$};
\node[above] at (7.2,6) {$\cdots$};
\node[below] at (1,0) {2};
\node[below] at (2,-0.2) {$\cdots$};
\node[below] at (3,-0.2) {$\cdots$};
\node[below] at (4,0) {2};
\node[below] at (5,0) {0};
\node[below] at (6,-0.2) {$\cdots$};
\node[below] at (7,-0.2) {$\cdots$};
\node[below] at (1,-0.7) {$\uparrow$};
\node[below] at (4,-0.7) {$\uparrow$};
\node[below] at (5,-0.7) {$\uparrow$};
\node[below] at (1,-1.4) {$z_1$};
\node[below] at (2,-1.4) {$\cdots$};
\node[below] at (3,-1.4) {$\cdots$};
\node[below] at (4,-1.4) {$z_N$};
\node[below] at (5,-1.4) {$z_{N+1}$};
\node[below] at (6,-1.4) {$\cdots$};
\node[below] at (7,-1.4) {$\cdots$};
}
\end{align}
In this picture, all left incoming and right outgoing edges are assigned the state 0. The first $N$ incoming bottom edges are assigned the state 2; the remaining ones are all set to 0. As in the case of the functions \eqref{pictureofrational1}, the top outgoing edges correspond with the entries of $S$. We emphasize that in this definition, the number of horizontal rows, $M$, is independent of $N$.

\begin{ex}
\label{ex:IKG}
Sample configuration of \eqref{eqdualrf1} when $N=4$, $M=5$, $S=(1,2,0,1,1,2,1,0,0,\dots)$:
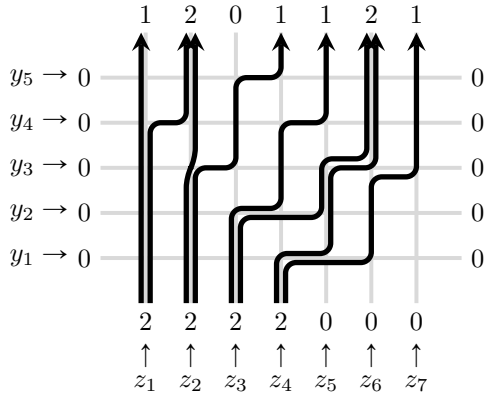
\begin{align*}
\tikz{0.6}{1.5cm}{
\foreach \y in {1,...,5}
\draw[lgray,line width=1.5pt] (0,\y) -- (8,\y);
\foreach \y in {1,...,5}
\node[left] at (0,\y) {$0$};
\foreach \y in {1,...,5}
\node[right] at (8,\y) {$0$};
\foreach \x in {1,...,7}
\draw[lgray,line width=1.5pt] (\x,0) -- (\x,6);
\node[left] at (-0.5,1) {$y_1\rightarrow$};
\node[left] at (-0.5,2) {$y_2 \rightarrow$};
\node[left] at (-0.5,3) {$y_3 \rightarrow$};
\node[left] at (-0.5,4) {$y_4 \rightarrow$};
\node[left] at (-0.5,5) {$y_5 \rightarrow$};
\node[above] at (1,6) {$1$};
\node[above] at (2,6) {$2$};
\node[above] at (3,6) {$0$};
\node[above] at (4,6) {$1$};
\node[above] at (5,6) {$1$};
\node[above] at (6,6) {$2$};
\node[above] at (7,6) {$1$};
\node[below] at (1,0) {$2$};
\node[below] at (2,0) {$2$};
\node[below] at (3,0) {$2$};
\node[below] at (4,0) {$2$};
\node[below] at (5,0) {$0$};
\node[below] at (6,0) {$0$};
\node[below] at (7,0) {$0$};
\node[below] at (1,-0.7) {$\uparrow$};
\node[below] at (2,-0.7) {$\uparrow$};
\node[below] at (3,-0.7) {$\uparrow$};
\node[below] at (4,-0.7) {$\uparrow$};
\node[below] at (5,-0.7) {$\uparrow$};
\node[below] at (6,-0.7) {$\uparrow$};
\node[below] at (7,-0.7) {$\uparrow$};
\node[below] at (1,-1.4) {$z_1$};
\node[below] at (2,-1.4) {$z_2$};
\node[below] at (3,-1.4) {$z_3$};
\node[below] at (4,-1.4) {$z_4$};
\node[below] at (5,-1.4) {$z_5$};
\node[below] at (6,-1.4) {$z_6$};
\node[below] at (7,-1.4) {$z_7$};
\draw[line width = 2,->,rounded corners] (0.9,0) -- (0.9,6);
\draw[line width = 2,->,rounded corners] (1.1,0) -- (1.1,4) -- (1.9,4) -- (1.9,6);
\draw[line width = 2,->,rounded corners] (1.9,0) -- (1.9,2.8) -- (2.1,3.2) -- (2.1,6);
\draw[line width = 2,->,rounded corners] (2.1,0) -- (2.1,3) -- (3,3) -- (3,5) -- (4,5) -- (4,6);
\draw[line width = 2,->,rounded corners] (2.9,0) -- (2.9,2.1) -- (4,2.1) -- (4,4) -- (5,4) -- (5,6);
\draw[line width = 2,->,rounded corners] (3.1,0) -- (3.1,1.9) -- (4.9,1.9) -- (4.9,3.2) -- (5.9,3.2) -- (5.9,6);
\draw[line width = 2,->,rounded corners] (3.9,0) -- (3.9,1.1) -- (5.1,1.1) -- (5.1,3) -- (6.1,3) -- (6.1,6);
\draw[line width = 2,->,rounded corners] (4.1,0) -- (4.1,0.9) -- (6,0.9) -- (6,2.8) -- (7,2.8) -- (7,6);
}
\end{align*}
Again, it is only necessary to draw the first seven columns of the lattice, as all subsequent columns are necessarily devoid of paths and have trivial weight.
\end{ex}

\subsection{Symmetry in primary alphabet}
\label{ssec:IK-symmetry}

Below we present our first main result of the chapter.

\begin{thm}
\label{thm:IK-FG-sym}
The rational functions $F_{S}(x_1,\dots,x_N;\bm{z})$ and $G_{S}(y_1,\dots,y_M;\bm{z})$, defined in $\eqref{eqspsymIK2}$ and $\eqref{eqdualsym1}$, are symmetric in their primary alphabets.
\end{thm}

\begin{proof}
The symmetry of $F_{S}(x_1,\dots,x_N;\bm{z})$ in $(x_1,\dots,x_N)$ follows from the commutation relation $[\mathcal{C}(x_i;\bm{z}),\mathcal{C}(x_j;\bm{z})] = 0$ for $i \not= j$; see equation \eqref{IK-ACD}. In a similar vein, the symmetry of $G_{S}(y_1,\dots,y_M;\bm{z})$ in $(y_1,\dots,y_M)$ follows from the commutation relation $[\mathcal{A}(y_i;\bm{z}),\mathcal{A}(y_j;\bm{z})] = 0$; see again \eqref{IK-ACD}.
\end{proof}

\section{Cauchy identities in the nineteen-vertex model}
\label{sec:IKcauchy}

In Section \ref{ssec:IK-cauchy}, we state a Cauchy summation identity for the rational functions $F_{S}$ and $G_{S}$ defined in $\eqref{eqspsymIK2}$ and $\eqref{eqdualsym1}$. The proof of this identity is split over the subsequent sections, and closely follows ideas from the proof of the Cauchy identity in the six-vertex model. In particular, Section \ref{ssec:IK-flip} introduces a further rational function $\dotr{G}_{S}$, which is related to the original $G_S$ via the symmetry \eqref{eqdottondot}. Making use of this alternative formulation of $G_S$, one is able to cast the Cauchy identity in algebraic form, and compute it using commutation relations among row operators; this is done in Section \ref{ssec:IK-cauchy-proof}.

\subsection{Cauchy identity}
\label{ssec:IK-cauchy}

\begin{defn}[Counting function]
\label{eqcounting1}
For each $2$-string $S=(S_1,S_2,\dots)$ and integer $1 \leq i \leq 2$, define the \textit{counting function} $\mathfrak{c}_{i}(S)$ to be multiplicity of $i$ in $S$; namely, $\mathfrak{c}_i(S) = {\rm Card}\{j: S_j=i\}$.
\end{defn}
\begin{thm}
\label{thm:IK-cauchy}
Let $N,M \in \mathbb{Z}_{>0}$ be fixed positive integers. Fix parameters $q,x_i,y_j,z_k \in \mathbb{C}$ such that for all $i \in \{1,\dots,N\}$, $j \in \{1,\dots,M\}$ and $k \in \mathbb{Z}_{>0}$, the following conditions hold:
\begin{align}
\label{eqres1-inhom}
&
\begin{tikzpicture}[baseline={(0,0)}]
\node[scale = 0.8] at (-1.5,0) {$y_j^{-1}$};
\node[scale = 0.8] at (-0.1,-1.2) {$z_k$};
\draw[thick,->] (-1.3,0) -- (-0.9,0);
\draw[thick,<-] (-0.1,-0.75) -- (-0.1,-1.05);
\draw[thick] (-1.9,-0.8) -- (-1.9,0.7);
\draw[->,lgray,line width=1.5pt] (-0.6,0) -- (0.4,0);
\draw[->,lgray,line width=1.5pt] (-0.1,-0.5) -- (-0.1,0.5);
\node[scale = 0.7] at (-0.7,0) {$1$};
\node[scale = 0.7] at (-0.1,-0.6) {$0$};
\node[scale = 0.7] at (-0.1,0.65) {$0$};
\node[scale = 0.7] at (0.5,0) {$1$};
\node at (-0.1,0) [circle,fill,inner sep=1.5pt] {};
\node at (0.75,0) {$\cdot$};
\begin{scope}[shift = {(2.5,0)}]
\node[scale = 0.8] at (-1.5,0) {$x_i$};
\node[scale = 0.8] at (-0.1,-1.2) {$z_k$};
\draw[thick,->] (-1.3,0) -- (-0.9,0);
\draw[thick,<-] (-0.1,-0.75) -- (-0.1,-1.05);
\draw[thick] (0.8,-0.8) -- (0.8,0.7);
\draw[->,lgray,line width=1.5pt] (-0.6,0) -- (0.4,0);
\draw[->,lgray,line width=1.5pt] (-0.1,-0.5) -- (-0.1,0.5);
\node[scale = 0.7] at (-0.7,0) {$1$};
\node[scale = 0.7] at (-0.1,-0.6) {$0$};
\node[scale = 0.7] at (-0.1,0.65) {$0$};
\node[scale = 0.7] at (0.5,0) {$1$};
\end{scope}
\end{tikzpicture}
=
\bigg| \frac{(1-q^3 y_j z_k)}{(1-q y_j z_k)} \cdot 
\frac{(x_i-z_k)}{(x_i-q^{2}z_k)}\bigg|
<
\epsilon
<1,
\\
\label{eqres2-inhom}
&
\begin{tikzpicture}[baseline=(current bounding box.center)]
\node[scale = 0.8] at (-1.5,0) {$y_j^{-1}$};
\node[scale = 0.8] at (-0.1,-1.2) {$z_k$};
\draw[thick,->] (-1.3,0) -- (-0.9,0);
\draw[thick,<-] (-0.1,-0.75) -- (-0.1,-1.05);
\draw[thick] (-1.9,-0.8) -- (-1.9,0.7);
\draw[->,lgray,line width=1.5pt] (-0.6,0) -- (0.4,0);
\draw[->,lgray,line width=1.5pt] (-0.1,-0.5) -- (-0.1,0.5);
\node[scale = 0.7] at (-0.7,0) {$0$};
\node[scale = 0.7] at (-0.1,-0.6) {$0$};
\node[scale = 0.7] at (-0.1,0.65) {$0$};
\node[scale = 0.7] at (0.5,0) {$0$};
\node at (-0.1,0) [circle,fill,inner sep=1.5pt] {};
\node at (0.75,0) {$\cdot$};
\begin{scope}[shift = {(2.5,0)}]
\node[scale = 0.8] at (-1.5,0) {$x_i$};
\node[scale = 0.8] at (-0.1,-1.2) {$z_k$};
\draw[thick,->] (-1.3,0) -- (-0.9,0);
\draw[thick,<-] (-0.1,-0.75) -- (-0.1,-1.05);
\draw[thick] (0.7,-0.8) -- (0.7,0.7);
\draw[->,lgray,line width=1.5pt] (-0.6,0) -- (0.4,0);
\draw[->,lgray,line width=1.5pt] (-0.1,-0.5) -- (-0.1,0.5);
\node[scale = 0.7] at (-0.7,0) {$2$};
\node[scale = 0.7] at (-0.1,-0.6) {$0$};
\node[scale = 0.7] at (-0.1,0.65) {$0$};
\node[scale = 0.7] at (0.5,0) {$2$};
\end{scope}
\end{tikzpicture}
=
\bigg| \frac{(1-q^2 y_j z_k)(1-q^3 y_j z_k)}{(1- y_j z_k)(1-q y_j z_k)} \cdot 
\frac{(x_i-z_k)(x_i-qz_k)}{(x_i-q^2z_k)(x_i-q^3z_k)}\bigg|
<
\epsilon
<1,
\end{align}
for some constant $\epsilon \in \mathbb{R}$ which is independent of $i,j,k$. For such a choice of parameters, the rational functions \eqref{eqspsymIK2} and \eqref{eqdualsym1} satisfy the summation identity
\begin{multline}
\label{eqcc5}
\sum_{S \in \mathfrak{s}(2)} 
c_{S}(q)
F_{S}(x_1,\dots,x_N;\bm{z})
G_{S}(y_1,\dots,y_M;q^{-3}\bm{z}^{-1})
\\
=
q^{N(2N+1)}
F_{(2^{N})}(x_1,\dots,x_N;\bm{z}) 
\prod_{i=1}^{N}
\prod_{j=1}^{M}
\frac{(1-q^2x_iy_j)(1-q^3x_iy_j)}{(1-x_iy_j)(1-qx_iy_j)},
\end{multline}
where the sum is taken over all $2$-strings $S$ such that $|S|=2N$. Here $\bm{z} = (z_1,z_2,\dots)$ as usual, while $q^{-3}\bm{z}^{-1} = (q^{-3}z_1^{-1},q^{-3}z_2^{-1},\dots)$. $F_{(2^{N})}$ appearing on the right hand side corresponds to the function \eqref{eqspsymIK2} for $S=(2^N,0,0,\dots)$, and the constant $c_S(q)$ is given by
\begin{equation}
\label{eqcsfact1}
c_{S}(q) 
=
(-1)^{\mathfrak{c}_1(S)/2}
(1-q)^{-\mathfrak{c}_1(S)}
q^{-\mathfrak{c}_2(S)}
\prod_{k=1}^{\infty} q^{2kS_k}.
\end{equation}
\end{thm}

The proof of Theorem \ref{thm:IK-cauchy} is deferred to Section \ref{ssec:IK-cauchy-proof}.

\begin{rmk}
As with the corresponding Cauchy identity \eqref{eqccsix5} in the six-vertex case, one again sees the appearance of a domain-wall-type partition function on the right hand side of \eqref{eqcc5}. In contrast to the six-vertex case, however, there is no known determinant formula for $F_{(2^{N})}$ other than when the parameter $q$ is set to a root of unity \cite{Garbali16}.
\end{rmk}

\subsection{Flip symmetry of $G_S$}
\label{ssec:IK-flip}

In a similar vein to the six-vertex model, we introduce here the rational function $\dotr{G}_{S}$, which will play a key role in the proof of the Cauchy identity \eqref{eqcc5}.

\begin{defn}
Fix $N \in \mathbb{Z}_{>0}$ and let $S \in \mathfrak{s}(2)$ be a $2$-string such that $|S|=2N$. Fix a further integer $M \in \mathbb{Z}_{>0}$ and two alphabets $(y_1,\dots,y_M)$, $\bm{z} = (z_1,z_2,\dots)$. We define the multivariate rational function
\begin{equation}
\label{eqrff2}
\dotr{G}_{S}(y_1,\dots,y_M;\bm{z}) = \bra{S} \dotr{\mathcal{D}}(y_M;\bm{z}) \cdots \dotr{\mathcal{D}}(y_1;\bm{z}) \ket{2^N,\varnothing},
\end{equation}
where the row operators $\dotr{\mathcal{D}}(y_i;\bm{z})$ are given by Proposition \ref{prop:IK-ACD-define}, while
\begin{equation*}
\ket{2^N,\varnothing} = \bigg(\bigotimes_{k = 1}^{N} \ket{2}_k\bigg) \otimes \bigg( \bigotimes_{k = N+1}^{\infty} \ket{0}_{k}\bigg)
\quad \text{ and } \quad 
\bra{S} = \bigotimes_{k = 1}^{\infty} \bra{S_{k}}_{k}
\end{equation*}
denote states in $\mathbb{V}$ and its dual $\mathbb{V}^{*}$.
\end{defn}

Translating \eqref{eqrff2} into its partition function form, it is given by
\begin{align}
\label{eqdualdot}
\dotr{G}_S(y_1,\dots,y_M;\bm{z})
=
\tikz{0.7}{2cm}{
\foreach \y in {1,...,5}
\draw[lgray,line width=1.5pt,->] (0,\y) -- (10,\y);
\foreach \x in {1,...,9}
\draw[lgray,line width=1.5pt,->] (\x,0) -- (\x,6);
\foreach \y in {1,...,5}
\node[left] at (0,\y) {2};
\foreach \y in {1,...,5}
\node[right] at (10,\y) {2};
\node[left] at (-1.2,1) {$y_M\rightarrow$};
\node[left] at (-1.2,2.2) {$\vdots$};
\node[left] at (-1.2,3.2) {$\vdots$};
\node[left] at (-1.2,4.2) {$\vdots$};
\node[left] at (-1.2,5) {$y_1\rightarrow$};
\node[above] at (1,6) {2};
\node[above] at (2,6) {$\cdots$};
\node[above] at (3,6) {$\cdots$};
\node[above] at (4,6) {2};
\node[above] at (5,6) {0};
\node[above] at (6.2,6) {$\cdots$};
\node[above] at (7.2,6) {$\cdots$};
\node[below] at (1,0) {$S_1$}; 
\node[below] at (2,-0.2) {$\cdots$};
\node[below] at (3,-0.2) {$\cdots$};
\node[below] at (4,0) {$S_N$}; 
\node[below] at (5.2,0) {$S_{N+1}$};
\node[below] at (6.4,-0.2) {$\cdots$};
\node[below] at (7.4,-0.2) {$\cdots$};
\node[below] at (1,-0.7) {$\uparrow$};
\node[below] at (4,-0.7) {$\uparrow$};
\node[below] at (5,-0.7) {$\uparrow$};
\node[below] at (1,-1.4) {$z_1$};
\node[below] at (2,-1.4) {$\cdots$};
\node[below] at (3,-1.4) {$\cdots$};
\node[below] at (4,-1.4) {$z_N$};
\node[below] at (5,-1.4) {$z_{N+1}$};
\node[below] at (6,-1.4) {$\cdots$};
\node[below] at (7,-1.4) {$\cdots$};
\foreach\x in {1,...,9}{
\foreach\y in {1,...,5}{
\node at (\x,\y) [circle,fill,inner sep=1.5pt] {};
}};
}
\end{align}
where the quantity \eqref{eqdualdot} is well-defined, despite the change in normalization of the underlying vertex weights. Indeed, one sees that the vertices in all columns sufficiently far to the right are frozen with weight $\dotr{W}_{z_j/y_i}(0,2;0,2)=1$.

\begin{prop}
\label{prop:IK-flip-sym}
Fix an integer $M \in \mathbb{Z}_{>0}$ and a $2$-string $S \in \mathfrak{s}(2)$ such that $|S|=2N$. The rational functions \eqref{eqdualsym1} and \eqref{eqrff2} are related under the following symmetry:
\begin{equation}
\label{eqdottondot}
q^{-N(2N+1)}
c_{S}(q)
G_{S}(y_1,\dots,y_M;q^{-3}\bm{z}^{-1}) 
=  
\dotr{G}_{S}(y_1^{-1},\dots,y_M^{-1};\bm{z}),
\end{equation}
where the constant $c_{S}(q)$ is defined in $\eqref{eqcsfact1}$.
\end{prop}

\begin{proof}
The proof is by repeated application of the local symmetry relation \eqref{eqnineteencym1}. To make the procedure more transparent, we begin by stating a version of \eqref{eqnineteencym1} that may be applied to towers of vertices. In particular, one notes that
\begin{align}
\label{IK-tower-sym}
\prod_{a=1}^{M} q^{m \bar{j}_a}
\begin{tikzpicture}[scale=0.8,baseline=0cm]
\draw[->,lgray,line width=1.5pt] (4,-2.75) -- (4,2.75);
\foreach \i in {-2,-1,0,1,2}
{\draw[->,lgray,line width=1.5pt] (3,\i) -- (5,\i);}
\node at (2.75,-2) {$j_M$};
\node at (2.75,-1) {$\vdots$};
\foreach \i in {0.1,1.1} {
\node at (2.75,\i) {$\vdots$};
}
\node at (2.75,2) {$j_1$};
\node at (5.35,-2) {$\ell_M$};
\node at (5.35,-1) {$\vdots$};
\foreach \i in {0.1,1.1} {
\node at (5.35,\i) {$\vdots$};
}
\node at (5.35,2) {$\ell_1$};
\node at (1.2,-2) {$y_M^{-1}$};
\node at (1.2,-1) {$\vdots$};
\node at (1.1,2) {$y_1^{-1}$};
\foreach \i in {0,1} {
\node at (1.2,\i) {$\vdots$};
}
\foreach \i in {-2,2} {
\draw[->,thick] (1.6,\i) -- (2.3,\i);
}
\node at (4,-3) {$i$};
\node at (4,3) {$k$};
\draw[->,thick] (4,-3.9) -- (4,-3.3);
\node at (4,-4.1) {$z$};
\foreach\y in {-2,...,2}{
\node at (4,\y) [circle,fill,inner sep=1.5pt] {};
}
\end{tikzpicture}
=
\prod_{a=1}^{M} q^{(m+2) \bar{\ell}_a}
\dfrac{q^{(m+2)i} 
(q-1)^{\bm{1}_{k=1}}
(-q)^{\bm{1}_{k=2}}}
{
q^{(m+2)k} 
(q-1)^{\bm{1}_{i=1}}
(-q)^{\bm{1}_{i=2}}
}
\begin{tikzpicture}[scale=0.8,baseline=0cm]
\begin{scope}[shift = {(8,0)}]
\draw[->,lgray,line width=1.5pt] (4,-2.75) -- (4,2.75);
\foreach \i in {-2,-1,0,1,2}
{\draw[->,lgray,line width=1.5pt] (3,\i) -- (5,\i);}
\node at (2.75,-2) {$\bar{j}_1$};
\node at (2.75,-1) {$\vdots$};
\foreach \i in {0.1,1.1} {
\node at (2.75,\i) {$\vdots$};
}
\node at (2.75,2) {$\bar{j}_M$};
\node at (5.35,-2) {$\bar{\ell}_1$};
\node at (5.35,-1) {$\vdots$};
\foreach \i in {0.1,1.1} {
\node at (5.35,\i) {$\vdots$};
}
\node at (5.35,2) {$\bar{\ell}_M$};
\node at (1.3,-2) {$y_1$};
\node at (1.3,-1) {$\vdots$};
\node at (1.3,2) {$y_M$};
\foreach \i in {0,1} {
\node at (1.3,\i) {$\vdots$};
}
\foreach \i in {-2,2} {
\draw[->,thick] (1.6,\i) -- (2.3,\i);
}
\node at (4,-3) {$k$};
\node at (4,3) {$i$};
\draw[->,thick] (4,-3.9) -- (4,-3.3);
\node at (4,-4.2) {$q^{-3}z^{-1}$};
\end{scope}
\end{tikzpicture}
\end{align}
where $m$ is an arbitrary integer. Equation \eqref{IK-tower-sym} is a straightforward consequence of \eqref{eqnineteencym1}, combined with conservation of paths through vertices, and appropriate telescopic cancellations of factors assigned to vertical edges.

The proof concludes by taking the lattice definition \eqref{eqdualdot} of $\dotr{G}_S$ and applying the relation \eqref{IK-tower-sym} column by column (with $m$ initially equal to $0$, and rising in value by $2$ after each application of the relation). In particular, one finds that the procedure stabilizes: all indices $i,k,\bar{\ell}_a$ on the right hand side of \eqref{IK-tower-sym} are ultimately $0$, as one traverses sufficiently far to the right through the lattice. Keeping track of all factors acquired through this process, we read off the relation
\begin{align}
\label{eq:sign-fix}
\dotr{G}_{S}(y_1^{-1},\dots,y_M^{-1};\bm{z})
=
\dfrac{(-1)^N}
{q^{N(2N+1)}
(q-1)^{\mathfrak{c}_1(S)}
(-q)^{\mathfrak{c}_2(S)}}
\prod_{k=1}^{\infty} q^{2kS_k}
G_{S}(y_1,\dots,y_M;q^{-3}\bm{z}^{-1}).
\end{align}
Noting that for any $2$-string $S$ such that $|S|=2N$ one has
\begin{align*}
\mathfrak{c}_1(S)+\mathfrak{c}_2(S)
=
\frac{\mathfrak{c}_1(S)}{2}
+
\frac{|S|}{2}
=
\frac{\mathfrak{c}_1(S)}{2}
+
N,
\end{align*}
we may rearrange signs in \eqref{eq:sign-fix} to recover \eqref{eqdottondot}.
\end{proof}

\subsection{Proof of the Cauchy identity}
\label{ssec:IK-cauchy-proof}

We now return to the proof of Theorem \ref{thm:IK-cauchy}.

\begin{proof}
In view of the symmetry property \eqref{eqdottondot},
to demonstrate the Cauchy identity stated in $\eqref{eqcc5}$ we need to show that
\begin{equation}
\label{eqcuachyidIK1}
\sum_{S \in \mathfrak{s}(2)}F_{S}(x_1,\dots,x_N;\bm{z})\dotr{G}_{S}(y_1^{-1},\dots,y_M^{-1};\bm{z}) = 
F_{(2^N)}(x_1,\dots,x_N;\bm{z})
\prod_{i=1}^{N}
\prod_{j=1}^{M}
\frac{(1-q^2x_iy_j)(1-q^3x_iy_j)}{(1-x_iy_j)(1-qx_iy_j)}.
\end{equation}
Define an expectation value
\begin{equation}
\label{eqexpectation1}
\mathcal{E}_{N,M} 
= 
\bra{\varnothing} 
\mathcal{C}(x_1)
\dots 
\mathcal{C}(x_N) 
\dotr{\mathcal{D}}(y_M^{-1}) 
\dots 
\dotr{\mathcal{D}}(y_1^{-1}) 
\ket{2^N,\varnothing}.
\end{equation}
Inserting the identity operator $\sum_{S\in \mathfrak s(2)} \ket{S}\bra{S}$ between the operators $\mathcal{C}(x_N) \dotr{\mathcal{D}}(y_M^{-1})$ on the right hand side of equation $\eqref{eqexpectation1}$, we obtain
\begin{align*}
\mathcal{E}_{N,M}
&= 
\sum_{S \in \mathfrak{s}(2)}
\bra{\varnothing} 
\mathcal{C}(x_1)\dots \mathcal{C}(x_N) 
\ket{S}\bra{S} 
\dotr{\mathcal{D}}(y_M^{-1})
\dots \dotr{\mathcal{D}}(y_1^{-1}) 
\ket{2^N,\varnothing}
\\
&=
\sum_{S \in \mathfrak{s}(2)} 
F_{S}(x_1,\dots,x_N;\bm{z})
\dotr{G}_{S}(y_1^{-1},\dots,y_M^{-1};\bm{z}),
\end{align*}
where the final line follows from the algebraic definitions \eqref{eqspsymIK2} and \eqref{eqrff2}. This matches the left hand side of \eqref{eqcuachyidIK1}.

On the other hand, commuting all $\mathcal{C}(x_i)$ and $\dotr{\mathcal{D}}(y_j^{-1})$ operators using the relation \eqref{eqspexchange1} (noting that convergence constraints \eqref{eqres1} and \eqref{eqres2} should be rewritten with $y \mapsto x_i$, $x \mapsto y_j^{-1}$), we find that
\begin{align*}
\mathcal{E}_{N,M} 
&= \prod_{i = 1}^N \prod_{j = 1}^M \frac{(1-q^2x_iy_j)(1-q^3x_iy_j)}{(1-x_iy_j)(1-qx_iy_j)}\bra{\varnothing} \dotr{\mathcal{D}}(y_M^{-1})\cdots \dotr{\mathcal{D}}(y_1^{-1})\mathcal{C}(x_1)\dots \mathcal{C}(x_N)\ket{2^N,\varnothing}.
\end{align*}
We conclude by noting that
$
\bra{\varnothing} \dotr{\mathcal{D}}(y_j^{-1})
=
\bra{\varnothing},
$
after which one has
\begin{align*}
\mathcal{E}_{N,M}
& = \prod_{i = 1}^N \prod_{j = 1}^M \frac{(1-q^2x_iy_j)(1-q^3x_iy_j)}{(1-x_iy_j)(1-qx_iy_j)}\bra{\varnothing} \mathcal{C}(x_1)\dots \mathcal{C}(x_N) \ket{2^N,\varnothing}
\\
& = \prod_{i = 1}^N \prod_{j = 1}^M
\frac{(1-q^2x_iy_j)(1-q^3x_iy_j)}{(1-x_iy_j)(1-qx_iy_j)}
F_{(2^N)}(x_1,\dots,x_N;\bm{z}), 
\end{align*}
which is precisely the right hand side of \eqref{eqcuachyidIK1}.
\end{proof}

\section{Stable symmetric functions}
\label{sec:IK_stable}

Our goal in this section is to generalize the definition \eqref{pictureofrational1} of the rational symmetric functions $F_S$ in a way that renders them stable; this generalization results in yet another family of functions denoted $H_S$. These functions are defined in Section \ref{ssec:IK-H}, and a direct limiting procedure that constructs them is outlined in Section \ref{ssec:IKstab-construct}. Their basic properties, namely their symmetry and stability, are proved in Section \ref{ssec:IK-stab}.

\subsection{Multivariate rational function $H_S$}
\label{ssec:IK-H}

Before defining the rational function $H_S$, we shall require an extension of the previous inversion statistic \eqref{eqinversionfor6} to $2$-strings:

\begin{defn}
Fix an integer $N \in \mathbb{Z}_{>0}$ and a vector $I = (I_1,\dots,I_N) \in \{0,1,2\}^N$. We define the inversion number ${\rm inv}(I)$ as follows:
\begin{equation}
\label{eqinversionforIK1}
\inv(I) = 
\sum_{1 \leq i<j \leq N} 
I_i (2-I_j).
\end{equation}
\end{defn}

\begin{defn}
Fix two integers $K \in \mathbb{Z}_{\geq 0}$, $N \in \mathbb{Z}_{>0}$ such that $0 \leq K \leq N$ and let $S \in \mathfrak{s}(2)$ be a $2$-string such that $|S|=2K$. We introduce the rational function $H_S$ as follows:
\begin{equation}
\label{eqstableIK1}
H_{S}(x_1,\dots,x_N;\bm{z}) = 
\sum_{I : |I| = 2K} 
(-1)^{\mathfrak{c}_1(I)/2}
(1-q^{-1})^{\mathfrak{c}_1(I)}
q^{-\inv(I)}
F_{S}^{I}(x_1,\dots,x_N;\bm{z}),
\end{equation}
with the sum taken over all vectors 
$I \in \{0,1,2\}^{N}$ such that $|I| = 2K$, and where
\begin{equation}
\label{eqstableIK191}
\raisebox{-37.5mm}{
\begin{tikzpicture}[scale = 0.65]
\foreach \i in {0,1,2,3,4,6,7} {
    \draw[->,lgray,line width=1.5pt] (\i,0) -- (\i,7);}
    \foreach \i in {1,2,3,4,5,6} {
    \draw[lgray,line width=1.5pt] (-1,\i) -- (4.3,\i);}
    \foreach \i in {1,2,3,4,5,6} {
    \draw[->,lgray,line width=1.5pt] (5.5,\i) -- (8,\i);}
    \foreach \i in {1,2,3,4,5,6} {
    \node at (5,\i) {$\dots$};}
    \foreach \i in {1,2,6} {
    \node at (8.3,\i) {$0$};
    }
    \foreach \i in {3.2,4.2,5.2} {
    \node at (8.3,\i) {$\vdots$};
    }
    \node[left] at (-1.8,1) {$x_1 \to$};
    \node[left] at (-1.8,2) {$x_2 \to$};
    \node[left] at (-1.8,3) {$\vdots$};
    \node[left] at (-1.8,4) {$\vdots$};
    \node[left] at (-1.8,5) {$\vdots$};
    \node[left] at (-1.8,6) {$x_N \to$};
    \node at (-1.4,1) {$I_1$};
    \node at (-1.4,2) {$I_2$};
    \node at (-1.4,3) {$\vdots$};
    \node at (-1.4,4) {$\vdots$};
    \node at (-1.4,5) {$\vdots$};
    \node at (-1.4,6) {$I_{N}$};
    \node[left] at (-3.8,3.5) {$F_{S}^{I}(x_1,\dots,x_N;\bm{z})=$};
    \node at (0,7.5) {$S_1$};
    \node at (1,7.5) {$S_2$};
    \node at (2,7.5) {$S_3$};
    \node at (3,7.5) {$\dots$};
    \node at (4,7.5) {$\dots$};
    \node at (0,-2) {$z_1$};
    \node at (1,-2) {$z_2$};
    \node at (2,-2) {$z_3$};
    \node at (3,-2) {$\dots$};
    \node at (4,-2) {$\dots$};
    \draw[thick,->] (0,-1.5) -- (0,-0.85);
    \draw[thick,->] (1,-1.5) -- (1,-0.85);
    \draw[thick,->] (2,-1.5) -- (2,-0.85);
    \node at (0,-0.5) {$0$};
    \node at (1,-0.5) {$0$};
    \node at (2,-0.5) {$0$};
    \node at (3,-0.5) {$\dots$};
    \node at (4,-0.5) {$\dots$};
\end{tikzpicture}}
\end{equation}
which is interpreted as a partition function in the same vein as \eqref{pictureofrational1}.
\end{defn}

\begin{rmk}\label{rmk:stable-cases}
The functions \eqref{eqstableIK1} generalize the family $F_S$. Indeed, taking $K=N$ trivializes the summation on the right hand side, and one is forced to have $(I_1,\dots,I_N) = (2^N)$. Since all of the statistics in the exponents of the summand of \eqref{eqstableIK1} reduce to $0$ when $I=(2^N)$, and $F_S^{(2^N)}(x_1,\dots,x_N;\bm{z}) = F_S(x_1,\dots,x_N;\bm{z})$, one finds that $H_S = F_S$ when $K=N$.

At the other extreme, when $K=0$, no particles may enter the lattice \eqref{eqstableIK191} and it freezes into a product of vertices $W_{z_j/x_i}(0,0;0,0)=1$. We therefore have $H_{(0,0,0,\dots)}(x_1,\dots,x_N;\bm{z}) = 1$ for any $N \geq 1$.
\end{rmk}

\subsection{Limiting procedure to obtain $H_S$}
\label{ssec:IKstab-construct}

The following theorem outlines how the functions $H_{S}$ may be recovered via a suitable series of limits applied to the original family \eqref{eqspsymIK2}. The proof follows very similar lines to the corresponding result (Theorem \eqref{6-speicalpart1}) in the six-vertex model: namely, we shall demonstrate certain simplified exchange relations for the row operators \eqref{eqrowopIK} that hold under the limits of the vertical spectral parameters that we impose.

\begin{thm}
\label{speicalpart1}
Fix two integers $K \in \mathbb{Z}_{\geq 0}$, $N \in \mathbb{Z}_{>0}$ such that $0 \leq K \leq N$, and let $S \in \mathfrak{s}(2)$ be a $2$-string such that $|S|=2K$. Fix three alphabets $(x_1,\dots,x_N)$, ${\bm u} = (u_1,\dots,u_{N-K})$ and ${\bm z} = (z_1,z_2,\dots)$ of arbitrary parameters. The following relation then holds:
\begin{equation}
\label{eqspeical1}
H_{S}(x_1,\dots,x_N;\bm{z}) 
= 
\lim_{\bm{u} \to \infty} 
\frac{F_{(2^{N-K},S)}(x_1,\dots,x_N;\bm{u} \cup \bm{z})}{F_{(2^{N-K})}(*;\bm{u})},
\end{equation}
where we write $\bm{u} \to \infty$ as shorthand for the limits $u_i \rightarrow \infty$, $1 \leq i \leq N-K$. The notation $(*;\bm{u})$ indicates that the primary alphabet of $F_{(2^{N-K})}$ may be chosen arbitrarily. 
\end{thm}

\begin{proof}
Our first observation is that all of the vertex weights in Figure \ref{fig:19weights} have a well-defined $y \rightarrow \infty$ limit. Moreover, due to the fact that the weights \eqref{eqdefn19w} depend on the spectral parameters $x,y$ only via their ratio, it is clear that all weights become independent of $x$ after taking $y \rightarrow \infty$. We shall reflect this fact by suppressing the value of horizontal spectral parameters in most of the partition functions that follow.

Now consider the lattice model interpretation of $F_{(2^{N-K},S)}(x_1,\dots,x_N;\bm{u} \cup \bm{z})$; see equation $\eqref{pictureofrational1}$. Sending $\bm{u} \to \infty$, we obtain the following:
\begin{equation}
\label{eqqexchangeIKrel1}
\begin{tikzpicture}[scale = 0.6,every node/.style={scale=0.5},baseline=(current bounding box.center)]
\node[scale = 2] at (-6.5,3) {$\displaystyle\lim_{\bm{u}\to \infty} F_{(2^{N-K},S)}(x_1,\dots,x_N;\bm{u} \cup \bm{z}) = \sum_{I:|I| = 2K}$};
\node[scale = 1.5] at (3.5,-1.5) {$\infty$};
\node[scale = 1.5] at (4.5,-1.5) {$\dots$};
\node[scale = 1.5] at (5.5,-1.5) {$\infty$};
\draw[thick,->] (3.5,-1.2) -- (3.5,-0.5);
\draw[thick,->] (5.5,-1.2) -- (5.5,-0.5);
\foreach \i in {3.5,4.5,5.5} {
\draw[lgray,line width=1.5pt,->] (\i,0) -- (\i,7);
}
\foreach \i in {1,2,3,4,5,6} {
\draw[lgray,line width=1.5pt] (2.5,\i) -- (6,\i);
}
\draw[red,dashed] (6.75,7) -- (6.75,0);
\foreach \i in {1,2,6} {
\node[scale = 1.25] at (2.25,\i) {$2$};
}
\foreach \i in {3,4,5} {
\node at (2.25,\i) {$\vdots$};
}
\foreach \i in {3.5,5.5} {
\node[scale = 1.25] at (\i,-0.25) {$0$};
}
\node[scale = 1.25] at (3.5,7.25) {$2$};
\node at (4.5,7.25) {$\dots$};
\node[scale = 1.25] at (5.5,7.25) {$2$};
\foreach \i in {3,4,5} {
\node at (6.2,\i) {$\vdots$};
}
\node[scale = 1.25] at (6.4,1) {$I_1$};
\node[scale = 1.25] at (6.4,2) {$I_2$};
\node[scale = 1.25] at (6.4,6) {$I_N$};
\foreach \i in {3,4,5} {
\node at (7.2,\i) {$\vdots$};
}
\node[scale = 1.25] at (7.2,1) {$I_1$};
\node[scale = 1.25] at (7.2,2) {$I_2$};
\node[scale = 1.25] at (7.2,6) {$I_N$};
\foreach \i in {8,9,10,11,12,13} {
\draw[lgray,line width=1.5pt,->] (\i,0) -- (\i,7);
}
\foreach \i in {1,2,3,4,5,6} {
\draw[lgray,line width=1.5pt,->] (7.5,\i) -- (13.5,\i);
}
\node at (4.5,-0.25) {$\dots$};
\foreach \i in {8,9} {
\node[scale = 1.25] at (\i,-0.25) {$0$};
}
\foreach \i in {10,11,12,13} {
\node at (\i,-0.25) {$\dots$};
}
\node[scale = 1.25] at (8,7.25) {$S_1$};
\node[scale = 1.25] at (9,7.25) {$S_2$};
\foreach \i in {10,11,12,13} {
\node at (\i,7.25) {$\dots$};
}
\foreach \i in {1,2,6} {
\node[scale = 1.25] at (13.75,\i) {$0$};
}
\node[scale = 1.5] at (8,-1.5) {$z_1$};
\node[scale = 1.5] at (9,-1.5) {$z_2$};
\foreach \i in {3,4,5} {
\node at (13.75,\i) {$\vdots$};}
\draw[thick,->] (8,-1.2) -- (8,-0.5);
\draw[thick,->] (9,-1.2) -- (9,-0.5);
\node[scale = 1.5] at (0.5,1) {$x_1$};
\node[scale = 1.5] at (0.5,2) {$x_2$};
\node[scale = 1.5] at (0.5,6) {$x_{N}$};
\foreach \i in {3,4,5} {
\node[scale = 1.5] at (0.5,\i) {$\vdots$};
}
\foreach \i in {1,2,6} {
\draw[thick,->] (1,\i) -- (2,\i);
}
\end{tikzpicture}
\end{equation}
where the red dashed line is used to separate out the first $N-K$ vertical lines of the lattice. In view of the comment above, we note that in these first $N-K$ columns, one may assign arbitrary values to the horizontal spectral parameters. To the right of the red dashed line we recognize the lattice definition of $F_{S}^{I}(x_1,\dots,x_N;\bm{z})$. Comparing with the definition \eqref{eqstableIK1} of $H_S$, equation $\eqref{eqspeical1}$ holds provided that 
\begin{equation}
\label{eqlatticeb1}
\begin{tikzpicture}[scale = 0.6,every node/.style={scale=0.5},baseline=1.9cm]
\foreach \i in {4,5,6} {
\draw[lgray,line width=1.5pt,->] (\i,0) -- (\i,7);
}
\foreach \i in {1,2,3,4,5,6} {
\draw[lgray,line width=1.5pt] (3,\i) -- (6.5,\i);
}
\foreach \i in {1,2,6} {
\node[scale = 1.25] at (2.8,\i) {$2$};
}
\foreach \i in {3,4,5} {
\node at (2.8,\i) {$\vdots$};
}
\foreach \i in {4,6} {
\node[scale = 1.25] at (\i,-0.2) {$0$};
\node[scale = 1.25] at (\i,7.2) {$2$};
}
\node at (5,7.2) {$\dots$};
\node at (5,-0.2) {$\dots$};
\foreach \i in {3,4,5} {
\node at (6.8,\i) {$\vdots$};
}
\node[scale = 1.25] at (6.8,1) {$I_1$};
\node[scale = 1.25] at (6.8,2) {$I_2$};
\node[scale = 1.25] at (6.8,6) {$I_{N}$};
\draw[->,thick] (4,-1.2) -- (4,-0.5);
\draw[->,thick] (6,-1.2) -- (6,-0.5);
\node[scale = 1.5] at (4,-1.5) {$\infty$};
\node[scale = 1.5] at (5,-1.5) {$\dots$};
\node[scale = 1.5] at (6,-1.5) {$\infty$};
\node[scale = 1.25] at (1,1) {$*$};
\node[scale = 1.25] at (1,2) {$*$};
\node[scale = 1.25] at (1,6) {$*$};
\draw[thick,->] (1.5,1) -- (2.5,1);
\draw[thick,->] (1.5,2) -- (2.5,2);
\draw[thick,->] (1.5,6) -- (2.5,6);
\end{tikzpicture}
=
(-1)^{\mathfrak{c}_1(I)/2}
(1-q^{-1})^{\mathfrak{c}_1(I)}
q^{-\inv(I)}
\lim_{\bm{u}\to \infty}
F_{(2^{N-K})}(*;\bm{u})
\end{equation}
for all $(I_1,\dots,I_N) \in \{0,1,2\}^N$ such that $|I|=2K$. To prove equation \eqref{eqlatticeb1} we require the following lemma, that allows us to incrementally modify the boundary conditions at the right edges of the partition function.

\begin{lem}
\label{qexchangeIKK1}
Fix integers $a,b \in \{0,1,2\}$ such that $a>b$, and arbitrary vectors $S,T \in \{0,1,2\}^N$. The following equality of partition functions holds:
\begin{equation}
\label{IKqexchange1}
q^{\inv(a,b)} \times \left(
\begin{tikzpicture}[scale=0.64,every node/.style={scale=0.7},baseline={([yshift=-0.5ex]current bounding box.center)}]
\draw[lgray,line width=1.5pt] (0,1) -- (2.5,1);
\draw[lgray,line width=1.5pt] (0,0) -- (2.5,0);
\draw[lgray,line width=1.5pt,->] (3.5,1) -- (6,1);
\draw[lgray,line width=1.5pt,->] (3.5,0) -- (6,0);
\node[scale = 1.5] at (3,0) {$\dots$};
\node[scale = 1.5] at (3,1) {$\dots$};
\foreach \i in {1,2,4,5} {
\draw[lgray,line width=1.5pt,<-] (\i,2) -- (\i,-1);
}
\foreach \i in {1,2} {
\draw[->,thick] (\i,-2.2) -- (\i,-1.5);
}
\foreach \i in {4,5} {
\draw[->,thick] (\i,-2.2) -- (\i,-1.5);
}
\node at (1,-2.5) {$\infty$};
\node at (2,-2.5) {$\infty$};
\node at (4,-2.5) {$\infty$};
\node at (5,-2.5) {$\infty$};
\node at (-0.25,0) {$2$};
\node at (-0.25,1) {$2$};
\node at (6.2,1) {$b$};
\node at (6.2,0) {$a$};
\foreach \i in {1,2} {
\node at (\i,2.2) {$T_{\i}$};
}
\node at (4,2.2) {$T_{N-1}$};
\node at (5,2.2) {$T_{N}$};
\foreach \i in {1,2} {
\node at (\i,-1.2) {$S_{\i}$};
}
\node at (4,-1.2) {$S_{N-1}$};
\node at (5,-1.2) {$S_{N}$};
\node at (-1.5,0) {$*$};
\node at (-1.5,1) {$*$};
\draw[thick,->] (-1.25,0) -- (-0.5,0);
\draw[thick,->] (-1.25,1) -- (-0.5,1);
\end{tikzpicture}\right)
= 
\begin{tikzpicture}[scale=0.64,every node/.style={scale=0.7},baseline={([yshift=-0.5ex]current bounding box.center)}]
\node at (8.5,0) {$*$};
\node at (8.5,1) {$*$};
\draw[->,thick] (9,0) -- (9.75,0);
\draw[->,thick] (9,1) -- (9.75,1);
\draw[lgray,line width=1.5pt] (10.25,1) -- (12.75,1);
\draw[lgray,line width=1.5pt,->] (13.75,1) -- (16.25,1);
\draw[lgray,line width=1.5pt] (10.25,0) -- (12.75,0);
\draw[lgray,line width=1.5pt,->] (13.75,0) -- (16.25,0);
\node[scale = 1.5] at (13.25,0) {$\dots$};
\node[scale = 1.5] at (13.25,1) {$\dots$};
\foreach \i in {11.25,12.25,14.25,15.25} 
{
\draw[lgray,line width=1.5pt,<-] (\i,2) -- (\i,-1);
}
\node at (10,1) {$2$};
\node[below] at (10,0.25) {$2$};
\node at (16.5,1) {$a$};
\node at (16.5,0) {$b$};
\foreach \i in {1,2}{
\node at (\i+10.25,2.2) {$T_{\i}$};
}
\node at (14.25,2.2) {$T_{N-1}$};
\node at (15.25,2.2) {$T_{N}$};
\foreach \i in {1,2} {
\node at (\i+10.25,-1.2) {$S_{\i}$};
}
\node at (14.25,-1.2) {$S_{N-1}$};
\node at (15.25,-1.2) {$S_{N}$};
\foreach \i in {11.2,12.2} {
\draw[->,thick] (\i,-2.2) -- (\i,-1.5);
}
\foreach \i in {14.2,15.2} {
\draw[->,thick] (\i,-2.2) -- (\i,-1.5);
}
\node at (11.2,-2.5) {$\infty$};
\node at (12.2,-2.5) {$\infty$};
\node at (14.2,-2.5) {$\infty$};
\node at (15.2,-2.5) {$\infty$};
\end{tikzpicture}
\end{equation}
where in this case ${\rm inv}(a,b)=a(2-b)$. 

Furthermore, the following equality holds:
\begin{equation}
\label{IKqexchange2}
-q (1-q^{-1})^{-2} \times
\left(
\begin{tikzpicture}[scale=0.64,every node/.style={scale=0.7},baseline={([yshift=-0.5ex]current bounding box.center)}]
\draw[lgray,line width=1.5pt] (0,1) -- (2.5,1);
\draw[lgray,line width=1.5pt] (0,0) -- (2.5,0);
\draw[lgray,line width=1.5pt,->] (3.5,1) -- (6,1);
\draw[lgray,line width=1.5pt,->] (3.5,0) -- (6,0);
\node[scale = 1.5] at (3,0) {$\dots$};
\node[scale = 1.5] at (3,1) {$\dots$};
\foreach \i in {1,2,4,5} {
\draw[lgray,line width=1.5pt,<-] (\i,2) -- (\i,-1);
}
\foreach \i in {1,2} {
\draw[->,thick] (\i,-2.2) -- (\i,-1.5);
}
\foreach \i in {4,5} {
\draw[->,thick] (\i,-2.2) -- (\i,-1.5);
}
\node at (1,-2.5) {$\infty$};
\node at (2,-2.5) {$\infty$};
\node at (4,-2.5) {$\infty$};
\node at (5,-2.5) {$\infty$};
\node at (-0.25,0) {$2$};
\node at (-0.25,1) {$2$};
\node at (6.2,1) {$1$};
\node at (6.2,0) {$1$};
\foreach \i in {1,2} {
\node at (\i,2.2) {$T_{\i}$};
}
\node at (4,2.2) {$T_{N-1}$};
\node at (5,2.2) {$T_{N}$};
\foreach \i in {1,2} {
\node at (\i,-1.2) {$S_{\i}$};
}
\node at (4,-1.2) {$S_{N-1}$};
\node at (5,-1.2) {$S_{N}$};
\node at (-1.5,0) {$*$};
\node at (-1.5,1) {$*$};
\draw[thick,->] (-1.25,0) -- (-0.5,0);
\draw[thick,->] (-1.25,1) -- (-0.5,1);
\end{tikzpicture} \right)
= 
\begin{tikzpicture}[scale=0.64,every node/.style={scale=0.7},baseline={([yshift=-0.5ex]current bounding box.center)}]
\node at (8.5,0) {$*$};
\node at (8.5,1) {$*$};
\draw[->,thick] (9,0) -- (9.75,0);
\draw[->,thick] (9,1) -- (9.75,1);
\draw[lgray,line width=1.5pt] (10.25,1) -- (12.75,1);
\draw[lgray,line width=1.5pt,->] (13.75,1) -- (16.25,1);
\draw[lgray,line width=1.5pt] (10.25,0) -- (12.75,0);
\draw[lgray,line width=1.5pt,->] (13.75,0) -- (16.25,0);
\node[scale = 1.5] at (13.25,0) {$\dots$};
\node[scale = 1.5] at (13.25,1) {$\dots$};
\foreach \i in {11.25,12.25,14.25,15.25} 
{
\draw[lgray,line width=1.5pt,<-] (\i,2) -- (\i,-1);
}
\node at (10,1) {$2$};
\node[below] at (10,0.25) {$2$};
\node at (16.5,1) {$2$};
\node at (16.5,0) {$0$};
\foreach \i in {1,2}{
\node at (\i+10.25,2.2) {$T_{\i}$};
}
\node at (14.25,2.2) {$T_{N-1}$};
\node at (15.25,2.2) {$T_{N}$};
\foreach \i in {1,2} {
\node at (\i+10.25,-1.2) {$S_{\i}$};
}
\node at (14.25,-1.2) {$S_{N-1}$};
\node at (15.25,-1.2) {$S_{N}$};
\foreach \i in {11.2,12.2} {
\draw[->,thick] (\i,-2.2) -- (\i,-1.5);
}
\foreach \i in {14.2,15.2} {
\draw[->,thick] (\i,-2.2) -- (\i,-1.5);
}
\node at (11.2,-2.5) {$\infty$};
\node at (12.2,-2.5) {$\infty$};
\node at (14.2,-2.5) {$\infty$};
\node at (15.2,-2.5) {$\infty$};
\end{tikzpicture}
\end{equation}
\end{lem}

\begin{proof}
We begin with the proof of \eqref{IKqexchange1}. Following very similar ideas to Proposition \ref{prop:IK-commute-nontrivial}, one may derive the following exchange relation between row-operators (of arbitrary length):
\begin{align}
\label{IK-general-commute}
\mathcal{T}_{2,a}(x;\bm{z}) \mathcal{T}_{2,b}(y;\bm{z})
=
\sum_{c,d:c+d=a+b}
W_{y/x}(d,c;b,a)
\mathcal{T}_{2,d}(y;\bm{z}) \mathcal{T}_{2,c}(x;\bm{z}),
\end{align}
where $a,b \in \{0,1,2\}$ are arbitrary. After sending $\bm{z} \rightarrow \infty$, the row operators in \eqref{IK-general-commute} become independent of their primary arguments:
\begin{align}
\label{IK-general-commute2}
\mathcal{T}_{2,a}(*;\infty) \mathcal{T}_{2,b}(*;\infty)
=
\sum_{c,d:c+d=a+b}
W_{y/x}(d,c;b,a)
\mathcal{T}_{2,d}(*;\infty) \mathcal{T}_{2,c}(*;\infty),
\end{align}
which holds for arbitrary values of the parameter $y/x$. Consulting the weights in Figure \ref{fig:19weights} and sending $y/x \rightarrow \infty$, one finds that
\begin{align*}
W_{\infty}(d,c;b,a)
=
\bm{1}_{c=a}
\cdot
\bm{1}_{d=b}
\cdot
q^{-{\rm inv}(a,b)},
\qquad
\text{for}
\ \ 
a>b,
\end{align*}
and \eqref{IKqexchange1} then follows directly from \eqref{IK-general-commute2}.

The proof of \eqref{IKqexchange2} follows by similar reasoning. This time, one starts from \eqref{IK-general-commute2} with $a=b=1$. Sending $y/x \rightarrow \infty$, there holds
\begin{align*}
W_{\infty}(d,c;1,1)
=
\left\{
\begin{array}{ll}
0, & \quad c=0 \ \text{and}\ d=2,
\\
\\
-q^{-1}, & \quad c=d=1,
\\
\\
-q^{-4}(1-q)(1-q^2), & \quad c=2 \ \text{and}\ d=0,
\end{array}
\right.
\end{align*}
and \eqref{IK-general-commute2} then reads
\begin{align*}
\mathcal{T}_{2,1}(*;\infty)
\mathcal{T}_{2,1}(*;\infty)
=
-q^{-1}
\mathcal{T}_{2,1}(*;\infty)
\mathcal{T}_{2,1}(*;\infty)
-q^{-4}(1-q)(1-q^2)
\mathcal{T}_{2,0}(*;\infty)
\mathcal{T}_{2,2}(*;\infty).
\end{align*}
This rearranges to yield \eqref{IKqexchange2}.
\end{proof}

An immediate corollary of the exchange relations \eqref{IKqexchange1} and \eqref{IKqexchange2} is that the quantity
\begin{align}
\label{IK-chi-inv}
\chi(I_1,\dots,I_N)
:=
(-1)^{\mathfrak{c}_1(I)/2}
(1-q^{-1})^{-\mathfrak{c}_1(I)}
q^{\inv(I)}
\times
\left(
\begin{tikzpicture}[scale = 0.6,every node/.style={scale=0.5},baseline=1.6cm]
\foreach \i in {4,5,6} {
\draw[lgray,line width=1.5pt,->] (\i,0) -- (\i,7);
}
\foreach \i in {1,2,3,4,5,6} {
\draw[lgray,line width=1.5pt] (3,\i) -- (6.5,\i);
}
\foreach \i in {1,2,6} {
\node[scale = 1.25] at (2.8,\i) {$2$};
}
\foreach \i in {3,4,5} {
\node at (2.8,\i) {$\vdots$};
}
\foreach \i in {4,6} {
\node[scale = 1.25] at (\i,-0.2) {$0$};
\node[scale = 1.25] at (\i,7.2) {$2$};
}
\node at (5,7.2) {$\dots$};
\node at (5,-0.2) {$\dots$};
\foreach \i in {3,4,5} {
\node at (6.8,\i) {$\vdots$};
}
\node[scale = 1.25] at (6.8,1) {$I_1$};
\node[scale = 1.25] at (6.8,2) {$I_2$};
\node[scale = 1.25] at (6.8,6) {$I_{N}$};
\draw[->,thick] (4,-1.2) -- (4,-0.5);
\draw[->,thick] (6,-1.2) -- (6,-0.5);
\node[scale = 1.5] at (4,-1.5) {$\infty$};
\node[scale = 1.5] at (5,-1.5) {$\dots$};
\node[scale = 1.5] at (6,-1.5) {$\infty$};
\node[scale = 1.25] at (1,1) {$*$};
\node[scale = 1.25] at (1,2) {$*$};
\node[scale = 1.25] at (1,6) {$*$};
\draw[thick,->] (1.5,1) -- (2.5,1);
\draw[thick,->] (1.5,2) -- (2.5,2);
\draw[thick,->] (1.5,6) -- (2.5,6);
\end{tikzpicture} \right)
\end{align}
is invariant for all $(I_1,\dots,I_N) \in \{0,1,2\}^N$ such that $|I|=2K$. Indeed, any re-partitioning of $(I_1,\dots,I_N)$ may be achieved by iteration of the relations \eqref{IKqexchange1} and \eqref{IKqexchange2}. We are therefore at liberty to choose $(I_1,\dots,I_N)$ in any way that facilitates the computation of $\chi(I_1,\dots,I_N)$. Choosing $(I_1,\dots,I_N)=(0^{N-K},2^K)$, equation \eqref{IK-chi-inv} becomes
\begin{align}
\label{IK-chi-inv2}
\chi(0^{N-K},2^K)
=
\begin{tikzpicture}[scale = 0.6,every node/.style={scale=0.5},baseline=1.9cm]
\begin{scope}[shift = {(0,1)}]
\foreach \i in {4,5,6} {
\draw[lgray,line width=1.5pt,->] (\i,3.75) -- (\i,7);
}
\foreach \i in {4,5,6} {
\draw[lgray,line width=1.5pt] (3,\i) -- (6.5,\i);
}
\node[scale = 1.25] at (6.8,4) {$2$};
\node[scale = 1.25] at (6.8,5) {$\vdots$};
\node[scale = 1.25] at (6.8,6) {$2$};
\node[scale = 1.25] at (1,4) {$*$};
\node[scale = 1.25] at (1,6) {$*$};
\draw[thick,->] (1.5,4) -- (2.5,4);
\draw[thick,->] (1.5,6) -- (2.5,6);
\foreach \i in {4,6} {
\node[scale = 1.25] at (\i,7.2) {$2$};
}
\node at (5,7.2) {$\dots$};
\foreach \i in {4,6} {
\node[scale = 1.25] at (2.8,\i) {$2$};
}
\node at (2.8,5) {$\vdots$};
\node at (4,3.5) {$2$};
\node at (5,3.5) {$\dots$};
\node at (6,3.5) {$2$};
\end{scope}
\draw[red,dashed] (3,4.1) -- (6.5,4.1);
\node at (4,3.75) {$2$};
\node at (5,3.75) {$\dots$};
\node at (6,3.75) {$2$};
\foreach \i in {4,5,6} {
\draw[lgray,line width=1.5pt] (\i,0) -- (\i,3.5);
}
\foreach \i in {1,2,3} {
\draw[lgray,line width=1.5pt] (3,\i) -- (6.5,\i);
}
\foreach \i in {1,3} {
\node[scale = 1.25] at (2.8,\i) {$2$};
}
\foreach \i in {2} {
\node at (2.8,\i) {$\vdots$};
}
\foreach \i in {4,6} {
\node[scale = 1.25] at (\i,-0.2) {$0$};
}
\node at (5,-0.2) {$\dots$};
\node[scale = 1.25] at (6.8,1) {$0$};
\node[scale = 1.25] at (6.8,2) {$\vdots$};
\node[scale = 1.25] at (6.8,3) {$0$};
\draw[->,thick] (4,-1.2) -- (4,-0.5);
\draw[->,thick] (6,-1.2) -- (6,-0.5);
\node[scale = 1.5] at (4,-1.5) {$\infty$};
\node[scale = 1.5] at (5,-1.5) {$\dots$};
\node[scale = 1.5] at (6,-1.5) {$\infty$};
\node[scale = 1.25] at (1,1) {$*$};
\node[scale = 1.25] at (1,3) {$*$};
\draw[thick,->] (1.5,1) -- (2.5,1);
\draw[thick,->] (1.5,3) -- (2.5,3);
\draw [decorate,decoration={brace,amplitude=5pt,mirror,raise=4ex}]
  (4,-0.8) -- (6,-0.8) node[midway,yshift=-5.5em]{$N-K$};
\draw [decorate,decoration={brace,amplitude=5pt,mirror,raise=4ex}]
  (6.25,1) -- (6.25,3) node[midway,xshift=7em]{$N-K$};
  \draw [decorate,decoration={brace,amplitude=5pt,mirror,raise=4ex}]
  (6.25,5) -- (6.25,7) node[midway,xshift=5.75em]{$K$};
\end{tikzpicture}
=
\lim_{\bm{u}\to \infty}
F_{(2^{N-K})}(*;\bm{u}),
\end{align}
where the final equality follows by noting that all vertices above the red dashed line are frozen: all of these vertices have the weight $W_{\infty}(2,2;2,2)=1$. Comparison of \eqref{IK-chi-inv} and \eqref{IK-chi-inv2} proves \eqref{eqlatticeb1}, thereby completing the proof of Theorem \ref{speicalpart1}.

\end{proof}

\subsection{Symmetry and stability of $H_S$}
\label{ssec:IK-stab}
In this section, we discuss some properties of the rational functions defined by $\eqref{eqstableIK1}$. Similarly to their six-vertex analogues, the first property is that the rational functions $H_{S}$ are symmetric with respect to their primary alphabet:

\begin{prop}
Fix two integers $K \in \mathbb{Z}_{\geq 0}$, $N \in \mathbb{Z}_{>0}$ such that $0 \leq K \leq N$ and let $S \in \mathfrak{s}(2)$ be a $2$-string such that $|S|=2K$. Then the rational function $H_{S}(x_1,\dots,x_N;\bm{z})$ is symmetric with respect to $(x_1,\dots,x_N)$.
\end{prop}

\begin{proof}
This is an immediate consequence of Theorem \ref{speicalpart1}, in view of the fact that the right hand side of \eqref{eqspeical1} is symmetric in $(x_1,\dots,x_N)$.
\end{proof}

The second property is that the rational functions $\eqref{eqstableIK1}$ are {\it stable} with respect to their alphabet
$(x_1,\dots,x_N)$:

\begin{prop}
\label{prop:IK-stab}
Fix two integers $K \in \mathbb{Z}_{\geq 0}$, $N \in \mathbb{Z}_{>0}$ such that $0 \leq K \leq N$ and let $S \in \mathfrak{s}(2)$ be a $2$-string such that $|S|=2K$. We then have the reduction property
\begin{align}
\label{stability-19v}
H_S(x_1,\dots,x_{N-1},x_N=\infty;\bm{z})
=
\left\{
\begin{array}{ll}
H_S(x_1,\dots,x_{N-1};\bm{z}),
&
\quad
|S| \leq 2N-2,
\\ \\
0,
&
\quad
|S|=2N.
\end{array}
\right.
\end{align}
\end{prop}

\begin{proof}
Since $H_S(x_1,\dots,x_N;\bm{z})$ is symmetric with respect to its primary alphabet, we may equivalently study the limit as $x_1 \rightarrow \infty$. This limit affects the bottommost row of vertices in the partition function \eqref{eqstableIK191}. Consulting the table in Figure \ref{fig:19weights}, one finds that all weights $W_{z/x}(i,j;k,\ell)$ have a well-defined limit as $x \rightarrow \infty$, and furthermore
\begin{align}
\label{vanishing-weights}
\lim_{x \rightarrow \infty}
W_{z/x}(0,1;1,0)
=
0,
\qquad
\lim_{x \rightarrow \infty}
W_{z/x}(0,2;1,1)
=
0,
\qquad
\lim_{x \rightarrow \infty}
W_{z/x}(0,2;2,0)
=
0.
\end{align}
It follows that $F^{I}_S(x_1=\infty,x_2,\dots,x_N;\bm{z})$ vanishes if any of the vertices \eqref{vanishing-weights} occur within the first row of the partition function \eqref{eqstableIK191}. 

Let us begin by assuming that $|S|<2N$. The only way to prevent the appearance of one of the vertices \eqref{vanishing-weights} is to fix $I_1=0$, which causes the entire bottom row of the lattice to freeze away (since it is then comprised solely of $W_{0}(0,0;0,0)$ vertices). As such, for any $I=(I_1,\dots,I_N)$ we find that
\begin{align*}
F^{I}_S(x_1=\infty,x_2,\dots,x_N;\bm{z})
=
\bm{1}_{I_1=0}
\cdot
F^{\tilde{I}}_S(x_2,\dots,x_N;\bm{z}),
\qquad
\tilde{I}=(I_2,\dots,I_N),
\end{align*}
which translates into the property
\begin{multline}
\label{H-reduce}
H_{S}(x_1=\infty,x_2,\dots,x_N;\bm{z}) = 
\sum_{\tilde{I} : |\tilde{I}| = 2K} 
(-1)^{\mathfrak{c}_1(\tilde{I})/2}
(1-q^{-1})^{\mathfrak{c}_1(\tilde{I})}
q^{-\inv(\tilde{I})}
F_{S}^{\tilde{I}}(x_2,\dots,x_N;\bm{z})
\\
=
H_{S}(x_2,\dots,x_N;\bm{z}),
\end{multline}
with the sum taken over vectors 
$\tilde{I}=(I_2,\dots,I_N)$, and where the first equality in \eqref{H-reduce} makes use of the fact that 
\begin{align*}
\mathfrak{c}_1(0,I_2,\dots,I_N) &= \mathfrak{c}_1(I_2,\dots,I_N) = \mathfrak{c}_1(\tilde{I}),
\\
\inv(0,I_2,\dots,I_N) &= \inv(I_2,\dots,I_N) = \inv(\tilde{I}).
\end{align*}
This completes the proof of \eqref{stability-19v} in the case $|S| < 2N$. 

The case $|S|=2N$ is handled very easily, since in this situation the vector that indexes the left incoming edges of the partition function \eqref{eqstableIK191} is forced, by conservation, to be $(I_1,\dots,I_N)=(2^N)$. As such, it is impossible to choose $I_1=0$, and $F^{I}_S(x_1=\infty,x_2,\dots,x_N;\bm{z})$ vanishes identically.
\end{proof}

\section{Cauchy identity for stable symmetric functions}
\label{ssec:IK-stable-cauchy}

In Section \ref{sec:IKcauchy} we stated a Cauchy-type identity between the rational symmetric functions \eqref{eqspsymIK2} and \eqref{eqdualsym1}. One aspect of the identity \eqref{eqcc5} which is somewhat undesirable is that its right hand side does not fully factorize. In this section we show that the stable symmetric functions \eqref{eqstableIK1} are more natural from this point of view, and satisfy a Cauchy summation identity with a fully factorized right hand side.

\begin{thm}
\label{thm:IK-stable-cauchy}
Let $N,M \in \mathbb{Z}_{>0}$ be fixed positive integers. Fix parameters $q,x_i,y_j,z_k \in \mathbb{C}$ such that for all $i \in \{1,\dots,N\}$, $j \in \{1,\dots,M\}$ and $k \in \mathbb{Z}_{>0}$, the conditions \eqref{eqres1-inhom} and \eqref{eqres2-inhom} hold. We then have the summation identity
\begin{align}
\label{eq:IK-stable-cauchy}
\sum_{S \in \mathfrak{s}(2)}
(-1)^{|S|/2}
q^{4MN-|S|^2}
c_S(q)
H_{S}(x_1,\dots,x_N;\bm{z})
H_{S}(y_1,\dots,y_M;q^{-3}\bm{z}^{-1}) 
=
\prod_{i=1}^{N}
\prod_{j=1}^{M}
\frac{(1-q^2x_iy_j)(1-q^3x_iy_j)}{(1-x_iy_j)(1-qx_iy_j)}, 
\end{align}
where the constant $c_S(q)$ is given by \eqref{eqcsfact1}.
\end{thm}

\begin{proof}
Our starting point is the identity \eqref{eqcuachyidIK1}, in which we choose the secondary alphabet to be $\bm{u} \cup \bm{z}$, where $\bm{u} = (u_1,\dots,u_N)$ and $\bm{z} = (z_1,z_2,\dots)$:
\begin{multline*}
\sum_{S \in \mathfrak{s}(2)}
F_{S}(x_1,\dots,x_N;\bm{u}\cup\bm{z})
\dotr{G}_{S}(y_1^{-1},\dots,y_M^{-1};\bm{u}\cup\bm{z}) 
\\
=
F_{(2^N)}(x_1,\dots,x_N;\bm{u}\cup\bm{z})
\prod_{i=1}^{N}
\prod_{j=1}^{M}
\frac{(1-q^2x_iy_j)(1-q^3x_iy_j)}{(1-x_iy_j)(1-qx_iy_j)}.
\end{multline*}
Dividing this equation by $F_{(2^N)}(*;\bm{u})$ and taking the limit $\bm{u} \rightarrow \infty$, in exactly the same vein as \eqref{eqspeical1}, we read off the identity
\begin{multline}
\label{eq:u-lim-H}
\lim_{\bm{u} \rightarrow \infty}
\frac{1}{F_{(2^N)}(*;\bm{u})}
\cdot
\sum_{S \in \mathfrak{s}(2)}
F_{S}(x_1,\dots,x_N;\bm{u}\cup\bm{z})
\dotr{G}_{S}(y_1^{-1},\dots,y_M^{-1};\bm{u}\cup\bm{z}) 
=
\prod_{i=1}^{N}
\prod_{j=1}^{M}
\frac{(1-q^2x_iy_j)(1-q^3x_iy_j)}{(1-x_iy_j)(1-qx_iy_j)},
\end{multline}
where the right hand side of \eqref{eq:u-lim-H} follows from \eqref{eqspeical1} with $S = (0,0,0,\dots)$. It remains to show that the left hand side of \eqref{eq:u-lim-H} matches that of \eqref{eq:IK-stable-cauchy}.

To that end, define the quantity
\begin{align*}
K_{N,M} := \lim_{\bm{u}\to \infty} \sum_{S \in \mathfrak s(2)} F_{S}(x_1,\dots,&x_N;\bm{u} \cup\bm{z}) \dotr{G}_{S}(y_1^{-1},\dots,y_M^{-1};\bm{u}\cup \bm{z}),
\end{align*}
which has the graphical representation
\begin{equation}
\label{eq:K-leftmost-columns}
\raisebox{-35mm}{
\begin{tikzpicture}[scale = 0.7,every node/.style={scale=0.7}]
\node[scale = 1.5] at (-5.5,4) {$K_{N,M}=
\displaystyle\sum_{\substack{(I_1,\dots,I_N) \\ (J_1,\dots,J_M) \\ |I|=|J|}}$};
\node[scale = 1.25] at (-2,1) {$x_1$};
\foreach \i in {2,6,7} {
\node at (-2,\i) {$\vdots$};
}
\node[scale = 1.25] at (-2,3) {$x_N$};
\node[scale = 1.25] at (-2,5) {$y_M^{-1}$};
\node[scale = 1.25] at (-2,8) {$y_1^{-1}$};
\foreach \i in {1,2,3,5,6,7,8} {
\draw[thick,->] (-1.5,\i) -- (-0.5,\i);
}
\foreach \i in {1,2,3,5,6,7,8} {
\draw[lgray,line width=1.5pt,->] (0,\i) -- (3,\i);
}
\foreach \i in {1,2,3,5,6,7,8} {
\draw[lgray,line width=1.5pt,->] (4,\i) -- (9,\i);
}
\foreach \i in {0.5,1.5,2.5} {
\draw[lgray,line width=1.5pt,->] (\i,0.5) -- (\i,8.5);
} 
\foreach \i in {4.5,5.5,6.5,7.5,8.5} {
\draw[lgray,line width=1.5pt,->] (\i,0.5) -- (\i,8.5);
} 
\foreach \i in {1,2,3,5,6,7,8} {
\node at (-0.2,\i) {$2$};
}
\foreach \i in {0.5,1.5,2.5,4.5,5.5,6.5} {
\node at (\i,0.2) {$0$};
}
\node at (3.5,1) {$I_1$};
\node at (3.5,2) {$\vdots$};
\node at (3.5,3) {$I_{N}$};
\node at (3.5,5) {$\bar{J}_M$};
\node at (3.5,6) {$\vdots$};
\node at (3.5,7) {$\vdots$};
\node at (3.5,8) {$\bar{J}_1$};
\node at (7.5,0.25) {$\dots$};
\node at (8.5,0.25) {$\dots$};
\node at (9.2,1) {$0$};
\node at (9.2,2) {$\vdots$};
\node at (9.2,3) {$0$};
\node at (9.2,5) {$2$};
\node at (9.2,6) {$2$};
\node at (9.2,7.1) {$\vdots$};
\node at (9.2,8) {$2$};
\foreach \i in {0.5,1.5,2.5} {
\node at (\i,8.8) {$2$};
}
\foreach \i in {4.5,5.5,6.5} {
\node at (\i,8.8) {$0$};
}
\foreach \i in {7.5,8.5} {
\node at (\i,8.8) {$\dots$};
}
\foreach \i in {0.5,1.5,2.5,4.5,5.5,6.5} {
\draw[thick,->] (\i,-0.8) -- (\i,-0.2);
}
\foreach \i in {0.5,1.5,2.5} {
\node at (\i,-1) {$\infty$};
}
\node at (4.5,-1) {$z_1$};
\node at (5.5,-1) {$z_2$};
\node at (6.5,-1) {$z_3$};
\foreach \i in {0.5,1.5,2.5,4.5,5.5,6.5,7.5,8.5} { \foreach \j in {5,6,7,8} {
\node at (\i,\j)[circle,fill,inner sep=2pt]{};
}}
\end{tikzpicture}}
\end{equation}
The leftmost $N$ columns of this lattice correspond with the variables $(u_1,\dots,u_N)$ that have been sent to infinity, while the remaining columns correspond with the secondary alphabet $(z_1,z_2,\dots)$ that remains free. We perform summation over the vectors $(I_1,\dots,I_N)$ and $(J_1,\dots,J_M)$ such that $\sum_{i=1}^{N} I_i = \sum_{i=1}^{M} J_i$, noting that this implies $\sum_{i=1}^{N} I_i + \sum_{i=1}^{M} \bar{J}_i = 2M$, so that conservation of paths is respected in the rightmost region of the partition function.

We turn our attention firstly to the leftmost $N$ columns of \eqref{eq:K-leftmost-columns}. We see that this region takes a very similar form to the quantity \eqref{IK-chi-inv} defined in Section \ref{ssec:IKstab-construct}. Indeed, the only significant difference is that top $M \times N$ block of these columns uses dotted vertices, which differ from their undotted counterparts only up to normalization. We thus find that
\begin{align}
\label{eq:K-leftmost-compute}
\begin{tikzpicture}[scale = 0.6,every node/.style={scale=0.6},baseline=2cm]
\foreach \i in {1,2,3} {
\draw[lgray,line width=1.5pt,->] (\i,0) -- (\i,8);
}
\foreach \i in {1,2,3,4,5,6,7} {
\draw[lgray,line width=1.5pt] (0,\i) -- (3.5,\i);
}
\foreach \i in {1,2,3,4,5,6,7} {
\node[scale = 1.25] at (-0.2,\i) {$2$};
}
\foreach \i in {1,3} {
\node[scale = 1.25] at (\i,-0.2) {$0$};
\node[scale = 1.25] at (\i,8.2) {$2$};
}
\node at (2,8.2) {$\dots$};
\node at (2,-0.2) {$\dots$};
\foreach \i in {2} {
\node at (3.8,\i) {$\vdots$};
}
\node[scale = 1.25] at (3.8,3) {$I_N$};
\node[scale = 1.25] at (3.8,1) {$I_1$};
\node[scale = 1.25] at (3.8,7) {$\bar{J}_1$};
\node[scale = 1.25] at (3.8,4) {$\bar{J}_M$};
\node at (3.8,5) {$\vdots$};
\node at (3.8,6) {
$\vdots$
};
\draw[->,thick] (1,-1.2) -- (1,-0.5);
\draw[->,thick] (3,-1.2) -- (3,-0.5);
\node[scale = 1.5] at (1,-1.5) {$\infty$};
\node[scale = 1.5] at (2,-1.5) {$\dots$};
\node[scale = 1.5] at (3,-1.5) {$\infty$};
\node[scale = 1.25] at (-1.8,1) {$*$};
\node[scale = 1.25] at (-1.8,2) {$\vdots$};
\node[scale = 1.25] at (-1.8,3) {$*$};
\node[scale = 1.25] at (-1.8,4) {$*$};
\node[scale = 1.25] at (-1.8,7) {$*$};
\foreach \i in {5,6} {
\node at (-1.8,\i) {$\vdots$};
}
\foreach \i in {1,2,3,4,5,6,7} {
\draw[thick,->] (-1.5,\i) -- (-0.5,\i);
}
\foreach \i in {1,2,3} { \foreach \j in {4,5,6,7}
{\node at (\i,\j) [circle,fill,inner sep=2pt]{};};}
\end{tikzpicture}
=
q^{4MN}
\times
\chi(I_1,\dots,I_N,\bar{J}_M,\dots,\bar{J}_1)
\frac{
(-1)^{\mathfrak{c}_1(I)/2+\mathfrak{c}_1(J)/2}
(1-q^{-1})^{\mathfrak{c}_1(I)+\mathfrak{c}_1(J)}
}
{q^{\inv(I)+\inv(J)+|I||J|}}
\end{align}
with $\chi$ defined in \eqref{IK-chi-inv}, and where the factor of $q^{4MN}$ arises from the above-mentioned normalization discrepancy between dotted and undotted vertices, while the $q$-inversions in the denominator of \eqref{eq:K-leftmost-compute} are computed using the fact that
\begin{align*}
\inv(I_1,\dots,I_N,\bar{J}_M,\dots,\bar{J}_1)
&=
\inv(I_1\dots,I_N)+\inv(\bar{J}_M,\dots,\bar{J}_1)+|I||J|
\\
&=
\inv(I_1\dots,I_N)+\inv(J_1,\dots,J_M)+|I||J|.
\end{align*}
By the invariance of $\chi$ under re-partitioning of its arguments, we may choose $(I_1,\dots,I_N,\bar{J}_M,\dots,\bar{J}_1) = (0^N,2^M)$, which by \eqref{IK-chi-inv2} yields 
\begin{align*}
\chi(I_1,\dots,I_N,\bar{J}_M,\dots,\bar{J}_1)
=
\chi(0^N,2^M) = \lim_{\bm{u}\to \infty} F_{(2^{N})}(*;\bm{u}).
\end{align*}
We have thus derived the equation
\begin{multline}
\label{eq:K-right-flip}
\lim_{\bm{u}\to \infty} \frac{K_{N,M}}{F_{(2^{N})}(*;\bm{u})}
=
\\
q^{4MN}
\displaystyle\sum_{\substack{(I_1,\dots,I_N) \\ (J_1,\dots,J_M) \\ |I|=|J|}}
\frac{
(-1)^{\mathfrak{c}_1(I)/2+\mathfrak{c}_1(J)/2}
(1-q^{-1})^{\mathfrak{c}_1(I)+\mathfrak{c}_1(J)}
}
{q^{\inv(I)+\inv(J)+|I||J|}}
\displaystyle\sum_{S\in \mathfrak s(2)}\raisebox{-35mm}
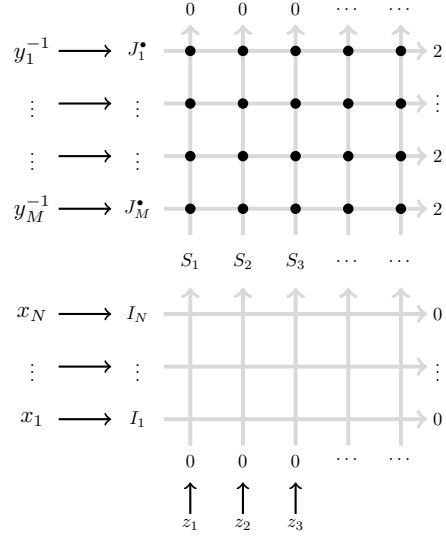
{
\begin{tikzpicture}[scale = 0.7,every node/.style={scale=0.7}]
\foreach \i in {4.5,5.5,6.5,7.5,8.5} {
\draw[lgray,line width=1.5pt,->] (\i,4.5) -- (\i,8.5);
} 
\begin{scope}[shift = {(3.5,0)}]
\foreach \i in {1,2,3,5,6,7,8} {
\draw[thick,->] (-1.5,\i) -- (-0.5,\i);
}
\node[scale = 1.25] at (-2,1) {$x_1$};
\foreach \i in {2,6,7} {
\node at (-2,\i) {$\vdots$};
}
\node[scale = 1.25] at (-2,3) {$x_N$};
\node[scale = 1.25] at (-2,5) {$y_M^{-1}$};
\node[scale = 1.25] at (-2,8) {$y_1^{-1}$};
\end{scope}
\foreach \i in {1,2,3,5,6,7,8} {
\draw[lgray,line width=1.5pt,->] (4,\i) -- (9,\i);
}
\foreach \i in {4.5,5.5,6.5,7.5,8.5} {
\draw[lgray,line width=1.5pt,->] (\i,0.5) -- (\i,3.5);
} 
\foreach \i in {4.5,5.5,6.5} {
\node at (\i,0.2) {$0$};
}
\node at (3.5,1) {$I_1$};
\node at (3.5,2) {$\vdots$};
\node at (3.5,3) {$I_{N}$};
\node at (3.5,5) {$\bar{J}_M$};
\node at (3.5,6) {$\vdots$};
\node at (3.5,7) {$\vdots$};
\node at (3.5,8) {$\bar{J}_1$};
\node at (4.5,4) {$S_{1}$};
\node at (5.5,4) {$S_{2}$};
\node at (6.5,4) {$S_{3}$};
\node at (7.5,0.25) {$\dots$};
\node at (8.5,0.25) {$\dots$};
\node at (7.5,4) {$\dots$};
\node at (8.5,4) {$\dots$};
\node at (9.2,1) {$0$};
\node at (9.2,2) {$\vdots$};
\node at (9.2,3) {$0$};
\node at (9.2,5) {$2$};
\node at (9.2,6) {$2$};
\node at (9.2,7.1) {$\vdots$};
\node at (9.2,8) {$2$};
\foreach \i in {4.5,5.5,6.5} {
\node at (\i,8.8) {$0$};
}
\foreach \i in {7.5,8.5} {
\node at (\i,8.8) {$\dots$};
}
\foreach \i in {4.5,5.5,6.5} {
\draw[thick,->] (\i,-0.8) -- (\i,-0.2);
}
\node at (4.5,-1) {$z_1$};
\node at (5.5,-1) {$z_2$};
\node at (6.5,-1) {$z_3$};
\foreach \i in {4.5,5.5,6.5,7.5,8.5} { \foreach \j in {5,6,7,8} {
\node at (\i,\j)[circle,fill,inner sep=2pt]{};
}}
\end{tikzpicture}}
\end{multline}
where summation over $S \in \mathfrak{s}(2)$ is restricted to $2$-strings such that $|S|=|I|=|J|$. 

The final step is to apply a symmetry relation to the top half of the lattice on the right hand side of \eqref{eq:K-right-flip}, similarly to that of Proposition \ref{prop:IK-flip-sym}. Letting the top half of the lattice be denoted by
\begin{equation}
\label{eq:IK-Fdot}
\raisebox{-37.5mm}{
\begin{tikzpicture}[scale = 0.6]
\foreach \i in {0,1,2,3,4,6,7} {
    \draw[->,lgray,line width=1.5pt] (\i,0) -- (\i,7);}
    \foreach \i in {1,2,3,4,5,6} {
    \draw[lgray,line width=1.5pt] (-1,\i) -- (4.3,\i);}
    \foreach \i in {1,2,3,4,5,6} {
    \draw[->,lgray,line width=1.5pt] (5.5,\i) -- (8,\i);}
    \foreach \i  in {0,1,2,3,4,6,7} { \foreach \j in {1,2,3,4,5,6} {
\node at (\i,\j)     [circle,fill, inner sep=1.5pt] {};}
}
    \foreach \i in {1,2,3,4,5,6} {
    \node at (5,\i) {$\dots$};}
    \foreach \i in {1,2,6} {
    \node at (8.3,\i) {\tiny $2$};
    }
    \foreach \i in {3,4,5} {
    \node at (8.3,\i+0.2) {$\vdots$};
    }
    \node[left] at (-1.8,1) {$y_M^{-1} \to$};
    \foreach \i in {3,4,5} {
    \node at (-2.2,\i) {$\vdots$};
    }
    \node[left] at (-1.8,6) {$y_1^{-1} \to$};
    \node at (-1.4,1) {\tiny $\bar{J}_M$};
    \node at (-1.4,6) {\tiny $\bar{J}_{1}$};
    \node[left] at (-3.8,3.5) {$F_{S}^{\bullet J}(y_1^{-1},\dots,y_M^{-1};\bm{z})=$};
    \node at (0,7.5) {\tiny $0$};
    \node at (1,7.5) {\tiny $0$};
    \node at (2,7.5) {\tiny $0$};
    \node at (3,7.5) {$\dots$};
    \node at (4,7.5) {$\dots$};
    \node at (0,-2) {$z_1$};
    \node at (1,-2) {$z_2$};
    \node at (2,-2) {$z_3$};
    \node at (3,-2) {$\dots$};
    \node at (4,-2) {$\dots$};
    \draw[thick,->] (0,-1.5) -- (0,-0.85);
    \draw[thick,->] (1,-1.5) -- (1,-0.85);
    \draw[thick,->] (2,-1.5) -- (2,-0.85);
    \node at (0,-0.5) {\tiny $S_1$};
    \node at (1,-0.5) {\tiny $S_2$};
    \node at (2,-0.5) {\tiny $S_3$};
    \node at (3,-0.5) {$\dots$};
    \node at (4,-0.5) {$\dots$};
\end{tikzpicture}}
\end{equation}
one may prove the relation
\begin{align*}
F^{\bullet J}_S(y_1^{-1},\dots,y_M^{-1};\bm{z})
=
(-1)^{|S|/2}
c_S(q)
F^{J}_S(y_1,\dots,y_M;q^{-3}\bm{z}^{-1})
\end{align*}
by repeated application of \eqref{IK-tower-sym} to the columns of the partition function \eqref{eq:IK-Fdot}, where $c_S(q)$ is given by \eqref{eqcsfact1} and $F^{J}_S(y_1,\dots,y_M;q^{-3}\bm{z}^{-1})$ by \eqref{eqstableIK191}. Combining everything, we have shown that
\begin{multline*}
\lim_{\bm{u}\to \infty} \frac{K_{N,M}}{F_{(2^{N})}(*;\bm{u})}
=
\sum_{S \in \mathfrak{s}(2)}
(-1)^{|S|/2}
q^{4MN-|S|^2}
c_S(q)
\times
\\
\displaystyle
\sum_{\substack{(I_1,\dots,I_N) \\ (J_1,\dots,J_M) \\ |I|=|J|=|S|}}
\frac{
(-1)^{\mathfrak{c}_1(I)/2+\mathfrak{c}_1(J)/2}
(1-q^{-1})^{\mathfrak{c}_1(I)+\mathfrak{c}_1(J)}
}
{q^{\inv(I)+\inv(J)}}
F^{I}_S(x_1,\dots,x_N;\bm{z})
F^{J}_S(y_1,\dots,y_M;q^{-3}\bm{z}^{-1})
\\
=
\sum_{S \in \mathfrak{s}(2)}
(-1)^{|S|/2}
q^{4MN-|S|^2}
c_S(q)
H_S(x_1,\dots,x_N;\bm{z})
H_S(y_1,\dots,y_M;q^{-3}\bm{z}^{-1}),
\end{multline*}
completing the proof of \eqref{eq:IK-stable-cauchy}.
\end{proof}

\section{Properties of rational symmetric functions}
\label{sec:19-prop}

This section serves as a prelude to the next one. Our aim is to provide a list of properties satisfied by the rational symmetric functions $F_S(x_1,\dots,x_N;\bm{z})$, before writing down a solution to these properties in Section \ref{sec:19-sym}.

In Section \ref{ssec:exchange}, we state a family of exchange relations obeyed by the function $F_S(x_1,\dots,x_N;\bm{z})$. These relations allow one to start from an arbitrary $2$-string $S$ and progressively order it. More precisely, these relations provide a way to convert an arbitrary function $F_S$ into a linear combination of functions $F_T$, such that each $T \in \mathfrak{s}(2)$ in this sum has the form $T_1 \geq T_2 \geq T_3 \geq  \cdots$. We shall refer to such $2$-strings as {\it ordered}. 

In Section \ref{ssec:recursion}, we use the lattice formulation of $F_T(x_1,\dots,x_N;\bm{z})$ to derive recursion relations that it satisfies when the indexing $2$-string $T$ is ordered. These recursion relations provide the foundation for standard Lagrange interpolation techniques. 

Finally, in Section \ref{ssec:unique}, we prove that all of these properties are uniquely-determining, meaning that any family of functions that obey the properties must be equal to $F_S(x_1,\dots,x_N;\bm{z})$ itself.

\subsection{Exchange relations}
\label{ssec:exchange}

\begin{defn}
Fix an integer $k \geq 1$ and let $S,T \in \mathfrak{s}(2)$ be $2$-strings such that $(S_k,S_{k+1}) \not= (T_k,T_{k+1})$, but $S_k+S_{k+1} = T_k+T_{k+1}$, and $S_i=T_i$ for all $i \not= k,k+1$. We say that $S,T$ are $k$-exchangeable and write $S \sim_k T$. 
\end{defn}

\begin{prop}
\label{prop:IK-exch}
Fix an integer $k \geq 1$ and a $2$-string $S \in \mathfrak{s}(2)$ such that $S_k < S_{k+1}$. The following relation then holds:
\begin{multline}
\label{eq:IK-exchange}
W_{z_{k+1}/z_k}(S_{k+1},S_k;S_{k+1},S_k)
F_S(x_1,\dots,x_N;\bm{z})
=
\\
F_{\sigma_k \cdot S}(x_1,\dots,x_N;\sigma_k \cdot \bm{z})
-
\sum_{T:S \sim_k T}
W_{z_{k+1}/z_k}(T_{k+1},T_k;S_{k+1},S_k)
F_T(x_1,\dots,x_N;\bm{z}),
\end{multline}
with the sum taken over $2$-strings $T$ such that $S,T$ are $k$-exchangeable, and where $\sigma_k\cdot S = (\dots,S_{k+1},S_k,\dots)$, $\sigma_k \cdot \bm{z} = (\dots,z_{k+1},z_k,\dots)$ denote transposition of the $k$-th and $(k+1)$-th elements of $S,\bm{z}$.
\end{prop}

\begin{proof}
The proof is by application of the Yang--Baxter equation \eqref{eq19ybe1} to the $k$-th and $(k+1)$-th columns of the partition function in question. 

In particular, starting from the lattice representation of $F_{\sigma_k \cdot S}(x_1,\dots,x_N;\sigma_k \cdot \bm{z})$, one may attach a frozen $W_{z_{k+1}/z_k}(0,0;0,0)$ vertex at the base of the $k$-th and $(k+1)$-th columns. By repeated application of the Yang--Baxter equation, this vertex emerges from the top of the same columns, leading to the identity
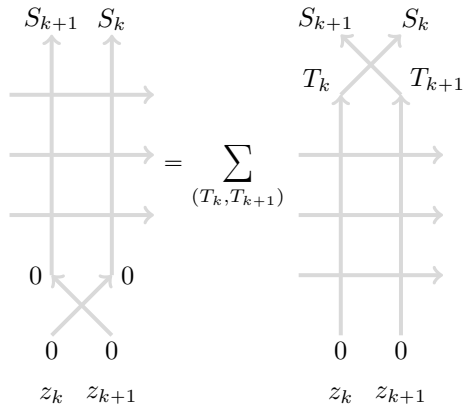
\begin{align*}
\begin{tikzpicture}[scale=0.8,baseline=1.4cm]
    \foreach \i in {3.5,4.5} {
    \draw[->,lgray,line width=1.5pt] (\i,0) -- (\i,4);}
    \foreach \i in {1,2,3} {
    \draw[->,lgray,line width=1.5pt] (2.8,\i) -- (5.2,\i);}
    \draw[<-,lgray,line width=1.5pt] (3.5,0) -- (4.5,-1);
    \draw[<-,lgray,line width=1.5pt] (4.5,0) -- (3.5,-1);
    \node[left] at (3.5,0) {$0$};
    \node[right] at (4.5,0) {$0$};
    \node at (3.5,4.25) {$S_{k+1}$};
    \node at (4.5,4.25) {$S_k$};
    \node at (3.5,-1.25) {$0$};
    \node at (4.5,-1.25) {$0$};
    \node at (3.5,-2) {$z_k$};
    \node at (4.5,-2) {$z_{k+1}$};
\end{tikzpicture}
=
\sum_{(T_k,T_{k+1})}
\begin{tikzpicture}[scale=0.8,baseline=1.4cm]
    \foreach \i in {6.5,7.5} {
    \draw[->,lgray,line width=1.5pt] (\i,-1) -- (\i,3);}
    \foreach \i in {1,2,3} {
    \draw[->,lgray,line width=1.5pt] (5.8,\i-1) -- (8.2,\i-1);}
    \draw[<-,lgray,line width=1.5pt] (6.5,4) -- (7.5,3);
    \draw[<-,lgray,line width=1.5pt] (7.5,4) -- (6.5,3);
    \node at (6.25,4.25) {$S_{k+1}$};
    \node at (7.75,4.25) {$S_k$};
    \node at (6.5,-1.25) {$0$};
    \node at (7.5,-1.25) {$0$};
    \node at (6.5,-2) {$z_k$};
    \node at (7.5,-2) {$z_{k+1}$};
    \node[left] at (6.5,3.25) {$T_k$};
    \node[right] at (7.5,3.25) {$T_{k+1}$};
\end{tikzpicture}
\end{align*}
where the states assigned to external horizontal edges are arbitrary, but held fixed on both sides of the equation. We then read off the equation
\begin{align*}
F_{\sigma_k \cdot S}(x_1,\dots,x_N;\sigma_k \cdot \bm{z})
=
\sum_{T}
W_{z_{k+1}/z_k}(T_{k+1},T_k;S_{k+1},S_k)
F_{T}(x_1,\dots,x_N;\bm{z})
\end{align*}
where the sum is over $2$-strings such that $T_i=S_i$ for all $i \not= k,k+1$. Equation \eqref{eq:IK-exchange} then follows by rearrangement.
\end{proof}

\begin{rmk}
While the $k$-th and $(k+1)$-th elements of $S$ in Proposition \ref{prop:IK-exch} are increasing, it is easily seen that all $2$-strings present on the right hand side of \eqref{eq:IK-exchange} are non-increasing with respect to the same two elements. It is therefore possible to iterate \eqref{eq:IK-exchange}, converting an arbitrary initial $2$-string into a linear combination of non-increasing ones.
\end{rmk}

\begin{prop}
\label{prop:11}
Fix an integer $N \geq 1$. The following relation holds:
\begin{multline}
\label{11-02}
W_{z_2/z_1}(1,1;2,0)
F_{(1^{2N})}(x_1,\dots,x_N;\bm{z})
=
\\
F_{(2,0,1^{2N-2})}(x_1,\dots,x_N;\sigma_1 \cdot \bm{z})
-
W_{z_2/z_1}(0,2;2,0)
F_{(2,0,1^{2N-2})}(x_1,\dots,x_N;\bm{z})
\\
-
W_{z_2/z_1}(2,0;2,0)
\prod_{i=1}^{N}
W_{z_1/x_i}(0,2;0,2)
F_{(2,1^{2N-2})}(x_1,\dots,x_N;\widehat{\bm{z}}_1),
\end{multline}
where $\widehat{\bm{z}}_1 = (z_2,z_3,\dots)$ denotes the secondary alphabet with the omission of $z_1$.
\end{prop}

\begin{proof}
The proof follows along similar lines to that of Proposition \ref{prop:IK-exch}. This time, one starts from the lattice representation of $F_{(2,0,1^{2N-2})}(x_1,\dots,x_N;\sigma_1 \cdot \bm{z})$ and attaches a frozen $W_{z_2/z_1}(0,0;0,0)$ vertex at the base of the first and second columns. By repeated application of the Yang--Baxter equation, this vertex emerges from the top of the same columns, leading to the identity
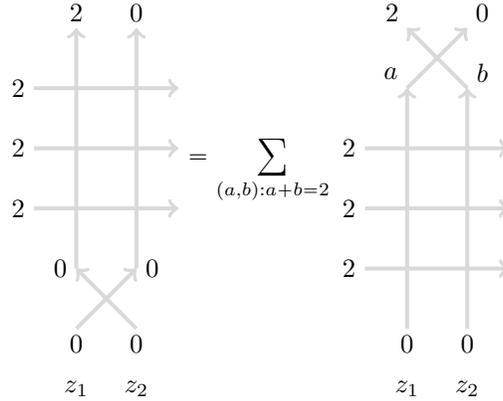
\begin{align*}
\begin{tikzpicture}[scale=0.8,baseline=1.4cm]
    \foreach \i in {3.5,4.5} {
    \draw[->,lgray,line width=1.5pt] (\i,0) -- (\i,4);}
    \foreach \i in {1,2,3} {
    \draw[->,lgray,line width=1.5pt] (2.8,\i) -- (5.2,\i);}
    \draw[<-,lgray,line width=1.5pt] (3.5,0) -- (4.5,-1);
    \draw[<-,lgray,line width=1.5pt] (4.5,0) -- (3.5,-1);
    \node[left] at (3.5,0) {$0$};
    \node[right] at (4.5,0) {$0$};
    \node at (3.5,4.25) {$2$};
    \node at (4.5,4.25) {$0$};
    \node at (3.5,-1.25) {$0$};
    \node at (4.5,-1.25) {$0$};
    \node at (3.5,-2) {$z_1$};
    \node at (4.5,-2) {$z_2$};
    \foreach\i in {1,2,3}{
    \node[left] at (2.8,\i) {$2$};
    } 
\end{tikzpicture}
=
\sum_{(a,b):a+b=2}
\begin{tikzpicture}[scale=0.8,baseline=1.4cm]
    \foreach \i in {6.5,7.5} {
    \draw[->,lgray,line width=1.5pt] (\i,-1) -- (\i,3);}
    \foreach \i in {1,2,3} {
    \draw[->,lgray,line width=1.5pt] (5.8,\i-1) -- (8.2,\i-1);}
    \draw[<-,lgray,line width=1.5pt] (6.5,4) -- (7.5,3);
    \draw[<-,lgray,line width=1.5pt] (7.5,4) -- (6.5,3);
    \node at (6.25,4.25) {$2$};
    \node at (7.75,4.25) {$0$};
    \node at (6.5,-1.25) {$0$};
    \node at (7.5,-1.25) {$0$};
    \node at (6.5,-2) {$z_1$};
    \node at (7.5,-2) {$z_2$};
    \node[left] at (6.5,3.25) {$a$};
    \node[right] at (7.5,3.25) {$b$};
    \foreach\i in {1,2,3}{
    \node[left] at (5.8,\i-1) {$2$};
    }
\end{tikzpicture}
\end{align*}
where the states assigned to left horizontal edges are all $2$, as that coincides with the left boundary of the partition function \eqref{pictureofrational1}. We then obtain the equation
\begin{align}
\label{eq:ab-sum}
F_{(2,0,1^{2N-2})}(x_1,\dots,x_N;\sigma_1 \cdot \bm{z})
=
\sum_{(a,b):a+b=2}
W_{z_2/z_1}(b,a;2,0)
F_{(a,b,1^{2N-2})}(x_1,\dots,x_N; \bm{z}).
\end{align}
The sum over $(a,b)$ such that $a+b=2$ consists of three possible pairs: namely, $(a,b)=(2,0)$, $(a,b)=(1,1)$ and $(a,b)=(0,2)$. The final of these leads to a simplification of the function $F_{(a,b,1^{2N-2})}(x_1,\dots,x_N; \bm{z})$:
\begin{align}
\label{eq:column-freeze}
F_{(0,2,1^{2N-2})}(x_1,\dots,x_N; \bm{z})
=
\prod_{i=1}^{N}
W_{z_1/x_i}(0,2;0,2)
F_{(2,1^{2N-2})}(x_1,\dots,x_N; \widehat{\bm{z}}_1),
\end{align}
which is easily deduced by freezing of the first column of the partition function \eqref{pictureofrational1} for $F_S(x_1,\dots,x_N;\bm{z})$ when $S_1=0$. Combining \eqref{eq:ab-sum} and \eqref{eq:column-freeze} then yields the result \eqref{11-02}.
\end{proof}

\begin{rmk}\label{rem:2string-order}
Collectively, Propositions \ref{prop:IK-exch} and \ref{prop:11} have the following implication: not only is it possible to completely order the $2$-strings that index our symmetric functions, it is also possible to ensure that every ordered $2$-string begins with a $2$. 

To see this, recall that Proposition \ref{prop:IK-exch} allows $F_S$ to be written as a linear combination of functions $F_T$, where $T$ is ordered. If $T=(1^{2N})$, the only ordered $2$-string that does not begin with $2$, application of Proposition \ref{prop:11} results in a linear combination of $2$-strings that begin with $2$. While not all of the $2$-strings appearing on the right hand side of \eqref{11-02} are non-increasing, renewed application of \eqref{eq:IK-exchange} allows the strings $(2,0,1^{2N-2})$ to be ordered (without disturbing the $2$ at the head of the string).

\end{rmk}

\subsection{Recursion relations}
\label{ssec:recursion}

\begin{prop}
\label{prop:IK-properties}
Fix integers $N \in \mathbb{Z}_{>0}$ and $M \in \mathbb{Z}_{\geq 0}$ such that $N>M$, and let $T=(2^{N-M},1^{2M})$ denote a non-increasing $2$-string. The function $F_T(x_1,\dots,x_N;\bm{z})$ satisfies the following properties:

\begin{enumerate}[wide,labelindent=0pt]

\item The function $F_T(x_1,\dots,x_N;\bm{z})$ is symmetric with respect to its primary alphabet $(x_1,\dots,x_N)$.
\medskip

\item Introduce the normalized function 
\begin{equation*}
\tilde{F}_{T}(x_1,\dots,x_N;\bm{z}) 
= 
\prod_{i=1}^{N}\prod_{j=1}^{N+M}
(x_i-q^2 z_j)(x_i-q^3 z_j)
F_{T}(x_1,\dots,x_N;\bm{z}).
\end{equation*}
Then $\tilde{F}_{T}$ is a polynomial in $z_1$ of degree $2N$.
\medskip

\item The function $F_{T}$ vanishes at $z_1=0$:
\begin{align*}
F_{T}(x_1,\dots,x_N;\bm{z})
\Big|_{z_1 = 0} = 0.
\end{align*}
\medskip

\item The function $F_{T}$ has the following recursion relation with respect to $z_1$:
\begin{align*}
F_{T}(x_1,\dots,x_N;\bm{z})
\Big|_{z_1 = x_N} = 
F_{T\backslash 2}(x_1,\dots,x_{N-1};\widehat{\bm{z}}_1), 
\end{align*}
where $\widehat{\bm{z}}_1 = (z_2,z_3,\dots)$ denotes the secondary alphabet with the omission of $z_1$ and $T\backslash2 = (2^{N-M-1},1^{2M})$.
\medskip

\item The function $F_{T}$ has a further recursion relation with respect to $z_1$:
\begin{multline}
\label{eqdomw3}
{\rm Res}
\Big\{
F_{T}(x_1,\dots,x_N;\bm{z})
\Big\}_{z_1 = q^{-3}x_N}
=
-\frac{(1-q^2)(1-q^3)}{q^6(1-q)}
x_N
\\[5pt]
\times
\prod_{i=1}^{N-1}
\frac{(x_N-q^2 x_i)(x_N-q^3 x_i)}{(x_N-x_i)(x_N-q x_i)}
\prod_{j=2}^{N+M}
\frac{(x_N-z_j)(x_N-q z_j)}{(x_N-q^2 z_j)(x_N-q^3 z_j)}
\sum_{U} \Phi_{U} F_U(x_1,\dots,x_{N-1};\widehat{\bm{z}}_1),
\end{multline}
with the sum taken over $2$-strings $U=(U_2,\dots,U_{N+M})$ such that $|U| = 2N-2$, and where the coefficients $\Phi_U$ are defined as the one-row partition functions
\begin{equation}
\label{cu-def}
\Phi_U
\equiv
\Phi_{(U_2,\dots,U_{N+M})}
=
\raisebox{-24mm}{
\begin{tikzpicture}[scale=0.8,baseline=-2.45cm]
\node[left] at (0.8,0) {$2$};
\draw[->,thick] (9.2,0) -- (8.7,0);
\node at (10,0) {$q^{-3}x_N$}; 
\draw[->,lgray,line width=1.5pt] (8,0) -- (1,0);
\foreach \i in {2,3,4,5,6,7} {
\draw[->,lgray,line width=1.5pt] (\i,-1) -- (\i,1);
}
\foreach \i in {4,5,6} {
\node at (\i,-2) {$\dots$};
}
\node at (2,-2.3) {$z_2$};
\node at (3,-2.3) {$z_3$};
\node at (7,-2.3) {$z_{N+M}$};
\node at (2,-1.2) {$U_2$};
\node at (3,-1.2) {$U_3$};
\foreach \i in {4,5,6} {
\node at (\i,-1.2) {$\dots$};
}
\node at (7,-1.2) {$U_{N+M}$};
\foreach \i in {2,3,7} {
\draw[->,thick
] (\i,-2) -- (\i,-1.5);
}
\node at (2,1.2) {$2$};
\node at (3,1.2) {$\dots$};
\node at (4,1.2) {$2$};
\node at (5,1.2) {$1$};
\node at (6,1.2) {$\dots$};
\node at (7,1.2) {$1$};
\node at (8.4,0) {$2$};
\end{tikzpicture}}
\end{equation}
with top boundary conditions corresponding with the string $(2^{N-M-1})\cup (1^{2M})$.
\medskip

\item In the case $N=1$, $M=0$ we have the explicit evaluation
\begin{align*}
F_{(2)}(x_1;\bm{z}) 
= 
W_{z_1/x_1}(0,2;2,0) 
= 
\frac{(1-q^2)(x_1+q^2 x_1-q^2 z_1-q^3 z_1)z_1}
{(x_1-q^2 z_1)(x_1-q^3 z_1)}.
\end{align*}

\end{enumerate}
\end{prop}

\begin{proof}
All of the properties (1)--(6) follow directly from the partition function formula \eqref{pictureofrational1}.

\begin{enumerate}[wide,labelindent=0pt]
\item The symmetry of $F_T(x_1,\dots,x_N;\bm{z})$ in $(x_1,\dots,x_N)$ has already been established in Theorem \ref{thm:IK-FG-sym}.
\medskip

\item Multiplying the partition function by $\prod_{i=1}^{N}\prod_{j=1}^{N+M}
(x_i-q^2 z_j)(x_i-q^3 z_j)$ is equivalent to converting the underlying vertex weights to their polynomial form. In particular, we define
\begin{align}
\label{poly-weights}
\tilde{W}_{x,z}(i,j;k,\ell)
:=
(x-q^2 z)(x-q^3 z)
W_{z/x}(i,j;k,\ell),
\end{align}
which are polynomials in $x$ and $z$ for all $i,j,k,\ell \in \{0,1,2\}$. The function $\tilde{F}_T(x_1,\dots,x_N;\bm{z})$ is then given by the partition function setup \eqref{pictureofrational1}, but using the weights \eqref{poly-weights}. Examining this partition function, it is clear that all dependence on the parameter $z_1$ is via the leftmost column of the lattice. Since the weights \eqref{poly-weights} of individual vertices in that column are at most degree-$2$ polynomials in $z_1$, it is then immediate that $\tilde{F}_T(x_1,\dots,x_N;\bm{z})$ is at most a degree-$2N$ polynomial in $z_1$.
\medskip

\item This again follows by studying the leftmost column of the partition function under consideration. It is clear that any configuration of this column must give rise to at least one of the following vertex weights:
\begin{align*}
W_{z_1/x_i}(1,2;2,1)
&=
\frac{(1-q^2)z_1}{(x_i-q^2 z_1)},
\\[5pt]
W_{z_1/x_i}(0,2;1,1)
&=
\frac{q(1-q)^2(1+q)(x_i-z_1)z_1}
{(x_i-q^2 z_1)(x_i-q^3 z_1)},
\\[5pt]
W_{z_1/x_i}(0,2;2,0)
&=
\frac{(1-q^2)(x_i+q^2x_i-q^2z_1-q^3z_1)z_1}
{(x_i-q^2 z_1)(x_i-q^3 z_1)},
\end{align*}
where $i$ takes values in $\{1,\dots,N\}$. All of these weights vanish at $z_1=0$, and the claim is then immediate.
\medskip

\item Take the lattice representation of $F_T(x_1,\dots,x_N;\bm{z})$ and use the symmetry in $(x_1,\dots,x_N)$ to permute the horizontal spectral parameters, so that $x_N$ is assigned to the lowest row of the lattice, followed by $(x_1,\dots,x_{N-1})$ assigned to the remaining rows. In any lattice configuration, the bottom-left vertex must then take one of the following forms:
\begin{align*}
W_{z_1/x_N}(0,2;0,2)
&=
\frac{(x_N-z_1)(x_N-qz_1)}{(x_N-q^2 z_1)(x_N-q^3 z_1)},
\\[5pt]
W_{z_1/x_N}(0,2;1,1)
&=
\frac{q(1-q)^2(1+q)(x_N-z_1)z_1}
{(x_N-q^2 z_1)(x_N-q^3 z_1)},
\\[5pt]
W_{z_1/x_N}(0,2;2,0)
&=
\frac{(1-q^2)(x_N+q^2x_N-q^2z_1-q^3z_1)z_1}
{(x_N-q^2 z_1)(x_N-q^3 z_1)}.
\end{align*}
The first two of these weights vanish at $z_1=x_N$; the third becomes $W_{1}(0,2;2,0)=1$. It is then easy to show that, after specializing $z_1=x_N$, the first row and column of $F_{T}(x_1,\dots,x_N;\bm{z})$ freeze out entirely:
\begin{equation}
\label{eqfreez1}
F_T(x_1,\dots,x_N;\bm{z})
\Big|_{z_1 = x_N}
=
\begin{tikzpicture}[scale = 0.7,every node/.style={scale=0.6},baseline=2.5cm]
\foreach \i in {-1,0,1,2,3,4,5,6} {
    \draw[->,lgray,line width=1.5pt] (\i,0) -- (\i,7);}
    \foreach \i in {1,2,3,4,5,6} {
    \draw[->,lgray,line width=1.5pt] (-2,\i) -- (7,\i);}
    \foreach \i in {1,2,6} {
    \node at (7.3,\i) {$0$};
    }
    \foreach \i in {3,4,5} {
    \node at (7.3,\i) {$\vdots$};
    }
    \node[left,scale = 1.25] at (-3.8,1) {$x_N$};
    \node[left,scale = 1.25] at (-3.8,2) {$x_1$};
    \node[left] at (-3.8,3) {$\vdots$};
    \node[left] at (-3.8,4) {$\vdots$};
    \node[left] at (-3.8,5) {$\vdots$};
    \node[left,scale = 1.25] at (-3.8,6) {$x_{N-1}$};
    \draw[thick,->] (-3.6,1) -- (-2.6,1);
    \draw[thick,->] (-3.6,2) -- (-2.6,2);
    \draw[thick,->] (-3.6,6) -- (-2.6,6);
    \foreach \i in {1,2,6} {
    \node at (-2.4,\i) {$2$};
    }
    \node at (-2.4,3) {$\vdots$};
    \node at (-2.4,4) {$\vdots$};
    \node at (-2.4,5) {$\vdots$};
    \node at (-1,1.5) {$2$};
    \foreach \i in {-1,0,2} {
    \node at (\i,7.3) {$2$};
    }
    \foreach \i in {3,6} {
    \node at (\i,7.3) {$1$};
    }
    \foreach \i in {1,4,5} {
    \node at (\i,7.3) {$\dots$};
    }
    \draw[->,thick] (0,-1) -- (0,-0.5);
    \draw[->,thick] (-1,-1) -- (-1,-0.5);
    \draw[->,thick] (6,-1) -- (6,-0.5);
    \node[scale = 1.25] at (-1,-1.2) {$x_N$};
    \node[scale = 1.25] at (0,-1.2) {$z_2$};
    \node[scale = 1.25] at (1,-1.2) {$\dots$};
    \node[scale = 1.25] at (6.1,-1.2) {$z_{N+M}$};
    \foreach \i in {-1,0,6} {
    \node at (\i,-0.2) {$0$};
    }
    \node at (1,-0.2) {$\dots$};
    \foreach \i in {4,5} {
    \node at (\i,-0.2) {$\dots$};
    }
    \foreach \i in {2,3,4,5,6} {
    \node at (-0.5,\i) {$2$};
    }
    \node at (-0.5,1) {$0$};
    \foreach \i in {0,1,2,3,4,5,6} {
    \node at (\i,1.5) {$0$};
    }
    \node at (-1,1) [circle,fill,inner sep=1.5pt,red] {};
    \end{tikzpicture}
\end{equation}
The total weight of the frozen regions in \eqref{eqfreez1} is identically $1$, and the remaining (unfrozen) portion of the lattice is simply $F_{T\backslash 2}(x_1,\dots,x_{N-1};\widehat{\bm{z}}_1)$, yielding the desired recursion.
\medskip

\item One begins by observing that the partition function $F_T(x_1,\dots,x_N;\bm{z})$ admits the following, alternative lattice formulation:
\begin{equation}
\label{eqpulllines} 
F_T(x_1,\dots,x_N;\bm{z})
=
\begin{tikzpicture}[scale = 0.5,every node/.style={scale=0.6},baseline=0.7cm]
\foreach \i in {-1,0,1,2,3,4} {
\draw[lgray,line width=1.5pt,->] (0,\i) -- (8,\i);
}
\draw[lgray,line width=1.5pt,->] (1,-1.5) -- (1,5);
\foreach \i in {2,3,4,5,6,7} {
\draw[lgray,line width=1.5pt,->] (\i,-3) -- (\i,5);
}
\foreach \i in {-1,0,4} {
\node at (-0.3,\i) {$2$};
}
\foreach \i in {1,2,3} {
\node at (-0.3,\i) {$\vdots$};
}
\foreach \i in {-1,4} {
\node at (8.3,\i) {$0$};
}
\node at (8.3,0) {$\vdots$};
\node at (8.3,1) {$\vdots$};
\foreach \i in {2,3} {
\node at (8.3,\i) {$\vdots$};
}
\foreach \i in {1,3} {
\node at (\i,5.3) {$2$};
}
\foreach \i in {4,7} {
\node at (\i,5.3) {$1$};
}
\foreach \i in {2,5,6} {
\node at (\i,5.3) {$\dots$};
}
\node at (-2,-1) {$x_N$};
\node at (-2,0) {$x_1$};
\node at (-2,1) {$\vdots$};
\node at (-2,2) {$\vdots$};
\node at (-2,3) {$\vdots$};
\node at (-2.2,4) {$x_{N-1}$};
\foreach \i in {-1,0,4} {
\draw[->,thick] (-1.7,\i) -- (-0.5,\i);
}
\draw[lgray,line width=1.5pt,rounded corners] (8,-2) -- (1,-2) -- (1,-1);
\node at (8.25,-2) {$0$};
\draw[->,thick] (9.5,-2) -- (8.5,-2);
\node at (9.75,-2) {$z_1$};
\begin{scope}[shift = {(0,-1)}]
\foreach \i in {2,7} {
\node at (\i,-2.2) {$0$};
}
\foreach \i in {3,5,6} {
\node at (\i,-2.2) {$\dots$};
}
\node at (2,-3.5) {$z_2$};
\node at (3,-3.5) {$\dots$};
\foreach \i in {5,6} {
\node at (\i,-3.5) {$\dots$};
}
\node at (7,-3.5) {$z_{N+M}$};
\foreach \i in {2,7} {
\draw[->,thick] (\i,-3.25) -- (\i,-2.5);
}
\foreach \i in {1,2,3,4,5,6,7} {
\node[scale = 1.5] at (\i,-0.5) {$\ast$};
}
\end{scope}
\end{tikzpicture}
=
\begin{tikzpicture}[scale = 0.5,every node/.style={scale=0.6},baseline=0.7cm]
\begin{scope}[shift = {(1.5,0.35)}]
\foreach \i in {-2,-1,0,1,2,3} {
\draw[lgray,line width=1.5pt,->] (12,\i) -- (20,\i);
}
\foreach \i in {13,14,15,16,17,18} {
\draw[lgray,line width=1.5pt] (\i,-3) -- (\i,4);
}
\draw[lgray,line width=1.5pt,->] (19,-3) -- (19,4);
\foreach \i in {13,14,18,19} {
\draw[->,thick] (\i,-4) -- (\i,-3.5);
}
\draw[red,dashed] (11.5,4.75) -- (20,4.75);
\begin{scope}[shift = {(0,2.75)}]
\draw[lgray,line width=1.5pt,->,rounded corners] (19,3) -- (19,4) -- (12,4);
\node at (13,5.25) {$2$};
\node at (13.55,5.25) {$\dots$};
\node at (14,5.25) {$2$};
\foreach \i in {17} {
\node at (\i,5.25) {$\dots$};
}
\node at (15,5.25) {$1$};
\node at (16,5.25) {$\dots$};
\node at (18,5.25) {$1$};
\node at (11.5,4) {$2$};
\foreach \i in {13,14,15,16,17,18} {
\draw[lgray,line width=1.5pt,->] (\i,3) -- (\i,5);
}
\node at (13,2.5) {$U_2$};
\node at (14,2.5) {$U_3$};
\foreach \i in {15,16,17} {
\node at (\i,2.5) {$\dots$};
}
\node at (18,2.5) {$U_{N+M}$};
\node at (19,2.5) {$U_1$};
\end{scope}
\begin{scope}[shift = {(0,1.75)}]
\node at (13,2.5) {$U_2$};
\node at (14,2.5) {$U_3$};
\foreach \i in {15,16,17} {
\node at (\i,2.5) {$\dots$};
}
\node at (18,2.5) {$U_{N+M}$};
\node at (19,2.5) {$U_1$};
\end{scope}
\node at (13,-4.3) {$z_{2}$};
\node at (14,-4.3) {$z_{3}$};
\foreach \i in {15,16,17} {
\node at (\i,-4.3) {$\dots$};
}
\node at (18,-4.3) {$z_{N+M}$};
\node at (19,-4.3) {$z_1$};
\foreach \i in {-2,-1,3} {
\node at (20.5,\i) {$0$};
}
\foreach \i in {0,1,2} {
\node at (20.5,\i) {$\vdots$};
}
\foreach \i in {-2,-1} {
\node at (11.5,\i) {$2$};
}
\node at (10,-2) {$x_N$};
\node at (10,-1) {$x_1$};
\node at (9.8,3) {$x_{N-1}$};
\foreach \i in {-2,-1,3}
{\draw[->,thick] (10.4,\i) -- (11.3,\i);}
\foreach \i in {0,1,2} {
\node at (10,\i) {$\vdots$};
}
\foreach \i in {0,1,2} {
\node at (11.5,\i) {$\vdots$};
}
\node at (11.5,3) {$2$};
\foreach \i in {13,14,18,19} {
\node at (\i,-3.2) {$0$};
}
\end{scope}
\end{tikzpicture} 
\end{equation}
The first equality in \eqref{eqpulllines} is due to the fact that all edges marked $*$ are forced to assume the state $0$, which then completely freezes the bottom row of the lattice, and one recovers $F_T(x_1,\dots,x_N;\bm{z})$ (where we have again exploited the symmetry in $(x_1,\dots,x_N)$ to transfer $x_N$ to the base of the lattice). The second equality in \eqref{eqpulllines} follows by repeated application of the Yang--Baxter equation \eqref{eq19ybe1}.

Expanding the rightmost partition function in \eqref{eqpulllines} along the edges marked with the red dashed line, we read off the algebraic equality
\begin{align}
\label{eq:FT-U}
F_T(x_1,\dots,x_N;\bm{z})
=
\sum_{(U_1,\dots,U_{N+M})}
\Psi_U(z_1)
F_{\omega\cdot U}(x_1,\dots,x_N;\omega\cdot\bm{z})
\end{align}
with the sum taken over $2$-strings $(U_1,\dots,U_{N+M})$ such that $|U|=2N$, and where $\omega \cdot U = (U_2,\dots,U_{N+M},U_1)$ and $\omega\cdot\bm{z} = (z_2,\dots,z_{N+M},z_1)$ denote cyclic permutations of the function index and alphabet, respectively. The function $F_{\omega\cdot U}(x_1,\dots,x_N;\omega\cdot\bm{z})$ appearing in the summand of \eqref{eq:FT-U} takes the form
\begin{equation}
\label{eqfreez2}
\raisebox{-25mm}{\begin{tikzpicture}[scale = 0.5,every node/.style={scale=0.6}]
\foreach \i in {-1,0,1,2,3,4,5,6} {
    \draw[->,lgray,line width=1.5pt] (\i,0) -- (\i,7);}
    \foreach \i in {1,2,3,4,5,6} {
    \draw[->,lgray,line width=1.5pt] (-2,\i) -- (7,\i);}
    \foreach \i in {1,2,6} {
    \node at (7.3,\i) {$0$};
    }
    \foreach \i in {3,4,5} {
    \node at (7.3,\i) {$\vdots$};
    }
    \node[left,scale = 1.25] at (-3.8,1) {$x_N$};
    \node[left,scale = 1.25] at (-3.8,2) {$x_1$};
    \node[left] at (-3.8,3) {$\vdots$};
    \node[left] at (-3.8,4) {$\vdots$};
    \node[left] at (-3.8,5) {$\vdots$};
    \node[left,scale = 1.25] at (-3.8,6) {$x_{N-1}$};
    \draw[thick,->] (-3.6,1) -- (-2.6,1);
    \draw[thick,->] (-3.6,2) -- (-2.6,2);
    \draw[thick,->] (-3.6,6) -- (-2.6,6);
    \foreach \i in {1,2,6} {
    \node at (-2.4,\i) {$2$};
    }
    \node at (-2.4,3) {$\vdots$};
    \node at (-2.4,4) {$\vdots$};
    \node at (-2.4,5) {$\vdots$};
    \node at (-1,7.3) {$U_2$};
    \node at (0,7.3) {$U_3$};
    \foreach \i in {1,2,3,4} {
    \node at (\i,7.3) {$\dots$};
    }
    \node at (5,7.3) {$U_{N+M}$};
    \node at (6,7.3) {$U_1$};
    \node at (-1,-1.7) {$z_2$};
    \node at (0,-1.7) {$z_3$};
    \foreach \i in {1,2,3,4} {
    \node at (\i,-1.7) {\dots};
    }
    \node at (5,-1.7) {$z_{N+M}$};
    \node at (6,-1.7) {$z_{1}$};
    \foreach \i in {-1,0,6} {
    \node at (\i,-0.2) {$0$};
    }
    \foreach \i in {1,2,3,4,5} {
    \node at (\i,-0.2) {$\dots$};
    }
    \foreach \i in {-1,0,1,2,3,4,5,6} {
    \draw[thick,->] (\i,-1.25) -- (\i,-0.5);
    }
    \node[scale = 1.5] at (-10,3.5) {$F_{\omega \cdot U}(x_1,\dots,x_N;\omega\cdot\bm{z})=$};
\end{tikzpicture}}
\end{equation}
while the coefficients $\Psi_U(z_1)$ are defined as the following one-row partition functions:
\begin{equation}
\label{eq:Psi-one-row}
\Psi_{U}(z_1)
\equiv
\Psi_{(U_1,U_2,\dots,U_{N+M})}(z_1)
=
\raisebox{-24mm}{
\begin{tikzpicture}[scale=0.8,baseline=-2.45cm]
\node[left] at (0.8,0) {$2$};
\draw[->,thick] (9.5,0) -- (9,0);
\node at (10,0) {$z_1$}; 
\draw[->,lgray,line width=1.5pt] (8,0) -- (1,0);
\foreach \i in {2,3,4,5,6,7} {
\draw[->,lgray,line width=1.5pt] (\i,-1) -- (\i,1);
}
\foreach \i in {4,5,6} {
\node at (\i,-2) {$\dots$};
}
\node at (2,-2.3) {$z_2$};
\node at (3,-2.3) {$z_3$};
\node at (7,-2.3) {$z_{N+M}$};
\node at (2,-1.2) {$U_2$};
\node at (3,-1.2) {$U_3$};
\foreach \i in {4,5,6} {
\node at (\i,-1.2) {$\dots$};
}
\node at (7,-1.2) {$U_{N+M}$};
\foreach \i in {2,3,7} {
\draw[->,thick
] (\i,-2) -- (\i,-1.5);
}
\node at (2,1.2) {$2$};
\node at (3,1.2) {$\dots$};
\node at (4,1.2) {$2$};
\node at (5,1.2) {$1$};
\node at (6,1.2) {$\dots$};
\node at (7,1.2) {$1$};
\node at (8.5,0) {$U_1$};
\end{tikzpicture}}
\end{equation}
Now consider the effect of computing the residue of $F_{\omega \cdot U}(x_1,\dots,x_N;\omega\cdot\bm{z})$ at the point $z_1=q^{-3}x_N$. The only contribution to this residue comes from the bottom-right vertex in \eqref{eqfreez2}. In particular, we require this vertex to assume the form $W_{z_1/x_N}(0,2;2,0)$, since both $W_{z_1/x_N}(0,1;1,0)$ and $W_{z_1/x_N}(0,0;0,0)$ have vanishing residue at $z_1=q^{-3}x_N$. We then observe a freezing of the first row and final column of \eqref{eqfreez2}:
\begin{align}
\label{eq:res-Fomega}
{\rm Res}\Big\{
F_{\omega \cdot U}(x_1,\dots,x_N;\omega\cdot\bm{z})
\Big\}_{z_1=q^{-3}x_N}
=
\bm{1}_{U_1=2}
\times
\begin{tikzpicture}[scale = 0.7,every node/.style={scale=0.7},baseline=2.5cm]
\foreach \i in {-1,0,1,2,3,4,5,6} {
    \draw[->,lgray,line width=1.5pt] (\i,0) -- (\i,7);}
    \foreach \i in {1,2,3,4,5,6} {
    \draw[->,lgray,line width=1.5pt] (-2,\i) -- (7,\i);}
    \foreach \i in {1,2,6} {
    \node at (7.3,\i) {$0$};
    }
    \foreach \i in {3,4,5} {
    \node at (7.3,\i) {$\vdots$};
    }
    \node[left,scale = 1.25] at (-3.8,1) {$x_N$};
    \node[left,scale = 1.25] at (-3.8,2) {$x_1$};
    \node[left] at (-3.8,3) {$\vdots$};
    \node[left] at (-3.8,4) {$\vdots$};
    \node[left] at (-3.8,5) {$\vdots$};
    \node[left,scale = 1.25] at (-3.8,6) {$x_{N-1}$};
    \draw[thick,->] (-3.6,1) -- (-2.6,1);
    \draw[thick,->] (-3.6,2) -- (-2.6,2);
    \draw[thick,->] (-3.6,6) -- (-2.6,6);
    \foreach \i in {1,2,6} {
    \node at (-2.4,\i) {$2$};
    }
    \node at (-2.4,3) {$\vdots$};
    \node at (-2.4,4) {$\vdots$};
    \node at (-2.4,5) {$\vdots$};
    \node at (-1,7.3) {$U_2$};
    \node at (0,7.3) {$U_3$};
    \foreach \i in {1,2,3,4} {
    \node at (\i,7.3) {$\dots$};
    }
    \node at (5,7.3) {$U_{N+M}$};
    \node at (6,7.3) {$2$};
    \draw[->,thick] (0,-1) -- (0,-0.5);
    \draw[->,thick] (-1,-1) -- (-1,-0.5);
    \draw[->,thick] (5,-1) -- (5,-0.5);
    \draw[->,thick] (6,-1) -- (6,-0.5);
    \node at (-1,-1.2) {$z_2$};
    \node at (0,-1.2) {$z_3$};
    \node at (5,-1.2) {$z_{N+M}$};
    \node at (6.3,-1.2) {$q^{-3}x_N$};
    \foreach \i in {-1,0,5,6} {
    \node at (\i,-0.2) {$0$};
    }
    \foreach \i in {1,2,3,4} {
    \node at (\i,-0.2) {$\dots$};
    }
    \foreach \i in {2,3,4,5,6} {
    \node at (5.5,\i) {$0$};
    }
    \node at (5.5,1) {$2$};
    \foreach \i in {-1,0,1,2,3,4,5} {
    \node at (\i,1.5) {$0$};
    }
    \node at (6,1.5) {$2$};
    \node at (6,1) [circle,fill,inner sep=1.5pt,red] {};
\end{tikzpicture}
\end{align}
where it is important to note that this residue vanishes if $U_1 \not= 2$. Next, we explicitly compute the weights of the frozen regions in \eqref{eq:res-Fomega}; namely,
\begin{align*}
{\rm Res}\Big\{ & W_{z_1/x_N}(0,2;2,0) \Big\}_{z_1=q^{-3}x_N}
=
-\frac{(1-q^2)(1-q^3)}{q^6(1-q)}x_N,
\\
\\
\prod_{i=1}^{N-1}
&
W_{z_1/x_i}(2,0;2,0)
\Big|_{z_1=q^{-3}x_N}
=
\prod_{i=1}^{N-1}
\frac{(x_N-q^2 x_i)(x_N-q^3 x_i)}{(x_N-x_i)(x_N-q x_i)},
\\
\\
\prod_{j=2}^{N+M}
& W_{z_j/x_N}(0,2;0,2)
=
\prod_{j=2}^{N+M}
\frac{(x_N-z_j)(x_N-q z_j)}{(x_N-q^2 z_j)(x_N-q^3 z_j)}.
\end{align*}
Making use of these evaluations, \eqref{eq:res-Fomega} becomes
\begin{multline}
\label{eq:res-Fomega2}
{\rm Res}\Big\{
F_{\omega \cdot U}(x_1,\dots,x_N;\omega\cdot\bm{z})
\Big\}_{z_1=q^{-3}x_N}
=
-\frac{(1-q^2)(1-q^3)}{q^6(1-q)}x_N
\\
\times
\prod_{i=1}^{N-1}
\frac{(x_N-q^2 x_i)(x_N-q^3 x_i)}{(x_N-x_i)(x_N-q x_i)}
\prod_{j=2}^{N+M}
\frac{(x_N-z_j)(x_N-q z_j)}{(x_N-q^2 z_j)(x_N-q^3 z_j)}
\cdot
\bm{1}_{U_1=2}
\cdot
F_{(U_2,\dots,U_{N+M})}(x_1,\dots,x_{N-1};\widehat{\bm{z}}_1).
\end{multline}
One may then directly compute the residue of $F_{T}(x_1,\dots,x_N;\bm{z})$ at the point $z_1=q^{-3} x_N$, using \eqref{eq:FT-U}, \eqref{eq:res-Fomega2}, and the fact that the one-row partition functions \eqref{cu-def} and \eqref{eq:Psi-one-row} are related via the identity $\Psi_{(2,U_2,\dots,U_{N+M})}(q^{-3}x_N) = \Phi_{(U_2,\dots,U_{N+M})}$. This completes the proof of \eqref{eqdomw3}.

\medskip

\item This property requires no further explanation.

\end{enumerate}
\end{proof}

\subsection{Proof of uniqueness}
\label{ssec:unique}

\begin{thm}
\label{thm:unique}
The functions $F_S(x_1,\dots,x_N;\bm{z})$ are uniquely characterized by the exchange relations of Propositions \ref{prop:IK-exch} and \ref{prop:11}, and by the properties listed in Proposition \ref{prop:IK-properties}.
\end{thm}

\begin{proof}
Suppose that $Z_S(x_1,\dots,x_N;\bm{z})$, $S \in \mathfrak{s}(2)$, $|S|=2N$ is a family of rational functions that satisfies all of the conditions of Propositions \ref{prop:IK-exch}, \ref{prop:11} and \ref{prop:IK-properties}. Our task is to show that, necessarily,
\begin{align}
\label{ZS=FS}
Z_S(x_1,\dots,x_N;\bm{z})
=
F_S(x_1,\dots,x_N;\bm{z}).
\end{align}
We begin with the observation that for any $2$-string $S$, one is able to write
\begin{align}
\label{Z-F-order}
Z_S(x_1,\dots,x_N;\bm{z})
-
F_S(x_1,\dots,x_N;\bm{z})
=
\sum_{T}
\sum_{\sigma}
d_{T,\sigma}^S 
\Big(
Z_T(x_1,\dots,x_N;\sigma \cdot \bm{z})
-
F_T(x_1,\dots,x_N;\sigma \cdot \bm{z})
\Big),
\end{align}
where the first sum is over $2$-strings of the form 
$T=(2^{N-M},1^{2M})$ with $N>M$ and the second sum is over permutations of the secondary alphabet $\bm{z}$, while $d_{T,\sigma}^S$ are appropriate coefficients. Indeed, as explained in Remark \ref{rem:2string-order}, Propositions \ref{prop:IK-exch} and \ref{prop:11} collectively ensure that one may write
\begin{align}
\label{F-order}
F_S(x_1,\dots,x_N;\bm{z})
=
\sum_{T}
\sum_{\sigma}
d_{T,\sigma}^{S} 
F_T(x_1,\dots,x_N;\sigma \cdot \bm{z})
\end{align}
for appropriate coefficients $d_{T,\sigma}^{S}$. By virtue of the assumption that $Z_S(x_1,\dots,x_N;\bm{z})$ satisfies the same exchange relations, it must then be possible to write an analogue of equation \eqref{F-order} for $Z_S(x_1,\dots,x_N;\bm{z})$ with the {\it same}\/ expansion coefficients; the claim \eqref{Z-F-order} follows immediately.

Thanks to equation \eqref{Z-F-order}, the task of proving \eqref{ZS=FS} is reduced to showing that
\begin{align}
\label{ZT=FT}
Z_T(x_1,\dots,x_N;\bm{z})
=
F_T(x_1,\dots,x_N;\bm{z}),
\qquad
\forall\ T=(2^{N-M},1^{2M})\ 
\text{with}\ N>M.
\end{align}
The rest of the proof proceeds by induction on $N$; namely, the size of the primary alphabet. Let $\mathcal{P}_N$ denote the proposition that \eqref{ZT=FT} holds. 

We begin with a verification of $\mathcal{P}_1$. The only ordered $2$-string that begins with a $2$ and has total weight $2$ is the string $T=(2)$. Property (6) of Proposition \ref{prop:IK-properties} provides an explicit evaluation of $F_T(x_1;\bm{z})$ and $Z_T(x_1;\bm{z})$ when $T=(2)$, so they are necessarily equal. Hence $\mathcal{P}_1$ is true.

Now we proceed with the inductive hypothesis that $\mathcal{P}_{N-1}$ is true for some $N \geq 2$, and consider the function
\begin{align}
\label{eq:delta-T}
\Delta_T=
\prod_{i=1}^{N}\prod_{j=1}^{N+M}
(x_i-q^2 z_j)(x_i-q^3 z_j)
\Big(
Z_T(x_1,\dots,x_N;\bm{z})-F_T(x_1,\dots,x_N;\bm{z})
\Big).
\end{align}
By property (2) of Proposition \ref{prop:IK-properties}, $\Delta_T$ is a polynomial in $z_1$ of degree $2N$. The symmetry property (1), together with properties (3), (4), (5) specifies $\Delta_T$ at $2N+1$ values of $z_1$. In particular, property (3) becomes $\Delta_T \Big|_{z_1=0}=0$. Property (4) becomes
\begin{align*}
\Delta_T
\Big|_{z_1 = x_k} = 
\prod_{j = 2}^{N+M}(x_k-q^2 z_j)(x_k -q^3z_{j})
\prod_{i = 1}^{N} (x_i - q^2 x_k)(x_i - q^3 x_k)
\Delta_{T\backslash 2}(x_1,\dots,\widehat{x}_k,\dots,x_N;
\widehat{\bm{z}}_1)
=
0
\end{align*}
for all $1 \leq k \leq N$, where the final equality with $0$ is due to the assumption that $\mathcal{P}_{N-1}$ holds. Finally, property (5) becomes
\begin{multline}
\label{eq:delta-eval}
\Delta_{T}
\Big|_{z_1 = q^{-3}x_k} 
=
-\frac{1}{q^{N+3}}
\prod_{j =2}^{N+M} (x_k - z_j)(x_k - qz_j)
\prod_{i=1}^{N} (x_k-q^2 x_i)(x_k-q^3 x_i)
\\
\times
\sum_{U} \Phi_{U}
\Delta_U(x_1,\dots,\widehat{x}_k,\dots,x_N;\widehat{\bm{z}}_1)
=0
\end{multline}
for all $1 \leq k \leq N$, with summation over $U$ and coefficients $\Phi_U$ as described in Proposition \ref{prop:IK-properties}. After reordering the $2$-strings appearing on the right hand side of \eqref{eq:delta-eval}, similarly to \eqref{Z-F-order}, we recover a sum over functions of the form \eqref{eq:delta-T} with non-increasing index and a primary alphabet of size $N-1$. This is again equal to $0$ due to the inductive hypothesis.

Collectively, these properties show that $\Delta_T$ vanishes at $2N+1$ values of $z_1$; it is therefore identically zero for arbitrary $z_1$. Hence $\mathcal{P}_N$ is true, and the proof is complete by induction on $N$.

\end{proof}

\section{Twisted columns presentation and symmetrization formula}
\label{sec:19-sym}

In this section we turn to the task of writing an explicit summation formula for the functions \eqref{eqspsymIK2}, by obtaining a solution to the set of properties that uniquely determine them (see Section \ref{ssec:unique}). To achieve this goal, we shall make use of a family of {\it twisted column operators} $\Gamma_k(z)$ (defined in Sections \ref{ssec:aux} and \ref{ssec:IK-twist}) which satisfy exchange relations from the Yang--Baxter algebra of the Izergin--Korepin model; this is the content of Theorem \ref{thm:IK-YB-twist}. Aided by these operators, we obtain an algebraic reformulation of the functions $F_S(x_1,\dots,x_N;\bm{z})$ in Section \ref{ssec:IK-twist-rep} which ultimately leads to an explicit symmetrization identity for them, in Section \ref{ssec:IK-sym-formula}.

Our construction of the operators $\Gamma_k(z)$ was strongly inspired by a useful procedure in the six-vertex model, where one applies {\it Drinfeld twists} to columns of six-vertex partition functions \cite{SantosMaillet00,WheelerZinn-Justin16,BorodinWheeler18}. Under such a transformation, it is well-known that column operators become manifestly symmetric in their spectral parameters $(x_1,\dots,x_N)$, providing a direct passage to symmetrization identities. In contrast with these works on the six-vertex model, however, little seems to be known about the construction of factorizing twists in the setting of the Izergin--Korepin model. This means that the operators \eqref{gam0}--\eqref{gam2} were in fact {\it not}\/ obtained via a change of basis applied to partition function columns in the Izergin--Korepin model, but found by computer experimentation. As such, we have to work a good deal harder to prove any results concerning them, in particular Theorems \ref{thm:IK-YB-twist} and \ref{thm:twisted-column-F}; these would be immediate if one knew how to recover $\Gamma_k(z)$ via Drinfeld twists.

\subsection{Auxiliary definitions}
\label{ssec:aux}

Before stating our expressions for the twisted column operators, we require a number of preliminary definitions. We shall make use of four diagonal matrices, defined as follows:
\begin{align}
\label{d-mat}
d_i(z)
= 
\begin{pmatrix}
1 & 0 & 0 \\
0 & \dfrac{x_i - z}{x_i - q^2z}& 0\\
0 & 0 & \dfrac{(x_i - z)(x_i - qz)}{(x_i - q^2z)(x_i - q^3z)}
\end{pmatrix}_i
\end{align}

\begin{align}
\label{dcirc-mat}
d_i^{\circ}(z) 
= 
\begin{pmatrix}
1 & 0 & 0\\
0 & \dfrac{(qx_i-z)(x_i - q^2 z)}{(1+q)(x_i-z)}& 0\\
0 & 0 & \dfrac{x_i - q^3 z}{x_i -q z}
\end{pmatrix}_i
\end{align}

\begin{align}
\label{dbull-mat}
d_i^{\bullet}(z)
=
\begin{pmatrix}
1 & 0 & 0\\
0 & \dfrac{1+q}{qx_i - z} & 0\\
0 & 0 & \dfrac{x_i - q^2 z}{x_i - z}
\end{pmatrix}_i
\end{align}

\begin{align}
\label{dboth-mat}
d_{i}^{\bullet\circ}(z) 
= 
\begin{pmatrix}
1 & 0 & 0\\
0 & \dfrac{x_i - q^2z}{x_i - z} & 0\\
0 & 0 & \dfrac{(x_i - q^2z)(x_i - q^3z)}{(x_i - z)(x_i - qz)}
\end{pmatrix}_i
\end{align}
where each matrix acts in a local space $V_i \cong \mathbb{C}^3$, as indicated by the use of the subscript $i \in \{1,\dots,N\}$. Each matrix also depends on two spectral parameters $z,x_i$, but we suppress the dependence on the second of these for compactness of notation. The matrices \eqref{d-mat}--\eqref{dboth-mat} satisfy the following relations:
\begin{align}
\label{eq:d-factor}
d_i^{\bullet}(z)
d_i^{\circ}(z)
=
d_i^{\bullet\circ}(z)
=
d_i(z)^{-1}.
\end{align}
We shall also require three types of non-diagonal matrices:
\begin{align}
\label{ecirc-mat}
e_i^{\circ}(z) = 
\dfrac{(1-q^2)z}{x_i - q^2z}
\begin{pmatrix}
0 & 0 & 0\\
1 & 0 & 0\\
0 & 0 & 0 
\end{pmatrix}_i
=
\begin{tikzpicture}[scale=0.5,baseline=-0.1cm]
\draw[lgray,line width=1.5pt,->] (14,0) -- (16,0);
\draw[lgray,line width=1.5pt,->] (15,-1) -- (15,1); 
\node at (13.8,0) {\tiny $1$};
\node at (16.2,0) {\tiny $0$};
\node at (15,-1.2) (A1) {\tiny $0$};
\node at (15,1.2) {\tiny $1$};
\node at (13.2,0) {$x_i$};
\node at (15,-2.35) (A) {$z$};
\draw[->,thick] (12.3,0) -- (12.8,0);
\draw[->,thick] (15,-2.1) -- (A1);
\end{tikzpicture}
\begin{pmatrix}
0 & 0 & 0\\
1 & 0 & 0\\
0 & 0 & 0 
\end{pmatrix}_i
\end{align}

\begin{align}
\label{ebull-mat}
e_i^{\bullet}(z) = 
\dfrac{q(1-q)^2(1+q)(x_i-z)z}{(x_i-q^2z)(x_i - q^3z)}
\begin{pmatrix}
0 & 0 & 0\\
0 & 0 & 0\\
0 & 1 & 0 
\end{pmatrix}_i
=
\begin{tikzpicture}[scale=0.5,baseline=-0.1cm]
\draw[lgray,line width=1.5pt,->] (14,0) -- (16,0);
\draw[lgray,line width=1.5pt,->] (15,-1) -- (15,1); 
\node at (13.8,0) {\tiny $2$};
\node at (16.2,0) {\tiny $1$};
\node at (15,-1.2) (A1) {\tiny $0$};
\node at (15,1.2) {\tiny $1$};
\node at (13.2,0) {$x_i$};
\node at (15,-2.35) (A) {$z$};
\draw[->,thick] (12.3,0) -- (12.8,0);
\draw[->,thick] (15,-2.1) -- (A1);
\end{tikzpicture}
\begin{pmatrix}
0 & 0 & 0\\
0 & 0 & 0\\
0 & 1 & 0 
\end{pmatrix}_i
\end{align}

\begin{align}
\label{eboth-mat}
e_i^{\bullet\circ}(z) = 
\dfrac{(1-q^2)(x_i+q^2x_i-q^2z-q^3z)z}
{(x_i-q^2z)(x_i-q^3z)}
\begin{pmatrix}
0 & 0 & 0\\
0 & 0 & 0\\
1 & 0 & 0
\end{pmatrix}_i
=
\begin{tikzpicture}[scale=0.5,baseline=-0.1cm]
\draw[lgray,line width=1.5pt,->] (14,0) -- (16,0);
\draw[lgray,line width=1.5pt,->] (15,-1) -- (15,1); 
\node at (13.8,0) {\tiny $2$};
\node at (16.2,0) {\tiny $0$};
\node at (15,-1.2) (A1) {\tiny $0$};
\node at (15,1.2) {\tiny $2$};
\node at (13.2,0) {$x_i$};
\node at (15,-2.35) (A) {$z$};
\draw[->,thick] (12.3,0) -- (12.8,0);
\draw[->,thick] (15,-2.1) -- (A1);
\end{tikzpicture}
\begin{pmatrix}
0 & 0 & 0\\
0 & 0 & 0\\
1 & 0 & 0
\end{pmatrix}_i
\end{align}
whose action is again in the local space $V_i$, as indicated by the subscript $i$. As before, each matrix depends on two spectral parameters $z,x_i$, but only the explicit dependence on $z$ is recorded in our notation. Note that the rational function that appears in each of \eqref{ecirc-mat}--\eqref{eboth-mat} corresponds precisely with the vertex weights $W_{z/x_i}(0,1;1,0)$, $W_{z/x_i}(0,2;1,1)$, $W_{z/x_i}(0,2;2,0)$ from Figure \ref{fig:19weights}.

Finally, we shall also make use of four types of $9 \times 9$ matrices, obtained by taking tensor products of the operators \eqref{ecirc-mat}--\eqref{eboth-mat}: 
\begin{align}
\label{ecc}
e_{ij}^{\circ\circ}(z)
=
\dfrac{x_ix_j+q^2x_ix_j-q^2x_iz-q^2x_jz}
{(1-q^2)z}
e_i^{\circ}(z) e_j^{\circ}(z)
\end{align}

\begin{align}
\label{ecb}
e_{ij}^{\bullet\circ}(z)
=
e_{ji}^{\circ\bullet}(z)
=
\dfrac{x_ix_j+q^2x_ix_j-q^2x_i z-q^3x_j z}
{q(1-q)(x_i-x_j)z}
e_i^{\bullet}(z)
e_j^{\circ}(z)
\end{align}

\begin{align}
\label{ebb}
e_{ij}^{\bullet\bullet}(z)
=
\dfrac{(1+q)(x_ix_j+q^2x_ix_j-q^3x_iz-q^3x_jz)}{q(1-q)(x_i-qx_j)(x_j-qx_i)z}
e_i^{\bullet}(z) e_j^{\bullet}(z)
\end{align}
for any pair of indices $i,j \in \{1,\dots,N\}$ such that $i\not=j$. Each of these matrices acts in $V_i \otimes V_j$, and depends on three spectral parameters $z,x_i,x_j$.

\subsection{Twisted column operators and Yang--Baxter algebra}
\label{ssec:IK-twist}

Using the elementary matrices from the previous subsection, we now define the operators $\Gamma_k(z)$, $k \in \{0,1,2\}$ that will be the main focus of this section. The operators $\Gamma_k(z)$ are $3^N \times 3^N$ matrices that act in a tensor product $V_1\otimes \cdots \otimes V_N$ of local spaces. They depend explicitly on a spectral parameter $z$, as well as implicitly on $N$ spectral parameters $(x_1,\dots,x_N)$. We shall refer to $\Gamma_k(z)$ as {\it twisted column operators}, as their structure closely resembles that of monodromy matrix operators under conjugation by a Drinfeld twist; see \cite{SantosMaillet00} for more on this subject, in the setting of the six-vertex model.\footnote{It is important to note, however, that we do not make explicit use of Drinfeld twists or the so-called $F$-basis in the present work.}

\begin{defn}[Twisted column operators]
Fix an integer $N$, a parameter $z$ and an alphabet of parameters $(x_1,\dots,x_N)$. The twisted column operators are defined as follows:
\begin{align}
\label{gam0}
\Gamma_0(z)
=
\prod_{i=1}^{N} d_i(z),
\end{align}
which is an $N$-fold tensor product of the diagonal matrices \eqref{d-mat};
\begin{align}
\label{gam1}
\Gamma_1(z)
=
\sum_{u \in \{\circ,\bullet\}}
\sum_{i=1}^{N}
e^{u}_i(z)
\prod_{j\not=i}^{N} d_j(z) d^{u}_j(x_i),
\end{align}
where the sum is over matrices \eqref{ecirc-mat}, \eqref{ebull-mat} acting on a single space, dressed by an $(N-1)$-fold tensor product of diagonal matrices \eqref{d-mat}--\eqref{dbull-mat}; and
\begin{align}
\label{gam2}
\Gamma_2(z)
=
\sum_{i=1}^{N}
e^{\bullet\circ}_i(z)
\prod_{j\not=i}^{N} 
d_j(z)
d^{\bullet\circ}_j(x_i) 
+
\sum_{u,v \in \{\circ,\bullet\}}
\sum_{1 \leq i<j \leq N}
e^{uv}_{ij}(z)
\prod_{k\not=i,j}^{N} 
d_k(z)
d^{u}_k(x_i) d^{v}_k(x_j) 
\end{align}
where the first sum is over matrices \eqref{eboth-mat} acting on a single space, dressed by an $(N-1)$-fold tensor product of diagonal matrices \eqref{d-mat}, \eqref{dboth-mat}, while the second sum is over matrices \eqref{ecc}--\eqref{ebb} acting on a pair of spaces, dressed appropriately by diagonal matrices on the remaining $(N-2)$ spaces.
\end{defn}

The twisted column operators \eqref{gam0}--\eqref{gam2} have two key properties that will be essential throughout the remainder of this section:

\begin{prop}
\label{prop:gam-sym}
For each $k \in \{0,1,2\}$, we let $\llangle i_1,\dots,i_N \| \Gamma_k(z;x_1,\dots,x_N) \| j_1,\dots,j_N \rrangle$ denote the entries of the matrix $\Gamma_k(z) \equiv \Gamma_k(z;x_1,\dots,x_N) \in {\rm End}(V_1 \otimes \cdots \otimes V_N)$. Then for any permutation $\sigma \in \mathfrak{S}_N$, there holds
\begin{align*}
\llangle i_{\sigma(1)},\dots,i_{\sigma(N)} \| \Gamma_k\left(
z;x_{\sigma(1)},\dots,x_{\sigma(N)} \right) \| j_{\sigma(1)},\dots,j_{\sigma(N)} \rrangle
=
\llangle i_1,\dots,i_N \| \Gamma_k(z;x_1,\dots,x_N) \| j_1,\dots,j_N \rrangle.
\end{align*}
In other words, $\Gamma_k(z)$ is invariant under simultaneous permutation of the alphabet $(x_1,\dots,x_N)$ and the spaces $V_1 \otimes \cdots \otimes V_N$ on which it acts.
\end{prop}

\begin{proof}
The simultaneous permutation of spaces and variables is achieved by replacing all subscripts in \eqref{gam0}--\eqref{gam2} by their image under $\sigma$; namely,
\begin{align*}
\Gamma_0^{\sigma}(z)
=
\prod_{i=1}^{N} d_{\sigma(i)}(z),
\end{align*}

\begin{align*}
\Gamma_1^{\sigma}(z)
=
\sum_{u \in \{\circ,\bullet\}}
\sum_{i=1}^{N}
e^{u}_{\sigma(i)}(z)
\prod_{j\not=i}^{N} d_{\sigma(j)}(z) d^{u}_{\sigma(j)}(x_{\sigma(i)}),
\end{align*}

\begin{multline*}
\Gamma_2^{\sigma}(z)
=
\sum_{i=1}^{N}
e^{\bullet\circ}_{\sigma(i)}(z)
\prod_{j\not=i}^{N} 
d_{\sigma(j)}(z)
d^{\bullet\circ}_{\sigma(j)}(x_{\sigma(i)}) 
+
\sum_{u,v \in \{\circ,\bullet\}}
\sum_{1 \leq i<j \leq N}
e^{uv}_{\sigma(i)\sigma(j)}(z)
\prod_{k\not=i,j}^{N} 
d_{\sigma(k)}(z)
d^{u}_{\sigma(k)}(x_{\sigma(i)}) 
d^{v}_{\sigma(k)}(x_{\sigma(j)}).
\qquad\qquad
\end{multline*}
Performing the change of indices $\sigma(i) \mapsto i$, $\sigma(j) \mapsto j$, $\sigma(k) \mapsto k$ and using the fact that $e^{uv}_{ab}(z) = e^{vu}_{ba}(z)$ for all $a\not=b$ and $u,v \in \{\circ,\bullet\}$, one immediately sees that $\Gamma_k^{\sigma}(z) = \Gamma_k(z)$ for $k \in \{0,1,2\}$.
\end{proof}

\begin{thm}
\label{thm:IK-YB-twist}
Fix two integers $k,\ell \in \{0,1,2\}$ and two parameters $z,y \in \mathbb{C}$. The product $\Gamma_k(z) \Gamma_{\ell}(y)$ of twisted column operators obeys the following exchange relation:
\begin{equation}
\label{gamma-com}
\Gamma_k(z)\Gamma_{\ell}(y) 
= 
\sum_{i=0}^2 \sum_{j=0}^2 
W_{z/y}(i,j;k,\ell) \Gamma_j(y)\Gamma_i(z),
\end{equation}
where $W_{z/y}(i,j;k,\ell)$ denotes the vertex weights \eqref{eqdefn19w} of the Izergin--Korepin model.
\end{thm}

\begin{proof}
Our proof will make reference to the number of factors in the tensor product $V_1 \otimes \cdots \otimes V_N$ on which the twisted column operators act. To that effect, throughout the proof we shall write $\Gamma_k^{(N)}(z) \equiv \Gamma_k(z)$, for all $k \in \{0,1,2\}$. We begin by defining the following matrix:
\begin{align}
\label{Emat=0}
E^{(k,\ell)}(x_1,\dots,x_N;z;y) 
:= 
\Gamma^{(N)}_{k}(z)
\Gamma^{(N)}_{\ell}(y) 
- 
\sum_{i = 0}^{2} 
\sum_{j = 0}^{2} 
W_{z/y}(i,j;k,\ell)
\Gamma_j^{(N)}(y)\Gamma_i^{(N)}(z)
\in 
{\rm End}(V_1\otimes\cdots\otimes V_N).
\end{align}
To show that equation \eqref{gamma-com} holds, we must demonstrate that $E^{(k,\ell)}(x_1,\dots,x_N;z;y) = 0$. To facilitate this, we make use of a certain decomposition of the matrix $E^{(k,\ell)}(x_1,\dots,x_N;z;y)$, that we now describe. 

For all $1 \leq k \leq N$, define
\begin{align*}
L_{k}^{\circ} = 
\begin{pmatrix}
 0 & 0 & 0\\
 1 & 0 & 0\\
 0 & 0 & 0	
 \end{pmatrix}_{k}, 
 \quad L_{k}^{\bullet} = 
 \begin{pmatrix}
 0 & 0 & 0\\
 0 & 0 & 0\\
 0 & 1 & 0	
 \end{pmatrix}_{k}, 
 \quad 
 L_{k}^{\bullet\circ} = \begin{pmatrix}
 0 & 0 & 0\\
 0 & 0 & 0\\
 1 & 0 & 0	
 \end{pmatrix}_{k},
\end{align*}
where the index $k$ specifies that the matrices act non-trivially on space $V_k$, and as the identity on all other spaces in the product $V_1 \otimes \cdots \otimes V_N$. We introduce the further notations
\begin{align*}
\bm{1}_{[u]} = \sum_{1 \leq i \leq N} L_i^{u},
\quad\quad
&
\bm{1}_{[u,v]} = \sum_{1 \leq i<j \leq N} L_i^{u} L_j^{v},
\\
\bm{1}_{[u,v,w]} = \sum_{1 \leq i<j<k \leq N} L_i^{u} L_j^{v} L_k^{w},
\quad\quad
&
\bm{1}_{[u,v,w,s]} = \sum_{1 \leq i<j<k<\ell \leq N} L_i^{u} L_j^{v} L_k^{w} L_{\ell}^{s},
\end{align*}
for any $u,v,w,s \in \{\circ,\bullet,\bullet\circ\}$. We refer to the subscripts $[u],[u,v],[u,v,w],[u,v,w,s]$, $u,v,w,s \in \{\circ,\bullet,\bullet\circ\}$ as {\it dot diagrams}. For any dot diagram $\mathcal{D}$ we let its {\it weight} be the total number of dots present in the diagram (where $\circ$ and $\bullet$ count as $1$, while $\bullet\circ$ counts as $2$), denoted $|\mathcal{D}|$. For any integer $m \in [1,4]$ we write $\mathcal{S}_m=\{\mathcal{D}:|\mathcal{D}|=m\}$. We also write $l(\mathcal{D})$ for the {\it length} of $\mathcal{D}$, defined as the number of components in the diagram, and $\mathcal{D}_i$ for the $i$-th component of $\mathcal{D}$. In general, one has that $l(\mathcal{D}) \leq |\mathcal{D}|$. For example, the dot diagram $\mathcal{D}=[\circ,\bullet\circ]$ satisfies $|\mathcal{D}|=3$, $l(\mathcal{D})=2$, $\mathcal{D}_1=\circ$, $\mathcal{D}_2=\bullet\circ$. 

Then for any matrix $M \in {\rm End}(V_1\otimes\cdots\otimes V_N)$ and dot diagram $\mathcal{D}$, we define
\begin{align*}
M_{\mathcal{D}} = M \odot \bm{1}_{\mathcal{D}},
\end{align*}
where $\odot$ denotes component-wise multiplication of matrices\footnote{More precisely, for any two matrices $P$ and $Q$ of equal dimensions, with entries $P_{ij}$ and $Q_{ij}$ respectively, we define the matrix $P\odot Q$ with entries $$[P\odot Q]_{ij} = P_{ij} Q_{ij}.$$}. $M \odot \bm{1}_{\mathcal{D}}$ is simply the matrix formed by selecting all entries of $M$ corresponding to the positions of ones in $\bm{1}_{\mathcal{D}}$, with zeros at all other entries. Having set up these notations, we see that
\begin{align*}
E^{(k,\ell)}(x_1,\dots,x_N;z;y)
=
\sum_{\mathcal{D} \in \mathcal{S}_{k+\ell}}
E^{(k,\ell)}_{\mathcal{D}}(x_1,\dots,x_N;z;y),
\end{align*}
which is our desired decomposition of the matrix \eqref{Emat=0}. The point of this decomposition is that for each dot diagram $\mathcal{D} \in \mathcal{S}_{k+\ell}$, the matrix $E^{(k,\ell)}_{\mathcal{D}}(x_1,\dots,x_N;z;y)$ has a nice recursive structure not present on $E^{(k,\ell)}(x_1,\dots,x_N;z;y)$ itself. In particular, we claim that for any $k,\ell \in \{0,1,2\}$ and dot diagram $\mathcal{D}$ such that $|\mathcal{D}|=k+\ell$, there holds
\begin{align}
\label{eq:E-decomp}
E^{(k,\ell)}_{\mathcal{D}}(x_1,\dots,x_N;z;y)
=
\sum_{S \subset [1,N]: {\rm Card}(S)= l(\mathcal{D})}
E^{(k,\ell)}_{\mathcal{D}}(x_S;z;y)
\prod_{j \not\in S}
d_j(z) d_j(y) 
\prod_{1 \leq i \leq l(\mathcal{D})}
d_j^{\mathcal{D}_i}(x_{S_i}),
\end{align}
where $E^{(k,\ell)}_{\mathcal{D}}(x_S;z;y) = E^{(k,\ell)}_{\mathcal{D}}(x_{S_1},\dots,x_{S_{\ell(\mathcal{D})}};z;y)$ has precisely the same definition as \eqref{Emat=0}, but acting in  the tensor product $V_{S_1} \otimes \cdots \otimes V_{S_{\ell(\mathcal{D})}}$ that contains at most four factors. It is not difficult to verify that \eqref{eq:E-decomp} holds, making use of the explicit formulas \eqref{gam0}--\eqref{gam2} for $\Gamma^{(N)}_k(z)$, $\Gamma^{(N)}_{\ell}(y)$ and the property \eqref{eq:d-factor} of the diagonal matrices \eqref{dcirc-mat}--\eqref{dboth-mat}.

The advantage of equation \eqref{eq:E-decomp} is that it expresses our original matrix $E^{(k,\ell)}_{\mathcal{D}}(x_1,\dots,x_N;z;y)$ in terms of matrices $E^{(k,\ell)}_{\mathcal{D}}(x_S;z;y)$ with a non-trivial action on at most four local spaces $V_i$. To prove that $E^{(k,\ell)}_{\mathcal{D}}(x_1,\dots,x_N;z;y)=0$, it is therefore sufficient to check it explicitly for the cases $1 \leq N \leq 4$, which may be done by computer. 
\end{proof}

\subsection{Twisted columns representation of $F_S(x_1,\dots,x_N;\bm{z})$}
\label{ssec:IK-twist-rep}

The purpose of this section is to show that the functions \eqref{eqspsymIK2} may be recovered as matrix elements of products of the twisted column operators \eqref{gam0}--\eqref{gam2}. In particular, we shall prove the following theorem:
\begin{thm}
\label{thm:twisted-column-F}
Fix $N \in \mathbb{Z}_{>0}$ and let $S \in \mathfrak{s}(2)$ be a $2$-string such that $|S|=2N$. Then one has the equality
\begin{align}
\label{F-gam=}
F_S(x_1,\dots,x_N;\bm{z})
=
\llangle 2^N\| 
\Gamma_{S_1}(z_1) \Gamma_{S_2}(z_2) \Gamma_{S_3}(z_3) 
\cdots
\|0^N \rrangle
=
\llangle 2^N\| 
\prod_{i \geq 1} \Gamma_{S_i}(z_i)
\|0^N \rrangle,
\end{align}
where the vectors $\llangle 2^N\| \in V_1^{*} \otimes \cdots \otimes V_N^{*}$ and $\|0^N \rrangle \in V_1 \otimes \cdots \otimes V_N$ are given by
\begin{align*}
\llangle 2^N\|
=
\bigotimes_{i=1}^{N}
\llangle 2 \|_i
=
\bigotimes_{i=1}^{N}
\left( \begin{array}{ccc} 0 & 0 & 1 \end{array} \right)_i
\qquad
\text{and}
\qquad
\| 0^N \rrangle
=
\bigotimes_{i=1}^{N}
\| 0 \rrangle_i
=
\bigotimes_{i=1}^{N}
\left( \begin{array}{c} 1 \\ 0 \\ 0 \end{array} \right)_i.
\end{align*}
\end{thm}

\begin{proof}
For convenience, in what follows we adopt the notation
\begin{align}
\label{gS-alg}
g_S(x_1,\dots,x_N;\bm{z})
:=
\llangle 2^N\| 
\Gamma_{S_1}(z_1) \Gamma_{S_2}(z_2) \Gamma_{S_3}(z_3) 
\cdots
\|0^N \rrangle
\end{align}
for the right hand side of \eqref{F-gam=}. Our proof rests on demonstrating that $g_S(x_1,\dots,x_N;\bm{z})$ satisfies the properties laid out in Propositions \ref{prop:IK-exch}, \ref{prop:11} and \ref{prop:IK-properties}, since these uniquely determine the family $F_S(x_1,\dots,x_N;\bm{z})$. 

\begin{prop}[Check of Proposition \ref{prop:IK-exch}]
Fix an integer $k \geq 1$ and a $2$-string $S \in \mathfrak{s}(2)$ such that $S_k < S_{k+1}$. The following relation then holds:
\begin{multline}
\label{eq:IK-exchange-g}
W_{z_{k+1}/z_k}(S_{k+1},S_k;S_{k+1},S_k)
g_S(x_1,\dots,x_N;\bm{z})
=
\\
g_{\sigma_k \cdot S}(x_1,\dots,x_N;\sigma_k \cdot \bm{z})
-
\sum_{T:S \sim_k T}
W_{z_{k+1}/z_k}(T_{k+1},T_k;S_{k+1},S_k)
g_T(x_1,\dots,x_N;\bm{z}),
\end{multline}
with the sum taken over $2$-strings $T$ such that $S,T$ are $k$-exchangeable, and where $\sigma_k\cdot S = (\dots,S_{k+1},S_k,\dots)$, $\sigma_k \cdot \bm{z} = (\dots,z_{k+1},z_k,\dots)$ denote transposition of the $k$-th and $(k+1)$-th elements of $S,\bm{z}$.
\end{prop}

\begin{proof}
Employing exchange relation \eqref{gamma-com} with $k \mapsto S_{k+1}$, $\ell \mapsto S_k$ and $z \mapsto z_{k+1}$, $y \mapsto z_k$, we have
\begin{multline*}
g_{\sigma_k \cdot S}(x_1,\dots,x_N;\sigma_k \cdot \bm{z})
=
\llangle 2^N\|
\prod_{i<k} \Gamma_{S_i}(z_i)
\cdot
\Gamma_{S_{k+1}}(z_{k+1}) \Gamma_{S_k}(z_k)
\cdot
\prod_{j > k+1} \Gamma_{S_j}(z_j)
\|0^N \rrangle
\\
=
\sum_{(T_k,T_{k+1})}
W_{z_{k+1}/z_k}(T_{k+1},T_k;S_{k+1},S_k)
\llangle 2^N\| 
\prod_{i<k} \Gamma_{S_i}(z_i)
\cdot
\Gamma_{T_k}(z_k) \Gamma_{T_{k+1}}(z_{k+1})
\cdot
\prod_{j > k+1} \Gamma_{S_j}(z_j)
\|0^N \rrangle
\\
=
W_{z_{k+1}/z_k}(S_{k+1},S_k;S_{k+1},S_k)
g_S(x_1,\dots,x_N;\bm{z})
+
\sum_{T: S\sim_k T}
W_{z_{k+1}/z_k}(T_{k+1},T_k;S_{k+1},S_k)
g_T(x_1,\dots,x_N;\bm{z}),
\end{multline*}
which is precisely \eqref{eq:IK-exchange-g}.
\end{proof}

\begin{prop}[Check of Proposition \ref{prop:11}]
Fix an integer $N \geq 1$. The following relation holds:
\begin{multline}
\label{11-02-g}
W_{z_2/z_1}(1,1;2,0)
g_{(1^{2N})}(x_1,\dots,x_N;\bm{z})
=
\\
g_{(2,0,1^{2N-2})}(x_1,\dots,x_N;\sigma_1 \cdot \bm{z})
-
W_{z_2/z_1}(0,2;2,0)
g_{(2,0,1^{2N-2})}(x_1,\dots,x_N;\bm{z})
\\
-
W_{z_2/z_1}(2,0;2,0)
\prod_{i=1}^{N}
W_{z_1/x_i}(0,2;0,2)
g_{(2,1^{2N-2})}(x_1,\dots,x_N;\widehat{\bm{z}}_1),
\end{multline}
where $\widehat{\bm{z}}_1 = (z_2,z_3,\dots)$ denotes the secondary alphabet with the omission of $z_1$.
\end{prop}

\begin{proof}
Employing exchange relation \eqref{gamma-com} with $k=2$, $\ell=0$ and $z \mapsto z_2$, $y \mapsto z_1$, we find that
\begin{multline}
\label{11-02-proof}
g_{(2,0,1^{2N-2})}(x_1,\dots,x_N;\sigma_1 \cdot \bm{z})
=
\llangle 2^N\|
\Gamma_{2}(z_2) \Gamma_{0}(z_1)
\cdot
\prod_{j=3}^{2N} \Gamma_1(z_j)
\|0^N \rrangle
\\
=
W_{z_2/z_1}(2,0;2,0)
\llangle 2^N\|
\Gamma_0(z_1) \Gamma_2(z_2)
\cdot
\prod_{j=3}^{2N} \Gamma_1(z_j)
\|0^N \rrangle
+
W_{z_2/z_1}(1,1;2,0)
\llangle 2^N\|
\Gamma_1(z_1) \Gamma_1(z_2)
\cdot
\prod_{j=3}^{2N} \Gamma_1(z_j)
\|0^N \rrangle
\\
+
W_{z_2/z_1}(0,2;2,0)
\llangle 2^N\|
\Gamma_2(z_1) \Gamma_0(z_2)
\cdot
\prod_{j=3}^{2N} \Gamma_1(z_j)
\|0^N \rrangle.
\end{multline}
The operator $\Gamma_0(z_1)$ is diagonal, and its action on $\llangle 2^N\|$ may be trivially computed:
\begin{align}
\label{gam0-leftact}
\llangle 2^N\|
\Gamma_0(z_1)
=
\prod_{i=1}^{N}
\frac{(x_i-q^2 z_1)(x_i-q^3 z_1)}
{(x_i-z_1)(x_i-q z_1)}
\llangle 2^N\|
=
\prod_{i=1}^{N}
W_{z_1/x_i}(0,2;0,2)
\llangle 2^N\|.
\end{align}
Using \eqref{gam0-leftact} in the first term on the right hand side of \eqref{11-02-proof}, and re-expressing each operator product in the form \eqref{gS-alg}, the result \eqref{11-02-g} follows.
\end{proof}

\begin{prop}[Check of Proposition \ref{prop:IK-properties}]

Fix integers $N \in \mathbb{Z}_{>0}$ and $M \in \mathbb{Z}_{\geq 0}$ such that $N>M$, and let $T=(2^{N-M},1^{2M})$ denote a non-increasing $2$-string. The function $g_T(x_1,\dots,x_N;\bm{z})$ satisfies the following properties:

\begin{enumerate}[wide,labelindent=0pt]

\item The function $g_T(x_1,\dots,x_N;\bm{z})$ is symmetric with respect to its primary alphabet $(x_1,\dots,x_N)$.
\medskip

\item Introduce the normalized function 
\begin{equation*}
\tilde{g}_{T}(x_1,\dots,x_N;\bm{z}) 
= 
\prod_{i=1}^{N}\prod_{j=1}^{N+M}
(x_i-q^2 z_j)(x_i-q^3 z_j)
g_{T}(x_1,\dots,x_N;\bm{z}).
\end{equation*}
Then $\tilde{g}_{T}$ is a polynomial in $z_1$ of degree $2N$.
\medskip

\item The function $g_{T}$ vanishes at $z_1=0$:
\begin{align*}
g_{T}(x_1,\dots,x_N;\bm{z})
\Big|_{z_1 = 0} = 0.
\end{align*}
\medskip

\item The function $g_{T}$ has the following recursion relation with respect to $z_1$:
\begin{align}
\label{eq:g-reduce-0}
g_{T}(x_1,\dots,x_N;\bm{z})
\Big|_{z_1 = x_N} = 
g_{T\backslash 2}
(x_1,\dots,x_{N-1};\widehat{\bm{z}}_1), 
\end{align}
where $\widehat{\bm{z}}_1 = (z_2,z_3,\dots)$ denotes the secondary alphabet with the omission of $z_1$ and $T\backslash2 = (2^{N-M-1},1^{2M})$.
\medskip

\item The function $g_{T}$ has a further recursion relation with respect to $z_1$:
\begin{multline}
\label{eq:res-gT}
{\rm Res}
\Big\{
g_{T}(x_1,\dots,x_N;\bm{z})
\Big\}_{z_1 = q^{-3}x_N}
=
-\frac{(1-q^2)(1-q^3)}{q^6(1-q)}
x_N
\\[5pt]
\times
\prod_{i=1}^{N-1}
\frac{(x_N-q^2 x_i)(x_N-q^3 x_i)}{(x_N-x_i)(x_N-q x_i)}
\prod_{j=2}^{N+M}
\frac{(x_N-z_j)(x_N-q z_j)}{(x_N-q^2 z_j)(x_N-q^3 z_j)}
\sum_{U} \Phi_{U} g_U(x_1,\dots,x_{N-1};\widehat{\bm{z}}_1),
\end{multline}
with the sum taken over $2$-strings $U=(U_2,\dots,U_{N+M})$ such that $|U| = 2N-2$, and where the coefficients $\Phi_U$ are defined as the one-row partition functions \eqref{cu-def}.
\medskip

\item In the case $N=1$, $M=0$ we have the explicit evaluation
\begin{align}
\label{eq:g-init}
g_{(2)}(x_1;\bm{z}) 
= 
W_{z_1/x_1}(0,2;2,0) 
= 
\frac{(1-q^2)(x_1+q^2 x_1-q^2 z_1-q^3 z_1)z_1}
{(x_1-q^2 z_1)(x_1-q^3 z_1)}.
\end{align}

\end{enumerate}
\end{prop}

\begin{proof}
We check each of these properties directly on $g_T(x_1,\dots,x_N)$:

\begin{enumerate}[wide, labelindent=0pt]
\item The symmetry with respect to $(x_1,\dots,x_N)$ is a direct consequence of Proposition \ref{prop:gam-sym}, as well as the fact that the vectors $\llangle 2^N \|$ and $\| 0^N \rrangle$ are invariant under permutation of the underlying spaces $V_1 \otimes \cdots \otimes V_N$.

\medskip
\item One may write $\tilde{g}_T(x_1,\dots,x_N;\bm{z})$ in the algebraic form
\begin{align*}
\tilde{g}_T(x_1,\dots,x_N;\bm{z})
=
\llangle 2^N\| 
\tilde{\Gamma}_{T_1}(z_1) 
\tilde{\Gamma}_{T_2}(z_2) 
\tilde{\Gamma}_{T_3}(z_3) 
\cdots
\|0^N \rrangle,
\end{align*}
where $\tilde{\Gamma}_k(z) = 
\prod_{i=1}^{N} (x_i-q^2 z)(x_i-q^3 z) \Gamma_k(z)$ for all $k \in \{0,1,2\}$. The normalized operators $\tilde{\Gamma}_k(z)$ are given explicitly by
\begin{align*}
\tilde{\Gamma}_0(z)
=
\prod_{i=1}^{N} \tilde{d}_i(z),
\end{align*}
\begin{align*}
\tilde{\Gamma}_1(z)
=
\sum_{u \in \{\circ,\bullet\}}
\sum_{i=1}^{N}
\tilde{e}^{u}_i(z)
\prod_{j\not=i}^{N} \tilde{d}_j(z) d^{u}_j(x_i),
\end{align*}
\begin{align*}
\tilde{\Gamma}_2(z)
=
\sum_{i=1}^{N}
\tilde{e}^{\bullet\circ}_i(z)
\prod_{j\not=i}^{N} 
\tilde{d}_j(z)
d^{\bullet\circ}_j(x_i) 
+
\sum_{u,v \in \{\circ,\bullet\}}
\sum_{1 \leq i<j \leq N}
\tilde{e}^{uv}_{ij}(z)
\prod_{k\not=i,j}^{N} 
\tilde{d}_k(z)
d^{u}_k(x_i) d^{v}_k(x_j),
\end{align*}
where we have defined the following normalized versions of the matrices \eqref{d-mat}, \eqref{ecirc-mat}--\eqref{ebb}:
\begin{align*}
\tilde{d}_j(z) &=(x_j-q^2 z)(x_j-q^3 z) d_j(z),
\\
\tilde{e}^{w}_i(z) &= (x_i-q^2 z)(x_i-q^3 z) e^{w}_i(z),
\quad
w \in \{\circ,\bullet,\bullet\circ\},
\\
\tilde{e}^{uv}_{ij}(z)
&= 
(x_i-q^2 z)(x_j-q^2 z)(x_i-q^3 z)(x_j-q^3 z) e^{uv}_{ij}(z),
\quad
u,v \in \{\circ,\bullet\}.
\end{align*}
From these representations it is clear that $\tilde{\Gamma}_k(z)$ has entries that are degree-$2N$ polynomials in $z$, and so in particular, $\tilde{g}_T(x_1,\dots,x_N;\bm{z})$ is a degree-$2N$ polynomial in $z_1$.
\medskip

\item The function $g_T(x_1,\dots,x_N;\bm{z})$ depends on $z_1$ only via the operator $\Gamma_{T_1}(z_1)$, and furthermore, one has that $T_1=2$. Thus, to demonstrate that $g_T(x_1,\dots,x_N;\bm{z})|_{z_1=0} = 0$ it is only necessary to show that $\Gamma_2(z)$ vanishes at $z=0$. Consulting \eqref{d-mat} and \eqref{ecirc-mat}--\eqref{ebb}, one sees that $d_i(0) = {\rm id}_i$ and $e^{\bullet\circ}_i(0) = e^{uv}_{ij}(0) = 0$ for all $u,v \in \{\circ,\bullet\}$, so that indeed $\Gamma_2(0) = 0$.
\medskip

\item
Using the fact that $T_1=2$, one has that
\begin{align*}
g_T(x_1,\dots,x_N;\bm{z})\Big|_{z_1=x_N}
=
\llangle 2^N\| 
\Gamma_{2}(x_N) 
\Gamma_{T_2}(z_2) 
\Gamma_{T_3}(z_3) 
\cdots
\|0^N \rrangle.
\end{align*}
To progress further, we examine the covector $\llangle 2^N\| \Gamma_2(x_N)$. One notices that
\begin{align*}
\llangle 2\|_N d_N(x_N) = 0,
\end{align*}
which is easily verified from \eqref{d-mat}, and 
\begin{align*}
\llangle 2\|_i \otimes \llangle 2\|_N
e^{uv}_{iN}(x_N)
=
0,
\quad
\text{for all}
\quad u,v \in \{\circ,\bullet\},
\qquad 1 \leq i \leq N-1,
\end{align*}
which follows from the fact that $\llangle 2\|_N e^{\circ}_N(z) = 0$ for arbitrary $z$, while $e^{\bullet}_N(x_N) = 0$ ({\it cf.}\/ equations \eqref{ecirc-mat} and \eqref{ebull-mat}). Collectively, these vanishing properties ensure that only the $i=N$ term from the first sum of \eqref{gam2} contributes non-trivially to the quantity $\llangle 2^N\| \Gamma_2(x_N)$:
\begin{align*}
\llangle 2^N\| \Gamma_2(x_N)
&=
\llangle 2^N\| 
e_N^{\bullet\circ}(x_N)
\prod_{j=1}^{N-1} d_j(x_N) d_j^{\bullet\circ}(x_N)
\\
&=
\llangle 2^N\| 
e_N^{\bullet\circ}(x_N)
\prod_{j=1}^{N-1} {\rm id}_j
=
\llangle 2 \|_1
\otimes\cdots \otimes 
\llangle 2 \|_{N-1}
\otimes
\llangle 0 \|_N.
\end{align*}
We have thus shown that 
\begin{align}
\label{eq:g-reduce-1}
g_T(x_1,\dots,x_N;\bm{z})\Big|_{z_1=x_N}
=
\llangle 2 \|_1
\otimes\cdots \otimes 
\llangle 2 \|_{N-1}
\otimes
\llangle 0 \|_N
\Gamma_{T_2}(z_2) 
\Gamma_{T_3}(z_3) 
\cdots
\|0^N \rrangle.
\end{align}
The proof concludes by noting the following projective property of the covector $\llangle 0 \|_N$ on the operators $\Gamma_k(z)$:
\begin{align}
\label{eq:0-project}
\llangle 0 \|_N
\Gamma_k(z;x_1,\dots,x_N)
=
\Gamma_k(z;x_1,\dots,x_{N-1})
\otimes
\llangle 0 \|_N,
\end{align}
where $\Gamma_k(z;x_1,\dots,x_N)$ denotes a twisted column operator acting on the full space $V_1\otimes \cdots \otimes V_N$, while $\Gamma_k(z;x_1,\dots,x_{N-1})$ denotes the same operator acting on $V_1\otimes \cdots \otimes V_{N-1}$. The property \eqref{eq:0-project} is easily verified using the fact that
\begin{align*}
\llangle 0 \|_N e_N^{u}(z)=0,\quad
\text{for all}\quad 
u \in \{\circ,\bullet,\bullet\circ\},
\end{align*}
and 
\begin{align*}
\llangle 0\|_N d_N(z) = \llangle 0\|_N,\
\qquad
\llangle 0\|_N d_N^{u}(z) = \llangle 0\|_N,\quad 
\text{for all}\quad
u \in \{\circ,\bullet,\bullet\circ\}.
\end{align*}
Iterating the relation \eqref{eq:0-project} within \eqref{eq:g-reduce-1}, each of the operators $\Gamma_{T_i}(z_i)$ is replaced by its $V_1 \otimes \cdots \otimes V_{N-1}$ counterpart, and all dependence on $V_N$ is ultimately contracted away due to the fact that $\llangle 0 \|_N \cdot \| 0 \rrangle_N =1$. The recursion relation \eqref{eq:g-reduce-0} is then established.
\medskip

\item
In the first step, one shows that the function $g_T(x_1,\dots,x_N;\bm{z})$ admits an algebraic equality of the same form as \eqref{eq:FT-U}. More precisely, we claim that
\begin{align}
\label{eq:gT-U}
g_T(x_1,\dots,x_N;\bm{z})
=
\sum_{(U_1,\dots,U_{N+M})}
\Psi_U(z_1)
g_{\omega\cdot U}(x_1,\dots,x_N;\omega\cdot\bm{z})
\end{align}
with the sum taken over $2$-strings $(U_1,\dots,U_{N+M})$ such that $|U|=2N$, where $\omega \cdot U = (U_2,\dots,U_{N+M},U_1)$ and $\omega\cdot\bm{z} = (z_2,\dots,z_{N+M},z_1)$, and the coefficients $\Psi_U(z_1)$ are given by the same one-row partition functions \eqref{eq:Psi-one-row} as previously. To see this, recall that the twisted column operators satisfy exactly the same family of exchange relations \eqref{gamma-com} as their untwisted counterparts. Since \eqref{eq:FT-U} is derived using nothing other than exchange relations of partition function columns, it follows that the same relation must apply to $g_T(x_1,\dots,x_N;\bm{z})$ itself.

The central part of the proof is then the computation of
\begin{align}
\label{eq:residue-key}
{\rm Res}\Big\{
g_{\omega\cdot U}(x_1,\dots,x_N;\omega\cdot\bm{z})
\Big\}_{z_1=q^{-3}x_N}
=
\llangle 2^N \|
\Gamma_{U_2}(z_2)
\cdots
\Gamma_{U_{N+M}}(z_{N+M})
{\rm Res}
\Big\{
\Gamma_{U_1}(z_1)
\Big\}_{z_1=q^{-3}x_N}
\| 0^N \rrangle,
\end{align}
considering the possible values of $U_1 \in \{0,1,2\}$. In the case $U_1=0$, one sees that
\begin{align*}
{\rm Res}
\Big\{
\Gamma_{0}(z)
\Big\}_{z=q^{-3}x_N}
\| 0^N \rrangle
=
0,
\end{align*}
since
\begin{align}
\label{dmat-res}
{\rm Res}\Big\{ d_N(z) \Big\}_{z=q^{-3} x_N} \|0\rrangle_N = 0,
\end{align}
the latter being easily verified from \eqref{d-mat}. One also finds that
\begin{align*}
{\rm Res}
\Big\{
\Gamma_{1}(z)
\Big\}_{z=q^{-3}x_N}
\| 0^N \rrangle
=
0,
\end{align*}
by combination of \eqref{dmat-res} together with the facts
\begin{align}
\label{emat-res}
{\rm Res}\Big\{ e^{\circ}_N(z) \Big\}_{z=q^{-3} x_N} = 0,
\qquad
e^{\bullet}_N(z) \|0\rrangle_N=0,
\end{align}
which follow immediately from \eqref{ecirc-mat} and \eqref{ebull-mat}.

The only non-trivial case of the residue \eqref{eq:residue-key} is thus when $U_1=2$. We therefore turn to examining the residue of the vector $\Gamma_2(z)\|0^N\rrangle$ at the point $z=q^{-3}x_N$. Using \eqref{dmat-res} and the vanishing property
\begin{align*}
{\rm Res}\Big\{ e_{iN}^{uv}(z) \Big\}_{z=q^{-3} x_N} 
\|0\rrangle_i \otimes \|0\rrangle_N
= 0,
\quad
\text{for all}
\quad u,v \in \{\circ,\bullet\},
\qquad 1 \leq i \leq N-1,
\end{align*}
which follows immediately from \eqref{emat-res}, one finds that only the $i=N$ term from the first sum of \eqref{gam2} contributes non-trivially to the residue of $\Gamma_2(z)\|0^N\rrangle$ at the point $z=q^{-3}x_N$:
\begin{align}
\label{eq:gam2-res}
{\rm Res}\Big\{ \Gamma_2(z) \Big\}_{z=q^{-3}x_N} \|0^N\rrangle
&= 
{\rm Res}\Big\{ e_N^{\bullet\circ}(z) \Big\}_{z=q^{-3}x_N}
\prod_{j=1}^{N-1} d_j(q^{-3}x_N) d_j^{\bullet\circ}(x_N)
\|0^N\rrangle
\\
&=
-\frac{(1-q^2)(1-q^3)}{q^6(1-q)}x_N
\cdot
\|0\rrangle_1 
\otimes
\cdots
\otimes
\|0\rrangle_{N-1}
\otimes
\|2\rrangle_N.
\nonumber
\end{align}
In preparation for applying \eqref{eq:gam2-res} to \eqref{eq:residue-key}, we claim that
\begin{multline}
\label{eq:N-spacefreeze}
\llangle 2^{N-1}\|
\otimes
\llangle 2 \|_N
\prod_{j=2}^{N+M}
\Gamma_{U_j}(z_j;x_1,\dots,x_N)
\| 0^{N-1} \rrangle \otimes \| 2 \rrangle_N
\\
=
\prod_{i=1}^{N-1}
\frac{(x_N-q^2 x_i)(x_N-q^3 x_i)}{(x_N-x_i)(x_N-q x_i)}
\prod_{j=2}^{N+M}
\frac{(x_N-z_j)(x_N-q z_j)}{(x_N-q^2 z_j)(x_N-q^3 z_j)}
\llangle 2^{N-1}\|
\prod_{j=2}^{N+M}
\Gamma_{U_j}(z_j;x_1,\dots,x_{N-1})
\| 0^{N-1} \rrangle.
\end{multline}
The derivation of equation \eqref{eq:N-spacefreeze} proceeds by analyzing the local operators that act in space $V_N$ within the product $\prod_{j=2}^{N+M}
\Gamma_{U_j}(z_j;x_1,\dots,x_N)$. Since the operators on this space are sandwiched between the covector $\llangle 2\|_N$ and vector $\| 2 \rrangle_N$, any non-diagonal operator acting on $V_N$ will give a vanishing contribution to the quantity on the left hand side of \eqref{eq:N-spacefreeze}. The only operators that may act on this space are thus $d_N(z_j)$ with $2 \leq j \leq N+M$, and $d_N^{u}(x_i)$ with $u \in \{\circ,\bullet,\bullet\circ\}$, $1 \leq i \leq N-1$. 

One finds that the operator $d_N(z_j)$ is represented exactly once for each $2 \leq j \leq N+M$, yielding the factor $\dfrac{(x_N-z_j)(x_N-q z_j)}{(x_N-q^2 z_j)(x_N-q^3 z_j)}$ when one sandwiches between $\llangle 2\|_N$ and $\| 2 \rrangle_N$. In a similar vein, the operator $d_N^{\bullet\circ}(x_i) = d_N^{\bullet}(x_i) d_N^{\circ}(x_i)$ is represented exactly once for each $1\leq i \leq N-1$, yielding $\dfrac{(x_N-q^2 x_i)(x_N-q^3 x_i)}{(x_N-x_i)(x_N-q x_i)}$ when one sandwiches between $\llangle 2\|_N$ and $\| 2 \rrangle_N$. The identity \eqref{eq:N-spacefreeze} then follows.

The proof concludes by combining \eqref{eq:gam2-res} and \eqref{eq:N-spacefreeze} in \eqref{eq:residue-key}, which gives
\begin{multline}
\label{eq:residue-key2}
{\rm Res}\Big\{
g_{\omega\cdot U}(x_1,\dots,x_N;\omega\cdot\bm{z})
\Big\}_{z_1=q^{-3}x_N}
=
-\frac{(1-q^2)(1-q^3)}{q^6(1-q)}x_N
\times
\\
\prod_{i=1}^{N-1}
\frac{(x_N-q^2 x_i)(x_N-q^3 x_i)}{(x_N-x_i)(x_N-q x_i)}
\prod_{j=2}^{N+M}
\frac{(x_N-z_j)(x_N-q z_j)}{(x_N-q^2 z_j)(x_N-q^3 z_j)}
\cdot\bm{1}_{U_1=2}\cdot
g_{(U_2,\dots,U_{N+M})}
(x_1,\dots,x_{N-1};\widehat{\bm{z}}_1).
\end{multline}
One may then use \eqref{eq:residue-key2} in \eqref{eq:gT-U} to compute 
${\rm Res}\Big\{g_{T}(x_1,\dots,x_N;\bm{z})\Big\}_{z_1=q^{-3}x_N}$ and the required recursion \eqref{eq:res-gT} follows immediately.

\medskip

\item One has that $g_{(2)}(x_1;\bm{z})
=
\llangle 2^1 \| \Gamma_2(z_1) \| 0^1 \rrangle$, where in this case 
\begin{align*}
\llangle 2^1 \| = \begin{pmatrix} 0 & 0 & 1 \end{pmatrix}_1,
\qquad
\| 0^1 \rrangle = \begin{pmatrix} 1 \\ 0 \\ 0 \end{pmatrix}_1,
\qquad
\Gamma_2(z_1)
=
\dfrac{(1-q^2)(x_1+q^2x_1-q^2z_1-q^3z_1)z_1}
{(x_1-q^2z_1)(x_1-q^3z_1)}
\begin{pmatrix}
0 & 0 & 0\\
0 & 0 & 0\\
1 & 0 & 0
\end{pmatrix}_1,
\end{align*}
and the claim \eqref{eq:g-init} is immediate.

\end{enumerate}
\end{proof}

This completes the check of Propositions \ref{prop:IK-exch}, \ref{prop:11} and \ref{prop:IK-properties} for the functions $g_S(x_1,\dots,x_N;\bm{z})$; in view of Theorem \ref{thm:unique}, it follows that $g_S(x_1,\dots,x_N;\bm{z}) = F_S(x_1,\dots,x_N;\bm{z})$ for all $S \in \mathfrak{s}(2)$ such that $|S|=2N$, thereby completing the proof of Theorem \ref{thm:twisted-column-F}.
\end{proof}

\subsection{Symmetrization formula}
\label{ssec:IK-sym-formula}

We are now ready to state and prove our symmetrization identity for the rational functions $F_S(x_1,\dots,x_N;\bm{z})$. Before doing so, we shall require some further notation.

\begin{defn}
Let $\sigma \in \mathfrak{M}_2(N,S)$ be a $2$-permutation matrix (see Definition \ref{defn:M-set}) with profile $S=(S_1,\dots,S_M)$. For any two rows $U=(U_1,\dots,U_M)$, $V=(V_1,\dots,V_M)$ of $\sigma$ and $a,b \in \{1,2\}$, we shall use the following pictorial shorthand:
\begin{align*}
&
\begin{tabular}{c||c|c|c}
    U & a & &\\ \hline
    V & & b &
\end{tabular}
\qquad
\text{ indicates that } a \text{ in } U \text{ occurs before } b \text{ in } V,
\\
&
\begin{tabular}{c||c|c|c}
    U & & a & \\ \hline
    V & b & &
\end{tabular}
\qquad
\text{ indicates that } a \text{ in } U \text{ occurs after } b \text{ in } V,
\\
&
\begin{tabular}{c||c|c|c}
    U & a &\phantom{b} &\\ \hline
    V & b & &
\end{tabular}
\qquad
\text{ indicates that } a \text{ in } U \text{ occurs at the same position as } b \text{ in } V.
\end{align*}
The purpose of this notation is to succinctly express the relative positions of the non-zero entries of $U,V$; it does not carry information about the absolute positions of the non-zero entries. In situations where the absolute positioning is relevant we shall place coordinate labels in the columns of these diagrams; for example,
\begin{align*}
\begin{tabular}{c||c|c|c}
     U & a & & \\ \hline
     V &  & b &\\
     \hdashline[1pt/1pt]
    & k & $\ell$ &
\end{tabular}
\qquad
\text{ indicates that } U_k=a \text{ and } V_{\ell}=b, \text{ with } 1 \leq k < \ell \leq M.
\end{align*}
\end{defn}

\begin{thm}\label{thm:symmetriz-IK}
Fix an integer $N \geq 1$ and let $S \in \mathfrak{s}(2)$ be a $2$-string such that $|S|=2N$. The rational symmetric function \eqref{eqspsymIK2} is given by the explicit formula
\begin{equation}
\label{symformula1}
F_S(x_1,\dots,x_N;\bm{z})
=
\sum_{\sigma \in \mathfrak{M}_2(N,S)}
\prod_{1 \leq i<j \leq N} \Delta_{\sigma(i),\sigma(j)}(x_i,x_j;\bm{z})
\prod_{i=1}^{N}
F_{\sigma(i)}(x_i;\bm{z}),
\end{equation}
where $\sigma(i)$ denotes the $i$-th row (read from top to bottom) of the $2$-permutation matrix $\sigma$. For any $U,V \in \mathfrak{s}(2)$ such that $|U|=|V|=2$, we define
\begin{align}
\label{eq:IK-delta1}
\Delta_{U,V}(x,y;\bm{z}) = 
\begin{cases} 
\smallskip
\dfrac{(y - q^2x)(y - q^3x)}{(y - x)(y - q x)}, 
&
\begin{tabular}{c||c|c|c}
    U & 2 & \\ \hline
    V & & 2 &
\end{tabular},\ \
\begin{tabular}{c||c|c|c}
    U & 2 &  &\\ \hline
    V & & 1& 1
\end{tabular},\ \
\begin{tabular}{c||c|c|c}
    U & 1 & 1 &\\ \hline
    V & & & 2
\end{tabular},\ \
\begin{tabular}{c||c|c|c|c}
    U & 1 & 1 & & \\ \hline
    V &  &  & 1 & 1  
\end{tabular}
\\
\smallskip
\dfrac{(y - q^2x)(q^2y - x)}{(y - x)^2}, 
&
\begin{tabular}{c||c|c|c}
    U & 1 & & 1 \\ \hline
    V & & 2 & 
\end{tabular},\ \
\begin{tabular}{c||c|c|c|c}
    U & 1 & & & 1 \\ \hline
    V & & 1 & 1 & 
\end{tabular}
\\
\smallskip
\dfrac{(y-q^2x)^2 (x-q y)}
{(y - x)^2 (y-q x)},
&
\begin{tabular}{c||c|c|c|c}
    U & 1 &  & 1 & \\ \hline
    V &  & 1 & & 1 
\end{tabular}
\end{cases}
\end{align}
and
\begin{align}
\label{eq:IK-delta2}
\Delta_{U,V}(x,y;\bm{z}) = \begin{cases}
\dfrac{(q^2x-y)(x y+q^2x y-q^3y z_k-q^3x z_k)}{q(1-q)(y-x)(y-qx)z_k}, \quad  &\begin{tabular}{c||c|c|c}
    U & 1 & 1 &  \\ \hline
    V & 1 &  & 1 \\ \hdashline[1pt/1pt]
    & k & &
\end{tabular}
\medskip
\\
\dfrac{(y-q^2x)(xy+q^2xy-q^2yz_k-q^3xz_k)}{q(1-q)(y-x)^2z_k}, 
&\begin{tabular}{c||c|c|c}
    U & 1 & 1 &  \\ \hline
    V &  & 1 & 1 \\ \hdashline[1pt/1pt]
    &  & k & 
\end{tabular}
\medskip
\\
\dfrac{(q^2x-y)(xy+q^2xy-q^2yz_k-q^2xz_k)}{(1-q)(y-x)(y-qx)z_k},
&\begin{tabular}{c||c|c|c}
    U & 1 &  & 1 \\ \hline
    V &  & 1 & 1 \\ \hdashline[1pt/1pt]
    &  & & k
\end{tabular}
\end{cases}
\end{align}
We also specify that
\begin{align}
\label{eq:IK-delta3}
\Delta_{U,V}(x,y;\bm{z}) = 
\frac{(xy+ q^2xy-q^3yz_k-q^3xz_k)(xy+q^2xy-q^2yz_{\ell}-q^2xz_{\ell})}{q(1-q)^2(y-qx)(x-qy)z_kz_{\ell}}, \quad\quad 
\begin{tabular}{c||c|c}
    U & 1 & 1  \\ \hline
    V & 1 & 1 \\ \hdashline[1pt/1pt]
    & k & $\ell$ 
\end{tabular}
\end{align}
Any case of $\Delta_{U,V}(x,y;\bm{z})$ not listed above is determined by imposing the symmetry
\begin{align}
\label{eq:delta-sym}
\Delta_{U,V}(x,y;\bm{z}) = \Delta_{V,U}(y,x;\bm{z}).
\end{align}
The function $F_{\sigma(i)}(x_i;\bm{z})$ appearing in the summand of \eqref{symformula1} denotes the $N=1$ version of the partition function \eqref{pictureofrational1}, and is given explicitly by
\[
F_U(x;\bm{z}) = 
\begin{cases}
\displaystyle\prod_{j = 1}^{k-1} 
\dfrac{(x-z_j)(x-qz_j)}{(x-q^2z_j)(x-q^3z_j)}
\dfrac{(1-q^2)(x+q^2x-q^2z_k-q^3z_k)z_k}{(x-q^2z_k)(x-q^3z_k)},
\qquad\
U_k=2
\\
\\
\displaystyle \prod_{j = 1}^{k-1} 
\dfrac{(x - z_j)(x - qz_j)}{(x - q^2z_j)(x-q^3z_j)}
\dfrac{q(1-q)^2(1+q)(x - z_k)z_k}{(x - q^2z_k)(x - q^3z_k)}
\cdot
\prod_{j = k + 1}^{\ell-1}
\dfrac{x - z_j}{x - q^2z_j}
\dfrac{(1-q^2)z_{\ell}}{x - q^2z_{\ell}},
\\
\\
\qquad
\qquad\qquad\qquad\qquad\qquad\qquad\qquad\qquad\qquad\qquad\qquad\qquad\qquad
U_k=U_{\ell}=1,\quad k < \ell.
\end{cases}
\]
\end{thm}

\begin{rmk}
The symmetrization formula \eqref{symformula1} has similar structure to \eqref{symformulasix1} in the six-vertex case, but is more involved due to the large number of possibilities that arise when selecting two rows of a $2$-permutation matrix: this is reflected in the large number of different {\it scattering factors} \eqref{eq:IK-delta1}--\eqref{eq:IK-delta3}. We note that some of these factors, namely \eqref{eq:IK-delta2}--\eqref{eq:IK-delta3}, depend on the vertical spectral parameters of the lattice model, which is a feature not observed in the six-vertex case.
\end{rmk}

\begin{proof}
In view of Theorem \ref{thm:twisted-column-F}, it suffices to show that
\begin{equation}
\label{col-to-sym}
\llangle 2^N\| \Gamma_{S_1}(z_1)\Gamma_{S_2}(z_2)\Gamma_{S_3}(z_3)
\cdots \| 0^N \rrangle
= 
\sum_{\sigma \in \mathfrak{M}_2(N,S)}
\prod_{1 \leq i<j \leq N} \Delta_{\sigma(i),\sigma(j)}(x_i,x_j;\bm{z})
\prod_{i=1}^{N}
F_{\sigma(i)}(x_i;\bm{z}),
\end{equation}
where the right-hand side of equation $\eqref{col-to-sym}$ is the same as the right-hand side of $\eqref{symformula1}$.

The proof of \eqref{col-to-sym} proceeds in two steps. In the first step, we establish a one-to-one correspondence between the elements of the set $\mathfrak{M}_2(N,S)$ and the products of operators $\mathfrak{P}$ that arise by explicitly expanding the left-hand side of \eqref{col-to-sym}. In the second step, we demonstrate that the functions associated to $\sigma$ in the summand of \eqref{col-to-sym} match $\llangle 2^N \| \mathfrak{P} \| 0^N \rrangle$ in the previously established correspondence. Our exposition will be fairly informal, as there are relatively simple ideas behind the matching \eqref{col-to-sym} that would be obscured by heavy notation.

For the first step, we note that the summation on both sides of \eqref{col-to-sym} is finite. This is due to the fact that $S \in \mathfrak{s}(2)$, and therefore there exists $M \geq 1$ such that $S_k = 0$ for all $k>M$ (without loss of generality, we may also assume that $M$ is the minimal such integer, meaning that $S_M \geq 1$). As such, the left-hand side of \eqref{col-to-sym} may be written in truncated form $\llangle 2^N \| \Gamma_{S_1}(z_1)\cdots \Gamma_{S_M}(z_M) \| 0^N \rrangle$, while the summation on the right-hand side of \eqref{col-to-sym} ranges over all $2$-permutation matrices in $\mathfrak{M}_2(N,S)$ with exactly $M$ columns.

Fix $\sigma \in \mathfrak{M}_2(N,S)$. We shall associate to it a product of operators $\mathfrak{P} =\prod_{i=1}^{N} \prod_{j=1}^{M} \mathfrak{P}^{\sigma}(i,j) \prod_{j=1}^{M} \eta^{\sigma}(j)$, where $\mathfrak{P}^{\sigma}(i,j)$ and $\eta^{\sigma}(j)$ are defined in the following way. Letting $\sigma(i,j)$ denote the row $i$ and column $j$ entry of $\sigma$, we define
\begin{align*}
\mathfrak{P}^{\sigma}(i,j)
=
\left\{
\begin{array}{ll}
e^{\bullet\circ}_i(z_j), 
&
\qquad
\sigma(i,j)=2,
\\ \\
e^{\bullet}_i(z_j),
&
\qquad
\sigma(i,j)=1,\ \exists\ k>j\ \text{such that}\ \sigma(i,k)=1,
\\ \\
e^{\circ}_i(z_j),
&
\qquad
\sigma(i,j)=1,\ \exists\ k<j\ \text{such that}\ \sigma(i,k)=1.
\end{array}
\right.
\end{align*}
This associates a non-diagonal operator to every non-zero entry of $\sigma$. The zero entries of $\sigma$, on the other hand, get associated with diagonal operators via the following choices:
\begin{align*}
\mathfrak{P}^{\sigma}(i,j)
=
\left\{
\begin{array}{ll}
d_i(z_j), 
&
\qquad
\sigma(i,j)=0,\ S_j=0,
\\ \\
d_i(z_j) d_i^u(x_k),
&
\qquad
\sigma(i,j)=0,\ S_j=1,\ \mathfrak{P}^{\sigma}(k,j) = e^{u}_k(z_j),
\\ \\
d_i(z_j)
d_i^{\bullet\circ}(x_k),
&
\qquad
\sigma(i,j)=0,\ S_j=2,\ \mathfrak{P}^{\sigma}(k,j) = e^{\bullet\circ}_k(z_j),
\\ \\
d_i(z_j)
d_i^{u}(x_k)
d_i^{v}(x_{\ell}),
&
\qquad
\sigma(i,j)=0,\ S_j=2,\ 
\mathfrak{P}^{\sigma}(k,j) = e^{u}_k(z_j),\ 
\mathfrak{P}^{\sigma}(\ell,j) = e^{v}_{\ell}(z_j),
\end{array}
\right.
\end{align*}
with $u,v$ taking values in $\{\circ,\bullet\}$. We also introduce the scalar

\begin{align*}
\eta^{\sigma}(j)
=
\left\{
\begin{array}{ll}
\dfrac{x_k x_\ell+q^2 x_k x_\ell-q^2 x_k z_j-q^2 x_\ell z_j}
{(1-q^2)z_j},
\quad
& S_j=2,\ 
\mathfrak{P}^{\sigma}(k,j) = e^{\circ}_k(z_j),\ 
\mathfrak{P}^{\sigma}(\ell,j) = e^{\circ}_{\ell}(z_j)
\\ \\
\dfrac{x_k x_\ell+q^2x_k x_\ell-q^2x_k z_j-q^3x_\ell z_j}
{q(1-q)(x_k-x_\ell)z_j},
\quad
& S_j=2,\ 
\mathfrak{P}^{\sigma}(k,j) = e^{\bullet}_{k}(z_j),\ 
\mathfrak{P}^{\sigma}(\ell,j) = e^{\circ}_\ell(z_j)
\\ \\
\dfrac{(1+q)(x_k x_\ell+q^2x_k x_\ell-q^3x_k z_j-q^3x_\ell z_j)}{q(1-q)(x_k-qx_\ell)(x_\ell-qx_k)z_j},
\quad
& S_j=2,\ 
\mathfrak{P}^{\sigma}(k,j) = e^{\bullet}_k(z_j),\ 
\mathfrak{P}^{\sigma}(\ell,j) = e^{\bullet}_{\ell}(z_j)
\\ \\
1,
\quad
&
{\rm otherwise.}
\end{array}
\right.
\end{align*}
It is not difficult to check that this defines a bijection between the elements of $\mathfrak{M}_2(N,S)$ and all operator products in the expansion of $\Gamma_{S_1}(z_1)\cdots \Gamma_{S_M}(z_M)$ which survive when sandwiched between $\llangle 2^N\|$ and $\|0^N \rrangle$. For an illustration of this correspondence, see Figure \ref{fig:sigma-correspond}.

\begin{figure}
$$
\begin{pmatrix}
0 & 1 & 1 & 0 & 0 & 0\\
0 & 0 & 0 & 0 & 1 & 1\\
2 & 0 & 0 & 0 & 0 & 0\\
0 & 1 & 0 & 0 & 1 & 0
\end{pmatrix}
$$
\bigskip
\medskip
$$\begin{pmatrix}
\medskip
d_1(z_1)
d^{\bullet\circ}_1(x_3)  
& 
e^{\bullet}_1(z_2)
&
e^{\circ}_1(z_3)
&
d_1(z_4)
&
d_1(z_5)
d^{\bullet}_1(x_2)
d^{\circ}_1(x_4)
&
d_1(z_6)
d^{\circ}_1(x_2)
\\ \medskip
d_2(z_1)
d^{\bullet\circ}_2(x_3)  
&
d_2(z_2)
d^{\bullet}_2(x_1)
d^{\bullet}_2(x_4)
&
d_2(z_3)
d^{\circ}_2(x_1)
&
d_2(z_4)
&
e^{\bullet}_2(z_5)
&
e^{\circ}_2(z_6)
\\ \medskip
e^{\bullet\circ}_3(z_1) 
& 
d_3(z_2)
d^{\bullet}_3(x_1)
d^{\bullet}_3(x_4)
& 
d_3(z_3)
d^{\circ}_3(x_1)
&  
d_3(z_4)
&
d_3(z_5)d^{\bullet}_3(x_2) d^{\circ}_3(x_4)
&
d_3(z_6)
d^{\circ}_3(x_2)
\\ \medskip
d_4(z_1)
d^{\bullet\circ}_4(x_3) 
& 
e^{\bullet}_4(z_2)
& 
d_4(z_3)
d^{\circ}_4(x_1)
& 
d_4(z_4)
&
e^{\circ}_4(z_5)
&
d_4(z_6)
d^{\circ}_4(x_2)
\end{pmatrix}
$$
\caption{An illustration of the correspondence between $\sigma$ and $\mathfrak{P}$, for 
$S=(2,2,1,0,2,1)$ and $N=4$. In the top half, we give a specific choice of $\sigma \in \mathfrak{M}_2(N,S)$. The corresponding operators $\mathfrak{P}^{\sigma}(i,j)$ are obtained from entry $(i,j)$ of the matrix in the bottom half. For this example, the factors $\eta^{\sigma}(2)$ and $\eta^{\sigma}(5)$ also contribute non-trivially to $\mathfrak{P}$, although we have suppressed them. Reading off the operators in the $j$-th column, one recovers a single term in the summation formula for $\Gamma_{S_j}(z_j)$.}
\label{fig:sigma-correspond}
\end{figure}

Now we turn our attention to the second step. Our claim is that under the above correspondence, one has
\begin{align*}
\llangle 2^N\|
\mathfrak{P}
\|0^N \rrangle
=
\prod_{j=1}^{M}
\eta^{\sigma}(j)
\llangle 2^N\| 
\prod_{i=1}^{N}
\prod_{j=1}^{M}
\mathfrak{P}^{\sigma}(i,j)
\|0^N \rrangle
=
\prod_{1 \leq i<j \leq N} \Delta_{\sigma(i),\sigma(j)}(x_i,x_j;\bm{z})
\prod_{i=1}^{N}
F_{\sigma(i)}(x_i;\bm{z}).
\end{align*}
It is easily verified that multiplying out all operators in $\mathfrak{P}$ with argument $z_j$, $1 \leq j \leq M$,  produces exactly the product $\prod_{i=1}^{N}
F_{\sigma(i)}(x_i;\bm{z})$ of single-variable functions. The remaining operators in $\mathfrak{P}$ (with arguments $x_i$, $1 \leq i \leq N$) as well as the factor $\prod_{j=1}^{M} \eta^{\sigma}(j)$ give rise to $\prod_{1 \leq i<j \leq N} \Delta_{\sigma(i),\sigma(j)}(x_i,x_j;\bm{z})$.
\end{proof}

We refer the reader to Appendix \ref{chapter:examples} for some explicit examples of the formula \eqref{symformula1}, in the case $N=2$.

\appendix

\chapter{Examples of Izergin--Korepin symmetrization formula}
\label{chapter:examples}

To help illustrate how the symmetrization formula \eqref{symformula1} works, we record three cases of it that may arise in the case $N=2$.

\section{The string \texorpdfstring{$S = (2,2)$}{Lg}}
    
    The set of $2$-permutations with profile $(2,2)$ is given by
	\begin{equation}
    \label{eq:=twoperm--(2,2)}
	\mathfrak{M}_{2}(2,(2,2)) = \left\{ \begin{pmatrix}
	2 & 0 \\
	0 & 2
	\end{pmatrix},\begin{pmatrix}
	1 & 1 \\
	1 & 1
	\end{pmatrix},\begin{pmatrix}
	0 & 2\\
	2 & 0	
	\end{pmatrix} \right\}.
	\end{equation}
	The formula \eqref{symformula1} produces the following sum over the elements of the set \eqref{eq:=twoperm--(2,2)}:
	\begin{align}
	\label{eq:=example--(2,2)}
	F_{(2,2)}(x_1,x_2;\bm{z}) &= \Delta_{(2,0),(0,2)}(x_1,x_2;\bm{z}) F_{(2,0)}(x_1;\bm{z})F_{(0,2)}(x_2;\bm{z}) \\
	&+ \Delta_{(1,1),(1,1)}(x_1,x_2;\bm{z})F_{(1,1)}(x_1;\bm{z})F_{(1,1)}(x_2;\bm{z}) \nonumber\\
	&+ \Delta_{(0,2),(2,0)}(x_1,x_2;\bm{z})F_{(0,2)}(x_1;\bm{z})F_{(2,0)}(x_2;\bm{z}), \nonumber
	\end{align}
where the indices of $\Delta$ and $F$ are determined by reading the rows of each $2$-permutation matrix from top to bottom. The one-variable partition functions that appear in \eqref{eq:=example--(2,2)} are given by
	\begin{equation}
	\label{pic:=onerowpart--(2,2)}
	F_{(a_1,a_2)}(x_1;\bm{z}) \times F_{(b_1,b_2)}(x_2;\bm{z}) =
\begin{tikzpicture}[scale = 0.7,every node/.style={scale=0.7},baseline = {(0,0)}]
	\foreach \i in {1,2} {
	\draw[lgray,line width=1.5pt,->] (\i,-1) -- (\i,1);
	}
	\draw[lgray,line width=1.5pt,->] (0,0) -- (3,0);
	\foreach \i in {1,2} {
	\node at (\i,-1.2) {$0$};
	}
	\foreach \i in {1} {
	\node at (\i,1.2) {$a_1$};
	}
	\node at (2,1.2) {$a_2$};
	\node at (3.2,0) {$0$};
	\node at (-0.2,0) {$2$};
	\node at (-1.5,0) {$x_1$};
	\draw[thick,->] (-1.25,0) -- (-0.5,0);
	\foreach \i in {1,2} {
	\draw[thick,->] (\i,-2.25) -- (\i,-1.5);
	}
	\foreach \i in {1,2} {
	\node at (\i,-2.5) {$z_{\i}$};}
	\node[scale = 1.5] at (3.75,0) {$\times$};
	\begin{scope}[shift = {(6,0)}]
	\foreach \i in {1,2} {
	\draw[lgray,line width=1.5pt,->] (\i,-1) -- (\i,1);
	}
	\draw[lgray,line width=1.5pt,->] (0,0) -- (3,0);
	\foreach \i in {1,2} {
	\node at (\i,-1.2) {$0$};
	}
	\foreach \i in {1} {
	\node at (\i,1.2) {$b_1$};
	}
	\node at (2,1.2) {$b_2$};
	\node at (3.2,0) {$0$};
	\node at (-0.2,0) {$2$};
	\node at (-1.5,0) {$x_2$};
	\draw[thick,->] (-1.25,0) -- (-0.5,0);
	\foreach \i in {1,2} {
	\draw[thick,->] (\i,-2.25) -- (\i,-1.5);
	}
	\foreach \i in {1,2} {
	\node at (\i,-2.5) {$z_{\i}$};}
	\end{scope}
	\end{tikzpicture}
	\end{equation}
	for any $a_1,a_2,b_1,b_2$ such that $a_1+a_2 = 2$, $b_1+b_2 = 2$, and the vertex weights are given in Figure \ref{fig:19weights} of Chapter \ref{chapter:19-vertex}. The scattering factors on the right hand side of \eqref{eq:=example--(2,2)} have the following pictorial presentations: 
	\[
	\Delta_{(2,0),(0,2)} \mapsto \begin{tabular}{c||c|c|c}
    \textit{U} & 2 & \\ \hline
    \textit{V} & & 2 &
	\end{tabular}, \qquad \Delta_{(1,1),(1,1)} \mapsto \begin{tabular}{c||c|c|c}
    \textit{U} & 1 & 1\\ \hline
    \textit{V} & 1 & 1 &
	\end{tabular}.
	\]
	Using the symmetry $\Delta_{U,V}(x,y;\bm{z}) = \Delta_{V,U}(y,x;\bm{z})$ and the explicit expressions \eqref{eq:IK-delta1} and \eqref{eq:IK-delta3}, we obtain the following expansion:
	\begin{align*}
	F_{(2,2)}(x_1,x_2;\bm{z}) &=\frac{(x_2-q^2x_1)(x_2-q^3x_1)}{(x_2-x_1)(x_2-qx_1)}F_{(2,0)}(x_1;\bm{z})F_{(0,2)}(x_2;\bm{z})\\
	& + \frac{(x_1x_2+q^2x_1x_2-q^3x_2z_1-q^3x_1z_1)(x_1x_2+q^2x_1x_2-q^2x_2z_2-q^2x_1z_2)}{q(1-q)^2(x_2-qx_1)(x_1-qx_2)z_1z_2}F_{(1,1)}(x_1;\bm{z})F_{(1,1)}(x_2;\bm{z})\\
	& + \frac{(x_1-q^2x_2)(x_1-q^3x_2)}{(x_1-x_2)(x_1-qx_2)} F_{(0,2)}(x_1;\bm{z})F_{(2,0)}(x_2;\bm{z}),
	\end{align*}
	where the one-row partition functions are depicted in equation \eqref{pic:=onerowpart--(2,2)}.

\section{The string \texorpdfstring{$S=(1,2,1)$}{Lg}}
    
	The set of $2$-permutations with profile $(1,2,1)$ is given by 
	\begin{equation}
    \label{eq:=twoperm--(1,2,1)}
	\mathfrak{M}_{2}(2,(1,2,1)) = \left\{\begin{pmatrix}
	1 & 0 & 1\\
	0 & 2 & 0
	\end{pmatrix},\begin{pmatrix}
	1 & 1 & 0\\
	0 & 1 & 1
	\end{pmatrix}, \begin{pmatrix}
	0 & 1 & 1 \\
	1 & 1 & 0
	\end{pmatrix},\begin{pmatrix}
	0 & 2 & 0 \\
	1 & 0 & 1
	\end{pmatrix} \right\}.
	\end{equation}
	The formula \eqref{symformula1} produces the following sum over the elements of the set \eqref{eq:=twoperm--(1,2,1)}:
	\begin{align}
    \label{eq:=example--(1,2,1)}
	F_{(1,2,1)}(x_1,x_2;\bm{z}) &= \Delta_{(1,0,1),(0,2,0)}(x_1,x_2;\bm{z})F_{(1,0,1)}(x_1;\bm{z})F_{(0,2,0)}(x_2;\bm{z})\\
    \nonumber
	&+ \Delta_{(1,1,0),(0,1,1)}(x_1,x_2;\bm{z})F_{(1,1,0)}(x_1;\bm{z})F_{(0,1,1)}(x_2;\bm{z})\\
    \nonumber
	&+ \Delta_{(0,1,1),(1,1,0)}(x_1,x_2;\bm{z}) F_{(0,1,1)}(x_1;\bm{z})F_{(1,1,0)}(x_2;\bm{z})\\
    \nonumber
	&+ \Delta_{(0,2,0),(1,0,1)}(x_1,x_2;\bm{z})F_{(0,2,0)}(x_1;\bm{z})F_{(1,0,1)}(x_2;\bm{z}).
	\end{align}
	 The one-variable partition functions that appear in \eqref{eq:=example--(1,2,1)} are given by
	\begin{equation}
    \label{eq:=onerowpart--(1,2,1)}
	F_{(a_1,a_2,a_3)}(x_1;\bm{z}) \times F_{(b_1,b_2,b_3)}(x_2;\bm{z}) =
\begin{tikzpicture}[scale = 0.7,every node/.style={scale=0.7},baseline = {(0,0)}]
	\foreach \i in {1,2,3} {
	\draw[lgray,line width=1.5pt,->] (\i,-1) -- (\i,1);
	}
	\draw[lgray,line width=1.5pt,->] (0,0) -- (4,0);
	\foreach \i in {1,2} {
	\node at (\i,-1.2) {$0$};
	}
	\foreach \i in {1,2,3} {
	\node at (\i,1.2) {$a_{\i}$};
	}
	\foreach \i in {1,2,3} {
	\node at (\i,-1.2) {$0$};
	}
	\node at (4.2,0) {$0$};
	\node at (-0.2,0) {$2$};
	\node at (-1.5,0) {$x_1$};
	\draw[thick,->] (-1.25,0) -- (-0.5,0);
	\foreach \i in {1,2,3} {
	\draw[thick,->] (\i,-2.25) -- (\i,-1.5);
	}
	\foreach \i in {1,2,3} {
	\node at (\i,-2.5) {$z_{\i}$};}
	\begin{scope}[shift = {(6,0)}]
	\end{scope}
	\end{tikzpicture} \hspace{2mm} \times \hspace{2mm}
	\begin{tikzpicture}[scale = 0.7,every node/.style={scale=0.7},baseline = {(0,0)}]
	\foreach \i in {1,2,3} {
	\draw[lgray,line width=1.5pt,->] (\i,-1) -- (\i,1);
	}
	\draw[lgray,line width=1.5pt,->] (0,0) -- (4,0);
	\foreach \i in {1,2} {
	\node at (\i,-1.2) {$0$};
	}
	\foreach \i in {1,2,3} {
	\node at (\i,1.2) {$b_{\i}$};
	}
	\foreach \i in {1,2,3} {
	\node at (\i,-1.2) {$0$};
	}
	\node at (4.2,0) {$0$};
	\node at (-0.2,0) {$2$};
	\node at (-1.5,0) {$x_2$};
	\draw[thick,->] (-1.25,0) -- (-0.5,0);
	\foreach \i in {1,2,3} {
	\draw[thick,->] (\i,-2.25) -- (\i,-1.5);
	}
	\foreach \i in {1,2,3} {
	\node at (\i,-2.5) {$z_{\i}$};}
	\begin{scope}[shift = {(6,0)}]
	\end{scope}
	\end{tikzpicture}
		\end{equation}
	for any $a_1,a_2,a_3,b_1,b_2,b_3$ such that $a_1+a_2+a_3 = 2$ and $b_1+b_2+b_3 = 2$. The scattering factors on the right hand side of \eqref{eq:=example--(1,2,1)} have the following pictorial presentations: 
	\[
	\Delta_{(1,0,1),(0,2,0)} \mapsto \begin{tabular}{c||c|c|c}
    \textit{U} & 1 & & 1 \\ \hline
    \textit{V} & & 2 & 
\end{tabular}, \quad \text{ and } \quad 
\Delta_{(1,1,0),(0,1,1)} \mapsto \begin{tabular}{c||c|c|c}
    \textit{U} & 1 & 1 &  \\ \hline
    \textit{V} &  & 1 & 1 \\ \hdashline[1pt/1pt]
    &  & 2 & 
\end{tabular}.
	\]
	Using the property that $\Delta_{U,V}(x,y;\bm{z}) = \Delta_{V,U}(y,x;\bm{z})$ and the explicit expressions \eqref{eq:IK-delta1} and \eqref{eq:IK-delta2}, we obtain the following expansion: 
	\begin{align*}
	F_{(1,2,1)}(x_1,x_2;\bm{z}) &= \frac{(q^2x_2-x_1)(x_2-q^2x_1)}{(x_2-x_1)^2} F_{(1,0,1)}(x_1;\bm{z})F_{(0,2,0)}(x_2;\bm{z})\\
	&+ \frac{(x_2-q^2x_1)(x_1x_2+q^2x_1x_2-q^2x_2z_2-q^3x_1z_2)}{q(1-q)(x_2-x_1)^2z_2} F_{(1,1,0)}(x_1;\bm{z})F_{(0,1,1)}(x_2;\bm{z})\\
	&+ \frac{(x_1-q^2x_2)(x_1x_2+q^2x_1x_2-q^2x_1z_2-q^3x_2z_2)}{q(1-q)(x_1-x_2)^2z_2} F_{(0,1,1)}(x_1;\bm{z})F_{(1,1,0)}(x_2;\bm{z})\\
	&+\frac{(q^2x_1-x_2)(x_1-q^2x_2)}{(x_1-x_2)^2} F_{(0,2,0)}(x_1;\bm{z})F_{(1,0,1)}(x_2;\bm{z}),
	\end{align*}
	where the one-row partition functions are depicted in equation \eqref{eq:=onerowpart--(1,2,1)}.

\section{The string \texorpdfstring{$S=(1,1,1,1)$}{Lg}}
    
	The set of $2$-permutations with profile $(1,1,1,1)$ is given by
	\begin{multline}
    \label{eq:=twoperm--(1,1,1,1)}
	\mathfrak{M}_{2}(2,(1,1,1,1)) = \\
	 \bigg\{\begin{pmatrix}
	1 & 1 & 0 & 0\\
	0 & 0 & 1 & 1
	\end{pmatrix},\begin{pmatrix}
	1 & 0 & 1 & 0\\
	0 & 1 & 0 & 1
	\end{pmatrix},\begin{pmatrix}
	1 & 0 & 0 & 1\\
	0 & 1 & 1 & 0
	\end{pmatrix},\begin{pmatrix}
	0 & 1 & 1 & 0\\
	1 & 0 & 0 & 1
	\end{pmatrix},\begin{pmatrix}
	0 & 1 & 0 & 1\\
	1 & 0 & 1 & 0
	\end{pmatrix},\begin{pmatrix}
	0 & 0 & 1 & 1\\
	1 & 1 & 0 & 0
	\end{pmatrix} \bigg\}.
	\end{multline}
	The formula \eqref{symformula1} produces the following sum over the elements of the set \eqref{eq:=twoperm--(1,1,1,1)}:
	\begin{align}
    \label{eq:=example--(1,1,1,1)}
	F_{(1,1,1,1)}(x_1,x_2;\bm{z}) &= \Delta_{(1,1,0,0),(0,0,1,1)}(x_1,x_2;\bm{z})F_{(1,1,0,0)}(x_1;\bm{z})F_{(0,0,1,1)}(x_2;\bm{z})\\
    \nonumber
	&+\Delta_{(1,0,1,0),(0,1,0,1)}(x_1,x_2;\bm{z})F_{(1,0,1,0)}(x_1;\bm{z})F_{(0,1,0,1)}(x_2;\bm{z})\\
    \nonumber
	&+\Delta_{(1,0,0,1),(0,1,1,0)}(x_1,x_2;\bm{z})F_{(1,0,0,1)}(x_1;\bm{z})F_{(0,1,1,0)}(x_2;\bm{z})\\
    \nonumber
	&+\Delta_{(0,1,1,0),(1,0,0,1)}(x_1,x_2;\bm{z})F_{(0,1,1,0)}(x_1;\bm{z})F_{(1,0,0,1)}(x_2;\bm{z})\\
    \nonumber
	&+\Delta_{(0,1,0,1),(1,0,1,0)}(x_1,x_2;\bm{z})F_{(0,1,0,1)}(x_1;\bm{z})F_{(1,0,1,0)}(x_2;\bm{z})\\
    \nonumber
	&+\Delta_{(0,0,1,1),(1,1,0,0)}(x_1,x_2;\bm{z})F_{(0,0,1,1)}(x_1;\bm{z})F_{(1,1,0,0)}(x_2;\bm{z}).
	\end{align}
	The one-variable partition functions that appear in \eqref{eq:=example--(1,1,1,1)} are given by
\begin{align}
\label{eq:=onerowpart--(1,1,1,1)}
	F_{(a_1,a_2,a_3,a_4)}(x_1;\bm{z}) \times F_{(b_1,b_2,b_3,b_4)}(x_2;\bm{z}) =
\begin{tikzpicture}[scale = 0.7,every node/.style={scale=0.7},baseline = {(0,0)}]
	\foreach \i in {1,2,3,4} {
	\draw[lgray,line width=1.5pt,->] (\i,-1) -- (\i,1);
	}
	\draw[lgray,line width=1.5pt,->] (0,0) -- (5,0);
	\foreach \i in {1,2,3,4} {
	\node at (\i,-1.2) {$0$};
	}
	\foreach \i in {1,2,3,4} {
	\node at (\i,1.2) {$a_{\i}$};
	}
	\node at (5.2,0) {$0$};
	\node at (-0.2,0) {$2$};
	\node at (-1.5,0) {$x_1$};
	\draw[thick,->] (-1.25,0) -- (-0.5,0);
	\foreach \i in {1,2,3,4} {
	\draw[thick,->] (\i,-2.25) -- (\i,-1.5);
	}
	\foreach \i in {1,2,3,4} {
	\node at (\i,-2.5) {$z_{\i}$};}
	\end{tikzpicture} \hspace{2mm} \times \hspace{2mm}
	\begin{tikzpicture}[scale = 0.7,every node/.style={scale=0.7},baseline = {(0,0)}]
	\foreach \i in {1,2,3,4} {
	\draw[lgray,line width=1.5pt,->] (\i,-1) -- (\i,1);
	}
	\draw[lgray,line width=1.5pt,->] (0,0) -- (5,0);
	\foreach \i in {1,2,3,4} {
	\node at (\i,-1.2) {$0$};
	}
	\foreach \i in {1,2,3,4} {
	\node at (\i,1.2) {$b_{\i}$};
	}
	\node at (5.2,0) {$0$};
	\node at (-0.2,0) {$2$};
	\node at (-1.5,0) {$x_2$};
	\draw[thick,->] (-1.25,0) -- (-0.5,0);
	\foreach \i in {1,2,3,4} {
	\draw[thick,->] (\i,-2.25) -- (\i,-1.5);
	}
	\foreach \i in {1,2,3,4} {
	\node at (\i,-2.5) {$z_{\i}$};}
	\end{tikzpicture}
\end{align}
	for any $a_1,a_2,a_3,a_4,b_1,b_2,b_3,b_4$ such that $a_1+a_2+a_3+a_4 = 2$ and $b_1+b_2+b_3+b_4 = 2$. The scattering factors on the right hand side of \eqref{eq:=example--(1,1,1,1)} are associated with the following diagrams:
	\[
\Delta_{(1,1,0,0),(0,0,1,1)} \mapsto \begin{tabular}{c||c|c|c|c}
    \textit{U} & 1 & 1 & & \\ \hline
    \textit{V} &  &  & 1 & 1  
\end{tabular}, 
\quad 
\Delta_{(1,0,1,0),(0,1,0,1)} \mapsto \begin{tabular}{c||c|c|c|c}
    \textit{U} & 1 &  & 1 & \\ \hline
    \textit{V} &  & 1 &  & 1  
\end{tabular},
	\]
	\\
	\[ 
\Delta_{(1,0,0,1),(0,1,1,0)} \mapsto \begin{tabular}{c||c|c|c|c}
    \textit{U} & 1 &  & & 1\\ \hline
    \textit{V} &  & 1 & 1 &   
\end{tabular}.
\]
In the same vein as the previous examples, the remaining scattering factors can be determined using the symmetry property \eqref{eq:delta-sym}. Using the explicit expressions in \eqref{eq:IK-delta1}, we obtain the following expansion:
\begin{align*}
\label{eq:=sym--(1,1,1,1)}
F_{(1,1,1,1)} &= \frac{(x_2-q^2x_1)(x_2-q^3x_1)}{(x_2-x_1)(x_2-qx_1)} F_{(1,1,0,0)}(x_1;\bm{z})F_{(0,0,1,1)}(x_2;\bm{z})\\
  &+\frac{(x_2-q^2x_1)^2(x_1-qx_2)}{(x_2-x_1)^2(x_2-qx_1)}F_{(1,0,1,0)}(x_1;\bm{z})F_{(0,1,0,1)}(x_2;\bm{z})\\
  &+\frac{(x_2-q^2x_1)(q^2x_2-x_1)}{(x_2-x_1)^2} F_{(1,0,0,1)}(x_1;\bm{z})F_{(0,1,1,0)}(x_2;\bm{z})\\
  &+ \frac{(x_1-q^2x_2)(q^2x_1-x_2)}{(x_1-x_2)^2}F_{(0,1,1,0)}(x_1;\bm{z})F_{(1,0,0,1)}(x_2;\bm{z})\\
  &+\frac{(x_1-q^2x_2)^2(x_2-qx_1)}{(x_1-x_2)^2(x_1-qx_2)} F_{(0,1,0,1)}(x_1;\bm{z})F_{(1,0,1,0)}(x_2;\bm{z})\\
 &+\frac{(x_1-q^2x_2)(x_1-q^3x_2)}{(x_1-x_2)(x_1-qx_2)}F_{(0,0,1,1)}(x_1;\bm{z})F_{(1,1,0,0)}(x_2;\bm{z}),
\end{align*}
where the one-row partition functions are depicted in equation \eqref{eq:=onerowpart--(1,1,1,1)}. Note that in this example (and in all examples in which the elements of $S$ are at most $1$), one observes neatly factorized scattering coefficients.

\bibliographystyle{alpha}
\bibliography{references}

\newpage

\end{document}